\documentclass[10pt]{amsart}

\usepackage{fullpage}
\usepackage{appendix}

\usepackage[utf8]{inputenc}
\usepackage{enumerate}
\usepackage{amssymb,amsfonts,amsmath,amsthm,mathtools}
\usepackage[colorlinks=true,linkcolor=red,citecolor=blue]{hyperref}
\usepackage{appendix}
\usepackage{graphicx,epsfig,color,axodraw4j}
\usepackage{tikz-cd}

\newtheorem{thm}{Theorem}[section]
\newtheorem{cor}[thm]{Corollary}
\newtheorem{lem}[thm]{Lemma}
\newtheorem{prop}[thm]{Proposition}
\theoremstyle{definition}
\newtheorem{defn}[thm]{Definition}
\newtheorem{conj}{Conjecture} 
\theoremstyle{remark}
\newtheorem{ex}[thm]{Example}
\newtheorem{rem}[thm]{Remark}

\numberwithin{equation}{section}

\newcommand{\defeq}{\coloneqq}

\mathcode`\;="303B

\newcommand{\CC}{\mathbb C}
\renewcommand{\H}{\mathfrak{H}}
\newcommand{\PP}{\mathbb P}
\newcommand{\QQ}{\mathbb Q}
\newcommand{\RR}{\mathbb R}
\newcommand{\ZZ}{\mathbb Z}

\newcommand{\To}{\longrightarrow}

\newcommand{\alg}{\mathrm{alg}}
\newcommand{\an}{\mathrm{an}}
\newcommand{\B}{\mathrm{B}}
\newcommand{\comp}{\mathrm{comp}}
\newcommand{\cusp}{\mathrm{cusp}}
\newcommand{\dR}{\mathrm{dR}}
\newcommand{\id}{\mathrm{id}}

\newcommand{\cc}{C}

\newcommand{\GM}{\mathcal{G}}

\newtheorem{example}[thm]{Example}

\DeclareMathOperator{\coker}{coker}
\DeclareMathOperator{\DM}{\mathbf{DM}}
\DeclareMathOperator{\DMT}{\mathbf{DMT}}

\DeclareMathOperator{\End}{End}
\DeclareMathOperator{\Gal}{Gal}
\DeclareMathOperator{\Hom}{Hom}
\let\Im\relax 
\DeclareMathOperator{\im}{im}
\DeclareMathOperator{\Im}{Im}
\DeclareMathOperator{\MT}{\mathbf{MT}}

\DeclareMathOperator{\Res}{Res}
\DeclareMathOperator{\SL}{SL}
\DeclareMathOperator{\GL}{GL}
\DeclareMathOperator{\Spec}{Spec}
\DeclareMathOperator{\Sym}{Sym}

\DeclareMathOperator{\eis}{eis}
\DeclareMathOperator{\node}{node}

\title{Single-valued periods of meromorphic modular forms \\
and a motivic interpretation of the Gross-Zagier conjecture}
\author{Francis Brown, Tiago J. Fonseca}

\date{6th August 2025}

\address{University of Oxford, Radcliffe Observatory, Andrew Wiles Building, Woodstock Rd, Oxford, United Kingdom, OX2 6GG}

\address{IMECC - UNICAMP, Departamento de matemática, Rua Sérgio Buarque de Holanda, 651, CEP 13083-859, Campinas, SP, Brazil}

\begin{document}

\begin{abstract}
A well-known conjecture of Gross and Zagier states that the values of the higher automorphic Green's function at pairs of points with complex multiplication in the upper half-plane are proportional to the logarithm of an algebraic number. It was recently settled in the case of congruence subgroups of the form $\Gamma_0(N)$ by analytic methods.  In this paper we provide a geometric and motivic interpretation of the general conjecture, and show that it is a consequence of a standard conjecture in the theory of motives.  In addition, we define a new class of  matrix-valued higher Green's functions for  both odd and even weight modular forms, and show that they are single-valued periods of a motive constructed from a suitable moduli stack of elliptic curves with marked points. The motive has the structure of a biextension involving symmetric powers of the motives of elliptic curves.  This suggests a very general extension of the Gross-Zagier conjecture relating values of matrix-valued higher Green's functions at  points  which do not necessarily have complex multiplication to special values of $L$-functions.   In particular, our motivic interpretation  of the Gross-Zagier log-algebraicity conjecture enables us to  give a completely geometric proof  in  level 1 and weight 4 by showing that the motive of the  moduli stack $\mathcal{M}_{1,3}$ of  elliptic curves with 3 marked points  is mixed Tate.

In the course of this paper we develop many new foundational results on: the theory of weak harmonic lifts,  meromorphic modular forms, biextensions of modular motives, and  their corresponding algebraic de Rham cohomology and single-valued periods, which may all be  of independent interest. 
\end{abstract}

\maketitle 

\setcounter{tocdepth}{1}
\tableofcontents

\section{Introduction}
Let $(z,w) \in \mathfrak{H} \times \mathfrak{H}$, where $\mathfrak{H}=\{z \in \mathbb{C}: \mathrm{Im}(z)>0\}$ is the upper half-plane, and consider the Green's function, defined for all  $(z,w)$  which do not lie in the same $\SL_2(\mathbb{Z})$-orbit, by 
\begin{equation} \label{intro: G2}
    G_1(z,w) = \log |j(z) - j(w)|^2,
\end{equation}
where $j: \mathfrak{H} \rightarrow \CC$ denotes the $j$-invariant. If $z,w$ correspond to elliptic curves with complex multiplication, then $j(z), j(w)$ are algebraic, and $G_1(z,w)$ is  the logarithm of an algebraic number.  

Let $\mathrm{Re}(s)>1$ and $\Gamma \leq \SL_2(\ZZ)$ be a subgroup of finite index. The \emph{higher Green's function} $G_{\Gamma,s}$ \cite{Roelcke1, Roelcke2, Hejhal}, also known as the resolvent kernel for the hyperbolic Laplacian, is the real-valued function defined on the open subset $\mathfrak{U} = \mathfrak{H} \times \mathfrak{H} \setminus \{(z,w) : z \in \Gamma w\}$ by the sum
\begin{equation}\label{eq:classical-GF}
    G_{\Gamma,s}(z,w) = -2 \sum_{\gamma \in \Gamma}  Q_{s-1} \left( 1+  \frac{|z-\gamma w|^2}{2 \, \mathrm{Im}(z)\, \mathrm{Im}(\gamma\, w)}\right) . 
\end{equation}
In this equation, $Q_{s-1}$ is the Legendre function of the 2nd kind, which satisfies Legendre's second-order differential equation and can be expressed using hypergeometric series. The argument of $Q_{s-1}$ can be interpreted as the hyperbolic cosine of the hyperbolic distance between $z$ and $\gamma w$ in $\mathfrak{H}$.

The function $G_{\Gamma,s}$ can be characterised by the following properties:
\begin{enumerate}
    \item It is $\Gamma\times \Gamma$-invariant: $G_{\Gamma,s}(\gamma_1 z, \gamma_2 w) = G_{\Gamma,s}(z,w)$  for all $\gamma_1,\gamma_2 \in \Gamma$,
    \item It satisfies  $\Delta_z G_{\Gamma,s}(z,w) = s (s-1) G_{\Gamma,s}(z,w)$ and is therefore an eigenfunction of the  hyperbolic Laplacian $\Delta_z$ in the variable $z$, and similarly with respect to the variable $w$, 
    \item  It defines a real-analytic function  on $\mathfrak{U}$ which tends to zero at cusps, and has a logarithmic singularity along the diagonal $z=w$. 
\end{enumerate}
For the most part, we shall take $\Gamma$ to be a congruence subgroup, but many of our constructions work equally well for non-congruence subgroups.

Let $m>0$ be an integer. Gross and Zagier conjectured \cite{GrossZagier, GrossKohnenZagier, ZagierICM}  that, after applying a suitable finite linear combination of Hecke operators $T=\sum_{n\ge 1} \lambda_n T_n$, for $z,w$ any two points of complex multiplication,
the corresponding combinations 
\[
    TG_{\Gamma,m+1}(z,w) = \sum_{n\ge 1} \lambda_n  T_nG_{\Gamma,m+1}(z, w)
\]
are proportional to logarithms of algebraic numbers. We call this the \emph{Gross-Zagier conjecture in level $\Gamma$ and weight $2m+2$} since it relates to modular forms of this weight for the group $\Gamma$.  This conjecture has attracted considerable interest and led to remarkable developments by both geometric \cite{ZhangHeights, Mellit,sreekantan2022,sreekantan2025} and analytic  \cite{Viazovska, ViazovskaPetersson, Li, BruinierLiYang} means. The analytic approach has proven very successful and  recently led to  a complete proof, due to Li, and Bruinier-Li-Yang, in the general case of the group $\Gamma = \Gamma_0(N)$. See also \cite{ZhouI, ZhouII} for a rather different perspective using explicit identities between periods.

It is a general theme in arithmetic and algebraic geometry that values of logarithms at algebraic arguments should have a geometric, or motivic, explanation. Indeed, the single-valued logarithm \eqref{intro: G2} has a  well-understood  interpretation in algebraic geometry as the regulator of a family of  Kummer extensions. The values of higher Green's functions at CM points similarly cry out for an arithmetic interpretation. This has remained elusive since the definition of $G_{\Gamma,s}$ 
is highly analytic and does not seem to be related to algebraic geometry or the theory of motives in any obvious way.
The hope that such an approach to the conjecture might exist is already alluded to in the very first papers in the subject (see \cite{GrossZagier},  page 315).

The goals of this paper are threefold: (i) we provide a geometric interpretation of the higher Green's functions and their special values at CM points as the single-valued periods of a certain biextension of mixed Hodge structures. In so doing, we (ii) generalise the higher Green's functions in several different directions and show that the Gross-Zagier conjecture (and generalisations thereof) are special cases of Beilinson's conjectures on special values of $L$-functions. Finally, (iii) we provide a motivic interpretation of the Gross-Zagier conjecture by constructing the corresponding motive in a triangulated category.  As a consequence of (iii), we show that the Gross-Zagier conjecture follows in great generality if one admits one or other of the `standard' conjectures for the motives of cusp forms.  As a proof of principle, we are able to prove the conjecture completely in level one and weight 4, \emph{i.e.}, for the function $G_{\SL_2(\mathbb{Z}),2}(z,w)$ for arbitrary CM points $z,w$, by entirely geometric methods using the modern theory of motives. 

Our approach bears some similiarity to ideas in Mellit's thesis \cite{Mellit} (\emph{cf.} \cite{sreekantan2022,sreekantan2025}), but we do not know how it relates to the analytic approach to the Gross-Zagier conjecture  mentioned above. 

\subsection{Results}\label{par:results}

Our first step is to introduce \emph{matrix-valued generalisations} of the higher Green's functions.
They are $\Gamma \times \Gamma$-invariant functions  $\GM_{\Gamma,k}: \mathfrak{U} \rightarrow M_{k+1 \times k+1}(\RR)$  which take values in $(k+1)\times(k+1)$ matrices, where  $\mathfrak{U}\subset \mathfrak{H}\times \mathfrak{H}$  was defined previously and $k\geq 1$ is any integer.  They may be defined in an elementary and completely explicit way via their generating series, which are obtained in very simple way by truncating the Taylor expansion of the logarithm function and taking complex conjugates (see Definition \ref{defn: MatrixGreenfunctionsGeneratingFunction}). This definition makes no reference to Legendre functions. 
An important point of this paper is that the matrix-valued Green's functions  $\GM_{\Gamma,k}(z,w)$ have  considerably more structure and better algebraic and analytic properties than the functions $G_{\Gamma,m+1}(z,w)$ considered in isolation.

When $k=2m$ is even, the matrix $\GM_{\Gamma,k}(z,w)$ has odd rank and hence has a distinguished central entry. We prove that this entry is proportional to the higher Green's function $G_{\Gamma,m+1}$ considered above. The remaining entries of the matrix $\GM_{\Gamma,k}$ are related to it by applying Maass raising and lowering operators. In the case when $k=2m+1$ is odd, the entries of the matrix-valued Green's functions do not appear to have been studied previously, and suggest  generalisations of the Gross-Zagier conjecture in  completely new directions (see \S \ref{sect: BeyondGZ}).  Some explicit examples of matrix-valued higher Green's functions can be found in  Appendix \ref{sect: AppendixC}.

Our first main theorem provides a geometric interpretation of $\GM_{\Gamma,k}$.

\begin{thm} \label{intro: thm1}
    Let $\mathcal{Y}_{\Gamma}$ denote the  modular curve, viewed as a  smooth stack, whose analytification is the orbifold $\mathcal{Y}_{\Gamma}(\mathbb{C}) = \Gamma \backslash \!\! \backslash \mathfrak{H}$. Let $\pi: \mathcal{E} \rightarrow \mathcal{Y}_{\Gamma}$ denote the universal elliptic curve and let $\mathcal{V}_k = \mathrm{Sym}^k R^1 \pi_* \QQ_{\mathcal{E}}$, where $\QQ_{\mathcal{E}}$ is the constant sheaf on  $\mathcal{E}$ with stalks $\QQ$. Then, for all but finitely many points $z,w \in \mathcal{Y}_{\Gamma}(\CC) $ (namely, $z,w$ are not elliptic if $k$ is even), the matrix $\GM_{\Gamma,k}(z,w)$  represents a certain block of the  single-valued period matrix
     of the  de Rham realisation of the mixed modular object
    \begin{equation}  \label{intro: modularmotive}
        M_{z,w}= H_{\mathrm{cusp}}^1\left( \mathcal{Y}_{\Gamma} \setminus \{ w\} , \{ z \} ; \mathcal{V}_k\right) 
    \end{equation} in  an explicitly-given  basis. 
   (In the case when $z,w$ are elliptic, the rank of  \eqref{intro: modularmotive} can drop, but a variant of this  statement still holds for a certain submatrix of $\mathcal{G}_{\Gamma,k}(z,w)$.) 
\end{thm}

Recall that the cuspidal (or parabolic) cohomology $H^1_{\cusp}(\mathcal{Y}_{\Gamma};\mathcal{V}_k)$ corresponds to (two copies of) the space of cusp forms for $\Gamma$ of weight $k+2$; the space $H^1_{\cusp}(\mathcal{Y}_{\Gamma} \setminus \{w\},\{z\};\mathcal{V}_k)$ is described by meromorphic cusp forms with poles along $w$ and prescribed values at $z$.

In the previous theorem, we view $M_{z,w}$ as  an object in a category of `realisations' consisting of compatible Betti and de Rham cohomology groups.  As we explain below, $M_{z,w}$ is the realisation of an object in a triangulated category of motives. Given such an object, the single-valued period map is the involution on algebraic de Rham cohomology obtained by  transporting the action of complex conjugation on  Betti cohomology  via the comparison isomorphism, and can be viewed as a $p$-adic period  for the infinite prime.   

\begin{rem}
    The  higher Green's functions  naturally fit  into the general framework of single-valued (families of) periods.
    This class of periods and families of periods contains many functions of importance in mathematics and physics, and include: the single-valued polylogarithms (for mixed Tate motives);  Néron-Tate heights (for motives of curves) \cite{brown-dupont}; Petersson inner products,  weak harmonic lifts of modular forms, and Fourier coefficients of Poincar\'e series (for motives of modular forms) \cite{CNHMF1, CNHMF3, CoeffPoincare} (\emph{cf.} \cite{bruininer-ono-rhoades,bringmann-ono}).  A natural next step in this `hierarchy' of periods are the single-valued periods associated to mixed modular motives, of which the higher Green's functions, by the previous theorem, are a simple example. 
\end{rem}

The object \eqref{intro: modularmotive} has the structure of a biextension (see \S\ref{par:intro-biextensions} below). We show that applying the linear combination of Hecke operators $T= \sum_{n\ge 1} \lambda_n T_n$, as in the Gross-Zagier conjecture, cuts out of this biextension a simple extension of the form
\begin{equation} \label{intro: extensionMTzw} 
    \begin{tikzcd}[column sep = small]
        0 \arrow{r} & \big(\Sym^{k} H^1(\mathcal{E}_{z})\big)^{\Gamma_z}  \arrow{r}& M^T_{z,w} \arrow{r} & \big(\Sym^k H^1(\mathcal{E}_{w})\big)^{\Gamma_w}(-1) \arrow{r}& 0 
    \end{tikzcd}
\end{equation}
where $\Gamma_z$ denotes the stabliser of $z$ in $\Gamma$ (which is contained in $\{\pm 1\}$ if $z$ is non-elliptic). 
This is a natural generalisation  of the relative cohomology object $\mathcal{K}_x=H^1( \PP^1 \setminus \{0,\infty\}, \{1,x\})$, or Kummer extension, which sits in a short exact sequence  
\begin{equation} \label{intro: Kummer}
    \begin{tikzcd}[column sep = small]
        0 \arrow{r}& \QQ(0) \arrow{r}& \mathcal{K}_x \arrow{r}& \QQ(-1) \arrow{r}& 0  \ . 
    \end{tikzcd}
\end{equation}
Its single-valued period is $\log |x|^2$, and in fact the object  \eqref{intro: modularmotive} reduces precisely to a Kummer extension  in the case when $k=0$ and $ \mathcal{Y}_{\Gamma}$ has genus zero.  

\begin{thm} \label{intro: thm2}
    Suppose that $z,w$ are CM points and $k=2m$ is even. By using the action of Hecke operators and the  action of complex multiplication on $\mathcal{E}_z$ and $\mathcal{E}_w$, we can isolate a subquotient of $M^T_{z,w}$  of the form 
    \begin{equation} \label{intro: GZextension} 
        \begin{tikzcd}[column sep = small]
            0 \arrow{r}& \QQ(-m) \arrow{r}& \mathcal{GZ}^{T}_{z,w} \arrow{r}& \QQ(-m-1) \arrow{r}& 0
        \end{tikzcd}
    \end{equation}
    in a suitable abelian category of realisations.  Its single-valued period is proportional to the value of the  higher Green's function at $(z,w)$. 
\end{thm}

Every extension of $\QQ(-m-1)$ by $\QQ(-m)$ in a category of realisations which arises as the cohomology of an algebraic variety is expected to be isomorphic to a Tate twist of a  Kummer extension  of the form \eqref{intro: Kummer}. In particular, every Gross-Zagier extension \eqref{intro: GZextension} is expected to be Kummer, which implies the conjecture of Gross and Zagier (see Theorem \ref{thm: GZ-conditional-proof}).   To make this precise we prove the following theorem. 

\begin{thm}
    Let $\Gamma=\SL_2(\ZZ)$ and let $\overline{\mathcal{M}}_{1,n}$ denote the Deligne-Mumford compactification of the moduli stack of genus one curves with $n$ marked points. Assume that $z,w \in \mathcal{M}_{1,1}(K)$ correspond to isomorphism classes of elliptic curves defined over a number field $K$.  Then the realisation of the motive
    \begin{equation} \label{intro: themotive}
        M(\overline{\mathcal{M}}_{1,n}\times_{\QQ}  K  \setminus \pi_n^{-1}(w), \pi_n^{-1}(z))_{\varepsilon} [n]
    \end{equation}
    in the triangulated category of motives $\DM(K)$ over $K$ is concentrated in degree $0$, and coincides with $M_{z,w}$ of \eqref{intro: modularmotive}. It is a biextension in  $\DM(K)$. Here, $\varepsilon$ denotes the alternating part with respect to the action of the symmetric group which permutes the marked points of $\mathcal{M}_{1,n}$, and $\pi_n$ denotes the  natural map $\overline{\mathcal{M}}_{1,n}\rightarrow \overline{\mathcal{M}}_{1,1}$. 
\end{thm}

The previous theorem uses the construction \cite{ConsaniFaber} of the motive of cusp forms in level 1, and may easily be extended to higher levels using \cite{Petersen}. An alternative approach would be to use the motives of Kuga-Sato varieties as in \cite{Scholl} or the six-functor formalism on triangulated categories of motivic sheaves \cite{AyoubICM}.  
The previous results which concern realisations can thus be promoted to a statement about  motives. The following theorem, in particular,  enables us to pinpoint the key content of the Gross-Zagier conjecture.

\begin{thm}
    Let $M^n_{\cusp}= M(\overline{\mathcal{M}}_{1,n})_{\varepsilon}[n]$ denote the motive associated to the space of cusp forms for $\Gamma= \SL_2(\mathbb{Z})$ of weight $n+1$, whose realisation is $H^1_{\cusp}(\mathcal{Y}_{\Gamma};\mathcal{V}_{n-1})$.  If the Hecke correspondence $T$ acts via zero on $M^n_{\cusp}$ then the extension \eqref{intro: extensionMTzw} is the realisation of an extension in $\DM(K)$. Furthermore,  if    $z,w$ have complex multiplication defined over $K$, then \eqref{intro: GZextension} is a Kummer extension in $\DM(K)$ and the Gross-Zagier conjecture holds. 
\end{thm}

As mentioned above, a version of this theorem also holds in higher level with no extra difficulty. 
The point is that the action of Hecke operators (and that of complex multiplication) is motivic and one may use the fact that,   unconditionally, every extension of $\mathbb{Q}(-1)$ by $\mathbb{Q}(0)$ in $\DM(K)$ is  a Kummer extension, or equivalently, $\Hom_{\DM(K)}(\QQ(-1),\QQ(0)[1])\cong K^{\times} \otimes \mathbb{Q}$. It follows that the  only missing part of the picture,  for a completely geometric proof the Gross-Zagier conjecture from start to finish, is the following the conjecture: 

\begin{conj} \label{conjectureTVanishing}
    If $T$ vanishes on the space of weight $n+1$ cusp forms, then it vanishes on $M^n_{\cusp}$. 
\end{conj}

In other words, we prove that the Gross-Zagier conjecture (in level 1) follows from Conjecture  \ref{conjectureTVanishing}. The case of higher level is similar (see \S \ref{sect: HigherLevels}). Conjecture \ref{conjectureTVanishing}  is a special case of a  standard conjecture in the theory of motives. Interestingly, it 
does not depend in any way on the points $z,w$, and  hence universal in each weight and level. In the case when there are no cusp forms of weight $n+1$, it is enough to show that the motive $M^n_{\cusp}$  vanishes, \emph{i.e.}, it is not  a `phantom motive' which is the name given to a  non-trivial motive which has trivial Betti realisation.  

As a proof of principle, we prove Conjecture \ref{conjectureTVanishing} in the first interesting case by geometric methods. 

\begin{thm}
    The motive of $\mathcal{M}_{1,3}$ is mixed Tate. Consequently, the motive $M^3_{\mathrm{cusp}}$ vanishes, \eqref{intro: GZextension}  is Kummer, and the Gross-Zagier conjecture holds in weight 4 and level 1 for any CM points $z,w$. 
\end{thm}

An unexpected consequence of our analysis is that it points to a family of new conjectures which go far beyond that of Gross and Zagier. Indeed, if we drop the condition that  $z$ or $w$ has complex multiplication, or even that $k$ is even, then Beilinson's conjecture predicts a relation between  the periods of extensions of the form \eqref{intro: extensionMTzw}  and special values of $L$-functions of symmetric powers of elliptic curves and their tensor products.  These relations open up a fascinating  avenue for further study, which  crucially involve the full matrix-valued Green's functions and cannot be formulated if one only considers  the central entries  $G_{\Gamma,m+1}(z,w)$. 

The proofs of the  theorems  set out above  involve  a number of constructions which relate to  different aspects of the theory of modular forms, and we have tried to give a definitive account of the  foundational results which are involved. They may be of independent interest. 
We now discuss each  in turn.

\subsection{Meromorphic modular forms and harmonic lifts}

First of all, consider the following differential forms of degree $0,1$, and $2$ on the complement of the diagonal in $\CC\times \CC$ with coordinates $(x,y)$: 
\begin{equation} \label{intro: mero012forms}
    \begin{tikzcd}
        \log (x-y)  \arrow[mapsto]{r}{d_x} & {\displaystyle \frac{dx}{x-y}} \arrow[mapsto]{r}{d_y} & -{\displaystyle \frac{dx\wedge dy}{(x-y)^2}}    \, ,
    \end{tikzcd}
\end{equation}
each of which is obtained from its predecessor by applying total derivatives with respect to $x$ or $y$.
The logarithm is multi-valued, but has a canonical single-valued version $\log |x-y|^2$ which satisfies 
\begin{equation} \label{intro: diffeqforlog}
    \begin{tikzcd}
        \log |x-y|^2  \arrow[mapsto]{r}{d_x} & {\displaystyle \frac{dx}{x-y} + \frac{d \overline{x}}{\overline{x}-\overline{y}} } \arrow[mapsto]{r}{d_y} &  - {\displaystyle \frac{dx\wedge dy}{(x-y)^2} -\frac{d\overline{x}\wedge d\overline{y}}{(\overline{x}-\overline{y})^2}} \, ,
    \end{tikzcd}
\end{equation}
where $d_x = \frac{\partial}{\partial x} dx +    \frac{\partial}{\partial \overline{x}} d\overline{x}$ is the total derivative with respect to $x, \overline{x}$, and similarly for $y$. The one-form $\frac{dx}{x-y} \in \Omega_{\log}^1( \CC \setminus \{y\})$ is the relative  logarithmic one-form on the family of punctured curves $\mathbb{P}^1(\mathbb{C})\setminus \{y, \infty\}$ of genus 0. The above functions and differential forms generalise to curves of arbitrary genus, a key difference being that logarithmic differential forms with prescribed poles and residues are no longer unique owing to the 
existence of global holomorphic differentials,   and so  the genus $g>0$ analogues of $\log |x-y|^2$ involve choices. 

In this paper we study generalisations  of \eqref{intro: diffeqforlog}  on quotients of the upper half-plane $\mathfrak{H}$ by finite-index subgroups $\Gamma \leq \SL_2(\ZZ)$ in higher weight (\emph{i.e.}, with coefficients in a vector bundle).
Before discussing the general case, consider the  analogues of \eqref{intro: mero012forms} on  the upper half-plane in the case $\Gamma = \SL_2(\ZZ)$:  they are simply the  pullback  of  \eqref{intro: mero012forms}  by the modular $j$-invariant $j:\mathfrak{H} \rightarrow \CC$ 
\begin{equation} \label{intro: mero012formsH}
    \begin{tikzcd}
        \log ( j(z) -j(w))  \arrow[mapsto]{r}{d_z} & {\displaystyle \Psi(z,w)\,  dz} \arrow[mapsto]{r}{d_w} & {\displaystyle \Phi(z,w)\,  dw \wedge dz} \, ,
    \end{tikzcd}
\end{equation}
where 
\[
    \Psi(z,w) =  \frac{j'(z)}{j(z)-j(w)}   \quad \hbox{ and } \quad  \Phi(z,w) =  \frac{j'(z)j'(w)}{(j(z)-j(w))^2} \ . 
\]
The function $\Psi(z,w)$  is a  meromorphic modular form of weight two  in $z$ with simple poles along the $\SL_2(\ZZ)$-orbit of $w$.  
The Green's function \eqref{intro: G2} is the  single-valued version of $\log(j(z)-j(w))$ and, if $w$ is fixed,  satisfies the differential equation (recall $d_z = \frac{\partial}{\partial z} dz +    \frac{\partial}{\partial \overline{z}} d\overline{z}$)
\[
  d_z G_1(z,w) = \Psi(z,w)dz + \overline{\Psi(z,w)}d\overline{z}\ . 
\]
It may thus be interpreted as a `harmonic lift' of $\Psi(z,w)$, as discussed  below. The function $\Phi$ and its generalisations play a minor role in this paper but are analogues of reproducing kernels (see below also).

\subsubsection{Rational Poincar\'e series and meromorphic modular forms}

The higher weight analogues of \eqref{intro: mero012formsH} are  defined  using Poincar\'e  series which are obtained by averaging a rational function over the action of $\Gamma$ (`s\'eries th\'etafuchsiennes', in Poincaré's own terminology; \emph{cf.} \cite{Poincare}). In the modern literature,  `Poincar\'e series' often  refers to series with  an additional exponential damping factor (also studied in \emph{loc. cit.} (6)), which makes their analytic properties much easier to study. For clarity we shall refer to the former as \emph{rational} Poincar\'e series.

Let $\Gamma \leq \SL_2(\ZZ)$ be a finite-index subgroup. For all $p,q \in \ZZ$ with $p+q>0$ and $z,w $ in either the upper or lower half-planes such that $z\notin \Gamma w \cup  \Gamma \overline{w}$, consider the rational Poincar\'e series
\begin{equation}\label{introPsipiqdef}
    \Psi^{p,q}_{\Gamma}(z,w) =  \sum_{\gamma \in \Gamma} \frac{1}{j^{p+q+2}_{\gamma}(z)} \frac{ w-\overline{w}} { (\gamma z - w)^{p+1} (\gamma z -\overline{w} )^{q+1}       } \ ,\qquad   \hbox{ where }   j_{\begin{psmallmatrix} a& b \\ c& d\end{psmallmatrix}}(z) =cz+d \ . 
\end{equation}
 It  converges normally on compacta of the upper and lower half-planes, and on horoball neighbourhoods of cusps. It therefore defines a function which admits two quite different interpretations according to whether it is viewed as  a meromorphic function of $z$, or a real-analytic function of $w$, respectively. The former point of view will be the most useful in the first instance. For fixed $w$, the functions $z \mapsto \Psi_{\Gamma}^{p,q}(z,w)$ define meromorphic modular forms for $\Gamma$ of weight $p+q+2$ which vanish at cusps, and have poles along $\Gamma w$ (or $\Gamma \overline{w}$, if $\mathrm{Im}(w)$ is negative) of order at most $p+1$ (resp. $q+1$). The functions $\Psi$ appear to have first been introduced in \cite{PeterssonMero}, and further studied in \cite{Petersson2}. See also \cite{BringmannKanevonPippich,LobrichSchwagenscheidt,ZemelHeegner}. 

In this paper we systematically take the point of view of vector-valued modular forms. Consider the $\QQ$-vector space $V_k^{\B}$ of homogeneous polynomials of degree $k$ in two generators $X,Y$ of degree $1$, equipped with the standard right action of $\Gamma$. It defines a local system on the orbifold quotient $\mathcal{Y}_{\Gamma}(\mathbb{C}) = \Gamma \backslash \!\!\backslash  \mathfrak{H}$, which induces a holomorphic vector bundle $\mathcal{V}_k^{\dR}$ of rank $k+1$. The higher weight analogues  of the differential form $d_z\log(j(z)-j(w))$ are the $\Gamma$-invariant 1-forms with values in $V_k^{\B}\otimes \mathbb{C}$:
\begin{equation} \label{intro: Psipqforms}
    \Psi_{\Gamma}^{p,q}(z,w) (X- zY)^{p+q} dz   \qquad \hbox{ where } \quad p+q = k \ \hbox{ and } p,q\geq 0\ , \
\end{equation}
which can also be seen, for fixed $w$, as global section of $\Omega^1\otimes \mathcal{V}_k^{\dR}$ over $\mathcal{Y}_{\Gamma}(\mathbb{C})\setminus \{w\}$.

The (vector-valued) residues of these functions are computed in Corollary \ref{cor:residue-Psi}; from this we deduce that all de Rham cohomology classes on the punctured modular curve may be represented using  the meromorphic modular forms \eqref{intro: Psipqforms} and  weakly holomorphic  modular forms. In some special cases (such as  $\Gamma=\SL_2(\ZZ)$) certain functions $\Psi_{\Gamma}^{p,q}$ can be written explicitly in terms of quotients of Eisenstein series (see Example \ref{ex:psi0k-sl2}).

\begin{rem}
    Sakharova  \cite{Sakharova} has shown, amongst other things, how  to derive the weight 2 case \eqref{intro: mero012formsH} from a similar construction using Hecke's trick.
\end{rem}

\subsubsection{The reproducing kernel}

For $k>2$, and $\Gamma$ as above, consider the rational Poincar\'e series:
\begin{equation} \label{intro: Phidef}
    \Phi_{\Gamma,k}(z,w) = (k-1) \sum_{\gamma \in \Gamma} \frac{1}{j_{\gamma}^k(z)} \frac{1}{(\gamma z- w)^k}\ .
\end{equation}
When $w$ is in the upper half-plane, $\Phi_{\Gamma,k}(z,w)$ is a meromorphic modular form of weight $k$ in  $z$ with poles along $\Gamma w$ of order at most $k$, and vanishing at cusps. It was first introduced by Poincar\'e \cite[page 143]{Poincare}. When $w$ is in the lower half-plane, it defines a cusp form  of weight $k$ in $z$, which is known as the `reproducing kernel'  \cite{Zagier}.  The analogue of the 2-form \eqref{intro: mero012formsH} is the $\Gamma \times \Gamma$-invariant 2-form with values in $(V_k^{\B}\otimes V_k^{\B})\otimes \mathbb{C}$ 
\begin{equation} \label{intro: Phitwoform}
    \Phi_{\Gamma,k}(z,w) (X-zY)^{k-2} (U-wV)^{k-2} dw \wedge dz \ . 
\end{equation}
The first copy of $V^{\B}_k$ has generators $X,Y$ and the second  has generators $U,V$. Geometrically, \eqref{intro: Phitwoform} is a differential form on (an open of) on $\mathcal{Y}_{\Gamma}(\mathbb{C}) \times \mathcal{Y}_{\Gamma}(\mathbb{C})$ with values in the external product vector bundle $\mathcal{V}^{\dR}_k \boxtimes \mathcal{V}^{\dR}_k$.

The relationship between $\Phi_{\Gamma,k}(z,w)$ and  $\Psi^{p,q}_{\Gamma}(z,w)$ comes about by viewing the latter as a real-analytic function of the variable $w$. From this perspective, it is modular in $w$ of weights $(p,q)$ where the first index denotes the holomorphic weight, and the second the anti-holomorphic weight. The generating series
\[
    \Psi_{\Gamma,k}(z,w) =  \sum_{\substack{p+q=k \\ p,q\geq 0}}  \Psi_{\Gamma}^{p,q}(z,w)(U-wV)^p (U-wV)^q
\]
is, for any fixed $z$, a modular invariant real-analytic section of $\mathcal{V}^{\dR}_k$ with respect to $w$ and satisfies the differential equation 
\begin{equation} \label{intro: dPsiisPhi}  
    d_w\Psi_{\Gamma,k}(z,w) =   \Phi_{\Gamma, k+2}(z,w)(U-wV)^k dw  - \overline{ \Phi_{\Gamma,k+2}(\overline{z}, w)}(U-\overline{w}V)^k d\overline{w}   \ .
\end{equation}
This equation implies that  the meromorphic modular forms $\Psi^{p,q}_{\Gamma}$ are  components of an Eichler integral (or harmonic lift) of the reproducing kernel $\Phi_{\Gamma,k+2}$.  

\subsubsection{The matrix-valued Green's functions} 

We have discussed the functions which generalise the middle and right-most terms in \eqref{intro: mero012formsH}; it remains to discuss the analogues of the single-valued logarithm,  which are  precisely the matrix-valued Green's functions. 

For any four indices $p,q,r,s\geq 0$ with $p+q=r+s=k>2$ we define functions (see \S\ref{sect: existenceHigherGreensMatrix})
\[
    {}^{\Gamma}\!G^{p,q}_{r,s} (z,w)
\]
which are real-analytic  on $\mathfrak{U} = \mathfrak{H}\times \mathfrak{H} \setminus \{(z,w):z \not\in \Gamma w\}$ and modular for $\Gamma$ of weights $(p,q)$ in $z$ and $(r,s)$ in $w$, \emph{i.e.}, 
\[
    {}^{\Gamma}\!G^{p,q}_{r,s} (\gamma_1 z, \gamma_2 w) =  \frac{1}{j^p_{\gamma_1}(z) j^q_{\gamma_1}(\overline{z})}\frac{1}{j^r_{\gamma_2}(w) j^s_{\gamma_2}(\overline{w})}    {}^{\Gamma}\!G^{p,q}_{r,s} ( z,  w) \qquad \hbox{ for all } (\gamma_1,\gamma_2) \in \Gamma \times \Gamma\ .
\]
The array of functions ${}^{\Gamma}\!G^{p,q}_{r,s}$ assemble into the $(k+1)\times(k+1)$ matrix-valued Green's function of Theorem \ref{intro: thm1} (see Definition \ref{defn:greens-fct-matrix}). For $k=2m$ even, the central function ${}^{\Gamma}\!G^{m,m}_{m,m}$ is, up to an explicit prefactor, the higher Green's function $G_{\Gamma,m+1}$ \eqref{eq:classical-GF}. All generalised Green's functions ${}^{\Gamma}\!G^{p,q}_{r,s}$ uniquely determine each other by applying Maass raising and lowering operators in $z$ and $w$. 

An equivalent way to interpret the defining  properties is to consider the   vector-valued generating  function 
\[
    \vec{G}_{\Gamma,k}(z,w) = \sum_{\substack{p+q=r+s=k \\ p,q,r,s\geq 0}}   {}^{\Gamma}\!G^{p,q}_{r,s} ( z, w) (X-zY)^r (X-\overline{z}Y)^s (U-wV)^p (U-\overline{w}V)^q
\]
which defines a real-analytic $\Gamma\times \Gamma$-invariant section of $\mathcal{V}^{\dR}_k \boxtimes \mathcal{V}^{\dR}_k$ and satisfies the differential equation 
\begin{equation} \label{intro: dGisPsi}
    d_z \vec{G}_{\Gamma,k}(z,w) = \Psi_{\Gamma,k}(z,w) (U,V)(X-zY)^{k}dz -  \overline{\Psi_{\Gamma,k}(z,\overline{w})} (U,V)(X-\overline{z}Y)^{k}d\overline{z}  
\end{equation}
and a similar equation for $d_w$, with $z,w$ interchanged. The  vector-valued  generating function $\vec{G}_{\Gamma,k}(z,w)$ is nothing other than an Eichler integral, or harmonic lift, of the meromorphic forms \eqref{intro: Psipqforms}. 

\begin{rem}
    We find, for example,  the following higher-weight analogue of $\frac{\partial}{\partial x}\frac{\partial}{\partial y} \log |x-y|^2 =  \frac{1}{(x-y)^2}$:
    \[
        \frac{\partial}{\partial z} \frac{\partial}{\partial w}  \vec{G}_{\Gamma,k}(z,w)= \Phi_{\Gamma,k+2}(z,w) (X-zY)^k(U-wV)^k 
    \]
    which implies that the generalised higher Green's functions are double Eichler integrals (or harmonic lifts with respect to $z$ and $w$) of the basic meromorphic forms $\Phi_{\Gamma,k+2}(z,w)$ given by the simple formula \eqref{intro: Phidef}.
\end{rem} 

The interpretation of the higher Green's functions via these differential equations is an essential step in proving Theorems \ref{intro: thm1} and \ref{intro: thm2}. In conclusion, the higher weight modular analogues of \eqref{intro: mero012forms} and \eqref{intro: mero012formsH} are, respectively the vector-valued higher Green's function $\vec{G}_{\Gamma,k}$, the one-forms \eqref{intro: Psipqforms} and the two-forms \eqref{intro: Phitwoform}. We believe that they are fundamental for the study of punctured modular curves and their periods.   

\subsection{Harmonic lifts and weak Maass forms}

Harmonic lifts of modular forms  have an  extensive literature.  In this paper we take the opportunity to formulate a very general theory of harmonic lifts using cohomological techniques which we apply to the higher Green's functions. In particular, we establish some general  theorems on the existence  of harmonic lifts which are both more precise, and more general than   existing results  which are proven by   ad-hoc methods. They may be of independent interest.  

In Appendix \ref{sec:complex-conjugation}, we study a general notion  of harmonic lifts  (Definition \ref{defn: B.6}) in the context of  vector bundles with integrable connection over a complex manifold. The only extra data required to define harmonic sections, in our sense, is the existence of a sub-bundle (which plays the role of the Hodge sub-bundle), and an anti-linear involution (which plays the role of complex conjugation).  We briefly review what this reduces to  in the case of modular curves  (see Sections \ref{sec: HarmonicLiftsGeneralities} and \ref{sec:cohomology-over-C}; \emph{cf.} \cite{Candelori}). Let $\mathfrak{U}\subset \mathfrak{H}$ be a $\Gamma$-invariant  open subset. Consider a holomorphic function $f: \mathfrak{U} \rightarrow \CC$  which is modular of weight $k+2\geq 2$ for $\Gamma$. A  \emph{harmonic lift} of $f$ is a vector-valued $\Gamma$-equivariant function   $F: \mathfrak{U} \rightarrow V_k^{\B}\otimes \CC$ satisfying
\begin{equation}  \label{into: harmonic}
    dF = f(z) (X-zY)^k dz -  \overline{g(z)}(X- \overline{z} Y)^k d\overline{z}
\end{equation}
for some other holomorphic  function $g: \mathfrak{U} \rightarrow \CC$ which is necessarily modular of weight $k+2$.  We call $g$ a \emph{Betti-conjugate} of $f$. It is not unique since $f$ can have more than one harmonic lift $F$, but to each such harmonic lift $F$ is associated a unique  Betti-conjugate $g$ of $f$.  We prove the following 

\begin{prop}
    For any $f$ as above, a harmonic lift always exists. Furthermore, if $f$ is meromorphic on $\mathfrak{H}$, then it admits a harmonic lift $F$ such that the corresponding Betti-conjugate $g$ is  also meromorphic on $\mathfrak{H}$. 
\end{prop}

Typically, one takes $f$ to be a cusp form, in which case $F$ is given by a weak harmonic Maass form as in \cite{bruininer-ono-rhoades} (see Remark \ref{rem:harmonic-maass}). The previous proposition also allows $f$ to have poles, and  implies that the higher Green's functions may be defined from the `bottom up', starting with $\Phi$, and using its interpretation \eqref{intro: dPsiisPhi} and \eqref{intro: dGisPsi} as a double harmonic lift. 

It is straightforward to show that equation \eqref{into: harmonic} is equivalent to a system of differential equations in the components $F^{p,q}$ of  $F = \sum_{p+q = k} F^{p,q}(z) (X-zY)^p (X- \overline{z}Y)^q$, which uniquely determine each other. In particular they are eigenfunctions of the hyperbolic Laplacian (see \S\ref{par:eigenfunctions}). When $k$ is even,  the central element $F^{p,p}$ plays a special role: Lemma \ref{prop: HarmonicCorrespondence} shows that harmonic lifts $F$ are in one-to-one correspondence with $\Gamma$-invariant eigenfunctions of  the  Laplace equation; Lemma \ref{lem: LaplacefromLegendre} shows that the latter  are in turn  equivalent to solutions to the Legendre differential equation.  These  elementary observations explain why Legendre functions necessarily appear in the theory of harmonic lifts, and why the `classical' higher Green's function naturally appears as the central element in  a two-dimensional array of generalised higher Green's functions. 

\subsection{Hodge theory of  modular curves and  biextensions} 

A necessary step in our study of higher Green's functions is a systematic treatment of the algebraic de Rham cohomology and Hodge theory of punctured modular curves. This is covered in Section \ref{sec:cohomology-over-C} entirely from the point of view of algebraic geometry and is somewhat independent from the rest of the  paper.

It is known that the  algebraic de Rham cohomology 
$H^1_{\dR}(\mathcal{Y}_{\Gamma}; \mathcal{V}_k^{\dR,\mathbb{C}}) $
is generated by weakly holomorphic modular forms $M_{k+2,\Gamma}^!$, modulo the image of weakly holomorphic modular forms of negative weight $M_{-k,\Gamma}^!$ under the so-called Bol map. It admits a Hodge filtration with non-zero Hodge numbers of type $(0,k+1)$,  $(k+1,0)$ and $(k+1,k+1)$, the subspace $F^{k+1}$ being generated by holomorphic modular forms. The corresponding statement (Proposition \ref{prop:dR-cohomology-meromorphic-modular-forms}) for the open  modular curve $\mathcal{U}_{\Gamma} = \mathcal{Y}_{\Gamma}\setminus S  $ with finitely many punctures $S= \{w_1,\ldots, w_r\}$ is  less well-known.
There is a long exact sequence of filtered  vector spaces
\[
    \begin{tikzcd}[column sep = small]
        0 \arrow{r} & H_{\dR}^1(\mathcal{Y}_{\Gamma};\mathcal{V}_k) \arrow{r} & H_{\dR}^1(\mathcal{U}_{\Gamma};\mathcal{V}_k) \arrow{r}{\Res} & \bigoplus_{i=1}^r(\Sym^kH_{\dR}^1(\mathcal{E}_{w_i}))^{\Gamma_{w_i}}(-1) \arrow{r} & 0 \ ,
    \end{tikzcd}
\]
where $\Gamma_{w_i}$ denotes the stabiliser of $w_i$ in $\Gamma$. The Hodge filtration on each $\Sym^kH_{\dR}^1(\mathcal{E}_{w_i})^{\Gamma_{w_i}}(-1) $ has non-vanishing Hodge numbers contained in the set  
$\{ (k+1,1), \ldots, (1, k+1)\}$.   Theorem \ref{thm:hodge-filtration-open-complex}  describes the Hodge filtration on  $H_{\dR}^1(\mathcal{U}_{\Gamma};\mathcal{V}_k)$ in terms of the  orders of the poles of meromorphic modular forms along $S$. Using this result we deduce 
that the meromorphic modular forms $\Psi^{p,q}_{\Gamma}$ provide a canonical splitting of the previous short exact sequence and that 
\[
    [\Psi_{\Gamma}^{p,q}(z,w_i)(X-zY)^k dz ]   \  \in \     F^{q+1} H^1_{\dR}(\mathcal{U}_{\Gamma};\mathcal{V}_k^{\dR,\mathbb{C}})
\]
for $p+q=k$ and $p,q\geq 0$. See  Corollary \ref{cor:Psi-splits-residue-seq} for a more precise statement.  In particular, the $\Psi_{\Gamma}^{p,q}$, together with weakly holomorphic modular forms, generate the vector space $H_{\dR}^1(\mathcal{U}_{\Gamma};\mathcal{V}_k)$. 
See Remark \ref{rem: Paloma} for an application to the meromorphic modular forms which are  obtained as
sums of negative  powers of integral binary quadratic forms  with negative discriminant.

\subsection{Biextensions}\label{par:intro-biextensions}

Let $k\ge 1$ be an integer.  In the following discussion, we  work in an abelian category of realisations  $\mathcal{H}_K$ over a number field $K\subset \mathbb{C}$ (namely:  Betti  and de Rham realisations, equipped with  a comparison isomorphism and real Frobenius action). Assume that $\mathcal{Y}_{\Gamma}$, $z$, $w$ are all defined over $K$, and consider the objects  
\[ 
    M_{w}= H_{\mathrm{cusp}}^1\left( \mathcal{Y}_{\Gamma} \setminus \{ w\}  ; \mathcal{V}_k\right) \quad \hbox{ and } \quad   M= H_{\mathrm{cusp}}^1\left( \mathcal{Y}_{\Gamma};\mathcal{V}_k \right)  
\]
of $\mathcal{H}_K$. Together with  \eqref{intro: modularmotive}, they fit into exact sequences in the category $\mathcal{H}_K$
\begin{equation}\label{biextensioninHK}
    \begin{gathered}
    \begin{tikzcd}[column sep = small]
        0 \arrow{r} & \big(\Sym^{k} H^1(\mathcal{E}_z)\big)^{\Gamma_z} \arrow{r}& M_{z,w} \arrow{r}&M_w \arrow{r}& 0
    \end{tikzcd}\\
    \begin{tikzcd}[column sep = small]
        0  \arrow{r}& M \arrow{r}& M_w \arrow{r}& \big(\Sym^{k} H^1(\mathcal{E}_w)\big)^{\Gamma_w}(-1) \arrow{r}& 0 \ ,
    \end{tikzcd} 
    \end{gathered}
\end{equation}
where $\Gamma_z$ (resp. $\Gamma_w$) denotes the stabiliser of $\Gamma$ at $z$ (resp. $w$). Thus $M_{z,w}$  has the structure of a biextension. 

In the case when there are no cusp forms of level $\Gamma$ and weight $k+2$, so that $M$ vanishes, it reduces to a simple extension. There is a generalisation where one replaces $z$ and $w$ with any finite number of (distinct) points. If $T$ is any linear combination of Hecke operators which annihilates the object $M$, then we show that one may extract from such a biextension a subquotient $M_{z,w}^T$ as in \eqref{intro: extensionMTzw}.

Let $k=2 m$ be even.  If  an elliptic curve $\mathcal{E}_z$ admits complex multiplication then one may show  that the object $\big(\Sym^{k} H^1(\mathcal{E}_z)\big)^{\Gamma_z}$ decomposes via the action of complex multiplication and always contains a  Tate motive $\QQ(-m)$ as a summand. Thus when $\mathcal{E}_z$ and $\mathcal{E}_w$ are both CM elliptic curves, we may extract from the previous extension the Gross-Zagier extension \eqref{intro: GZextension}.  
  
\subsection{Motives and  Beilinson's conjecture} 

The  above constructions are motivic. For simplicity we consider the case of full level $\Gamma= \SL_2(\ZZ)$ in detail, but the argument generalises. We may construct a relative cohomology motive \eqref{intro: themotive} in Voevodsky's triangulated category $\DM(K)$ which admits the structure of a biextension, \emph{i.e.}, fits into two distinguished triangles whose realisations correspond to \eqref{biextensioninHK}: 
\[
    \begin{tikzcd}[column sep = small]
        \left(\Sym^{n-1}H^1(\mathcal{E}_z)\right)^{\Gamma_z} \arrow{r}&M_{z,w} \arrow{r} & M_w \arrow{r}{+1}&{}
    \end{tikzcd}
\]
\[
    \begin{tikzcd}[column sep = small]
        M \arrow{r}& M_w \arrow{r}& \left(\Sym^{n-1} H^1(\mathcal{E}_w)\right)^{\Gamma_w}(-1) \arrow{r}{+1}&{}
    \end{tikzcd}
\]
Here, $n = k+1$ and $M = M^n_{\cusp}$, as in the end of \S\ref{par:results}. This relies on the work of Consani-Faber \cite{ConsaniFaber} in the case of level $1$, but may be extended to higher level using \cite{Petersen}. We show in Theorem \ref{thm: SimpleMotivicExtension}  that  any linear combination of Hecke operators $T$ which acts by zero on $M^n_{\mathrm{cusp}}$  cuts out a motive which is a simple extension  whose realisation is \eqref{intro: extensionMTzw}.  In the case where $z,w$ have complex multiplication, it can be broken down further into a Tate twist of a Kummer extension whose realisation is the Gross-Zagier extension \eqref{intro: GZextension}. 

The subtle point is that this assumption on $T$ is a priori stronger, but conjecturally equivalent, to the statement that $T$ acts by zero on the realisation  of $M^n_{\mathrm{cusp}}$ (or, equivalently, on the space of cusp forms of weight $n+1$), which is the hypothesis of the Gross-Zagier conjecture. The equivalence of these hypotheses would follow, for example, from the faithfulness of the Betti realisation functor (or related conjectures, which are a basic cornerstone in the expected theory of motives) which is known in many cases, but alas not presently  for motives of cusp forms. We conjecture nevertheless that it is true and prove it in weight 4.   We therefore obtain a \emph{universal construction} of a simple  extension in the triangulated category of motives
\[
    \begin{tikzcd}[column sep = small]
        \left(\Sym^{2} H^1(\mathcal{E}_z)\right)^{\Gamma_z} \arrow{r}& M^T_{z,w} \arrow{r}& \left(\Sym^{2} H^1(\mathcal{E}_w)\right)^{\Gamma_w}(-1) \arrow{r}{+1} & {}
    \end{tikzcd}
\]
whose realisation is \eqref{intro: extensionMTzw}, and which may be of independent arithmetic interest.

Beilinson's conjecture leads to  predictions between the entries of our generalised matrix-valued Green's functions and special values of $L$-functions; see \S \ref{sect: BeyondGZ} for a brief discussion. This points to a vast, and fascinating, generalisation of the Gross-Zagier conjecture which is amenable to numerical verification but takes us beyond the scope of the present work. 

\begin{rem}
    Although we shall not specifically study the simpler  situation  when $w$ is a cusp, the construction of biextensions and motives works in a similar manner in this case.  It may be helpful when reading this paper to bear in mind the corresponding  analytic objects in this, possibly more familiar, situation. First of all, the analogues of the (cohomology classes of) meromorphic modular forms  $\Psi_{\Gamma}^{p,q}$ with poles at a cusp are given by holomorphic  Eisenstein series.  The analogues of higher Green's functions when $w$ is a cusp are their harmonic lifts (or indefinite equivariant Eichler integrals)  which are real-analytic Eisenstein series. 
\end{rem}

\subsection{Acknowledgements}

Many thanks to Giuseppe Ancona, Cl\'ement Dupont and \"Ozlem Imamoglu for comments. This project has received funding from the European Research Council (ERC) under the European Union’s Horizon Europe programme (grant agreement No. 101167287). The second-named author is currently supported by the grant $\#$2020/15804-1, São Paulo Research Foundation (FAPESP). For the purpose of Open Access, the first-named author has applied a CC BY public copyright licence to any Author Accepted Manuscript (AAM) version arising from this submission.

\section{Harmonic lifts, Betti-conjugates, and Laplace eigenfunctions} \label{sec: HarmonicLiftsGeneralities}

We recall some results on harmonic lifts and  Betti-conjugates of modular forms, and  establish a one-to-one  correspondence between harmonic lifts and  eigenfunctions of the Laplacian.

\subsection{Reminders on differential operators}\label{sec: raisin-lowering}

We denote the upper half-plane by $\mathfrak{H} = \{z \in \mathbb{C} : \Im(z) >0\}$, and the lower half-plane by  $\mathfrak{H}^- = \{z \in \mathbb{C} : \Im(z) < 0\}$. Denote their union by $\mathfrak{H}^{\pm} = \mathbb{C} \setminus \mathbb{R}$. 

The \emph{raising} and \emph{lowering} operators are differential operators on $\mathfrak{H}^{\pm}$ denoted, for any $r,s \in \ZZ$, by
\[
    \partial_r = (z-\overline{z}) \frac{\partial}{\partial z} +r \qquad \text{and} \qquad  \overline{\partial}_s = (\overline{z}-z) \frac{\partial}{\partial \overline{z}} +s  \, .
\] 
When considering functions of two variables $(z,w) \in \mathfrak{H}^{\pm}\times \mathfrak{H}^{\pm}$, lower (resp. upper) indices denote differentiation with respect to  the first coordinate $z$ (resp. the  second coordinate $w$):
\[
    \partial^p = (w-\overline{w}) \frac{\partial}{\partial w} +p  \qquad \text{ and } \qquad  \overline{\partial}^q = (\overline{w}-w) \frac{\partial}{\partial \overline{w}} +q \,  .
\]
An upper-indexed  operator commutes with a  lower-indexed operator:
$
    [\partial_r, \partial^p] = [\overline{\partial}_s, \partial^p ] =  [\partial_r, \overline{\partial}^q ] =  [\overline{\partial}_s, \overline{\partial}^q ] =0 .
$  
We shall use the identities
\begin{equation} \label{eqn:shift}  
    \partial_{r}  \left( (z-\overline{z})^n f  \right)=  (z-\overline{z})^n \partial_{r+n} f\quad \text{ and } \quad \overline{\partial}_{s}  \left( (z-\overline{z})^n f  \right)=  (z-\overline{z})^n \overline{\partial}_{s+n} f
\end{equation} 
valid for all integers $r,s,n$ and any differentiable function $f$ on an open subset of $\mathfrak{H}^{\pm}$. A  similar  formula holds with $z$ replaced by $w$, and $\partial_r$ (resp. $\overline{\partial}_r$) replaced by $\partial^r$ (resp. $\overline{\partial}^r$). 

We denote the so-called \emph{automorphy factor} by:
\[
    j_{\gamma}(z)    =   cz+d\ , \qquad \hbox{ for any   } \gamma =  \left(\begin{smallmatrix} a & b \\ c& d \end{smallmatrix} \right) \in \GL_2(\RR) \, . 
\]

Let $\Gamma \leq \mathrm{SL}_2(\RR)$ be a subgroup and $r,s\in \ZZ$. A function $f$ defined on a $\Gamma$-invariant open subset  $\mathfrak{U}$ of $\mathfrak{H}$ is called modular of weights $(r,s)$ with respect to $\Gamma$, if $f(\gamma z) = j^r_{\gamma}(z) j^s_{\gamma}(\overline{z}) f(z)$ for all $\gamma \in \Gamma$, and $z\in \mathfrak{U}$. When $s = 0$, we simply say `modular of weight $r$'.

The map $\partial_r$  (resp. $\overline{\partial}_s$) sends a modular function of weights $(r,s)$ to one of weights $(r+1,s-1)$  (resp. $(r-1,s+1)$), and similarly for upper indices. 

The \emph{Laplace operator} in weights $(r,s)$ with respect to the variable $z$ is defined by the formula:
 \begin{equation} \label{LaplaceOperator}
     \Delta_{r,s} = - \overline{\partial}_{s-1} \partial_r +r(s-1) = - \partial_{r-1} \overline{\partial}_s + s(r-1) \, .
 \end{equation}
When $r=s=0$ it reduces to the usual hyperbolic Laplacian
 \begin{equation} \label{Delta00} 
    \Delta_{0,0} =(z- \overline{z})^2 \frac{\partial^2}{\partial z \partial \overline{z}}  = - y^2 \left( \frac{\partial^2}{\partial x^2} +  \frac{\partial^2}{\partial y^2}   \right)\, ,
\end{equation}
where we write $z=x+iy$. Let us denote by $\Delta^{p,q}= - \overline{\partial}^{q-1} \partial^p +p(q-1)$, with upper indices, to be the corresponding Laplace operators in the variable $w$. 

\subsection{Harmonic vector-valued functions} \label{par:harmonic}

In what follows, tensor products without subscripts are over $\mathbb{Q}$. For every integer $k\ge 0$, we denote by $V^{\B}_k$ the $\mathbb{Q}$-vector space of homogeneous polynomials of degree $k$ in $\mathbb{Q}[X,Y]$,  with a right action of $\SL_2(\mathbb{Z})$ given by
\begin{equation}\label{eq:right-action-SL2}
    (X,Y)|_{\gamma} = (aX+bY,cX+dY)\ \text{, }\qquad \gamma = \left(\begin{array}{cc}
    a & b \\
    c & d\end{array}\right) \ .
\end{equation}
If $\mathfrak{U}$ is an open subset of $\mathfrak{H}$ and $F: \mathfrak{U} \to V_k^{\B}\otimes \mathbb{C}$ is a real-analytic  vector-valued function, we can  write
\begin{equation}\label{eq:Frs}
    F(z) = \sum_{\substack{r+s=k \\ r,s\geq 0}}F_{r,s}(z)(X-zY)^r(X-\overline{z}Y)^s  \, ,
\end{equation} 
where each $F_{r,s}:\mathfrak{U} \to \mathbb{C}$ is a uniquely defined real-analytic function on $\mathfrak{U}$.

\begin{defn} \label{defn:harmonic}
  We shall call $F$ \emph{harmonic} if there exist holomorphic functions $f,g:\mathfrak{U} \rightarrow \CC$ such that
  \begin{eqnarray} \label{eq:harmonic}
    d F =  f(z) (X-zY)^{k} dz - \overline{g(z)} (X-\overline{z}Y)^{k} d\overline{z}\ . 
  \end{eqnarray} 
\end{defn}

Equation \eqref{eq:harmonic} is equivalent to the following system of equations (\emph{cf.} \cite[Prop. 7.2]{CNHMF1}): 
    \begin{equation}\label{eq:raising-lowering-middle} 
        \begin{aligned}
        \partial_r F_{r,s}  =  (r+1)\, F_{r+1,s-1} \ , \qquad \hbox{ for }   0\leq r\leq k-1  \\
        \overline{\partial}_s F_{r,s}  =  (s+1) \,F_{r-1,s+1} \ ,  \qquad \hbox{ for }   0\leq s\leq k-1  
        \end{aligned}
    \end{equation} 
    and 
    \begin{equation} \label{eq:raising-lowering-end}
        \partial_{k} F_{k,0}  =  (z- \overline{z}) \, f \  ,  \quad 
        \overline{\partial}_{k} F_{0, k}   =  (z-\overline{z}) \, \overline{g} \,  .
     \end{equation}
  
For every $\gamma \in \SL_2(\mathbb{Z})$, we have
\begin{align}\label{eq:gamma-equivariance}
        (X-\gamma z\,  Y)|_{\gamma} = j^{-1}_{\gamma}(z)(X-zY) \ , \quad 
        (X-\gamma \overline{z}\,  Y)|_{\gamma} = j^{-1}_{\gamma}(\overline{z})(X-\overline{z}Y) \, . 
\end{align}
Let $\Gamma \le \SL_2(\mathbb{Z})$ be a subgroup and let $\mathfrak{U}$ be a $\Gamma$-invariant open subset of $\H$. A vector-valued function $F: \mathfrak{U} \rightarrow V_k^\B\otimes \CC$  will be called  \emph{$\Gamma$-equivariant}, if 
\[
  F(\gamma z)\big|_{\gamma} = F(z)\ , \qquad \hbox{ for all  }  z \in \mathfrak{U} , \hbox{ and }  \gamma \in \Gamma ,\ \]
 which amounts by \eqref{eq:gamma-equivariance} to $F_{r,s}$ being modular for $\Gamma$ of weights $(r,s)$ for every $r+s=k$.

\begin{lem}\label{lem:equivariant-harmonic}
     Let $F: \mathfrak{U} \to V^{\B}_k\otimes\CC$ be a solution to equation \eqref{eq:harmonic}. If $F$ is $\Gamma$-equivariant, then $f,g$ are modular for $\Gamma$ of weight $k+2$.
\end{lem}

\begin{proof}
    Since $F_{k,0}$ (resp. $F_{0,k}$) is modular for $\Gamma$ of weights $(k,0)$ (resp. $(0,k)$), $\partial_k F_{k,0}$ (resp. $\overline{\partial}_k F_{0,k}$) is modular of weights $(k+1,-1)$ (resp. $(-1,k+1)$). The result follows from equation \eqref{eq:raising-lowering-end} together with the  fact that $z-\overline{z}$ is modular of weights $(-1,-1)$.
\end{proof}

\begin{rem}\label{rem:variable-w}
    All of the above definitions and constructions have analogues for functions $F$ in the variable $w$. In this case, $V_k^{\B}$ denotes the $\mathbb{Q}$-vector space of homogeneous polynomials of degree $k$ in the ring $\mathbb{Q}[U,V]$.  A real-analytic vector-valued function admits a decomposition of the following form
    \[
        F(w) = \sum_{\substack{p+q = k \\ p,q\ge 0}} F^{p,q}(w) (U-wV)^p(U-\overline{w}V)^q,
    \]
    and equations \eqref{eq:raising-lowering-middle} and \eqref{eq:raising-lowering-end} have natural counterparts with operators $\partial^p$ and $\overline{\partial}^q$ in place of $\partial_r$ and $\overline{\partial}_s$.
\end{rem} 

\subsection{Harmonic lifts and Betti-conjugates}

Let $\Gamma$ be a subgroup of $\SL_2(\mathbb{Z})$ and $\mathfrak{U}\subset \mathfrak{H}$ be a connected $\Gamma$-invariant open subset. Let $k\ge 0$ be an integer and $f:\mathfrak{U} \to \CC$ be a holomorphic function which is  modular for $\Gamma$ of weight $k+2$.

\begin{defn} \label{defn:harmoniclift}
    A \emph{harmonic lift} of $f$ is a $\Gamma$-equivariant harmonic vector-valued function $F: \mathfrak{U} \to V_k^{\B}\otimes \mathbb{C}$ satisfying equation \eqref{eq:harmonic} for some holomorphic function $g:\mathfrak{U} \to \CC$; we may also say that $F$ is a harmonic lift of the pair $(f,g)$. When such a harmonic lift exists, we say that $f$ is a \emph{Betti-conjugate} of $g$. 
\end{defn}

By Lemma \ref{lem:equivariant-harmonic}, $g$ is necessarily modular for $\Gamma$ of weight $k+2$. By applying the involution $F \mapsto -\overline{F}$ on the harmonic lift, we deduce that $f$ is a Betti-conjugate of $g$ if and only if $g$ is a Betti-conjugate of $f$.

We shall prove in Section \ref{sec:cohomology-over-C} that, when $\Gamma$ has finite index in $\SL_2(\mathbb{Z})$ and $\mathfrak{U}$ is given by the complement of a finite number of $\Gamma$-orbits, then harmonic lifts and Betti-conjugates always exist (see Corollary \ref{cor:existence-harmonic-lift} and  Corollary \ref{cor:existence-harmonic-lift-meromorphic} for a refined result in the setting of meromorphic modular forms). When the pair $(f,g)$ is fixed, we have the following uniqueness statement.

\begin{lem}\label{lem:uniqueness-harmonic}
    Assume that $\Gamma$ is a finite index subgroup of $\mathrm{SL}_2(\ZZ)$ and $k\ge 1$. The harmonic lift of a pair $(f,g)$ as above is unique, if it exists. 
\end{lem}

\begin{proof}
    The difference between any two harmonic lifts of $(f,g)$ is a vector-valued function $F : \mathfrak{U} \to V_k^{\B}\otimes \mathbb{C}$ satisfying $dF = 0$. Since $\mathfrak{U}$ is connected, $F$ is a constant. By definition, $F$ is also $\Gamma$-equivariant. The lemma follows from the fact that $H^0(\Gamma;V_k^{\B}\otimes \mathbb{C}) = 0$
    when $k\geq 1.$
\end{proof}

When only $f$ is fixed and a choice of Betti-conjugate $g$ is not, harmonic lifts will not be unique in general (even under the hypotheses of the above lemma) due to the existence of `Betti-conjugates of $0$'. 

\begin{prop} \label{prop:holomorphic-lift}
  The following are equivalent:
  \begin{enumerate}[(i)]
  \item $f$ is a Betti-conjugate of $0$,
  \item $f$ admits a holomorphic harmonic lift,
  \item there exists a holomorphic function $h:\mathfrak{U} \to \CC$ modular for $\Gamma$ of weight $-k$ such that
    \begin{equation} \label{eq:coboundary}
        \frac{\partial^{k+1}h}{\partial z^{k+1}} =f.
      \end{equation}
   \end{enumerate}
\end{prop}

\begin{proof}
  For a vector-valued function $F$ satisfying equation \eqref{eq:harmonic}, we have
  \[
    \frac{\partial F}{\partial \overline{z}} = -\overline{g(z)}(X-\overline{z}Y)^k,
  \]
  so that $g=0$ if and only if $F$ is holomorphic. Thus, $(i)$ is equivalent to $(ii)$.

  Assume that $f$ is a Betti-conjugate of $0$ and let $F$ be a harmonic lift of the pair $(f,0)$:
  \begin{equation}\label{eq:holomorphic-harmonic-lift}
    dF = f(z)(X-zY)^kdz \, .
  \end{equation}
  Using the decomposition \eqref{eq:Frs}, define 
  \begin{equation}\label{eq:defn-h}
   h(z) = \frac{(z-\overline{z})^k}{k!}F_{0,k}(z)\, .
  \end{equation}
  The $\Gamma$-equivariance of $F$ implies that $h$ is modular for $\Gamma$ of weights $(-k,0)$. Moreover,
  by  \eqref{eqn:shift}
  \begin{equation}\label{eq:h-holomorphic}
    \overline{\partial}_0 h = \frac{(z-\overline{z})^k}{k!}\overline{\partial}_kF_{0,k} = 0\, ,
  \end{equation}
  where the second equation follows from \eqref{eq:raising-lowering-end}
 and \eqref{eq:holomorphic-harmonic-lift}. Thus $h$ is holomorphic. Equation  \eqref{eq:coboundary} follows from  successive application of \eqref{eq:raising-lowering-middle} and \eqref{eq:raising-lowering-end} to \eqref{eq:defn-h}, and `Bol's formula' (\emph{cf.} \cite[Lemma 3.3]{CNHMF3}):
  \begin{equation}\label{eq:bol-formula}
    \partial_0\partial_1\cdots \partial_{-k} = (z-\overline{z})^{k+1}\frac{\partial^{k+1}}{\partial z^{k+1}}.
  \end{equation}
  This shows that $(i)$ implies $(iii)$. For the converse, let $F$ be the unique real-analytic vector-valued function with $F_{0,k}$ defined by  \eqref{eq:defn-h}, which satisfies equations \eqref{eq:raising-lowering-middle} (define $F_{r,s}$ for $r>0$ by successively applying operators $\partial_r$ to $F_{0,k}$: the consistency of equations     
  \eqref{eq:raising-lowering-middle} follows ultimately from the holomorphicity of $h$). The modularity of $h$ implies the $\Gamma$-equivariance of $F$, and equations  \eqref{eq:bol-formula}, and \eqref{eq:coboundary} imply that $F$ satisfies \eqref{eq:holomorphic-harmonic-lift}. It follows that $f$ is a Betti-conjugate of 0.
\end{proof}

\begin{cor}
  Let $g_0$ be a Betti-conjugate of $f$. Another  function $g$ is also a Betti-conjugate of $f$ if and only if there exists a holomorphic function $h:\mathfrak{U} \to \CC$, which is modular for $\Gamma$ of weight $-k$, such that
  \[
    g = g_0 + \frac{\partial^{k+1}h}{\partial z^{k+1}} \  .
  \]
\end{cor}

\begin{proof}
   From the definition, $g - g_0$ is a Betti-conjugate of 0. Now apply Proposition \ref{prop:holomorphic-lift}.
\end{proof}

\begin{rem}\label{rem:harmonic-maass}
   In the theory of (weak) harmonic Maass forms and mock modular forms (see \cite{bruininer-ono-rhoades,Candelori}), one considers the  following differential operators applied to  a  harmonic Maass form $\varphi$ of weight $-k$: 
   \[
        \xi_{-k}\varphi =  \frac{2i}{\mathrm{Im}(z)^k}\overline{\left(\frac{\partial \varphi}{\partial \overline{z}}\right)}\text{, }\qquad D^{k+1}\varphi = \frac{1}{(2\pi i)^{k+1}}\frac{\partial^{k+1}\varphi}{\partial z^{k+1}}
   \]
    If $f = \xi_{-k}\varphi$, then the unique harmonic vector-valued function $F: \mathfrak{H} \to V^{\B}_k\otimes \CC$ with 
    \[
        F_{k,0}(z) = \frac{(2i)^{k+1}}{(z-\overline{z})^k}\overline{\varphi(z)}
    \]
    is a harmonic lift of $f$ as in Definition \ref{defn:harmoniclift}. The  Betti-conjugate of $f(z)$ corresponding to  the lift $F$   is 
    \[
        g(z) = \frac{(-4\pi)^{k+1}}{k!}D^{k+1}\varphi(z).
    \]
\end{rem}

\subsection{Harmonic functions and eigenfunctions of the  Laplacian}\label{par:eigenfunctions}

The equations \eqref{eq:raising-lowering-middle}  imply that the components $F_{r,s}$ are all eigenfunctions of the generalised Laplacian: 
    \begin{equation} \label{grsEigenfunction} 
        (\Delta_{r,s} + k ) F_{r,s}  =0  \qquad \hbox{ for all } r+s = k   \  ,
    \end{equation} 
which follows  from one  of the two equivalent definitions \eqref{LaplaceOperator}  of $\Delta_{r,s}$.  Note that, after suitable shifts in the indices $r,s$,  the set of eigenfunctions of the Laplace operator $\Delta_{r,s}$ are stable under the operations $\partial, \overline{\partial}$. 

\begin{lem}\label{prop: HarmonicCorrespondence} 
    Let $r,s \geq 0$.  There is a one-to-one correspondence between real-analytic solutions $H_{r,s}: \mathfrak{U} \rightarrow \CC$ of the Laplace eigenfunction equation 
    \begin{equation} \label{rsseedEigenfunctionequation} 
        \left(\Delta_{r,s}+r+s \right) H_{r,s} = 0
    \end{equation}
    and harmonic vector-valued functions $F: \mathfrak{U} \rightarrow V_{r+s}^{\B}\otimes\CC$ whose  $(r,s)$-component in the decomposition \eqref{eq:Frs}  satisfies $F_{r,s} = H_{r,s}$. Furthermore, $H_{r,s}$ is modular of weights $(r,s)$ with respect to a subgroup $\Gamma\leq \SL_2(\ZZ)$ which preserves $\mathfrak{U}$ if and only if the vector-valued function $F$ is $\Gamma$-equivariant.  
\end{lem}

\begin{proof}
    The proposition is straightforward and paraphrases \cite[Proposition 5.4]{CNHMF3}. The main idea is  that the consistency of the equations \eqref{eq:raising-lowering-middle}, and the  holomorphicity of $f, g$, follow from the eigenfunction equation \eqref{grsEigenfunction}. Therefore, starting from  $F_{r,s}=H_{r,s}$ a solution to \eqref{rsseedEigenfunctionequation}, one can uniquely define the functions $F_{r',s'}$  for all other $r'+s'=r+s$, $r',s'\geq 0$ and deduce  $f, g$ by applying raising and lowering operators using \eqref{eq:raising-lowering-middle} and \eqref{eq:raising-lowering-end}. This defines a harmonic vector-valued function $F$. The converse statement is \eqref{grsEigenfunction}. 
\end{proof}

In the case  when $k=2m$ is even, the central component $r=s=m$ plays a special role.

\begin{lem} \label{lem: LaplaceHarmonicCorrespondence}
    For any integer $m\geq 0$,  there is a one-to-one correspondence between real-analytic solutions $\phi : \mathfrak{U} \to \mathbb{C}$ of the classical hyperbolic Laplace equation
    \begin{equation} \label{eqn: ClassicalLaplace}
        \left(  \Delta_{0,0}    + m(m+1) \right) \phi= 0
    \end{equation}
    and harmonic vector-valued functions $F : \mathfrak{U} \to V^{\B}_{2m}\otimes\CC$ satisfying
    $F_{m,m}(z)=    (z-\overline{z})^{-m} \phi(z)$. The function $F$ is $\Gamma$-equivariant if and only if $\phi$ is $\Gamma$-invariant. 
\end{lem}

\begin{proof}
    Lemma \ref{prop: HarmonicCorrespondence} gives a correspondence between harmonic vector-valued functions $F: \mathfrak{U} \rightarrow V_{2m}^{\B}\otimes \CC$ and real-analytic solutions $H_{m,m}:\mathfrak{U} \to \mathbb{C}$ of the  generalised hyperbolic Laplace equation
    \begin{equation} \label{seedEigenfunctionequation} 
        \left(\Delta_{m,m}+2m \right) H_{m,m} = 0 \ . 
    \end{equation}
    It is enough to show that  $H_{m,m}$ is a solution to \eqref{seedEigenfunctionequation} if and only if 
    $\phi(z)= (z- \overline{z})^m H_{m,m}(z)$ is a solution to  \eqref{eqn: ClassicalLaplace}. By \cite[(2.23)]{CNHMF1} for any real-analytic function $\psi(z)$, 
    \[
        (z- \overline{z}) \Delta_{n,n}\,  \psi(z)  - \Delta_{n-1,n-1} (z- \overline{z})\, \psi(z) = (2n-2)\, (z-\overline{z}) \, \psi(z) \   , 
    \]
    for any integer $n$. Thus, if $H_{m,m}$ is an eigenfunction of $\Delta_{m,m}$ with eigenvalue $\lambda_m$, it follows by induction on $i$ that $H_{m-i,m-i}= (z-\overline{z})^i H_{m,m} $ is an eigenfunction of $\Delta_{m-i,m-i}$ with eigenvalue 
    \[
        \lambda_{m-i}= \lambda_m - \left((2m-2) +  (2m-4) + \cdots + (2m-2i) \right) \, .
    \]
    In particular $\lambda_0= \lambda_m - m(m-1)$. The equivalence of \eqref{seedEigenfunctionequation} and \eqref{eqn: ClassicalLaplace} follows from the case $\lambda_{2m} = -m$. 
\end{proof}

The holomorphic functions $f, g$ in equation \eqref{eq:harmonic} corresponding to $F$ are explicitly given by
\begin{equation} \label{eqn: fgfromphi}
    f = \frac{m!}{(2m)!}\, \frac{1}{(z-\overline{z})^{m+1}}\, \partial_m \cdots \partial_1 \partial_0\,  \phi
\qquad 
\hbox{ and } 
\qquad
    \overline{g} = -\frac{m!}{(2m)!}\, \frac{1}{(z-\overline{z})^{m+1}} \, \overline{\partial}_m \cdots \overline{\partial}_1 \overline{\partial}_0\,  \phi \ . 
\end{equation}

\section{Meromorphic modular forms }

Using Poincar\'e series of rational functions  we construct the meromorphic modular forms $\Psi^{p,q}_{\Gamma}(z,w)$, which for $p,q\ge 0$ and $p+q>0$ can be seen as higher analogues of `differential forms of the 3rd kind' on a modular curve. We then  compute their (vector-valued) residues.

\subsection{The meromorphic modular forms  $\Psi^{p,q}_{\Gamma}$}

For $p,q \in \ZZ$, and $z\notin \{ w, \overline{w}\}$ we denote 
\begin{equation} \label{eq:psi} 
    \psi^{p,q} (z,w) \coloneqq \frac{w-\overline{w}}{(z-w)^{p+1} (z-\overline{w})^{q+1}} \ .
\end{equation}
  
\begin{defn} \label{defn: psiGamma}
    Let $\Gamma \leq \SL_2(\QQ)$ be commensurable to $\SL_2(\ZZ)$ (for the most part, $\Gamma $ will be a  subgroup of $\SL_2(\ZZ)$).  For all $p,q \in \ZZ$ with $p+q > 0$ and $(z,w) \in  \H^{\pm}\times \H^{\pm}$ with $z\not\in \Gamma w \cup \Gamma \overline{w}$, we define
    \begin{equation} \label{eq:Psi}
        \Psi_{\Gamma}^{p,q}(z,w)  \coloneqq  \sum_{\gamma \in \Gamma}    \frac{1}{j^{p+q+2}_{\gamma}(z)}  \psi^{p,q}(\gamma z,w)\ .
    \end{equation}
\end{defn} 

\begin{lem}[Poincaré]\label{lem:Poincare}
    Let $\Gamma$ be as above. If $R$ is a rational function on $\mathbb{C}$ with no poles along the real axis  which vanishes to order at least $n >2$ at infinity, then the series $\sum_{\gamma \in \Gamma}j^{-n}_{\gamma}(z)R(\gamma z)$ converges normally on compact subsets of $\mathfrak{H}$, and defines a meromorphic modular form for $\Gamma$ of weight $n$ which vanishes at the cusps.
\end{lem}

This is a classical result of Poincaré (\emph{cf.} \cite[Chapter V]{Ford}). A more streamlined proof can be obtained from the integral estimate in \cite[Theorem 2.6.1]{Miyake}.
 
\begin{prop} \label{prop:Psi}
    The  series \eqref{eq:Psi} converges normally on compact subsets of $\H^{\pm}\times \H^{\pm}$, and defines a function $\Psi^{p,q}_{\Gamma}$ on the open subset $\{(z,w) \in \H^{\pm}\times \H^{\pm} : z \not\in \Gamma w \cup \Gamma \overline{w}\}$ which is holomorphic in $z$ and real-analytic in $w$. Moreover:
    \begin{enumerate}
        \item It can be written in the alternative form:
            \begin{equation} \label{eq:Psi-average-w}
                \Psi_{\Gamma}^{p,q}(z,w) =  \sum_{\gamma \in \Gamma} \frac{1}{j^{p}_{\gamma}(w)j^{q}_{\gamma}(\overline{w})} \psi^{p,q}(z,\gamma w).
            \end{equation}
      The function $\Psi^{p,q}_{\Gamma}(z,w)$ is modular of weight $p+q+2$ in $z$ and weights $(p,q)$ in $w$:
            \begin{equation} \label{eq:Psi-modularity}
                \Psi^{p,q}_{\Gamma}(\gamma_1 z,\gamma_2 w) = j_{\gamma_1}^{p+q+2}(z) j_{\gamma_2}^p(w)j_{\gamma_2}^q(\overline{w}) \Psi_{\Gamma}^{p,q}(z,w)\, ,   \qquad \hbox{ for all  } (\gamma_1,\gamma_2) \in \Gamma\times \Gamma \, . 
            \end{equation}
        \item It also converges normally on compact subsets of $D\times \H^{\pm}$, where $D$ is a disc neighbourhood of any cusp for $\Gamma$. For any fixed $w \in \H^{\pm}$, the function $z\mapsto \Psi^{p,q}_{\Gamma}(z,w)$ is a meromorphic modular form of weight $p+q+2$ for $\Gamma$ which vanishes at the cusps. Furthermore, if $w \in \mathfrak{H}$ (resp. $w \in \mathfrak{H}^-$), then it has poles of order at most $p+1$ (resp. $q+1$) along $z \in \Gamma w$ (resp. $z \in \Gamma \overline{w}$).
        \item For any $\delta\in \GL_2(\QQ)$ we have 
            \begin{equation} \label{eq:Psi-g-action}
                \Psi_{\Gamma}^{p,q}(\delta z,w) = \frac{j_{\delta}^{p+q+2}(z)}{(\det \delta)^{p+q+1} j^{p}_{\delta^{-1}}(w) j^{q}_{\delta^{-1}}(\overline{w})} \Psi_{\delta^{-1}\Gamma \delta}^{p,q}(z, \delta^{-1}w) .
            \end{equation}
    \end{enumerate}
\end{prop}
  
\begin{proof}[Proof sketch]
    The convergence statements follow from Lemma \ref{lem:Poincare}, with a minor extra argument to obtain normal convergence in both variables $(z,w)$. Then, equation \eqref{eq:Psi-average-w} follows from the general formula
        \begin{equation}\label{eq:action-g-z-w}
            \delta z - \delta w = \frac{\det \delta}{j_{\delta}(z)j_{\delta}(w)}(z-w) \ ,
        \end{equation}
    which holds for any $z,w \in \mathfrak{H}^{\pm}$ and any $\delta \in \GL_2(\mathbb{Q})$. The modularity \eqref{eq:Psi-modularity} follows from a standard argument.  Equation \eqref{eq:Psi-g-action} follows from
        \[
            \frac{\psi^{p,q}(\delta z,w)}{j_{\delta}^{p+q+2}(z)} = \frac{\psi^{p,q}(z,\delta^{-1}w)}{(\det\delta)^{p+q+1} j^{p}_{\delta^{-1}}(w) j^{q}_{\delta^{-1}}(w)  }\ ,
        \]
    which is another application of \eqref{eq:action-g-z-w}. 
\end{proof} 

\subsection{$\Psi^{p,q}_{\Gamma}(z,w)$ as components of a harmonic lift}

Let $k>2$ be an integer and set
    \begin{equation}\label{eq:Phi}
        \Phi_{\Gamma,k}(z,w) \defeq \frac{k-1}{w-\overline{w}}\Psi_{\Gamma}^{k-1,-1}(z,w) = (k-1)\sum_{\gamma \in\Gamma}\frac{1}{j_{\gamma}^{k}(z)}\frac{1}{(\gamma z - w)^k}.
    \end{equation}
Viewing $\Phi_{\Gamma,k}(z,w)$ as a function of $z$ in $\H \setminus \Gamma w$, it follows from Proposition \ref{prop:Psi} that:
\begin{itemize}
    \item for fixed $w \in \mathfrak{H}$, $\Phi_{\Gamma,k}(z,w)$ is a meromorphic modular form for $\Gamma$ of weight $k$ which vanishes at the cusps and has poles of order at most $k$ along $z \in \Gamma w$.
    \item for fixed $w \in \mathfrak{H}^-$, $\Phi_{\Gamma,k}(z,w)$ is a cusp form for $\Gamma$ of weight $k$, known in the literature as a `reproducing kernel' (\emph{cf.} \cite[\S 2]{zagier77}).
\end{itemize}
For any $w \in \mathfrak{H}^{\pm}$, we can also write
    \[
        \Phi_{\Gamma,k}(z,\overline{w}) = \frac{k-1}{w-\overline{w}}\Psi_{\Gamma}^{-1,k-1}(z,w).
    \]

\begin{rem}\label{rem:symmetry-Phi}
    For $z,w \in \mathfrak{H}^{\pm}$, with $z \not\in \Gamma w$, we have the following symmetry properties:
        \[
            \Phi_{\Gamma,k}(w,z) = (-1)^k \Phi_{\Gamma,k}(z,w)\text{, }\qquad \overline{\Phi_{\Gamma,k}(z,w)} = \Phi_{\Gamma,k}(\overline{z},\overline{w}).
        \]
\end{rem}

For any $p,q \ge 0$ with $p+q = k$, we shall now recover $\Psi^{p,q}_{\Gamma}(z,w)$ as a component of a harmonic lift, in the variable $w$, of $\Phi_{\Gamma,k+2}(z,w)$ (\emph{cf.} Remark \ref{rem:variable-w}). For this,  consider the vector-valued function
    \begin{equation} \label{eqndefn:PsiGammak}
        \Psi_{\Gamma,k}(z,w) = \sum_{\substack{p+q=k\\ p,q\ge 0}} \Psi_{\Gamma}^{p,q}(z,w)(U-wV)^p (U- \overline{w} V)^q.
    \end{equation}
It is real-analytic and equivariant with respect to the  simultaneous action of $\Gamma$ on $w$ and $U,V$.

\begin{lem}\label{lem:identity-psi}
    Consider the vector-valued rational function of $z, w, \overline{w}$:
        \begin{equation} \label{eq:psiklowercasedef}
            \psi_{k}(z,w) = \sum_{\substack{p+q = k \\ p,q\ge 0}}\psi^{p,q}(z,w) (U-wV)^p(U- \overline{w} V)^q.
        \end{equation}
    The following identity holds in $\mathbb{Q}(z,w,\overline{w}, U,V)$:
        \begin{equation} \label{psikrational}
            \psi_k(z,w) = \frac{1}{U-zV}\left[\left(\frac{U-wV}{z-w} \right)^{k+1} - \left(\frac{U-\overline{w}V}{z-\overline{w}} \right)^{k+1} \right].
        \end{equation}
\end{lem}

\begin{proof}
    By definition of $\psi^{p,q}(z,w)$, we have
        \[
            \psi_k(z,w) = \sum_{\substack{p+q=k \\ p,q\ge 0}}\frac{w-\overline{w}}{(z-w)(z-\overline{w})}\left(\frac{U-wV}{z-w}\right)^p\left(\frac{U-\overline{w}V}{z-\overline{w}} \right)^q.
        \]
    The statement follows by applying the identity
        \[
            \sum_{\substack{p+q=k\\ p,q\geq 0}} A^p B^q = \frac{A^{k+1}-B^{k+1}}{A-B}
        \]
    with $A= \frac{U-w V}{z- w }$, $B=\frac{U-\overline{w}V}{z-\overline{w}}$ to the previous expression, and noting that
        \[
            A-B = \frac{(w-\overline{w})(U-zV)}{(z-w)(z-\overline{w})}.\qedhere
        \]
\end{proof}

\begin{rem}
    It follows immediately from the expression \eqref{psikrational} that $\psi_k(z,w)$ satisfies the following symmetries with respect to complex conjugation:
    \begin{equation}\label{eq:complex-conjugate-psi}
        \psi_k(z,\overline{w}) = -\psi_k(z,w)\text{, }\qquad \overline{\psi_k(z,w)} = \psi_k(\overline{z},\overline{w}) = -\psi_k(\overline{z},w).
    \end{equation}
\end{rem}

\begin{prop} \label{prop: PsiliftsPhi}
   Let $k>0$ and $z\in \H^{\pm}$. The vector-valued function $\Psi_{\Gamma,k}(z,w)$, viewed as a function of $w$ on $\H^{\pm}\setminus \Gamma z \cup \Gamma \overline{z}$, satisfies  the harmonic lift equation \eqref{eq:harmonic}
    \[
      d_w \Psi_{\Gamma,k} = \Phi_{\Gamma,k+2}(z,w) (U-w V)^k\,  dw -    \Phi_{\Gamma,k+2}(z, \overline{w}) (U- \overline{w} V)^k \,d \overline{w}  \ , 
    \]
    since $\overline{ \Phi_{\Gamma,k+2}(\overline{z}, w)}=  \Phi_{\Gamma,k+2}(z, \overline{w})$ by Remark \eqref{rem:symmetry-Phi}.
\end{prop}

\begin{proof}
    By Lemma \ref{lem:identity-psi} and direct computation, we have
    \begin{equation}\label{eqn:dpsi}
        d_w \psi_k(z,w) = (k+1) \left(\frac{(U-wV)^k}{(z-w)^{k+2} }dw -  \frac{(U-\overline{w}V)^k}{(z-\overline{w})^{k+2} }d\overline{w}   \right) .
    \end{equation} 
    Note that by \eqref{eq:Psi-average-w} and  \eqref{eq:gamma-equivariance},
    \[
        \Psi_{\Gamma,k}(z,w) = \sum_{\gamma \in  \Gamma} \frac{\psi_k(\gamma z,w)}{j_{\gamma}^{p+q+2}(z)}= \sum_{\gamma \in \Gamma}  \psi_k( z , \gamma w) \Big|_{(1,\gamma)} ,
    \]
    where $(1 , \gamma) \in \Gamma \times \Gamma$ acts only on  $U,V$. A similar formula holds for $\Phi_{\Gamma,k+2}(z,w)$ and $\Phi_{\Gamma,k+2}(z,\overline{w})$. The statement then follows by averaging both sides of Equation \eqref{eqn:dpsi} by the action of $\Gamma$ on $w$ and $(U,V)$, and applying the normal convergence of the series defining $\Psi_{\Gamma}^{p,q}(z,w)$ (Proposition \ref{prop:Psi}).
\end{proof}

\begin{cor} \label{cor: PhiBettiConjugates}
    As meromorphic modular forms in $w$, the functions  $\Phi_{\Gamma,k}(z,w)$ and $\Phi_{\Gamma,k}(\overline{z},w)$  for $k>2$ are Betti-conjugate. By symmetry,  the functions $\Phi_{\Gamma,k}(z,w)$ and $\Phi_{\Gamma,k}(z,\overline{w})$ are Betti-conjugate as meromorphic  modular forms in $z$. \hfill $\square$
\end{cor} 

\begin{rem}
    Proposition \ref{prop: PsiliftsPhi}  implies that $\Psi_{\Gamma,k}(z,w) $ may be written as an `equivariant single-valued integral' of  $\Phi_{\Gamma,k+2}(z,\tau)(U-\tau V)^kd\tau$, which is a modular version of  a single-valued integral \cite{brown-dupont}. Concretely, it is a combination of an Eichler integral of $\Phi_{\Gamma,k+2}(z,\tau)(U-\tau V)^kd\tau$  and its complex conjugate, with an appropriate constant of integration to ensure modularity.  Even though this point of view is  fruitful \cite{CNHMF3}, it will not be pursued  in this paper. 
\end{rem}

\subsection{Residues}

Fix $w \in \mathfrak{H}^{\pm}$ and regard each $\Psi^{p,q}_{\Gamma}(z,w)$, where $p+q=k>0$ and $p,q\ge 0$, as a function of $z$ defined on the open subset $\mathfrak{U} = \{z \in \mathfrak{H} : z \not\in \Gamma w \cup \Gamma \overline{w}\}$. In order to compute the residues of $\Psi^{p,q}_{\Gamma}(z,w)$, it is convenient to first consider the one-form with values in $(V^{\B}_{k} \boxtimes V_{k}^{\B})\otimes \mathbb{C}$ 
(external tensor product on $\mathfrak{H}^{\pm} \times \mathfrak{H}^{\pm}$):
\begin{equation}\label{eq:psi-differential}
  \Psi_{\Gamma,k}(z,w)(X-zY)^kdz = \sum_{\substack{p+q=k \\ p,q\ge 0 }} \Psi_{\Gamma}^{p,q}(z,w)  (U-wV)^p (U- \overline{w} V)^q  (X-zY)^k \, dz \ .
\end{equation}
Note that \eqref{eq:psi-differential} is equivariant with respect to both $z$ and $w$, and its residue along $z= \gamma w$ will give the same element of $\left(V^{\B}_{k} \otimes V_{k}^{\B}\right)^{\Gamma_w}\otimes \mathbb{C}$ for any $\gamma \in \Gamma$. Here, $\Gamma_w=\{\gamma \in \Gamma: \gamma w= w\}$ denotes the stabiliser of $w$ in $\Gamma$ and acts via the diagonal embedding  $\Gamma_w \hookrightarrow \Gamma_w \times \Gamma_w$ on $V_k^{\B}\otimes V_k^{\B}$. 

\begin{thm} \label{thm:residue-Psi}
    We have:
        \[
            \Res_{z=w} \Psi_{\Gamma,k}(z,w)(X-zY)^kdz   = \sum_{\gamma \in \Gamma_w}  (XV-YU)^k\Big|_{(\gamma,1)} = \sum_{\gamma \in \Gamma_w}  (XV-YU)^k\Big|_{(1,  \gamma)},
        \]
    where $(\gamma,1)$ acts via $|_{\gamma}$ on $(X,Y)$ and acts trivially on $(U,V)$, and  vice-versa for $(1,\gamma)$. In particular, the above residue is rational with respect to the Betti structure, \emph{i.e.}, it lies in  $(V^{\B}_{k} \otimes V_{k}^{\B})^{\Gamma_w}$.
\end{thm}

\begin{proof}
    First of all, it follows from the normal convergence on compact subsets of the defining series \eqref{eq:Psi-average-w} for $\Psi_{\Gamma}^{p,q}(z,w)$ as a sum over $\Gamma$ that only the terms corresponding to $\gamma \in \Gamma_w$   contribute to the residue at $z=w$. It therefore suffices to compute:
        \[
            R\defeq\operatorname{Res}_{z=w}  \Bigg(\sum_{\gamma \in \Gamma_w} \sum_{\substack{p+q=k \\ p,q\ge 0}}
            \frac{\psi^{p,q}(z,w)}{j^p_{\gamma}(w) j^q_{\gamma}(\overline{w})} (U-wV)^p(U-\overline{w}V)^q (X-zY)^k dz\Bigg).   
        \]
    We may treat $X,Y,U,V$ as formal variables and work in the field $\QQ(X,Y,Z,W,z,\overline{z}, w, \overline{w})$. The denominators $j_{\gamma}^p(w), j_{\gamma}^q(\overline{w})$ can be absorbed, via \eqref{eq:gamma-equivariance}, by slashing on the right by $(1,\gamma)$, which yields
        \[
             R=\Res_{z=w}  \Bigg(\sum_{\gamma \in \Gamma_w}\psi_k(z,w)\Big|_{(1,\gamma)}(X-zY)^kdz \Bigg),
        \]
    where $\psi_k(z,w)$ is as in \eqref{eq:psiklowercasedef}.  By Lemma \ref{lem:identity-psi}, we obtain
        \[
            R=  \sum_{\gamma \in \Gamma_w} R_{0} \Big|_{(1, \gamma)},
        \]
    where
        \[
            R_{0}=  \mathrm{Res}_{z=w}   \left(   \left(\frac{U-wV}{z-w}\right)^{k+1} -   \left(\frac{U-\overline{w}V}{z-\overline{w}}\right)^{k+1} \right) \frac{(X-zY)^k}{U-zV}dz.
        \]
    The term $\left(\frac{U-\overline{w}V}{z-\overline{w}}\right)^{k+1}$ has no pole at $z=w$ and so
        \[
            R_{0}=   \mathrm{Res}_{z=w}  \left(\left(\frac{U-wV}{z-w}\right)^{k+1}  \frac{(X-zY)^k}{U-zV}  dz\right).
        \]
    We can write
        \[
            \left(\frac{U-wV}{z-w}\right)^{k+1}  \frac{(X-zY)^k}{U-zV} =  \frac{(U-wV)(U-zV)^{k-1}}{z-w}  \left(\frac{(X-zY)(U-wV)}{(z-w)(U-zV) } \right)^k
        \]
    Using the identity
        \[
            \frac{(X-zY)(U-wV)}{(z-w)(U-zV)} =   \frac{X-w Y}{z-w} + \frac{XV-YU}{U-zV}
        \]
    we deduce that
        \begin{align*}
            R_0 &=\operatorname{Res}_{z=w}   \Bigg(   \frac{(U-wV)(U-zV)^{k-1}}{z-w}  \left(\frac{X-w Y}{z-w} + \frac{XV-YU}{U-zV} \Bigg)^k dz\right)\\
            &=\operatorname{Res}_{z=w}   \Bigg(   \frac{(U-wV)(U-zV)^{k-1}}{z-w}  \left(\frac{XV-YU}{U-zV} \right)^k dz\Bigg)\\
            &=\frac{(U-wV)(U-wV)^{k-1}}{(U-wV)^k}(XV-YU)^k\\
            &=(XV-YU)^k,  
        \end{align*}
    which completes the proof of the first equality in the statement. The second equality follows from the fact that $XV- YU$ is invariant under the diagonal action $|_{(\gamma, \gamma)}$. 
\end{proof}

We can now deduce a formula for the residues of each component of \eqref{eq:psi-differential}.

\begin{cor}\label{cor:residue-Psi}
    For all $p,q\geq 0$ such that $p+q>0$, let
    \begin{equation}\label{kappdef}
        \kappa_{\Gamma,w}^{p,q} \defeq \sum_{\gamma \in \Gamma_{w}}\frac{1}{j^p_{\gamma}(w)j^q_{\gamma}(\overline{w})}.
    \end{equation}
    Then:
    \[
        \Res_{z=w} \left( \Psi_{\Gamma}^{p,q}(z,w)   (X-zY)^{p+q} dz  \right)
        = \frac{(-1)^p\kappa_{\Gamma,w}^{p,q}}{(w-\overline{w})^{p+q}}  \binom{p+q}{p} (X- \overline{w}Y)^p(X-  w Y)^q \ \in \  (V_k^{\B} \otimes V_k^{\B})\otimes \CC 
    \]
\end{cor}

\begin{proof}
    Let $k=p+q>0$. We first note that
        \[
            \Res_{z=w} (\Psi_{\Gamma,k}(z,w)(X-zY)^kdz) = \sum_{p+q=k} \Res_{z=w}\left(\Psi^{p,q}_{\Gamma}(z,w) (X-zY)^kdz \right)(U-wV)^p(U-\overline{w}V)^q.
        \]
    Now, write 
        \[
            XV-YU = \frac{(U-\overline{w}V)}{w-\overline{w}}(X-w  Y) - \frac{(U-wV)}{w-\overline{w}}(X-\overline{w} Y)
        \]
    and use the binomial expansion to read off the coefficient of $(U-w V )^p(U-\overline{w} V)^q$ in $(XV-YU)^k$. The result then follows from summing over $\Gamma_w$, applying  $(\gamma,1)$ and equation \eqref{eq:gamma-equivariance} to obtain 
        \[
            \Res_{z=w}\left(\Psi^{p,q}_{\Gamma}(z,w) (X-zY)^kdz \right) = \sum_{\gamma \in \Gamma_w} \frac{(-1)^p}{(w-\overline{w})^{k}}  \binom{p+q}{p}  (X- \overline{w}Y)^p   (X-  w Y)^q  \Big|_{(\gamma,1)} . \qedhere
        \]
\end{proof}

\begin{rem}\label{rem:kappa}
    The number $\kappa_{\Gamma,w}^{p,q}$ is always an integer, and is either 0 or $|\Gamma_w|$. This follows from the classical orthogonality relations between irreducible characters of a finite group and from the formula $\kappa_{\Gamma,w}^{p,q} = |\Gamma_w| \langle \chi_{\Gamma,w}^{p,q}, 1 \rangle$, where $\chi_{\Gamma,w}^{p,q}$ is the character of $\Gamma_w$ corresponding to the 1-dimensional representation $\mathbb{C} (X-wY)^p (X-\overline{w}Y)^q$, and $1$ denotes the trivial character. This can also be deduced from the following explicit formulae. For a non-elliptic point $w$, we have $\Gamma_w = \{ I\}$ or $\Gamma_w = \{\pm I\}$, so that
    \[
        \kappa_{\Gamma,w}^{p,q}  = 1 \qquad \text{or}\qquad \kappa_{\Gamma,w}^{p,q}= 1 + (-1)^{p+q}  \  \in  \  \{0,2\} \quad \hbox{respectively}\ . 
    \]
    For the elliptic points, it is sufficient to consider the cases $w = i$ and $w = \rho = e^{\frac{2\pi i}{3}}$, since $\kappa_{\Gamma,w'}^{p,q} = \kappa_{\Gamma, w}^{p,q}$ for any $w' \in \Gamma w$. In the case $w=i$, we can assume $\Gamma_i = \{\pm I, \pm S\}$, where $S^2=-I$ which gives
    \[
        \kappa_{\Gamma,i}^{p,q} = 1 + (-1)^{p+q} + i^{p-q} + i^{q-p} \qquad \in \qquad \{0, 4\}
    \]
    For $w=\rho$, we can assume $\Gamma_{\rho} = \{I,(ST)^2,(ST)^4\}$ or $\Gamma_{\rho} = \{\pm I,\pm ST, \pm (ST)^2\}$, yielding
    \[
        \kappa_{\Gamma,\rho}^{p,q} = 1 + \rho^{p-q} + \rho^{q-p} \  \in  \ \{0,3\} \qquad \text{or} \qquad \kappa_{\Gamma,\rho}^{p,q} = (1+(-1)^{p+q})(1 + \rho^{p-q} + \rho^{q-p})  \ \in \  \{0,6\} \ .
    \]
\end{rem}

\begin{ex}\label{ex:psi0k-sl2}
    Let $\Gamma = \SL_2(\mathbb{Z})$ and $k \in \{2,4,6,8,12\}$, so that there are no holomorphic cusp forms of weight $k+2$ for $\Gamma$. Then, for a fixed $w$ in the upper half-plane, 
        \[
            \Psi^{0,k}_{\Gamma}(z,w) = \frac{2}{(w-\overline{w})^k}\frac{j'(w)}{E_{k+2}(w)}\frac{E_{k+2}(z)}{j(z)-j(w)},
        \]
    where $E_{k+2}$ denotes the classical Eisenstein series of level 1 normalised with first Fourier coefficient equal to $1$. To prove the above identity, it suffices to note that, as functions of $z$, both sides are meromorphic modular forms of weight $k+2$ which vanish at the cusp at infinity (\emph{cf.} Proposition \ref{prop:Psi}) and have simple poles with the same residues along $\Gamma w$ (\emph{cf.} Corollary \ref{cor:residue-Psi}). For applications of these forms, see  \cite{Magnetic}. 
\end{ex}

\section{Higher Green's functions}

In \S\ref{sect: existenceHigherGreensMatrix}, we establish the existence of higher vector-valued Green's functions, which may also be regarded as $(k+1)\times (k+1)$ matrix-valued Green's functions (Definition \ref{defn:greens-fct-matrix}). There are two cases depending on the parity of $k$. When $k$ is even, the central element of the matrix is proportional to  the classical higher Green's  function \eqref{eq:classical-GF}, from which the other entries may be derived. This leads to  an alternative  approach to the existence of the Green's matrix which is spelled out starting from \S\ref{sect: Legendre2ndkind}. The case when $k$ is odd seems to be new since it lies completely outside the scope of the Gross-Zagier conjecture, but is no less interesting. Explicit formulae for the matrix-valued Green's functions for $k=1,2$ are given in Appendix \ref{sect: AppendixC}.

\subsection{Definition of higher vector-valued Green's functions.} \label{sect: existenceHigherGreensMatrix} 

The existence of higher vector-valued Green's functions on the upper half-plane may  be established using cohomological methods and our results on the existence of harmonic lifts. We choose a direct but  more computational approach for expediency. 

\begin{defn} \label{defn: MatrixGreenfunctionsGeneratingFunction} Let $k>0$.  
Consider the following  expression 
    \begin{equation}
        p_k(z,w)  =  -  \sum_{p=1}^k  \frac{1}{p}  \left(  \frac{(X-zY)(U-wV)}{z-w}\right)^p  (VX-UY)^{k-p}   \ . 
    \end{equation} 
    It is a homogeneous polynomial in $X,Y,U,V$ of degree $k$ with coefficients in    $\QQ[z,w,(z-w)^{-1}]$.  Let us set 
    \begin{equation} \label{Gpolygenfunction} 
        g_k(z,w) =   p_k(z,w) - p_k(z,\overline{w}) - p_k(\overline{z}, w) + p_k(\overline{z},\overline{w})   +  
        (VX - UY)^k\,  \log \left( \frac{ |z-w|^2}{| z-\overline{w}|^2}\right) \ . 
    \end{equation}
\end{defn}

It follows from its definition that $g_k(z,w)$ satisfies
\begin{equation} \label{GeneralGkzwsymmetriesandlimit} 
    \begin{split}
    \frac{\partial^2g_k}{\partial z \partial \overline{z}} (z,w) &  =  \frac{\partial^2g_k}{\partial w \partial \overline{w}}(z,w)=0 \\ 
        g_k(z,w)(X,Y,U,V) & =  (-1)^k g_k(w,z)(U,V,X,Y)   \\ 
        g_k(z,w) + g_k(\overline{z},w) &=  g_k(z,w) + g_k(z,\overline{w}) = 0   \\ 
         \lim_{z\rightarrow \infty} g_k(z,w) & = \lim_{w\rightarrow \infty} g_k(z,w) = 0 .
    \end{split}
\end{equation}    
Furthermore, it satisfies for any invertible $2\times 2$ real matrix $\gamma$,
\[
    g_k(\gamma z, \gamma w)\big|_{(\gamma, \gamma)}= (\det\gamma)^k g_k(z,w)\ ,    
\]
where $(\gamma,\gamma)$ acts via $\gamma$  on the right on $X,Y, U, V$ simultaneously, in the usual manner. 

The vector-valued function $g_k(z,w)$ may be uniquely written in the form
\begin{equation}\label{Gkzwcomponentspqrs} 
    g_k(z,w) = \sum_{\substack{p+q=r+s=k \\ p,q,r,s\geq 0}}   g^{p,q}_{r,s}(z,w)   (X-zY)^r (X-\overline{z}Y)^s (U-wV)^p (U- \overline{w} V)^q  \ ,
\end{equation}
where $g_{r,s}^{p,q}(z,w) $  are explicitly   determined upon making  the following substitution: 
\begin{multline*}
    VX-UY  =   \frac{(\overline{z}-\overline{w})}{ (z-\overline{z}) (w-\overline{w}) } (X- zY)(U- w V)  + \frac{  (w-\overline{z})}{   (z-\overline{z}) (w-\overline{w}) } (X-zY)   (U-\overline{w}V) + \\ 
    + \frac{ (\overline{w}-z)}{  (z-\overline{z}) (w-\overline{w}) }   (X-\overline{z}Y)(U-wV)     + \frac{(z-w)}{ (z-\overline{z}) (w-\overline{w})  }(X-\overline{z}Y)  (U-\overline{w}V) \ . 
\end{multline*}
See Appendix \ref{sect: AppendixC} for examples in the cases $k=1,2.$ 

\begin{lem}
    Let us denote by 
    \begin{equation} \label{notation: xzw} 
        x(z,w) = \frac{(z-\overline{z})(w-\overline{w})}{(z-w)(\overline{z}-\overline{w})}  =  -  \frac{4 \,\mathrm{Im}(z) \mathrm{Im}(w)  }{|z-w|^2}    \ .   
    \end{equation}
    Then, for all integers $k>0$, we have the formula
    \begin{equation} \label{gkoformulaaslogarithm}
        g^{k,0}_{k,0}(z,w) =  \frac{1}{(z-w)^k} \,  \mathcal{L}_k(x(z,w)) \ ,
    \end{equation}
    where 
    \begin{equation} \label{truncatedlogarithm} \mathcal{L}_k(x) =  -    \sum_{p=1}^k  \frac{1}{p} \,  x^{p-k}   -   x^{-k}    \log ( 1-x  ) \qquad \Bigg(    =  \sum_{m\geq 1} \frac{1}{k+m} \, x^m\ \hbox{ for small } x    \Bigg),
    \end{equation}
    that is, $x^k\mathcal{L}_k(x)$ is the $k$th `Taylor remainder' at $x=0$ of $\operatorname{Li}_1(x) = -\log(1-x)$.
\end{lem}

\begin{proof}
    It follows from \eqref{Gkzwcomponentspqrs} that 
    the coefficient $g^{k,0}_{k,0}(z,w)$ satisfies
    \[
        g^{k,0}_{k,0}(z,w) =  \frac{1}{(\overline{z}-z)^k (\overline{w} - w)^k} \,  g_k(z,w) ( \overline{z} ,1 ,\overline{w},1)\ .
    \]
    Since  $p_k(z,w)$ has $(X-zY)(U-wV)$ as a factor, setting $(X,Y,U,V)= ( \overline{z} ,1 ,\overline{w},1)$ annihilates all terms in the right-hand side of \eqref{Gpolygenfunction} except the first ($p_k(z,w)$) and  the last (logarithmic term), giving:
    \[
        (\overline{z}-z)^k (\overline{w}-w)^k g^{k,0}_{k,0}(z,w) = - \sum_{p=1}^k  \frac{1}{p}  \left(  \frac{(\overline{z}-z)(\overline{w}-w)}{z-w}\right)^p  (\overline{z}- \overline{w})^{k-p}   +  (\overline{z}-\overline{w})^k\,  \log \left( \frac{ |z-w|^2}{| z-\overline{w}|^2}\right)
    \]
    and hence
    \[ 
        (z-w)^k \,  g^{k,0}_{k,0}(z,w) =    -   \sum_{p=1}^k  \frac{1}{p}  \left(\frac{(z-w)(\overline{z}-\overline{w})}{(\overline{z}-z)(\overline{w}-w)}  \right)^{k-p}    +     \left(\frac{(z-w)(\overline{z}-\overline{w})}{(\overline{z}-z)(\overline{w}-w)}  \right)^k \log \left( \frac{ |z-w|^2}{| z-\overline{w}|^2}\right) \, ,
    \]
    from which the stated formula follows on observing that $(1-x(z,w))^{-1} =  \frac{ (z-w)(\overline{z}-\overline{w})}{(z- \overline{w})(\overline{z}- w)   } .$
\end{proof}

\begin{prop} \label{prop: existsGm}
    The function $g_k(z,w)$ satisfies the following differential equations:
     \begin{equation} \label{DiffGzk}
        \begin{split}
        d_z  g_k(z,w)  &=      \psi_k(z,w)(U,V) (X-zY)^k dz -    \psi_k(\overline{z}, w ) (U,V)  (X- \overline{z} Y)^k d\overline{z} \\ 
        d_w  g_k(z,w) & =     -   \psi_k(w,z)(X,Y) (U-wV)^k dw  +   \psi_k(\overline{w}, z ) (X,Y) (U- \overline{w} V)^k d\overline{w}  \ . 
        \end{split}
    \end{equation}
    In particular, we have 
    \[ 
        \frac{\partial^2g_k}{\partial z \partial w}   (z,w) =  (k+1) \frac{(X-zY)^k (U-wV)^k}{(z-w)^{k+2}} \, , \qquad  \frac{\partial^2g_k}{\partial \overline{z} \partial w}   (z,w) =   - (k+1) \frac{(X-\overline{z}Y)^k (U-wV)^k}{(\overline{z}-w)^{k+2}} \, ,
    \]
    and similarly upon replacing $w$ with its complex conjugate (and multiplying the right-hand side by $-1$). 
\end{prop}

\begin{rem}
    By \eqref{eq:complex-conjugate-psi}, the first equation of \eqref{DiffGzk} may equivalently be written 
    \[
        d_z  g_k(z,w) = \psi_k(z,w)(U,V) (X-zY)^k dz +    \overline{\psi_k(z, w ) (U,V)  (X- z Y)^k dz}
    \]
    which shows, by analogy with the formula $d_z\log |z-w|^2 = \frac{dz}{z-w} + \overline{\frac{dz}{z-w}}$, that $g_k(z,w)$ is a `vector-valued version' of the single-valued logarithm function. 
\end{rem}

\begin{proof}
    We shall work formally with rational functions in  $X,Y,U,V$ whose coefficients are real analytic functions in $z,w$ although all expressions will, \emph{in fine}, be polynomials in $X,Y,U,V$. Define
    \[
        f_k(z,w) =  \frac{(X - zY)^k }{U-zV} \left(\frac{U-wV}{z-w} \right)^{k+1}   \, , \qquad r_k(z,w) = (VX-UY)^k \frac{U-wV}{(z-w) (U-zV)} \ . 
    \]
    Let us denote by $f_k^0(z,w) = f_k(z,w) - r_k(z,w)$. By equation \eqref{psikrational} we have 
    \[
        (X-zY)^k \, \psi_k(z,w) =  f_k(z,w) - f_k(z, \overline{w}) =   f^0_k(z,w) - f^0_k(z, \overline{w}) + \left( r_k(z,w) - r_k(z,\overline{w} ) \right) \ . 
    \]
    A formula for its primitive may be derived as follows. The final term may be simplified to 
     \[
        r_k(z,w) - r_k(z,\overline{w}) =  (VX - UY)^k \frac{(w-\overline{w})}{ (z-w)(z-\overline{w})}  \ ,  
    \]
    which has the following  multi-valued primitive with respect to the variable $z$: 
    \begin{equation} \label{primR}
        \ell_k(z,w)= (VX - UY)^k\,   \log \left( \frac{z-w}{z-\overline{w}}\right) \, . 
    \end{equation}
    To find a primitive of $f^0_k(z,w)$ with respect to $z$, we may write it in the form: 
    \begin{equation} \label{inproofF0sum} 
        f^0_k(z,w) =  \frac{(U-wV)(X-wY)}{(z-w)^2} \frac{(A^k -B^k )}{(A-B)} =  \frac{(U-wV)(X-wY)}{(z-w)^2} \left( \sum_{p+q=k-1} A^p B^q  \right) 
    \end{equation}
    where  
    \[
        A=  \frac{(X-zY)(U-wV)}{z-w} \,  ,  \qquad  B = VX-UY
    \]
    which satisfy $A-B= \frac{(X-wY)(U-zV)}{z-w}$. Each term in \eqref{inproofF0sum} may be written
    \[
        f^0_k(z,w) =   (X-wY)  (VX-UY)^{k}  \, \sum^{k-1}_{p=0}   \frac{(X-zY)^p}{(z-w)^{p+2}}   \frac{(U-wV)^{p+1}}{(VX-UY)^{p+1}}\ .
    \]
    By using the following identity:  
    \[ 
        \frac{d}{dz} \left(    \frac{1}{X-wY} \left(\frac{X-zY}{z-w} \right)^{p+1}  \right)  =    - (p+1) \,   \frac{(X-zY)^p}{(z-w)^{p+2}}      \ ,    
    \]
    we may write down the following  primitive of $f^0_k(z,w)$ with respect to $z$: 
    \[  
        P\! f^0_k (z,w) =  - (VX-UY)^{k}  \, \sum^{k-1}_{p=0}   \frac{1}{p+1} \left(\frac{(X-zY)(U-wV)}{(z-w)(VX-UY)} \right)^{p+1}  = p_k(z,w)    \ .
    \]
    Putting  the above  equations together, we conclude from \eqref{primR} that 
    \[ 
        \frac{d}{dz} \left(   p_k (z,w)  - p_k (z,\overline{w})   +      \ell_k(z,w)  \right)  =  (X-zY)^k \psi_k(z,w)\ .  
    \]
    Since the term in brackets on the left-hand side does not depend on $\overline{z}$, we may replace $z$ by $\overline{z}$ in the above expression to conclude an  analogous equation holds for $\overline{z}$. The first line of  \eqref{DiffGzk} follows from this (see \eqref{Gpolygenfunction}). The second follows by symmetry in $z$ and $w$ since $g_k(z,w)(X,Y,U,V) = (-1)^k  g_k(w,z)(U,V,X,Y).$

    Note that, although $\ell_k(z,w)$ is a multi-valued function, the difference $\ell_k(z,w) - \ell_k(\overline{z},w)$ is single-valued and in the end the function $g_k(z,w)$ is well-defined as a function of $z,w$ away from $z=w$, $z=\overline{w}$. 
\end{proof}

Now let $k>0$ and let $\Gamma$ satisfy the conditions of Definition \ref{defn: psiGamma}. The higher vector-valued Green's function may be obtained by averaging over the action of $\Gamma$. Recall that $V_{k}^{\B} \otimes V_{k}^{\B}$ is given by homogeneous polynomials of degree $k$ in $X,Y,U,V$. 

\begin{defn}
   For all  $(z,w) \in  \H^{\pm}\times \H^{\pm}$ with $z\not\in \Gamma w \cup \Gamma \overline{w}$, we define
    \begin{equation} \label{eq:G}
        \vec{G}_{\Gamma,k}(z,w)  =  \sum_{\gamma \in \Gamma}    g_k(\gamma z,w)\Big|_{(\gamma,1)} =  \sum_{\gamma \in \Gamma}    g_k( z, \gamma w)\Big|_{(1,\gamma)}
    \end{equation}
    where $(\gamma,1)$ acts on $(X,Y)$ only, and $(1, \gamma)$ acts only on $ (U, V)$. 
\end{defn} 

\begin{thm}
    The series \eqref{eq:G} converges normally on compact subsets of $\mathfrak{H}\times \mathfrak{H}$, and defines a real-analytic vector-valued function on the open subset $\mathfrak{U} =  \mathfrak{H} \times \mathfrak{H} \setminus \{(z,w) : z \not\in \Gamma w\}$: 
    \[
        \vec{G}_{\Gamma, k}(z,w):   \mathfrak{U}  \To (V_{k}^{\B} \otimes V_{k}^{\B})\otimes \mathbb{C},
    \]
    which is a harmonic lift of $\Psi_{\Gamma,k}(z,w)$ with respect to $z$ and of $\Psi_{\Gamma,k}(w,z)$ with respect to $w$. More precisely: 
    \begin{equation}
        \begin{split}
        d_z  \vec{G}_{\Gamma, k}(z,w) & =  \Psi_{\Gamma,k}(z,w)(U,V)(X-zY)^{k} dz - \overline{\Psi_{\Gamma,k}(z,\overline{w})}(U,V)(X-\overline{z}Y)^{k} d\overline{z}   \label{dzGfullversion} \\ 
        d_w  \vec{G}_{\Gamma, k}(z,w) &  =  \Psi_{\Gamma,k}(w,z)(X,Y)(U-wV)^{k} dw - \overline{\Psi_{\Gamma,k}(w,\overline{z})}(X,Y)(U-\overline{w}V)^{k} d\overline{w} \ .   
        \end{split}
    \end{equation} 
\end{thm}

\begin{proof}
    It suffices to prove the normal convergence of the component $\sum_{\gamma \in \Gamma}j^{-k}_{\gamma}(z)g^{k,0}_{k,0}(\gamma z, w)$ on compact subsets. The $\Gamma$-equivariance follows from the definition as a sum over $\Gamma$, and all of the other claims follow from the differential equations in Proposition \ref{prop: existsGm}, which also imply, in particular, that the other components $g^{p,q}_{r,s}$ of $g_k$ are determined by $g^{k,0}_{k,0}$.

    From Equation \eqref{gkoformulaaslogarithm}, we can write $g^{k,0}_{k,0}(z,w) = (z-w)^{-k}x(z,w)F(x(z,w))$, where $F$ is a real-analytic function of $x\in (-\infty,1]$, and $x(z,w)$ was defined  in \eqref{notation: xzw}. Let $K\subset \mathfrak{H} \times \mathfrak{H}$ be a compact subset. Since $\Gamma$ is discrete and acts properly on $\mathfrak{H}$, there is a finite subset $S \subset \Gamma$ such that $\gamma z \neq w$ for every $(z,w) \in K$ and every $\gamma \in \Gamma \setminus S$. This implies (for instance, by relating $x(z,w)$ to the hyperbolic distance between $z$ and $w$) that $x(\gamma z, w)$, and hence $F(x(\gamma z, w))$, is uniformly bounded for $(z,w) \in K$ and $\gamma \in \Gamma \setminus S$. Hence, we are left to prove the convergence of
    \[
        \sum_{\gamma \in \Gamma \setminus S} \sup_{(z,w) \in K}\frac{|x(\gamma z, w)|}{|j_{\gamma}^k(z)||\gamma z - w|^k} = \sum_{\gamma \in \Gamma \setminus S} \sup_{(z,w) \in K}\frac{4\Im(z)\Im(w)}{|j_{\gamma}^{k+2}(z)||\gamma z - w|^{k+2}},
    \]
    but this follows from the normal convergence on compact subsets of the Poincaré series defining $\Phi_{\Gamma,k+2}(z,w)$ (\emph{cf}. Equation \eqref{eq:Phi} and Proposition \ref{prop:Psi}). 
\end{proof}

\begin{rem}
    Consequently, the equivariant vector-valued functions $\vec{G}_{\Gamma,k}$ may be interpreted as two-fold equivariant Eichler integrals (with respect to $z$ and $w$) of the functions $\Phi_{\Gamma,k}(z,w)$.  It would be interesting to write them as `single-valued'   integrals \cite{brown-dupont} with respect to the measure given by $dz \wedge  d\overline{z} \wedge dw \wedge  d \overline{w}$.
\end{rem}

Let us denote the components of $\vec{G}_{\Gamma,k }(z,w)$ by 
\begin{equation} \label{Gpqrsdefn}
    \vec{G}_{\Gamma,k} (z,w)= \sum_{\substack{p+q=r+s=k \\ p,q,r,s\geq 0}}  {}^{\Gamma}\! G_{r,s}^{p,q} (z,w)  (X-zY)^r (X-\overline{z} Y)^s (U-wV)^p (U-\overline{w} V)^q \ . 
\end{equation}
It follows from the general discussion of Section \ref{sec: HarmonicLiftsGeneralities}  that the functions ${}^{\Gamma}\! G^{p,q}_{r,s}(z,w)$ are eigenfunctions of the hyperbolic Laplacian in both $z$ and $w$:
\[ 
    (\Delta_{r,s} + k)\,   {}^{\Gamma} \!G^{p,q}_{r,s}(z,w) =   (\Delta^{p,q} + k) \,  {}^{\Gamma}\!  G^{p,q}_{r,s}(z,w)= 0  \ , 
\]
are modular in $z$ with  weights $(r,s)$ and in $w$ with weights $(p,q)$, and satisfy the  following equations with respect to raising and lowering operators:
\[
    \partial_r \, {}^{\Gamma} \! G^{p,q}_{r,s}   = (r+1) \,  {}^{\Gamma} \! G^{p,q}_{r+1,s-1}  \ , \ \qquad   \overline{\partial}_s \,  {}^{\Gamma}\! G^{p,q}_{r,s}   = (s+1) \,  {}^{\Gamma}\! G^{p,q}_{r-1,s+1}  
\]
\[
    \ \partial^p \,    {}^{\Gamma} \! G^{p,q}_{r,s}   = (p+1) \,  {}^{\Gamma}\! G^{p+1,q-1}_{r,s}   \ , \  \qquad  \overline{\partial}^q \,    {}^{\Gamma}\! G^{p,q}_{r,s}   = (q+1) \,  {}^{\Gamma} \! G^{p-1,q+1}_{r,s}  \ . 
\]

\begin{cor} \label{cor: GisWHLofPsis}
    Let $k>0$. If we denote by   
    \[
        \vec{G}_{\Gamma,k}^{p,q}  (z,w) =  \sum_{\substack{r+s=k \\ r,s\ge 0}} {}^{\Gamma}\!G^{p,q}_{r,s}(z,w) (X-zY)^r (X- \overline{z}Y)^s  \ ,
    \]
    then the vector-valued functions $\vec{G}_{\Gamma,k}^{p,q}  (z,w) $ are harmonic lifts of the meromorphic modular forms $\Psi_{\Gamma}^{p,q}(z,w)$: 
    \begin{equation} \label{eqn: GisWHLofPsis}
        d_z   \vec{G}_{\Gamma,k}^{p,q}  (z,w)   =   \Psi_{\Gamma}^{p,q} (z,w) (X-zY)^k dz  +      \overline{\Psi_{\Gamma}^{q,p} (z,w)} (X-\overline{z}Y)^k d\overline{z}   .
    \end{equation}
    In particular, $-\Psi^{q,p}_{\Gamma}(z,w)$ is a Betti-conjugate of $\Psi^{p,q}_{\Gamma}(z,w)$.
\end{cor}

\begin{proof}
    Extract the coefficient of $(U-wV)^p (U-\overline{w} V)^q$ from the first equation of \eqref{dzGfullversion}, and use the fact that $\Psi_{\Gamma}^{p,q}(z,\overline{w}) = -\Psi_{\Gamma}^{q,p}(z, w)$. 
\end{proof}

This completes the existence and uniqueness of the higher generalised Green's functions, which, to our knowledge, are new in the case when $k$ is odd. When $k$ is even, it remains to relate the central elements to the `classical' higher Green's function. This will occupy the remaining paragraphs of this section. Alternatively, one might prove this using the properties which uniquely characterise  both functions.

\subsection{Legendre functions of the second kind}    \label{sect: Legendre2ndkind}

The functions $Q_{s}(t)$ defined  for $t>1$ and $s>-1$ by 
\begin{equation} \label{intro: Qst}
    Q_{s}(t) = \int_0^{\infty} \left( t + \sqrt{t^2-1} \cosh u\right)^{-s-1} du \ , 
\end{equation}
solve the second-order differential equation:
\begin{equation} \label{Legendrediffeq}
    (1-t^2)  Q''_{s}(t) - 2t Q'_{s}(t) +(s+1)sQ_{s}(t)=0. 
\end{equation}
We are interested in those functions for which $s=m\geq 0$ is an integer.

\begin{example}
    The first few cases are:
    \[
        Q_0 (t) =  \frac{1}{2} \log \left( \frac{t+1}{t-1} \right)  \ , \quad  
        Q_1(t)  =  \frac{t}{2} \log \left( \frac{t+1}{t-1} \right) -1   \ , \quad 
        Q_2(t)  =  \frac{3\,t^2-1}{4} \log \left( \frac{t+1}{t-1} \right)  - \frac{3t}{2}. \] 
\end{example}

In this situation,  the functions $Q_m(t)$ can also be defined via the  recurrence relation for all $m\geq 1$:
\begin{equation}  \label{Qrec}
    (m+1) Q_{m+1}(t) - (2m+1)\, t Q_m(t)  + m\,  Q_{m-1}(t)=0
\end{equation}
together with the initial values $Q_0, Q_1$ given above. In general one has $Q_m(t) \in \QQ[t] \oplus \QQ[t] Q_0$.

\begin{lem} \label{lem: Qdiff}
    For all $m\geq 0$ we have the identity
    \begin{equation} \label{Qdiffidentity}
        \frac{d^{m+1}}{dt^{m+1}}\,  Q_{m}(t) =   \frac{ 2^m \, m!} {(1- t^2)^{m+1}} \ .
    \end{equation} 
\end{lem} 

\begin{proof} 
    Since we were unable to find a reference in the literature, we give a proof. It may be  checked for $m=0,1$ directly. For $m\geq 1$, use \eqref{Qrec} to deduce that
    \[
        (m+1) \frac{d^{m+2}}{dt^{m+2}}\,  Q_{m+1}  =(2 m+1) \left(    t\, \frac{d^{m+2}}{dt^{m+2}}\,  Q_{m}  + (m+2) \frac{d^{m+1}}{dt^{m+1}}\,  Q_{m} \right)  -m  \frac{d^{m+2}}{dt^{m+2}}\,  Q_{m-1}\ .
    \]
    By induction, assume the statement holds for $Q_m, Q_{m-1}$. The right-hand side is
    \begin{multline*}
        (2m+1) \, t  \frac{d}{dt}  \left(\frac{2^m  \, m! } {(1- t^2)^{m+1}} \right) +(2m+1)(m+2)   \frac{ 2^m \, m! } {(1- t^2)^{m+1}}  -   m \frac{d^2}{dt^2}  \left(\frac{  2^{m-1}\, (m-1)! } {(1- t^2)^{m}}\right) \\
        = \frac{2^m\, m!}{(1-t^2)^{m+2}}\left(  2 (2m+1)(m+1)  t^2 +  (2m+1)(m+2)(1-t^2)   - m (1+ (2m+1)t^2)   \right)  \ .
    \end{multline*}
    The term in brackets simplifies to $2 (m+1)^2$ and completes the induction step. 
\end{proof}

\begin{rem}
    The other solutions to \eqref{Legendrediffeq} in the case when $s=m\geq 0$ is an integer are given by the Legendre polynomials $P_m(t) \in \QQ[t]$ which are of degree $m$. Consequently they satisfy $ \frac{d^{m+1}}{dt^{m+1}}\,  P_{m}(t) =0.$ A key property of the functions  $Q_m(t)$ is that they  vanish to order $m+1$ at $t=\infty$. 
\end{rem}

\subsection{Legendre equation and Laplacian}

Define the cross-ratio of $(z,w) \in \mathfrak{H}^{\pm}\times \mathfrak{H}^{\pm}$ to be: 
\begin{equation}
    \cc(z,w) \coloneqq \frac{(z-w)(\overline{z}-\overline{w}) }{ (z-\overline{z}) (w- \overline{w}) }  = -\frac{ 1}{4} \frac{ | z-w |^2}{\mathrm{Im}(z) \,\mathrm{Im}(w)} \ .
\end{equation}
It  satisfies $\cc(\gamma z, \gamma w) = \cc(z,  w)$  for all  $\gamma \in \GL_2(\RR)$. Note that $\cc(z,w)=  x(z,w)^{-1}$ (see \eqref{notation: xzw}) is symmetric, real and non-positive on $\mathfrak{H}\times \mathfrak{H}$: $\cc(z,w) = \cc(w,z) \leq 0$. We often write $\cc$ simply for $\cc(z,w)$ when there is no possible ambiguity.  

The following lemma is implicitly stated without proof  in \cite{GrossZagier}.
    
\begin{lem}  \label{lem: LaplacefromLegendre}
    Let $f : (1,+\infty)\rightarrow \CC$ be a twice differentiable function and let $\lambda \in \CC$ be any complex number. Recall that $\Delta_{0,0}$ denotes the hyperbolic Laplacian \eqref{Delta00}. Then the Laplace eigenvalue equation 
    \begin{equation} \label{Laplaceforfc} 
        \left(\Delta_{0,0} +\lambda\right)  f(1-2\cc)= 0 
    \end{equation}
    is equivalent to the Legendre second-order differential equation
    \begin{equation} \label{Legendreequation} 
        (1-t^2)  f''(t) - 2t f'(t) + \lambda\, f(t)= 0\ .
    \end{equation} 
\end{lem} 

\begin{proof}
    One checks that $\overline{\partial}_{-1} \partial_0 \cc = 2 \cc -1$ and  $\left(\partial_0 \cc \right) \left( \overline{\partial}_0 \cc \right) =\cc (\cc-1)$, or equivalently,
    \begin{equation} \label{ccdiffequationsinzandzb}
        (z-\overline{z})^2\, \frac{\partial^2 \cc}{\partial z \partial \overline{z}}   =    1- 2 \cc \quad \hbox{ and } \quad (z-\overline{z})^2\,  \frac{\partial \cc}{\partial z}  \frac{\partial \cc}{\partial \overline{z}}  =  \cc (1- \cc) \ .  
    \end{equation}
    We have
    \begin{align*}
        \Delta_{0,0}\,  f(1- 2\cc)   &\  =   -  2  (z- \overline{z})^2  \frac{\partial}{\partial \overline{z}}\left( f'(1- 2\cc) \frac{\partial \cc}{\partial z} \right) \nonumber \\ 
        &\ =   - 2  (z- \overline{z})^2 \left( - 2 f''(1- 2\cc)\frac{\partial \cc}{\partial  \overline{z}}\frac{\partial \cc}{\partial z}    + f'(1-2\cc)  \frac{\partial^2 \cc}{\partial  z  \partial \overline{z}}\right)  \nonumber \\
        & \stackrel{\eqref{ccdiffequationsinzandzb}}{=}   4  \cc(1-\cc)\,  f''(1- 2\cc)    -  2 (1-2 \cc)  f'(1- 2\cc) \ . \nonumber
    \end{align*} 
    If we set $t= 1 - 2\cc$ then  $(1-t^2) = 4 \cc(1-\cc)$ and the above computation  reduces to 
    \[
        \Delta_{0,0}\,  f(t) = (1-t^2 )f''(t) - 2t f'(t)\ ,
    \]
    which implies the equivalence of \eqref{Laplaceforfc} and \eqref{Legendreequation}. 
\end{proof} 
 
By symmetry in $z$ and $w$, equation \eqref{Laplaceforfc} is equivalent to the identical equation with respect to the hyperbolic Laplacian $\Delta^{0,0} =  (w- \overline{w})^2 \frac{\partial}{\partial w}\frac{\partial}{\partial \overline{w}}$ with respect to $w$.

\subsection{Central element of the higher Green's functions on $\mathfrak{H}^{\pm}\times \mathfrak{H}^{\pm}$}

By Lemma \ref{lem: LaplacefromLegendre}, the function 
\[
    f_m(z,w)\coloneqq -2 \, Q_{m} (1 - 2 \cc(z,w))\ , \qquad \hbox{ for  }  (z,w) \  \in \  \mathfrak{U}  \ ,
\] 
where $\mathfrak{U}= \mathfrak{H} \times \mathfrak{H} \setminus\{(z,w) : z = w\} $  is an eigenfunction of the hyperbolic Laplacian in both $z$ and $w$ with eigenvalue $m(m+1)$. By Lemma \ref{lem: LaplaceHarmonicCorrespondence}, it corresponds to a  vector-valued function which is  harmonic with respect to both variables $z$ and $w$. 

\begin{prop} \label{prop:CentralGmandGZ}
    Let $g_{2m}(z,w):\mathfrak{U} \rightarrow  (V_{2m}^{\B}\otimes V_{2m}^{\B})\otimes \mathbb{C}$ be as defined in \eqref{Gpolygenfunction}, written in the form  \eqref{Gkzwcomponentspqrs}. Its central component is proportional to $f_m$: 
    \begin{equation} \label{gfmmmm} 
        g^{m,m}_{m,m}(z,w)  =(-1)^m \binom{2m}{m}\,  \frac{  f_m(z,w)    }{(z- \overline{z})^{m} (w- \overline{w})^{m}}\ , 
    \end{equation} 
    and all other components $g^{p,q}_{r,s}(z,w)$ may be deduced from it by applying raising and lowering operators. 
\end{prop}

\begin{proof}
    It suffices to show that, if we define $h^{m,m}_{m,m}(z,w)$ by equation \eqref{gfmmmm}, and define the remaining components $ h^{p,q}_{r,s}(z,w)$  by applying raising and lowering operators to it, then the resulting generating function 
    \[
        h_{2m}(z,w) =  \sum_{\substack{ p+q=r+s=2m\\ p,q,r,s\geq 0 }} h^{p,q}_{r,s}(z,w) (X-zY)^r(X- \overline{z} Y)^s (U -wV)^p (U-\overline{w} V)^q
    \]
    satisfies the equations 
    \begin{equation} \label{eq:dzGm}
        \begin{split}
            d_z h_{2m}(z,w) & =  \psi_{2m} (z,w)(U,V) (X-zY)^{2m} dz  - \psi_{2m}(\overline{z},w)(U,V) (X-\overline{z}Y)^{2m} d\overline{z}  \\
            d_w h_{2m} (z,w) & = -  \psi_{2m} (w,z)(X,Y) (U-wV)^{2m} dw + \psi_{2m} (\overline{w},z)(X,Y) (U-\overline{w} V)^{2m} d\overline{w}   \ , 
        \end{split}
    \end{equation}
    since these equations, which are the same as \eqref{DiffGzk}, determine it uniquely up to a constant polynomial in $X,Y,U,V$. 
    This constant is in turn uniquely determined from either equation in the third line of \eqref{GeneralGkzwsymmetriesandlimit}, namely anti-invariance with respect to complex conjugation in $z$ (or $w$). 
    
    The fact that $h_{2m}(z,w)$ is harmonic in $z$ and $w$ is equivalent, by Lemma \ref{lem: LaplaceHarmonicCorrespondence} and  equation \eqref{eq:raising-lowering-middle} applied to both $z$ and $w$, to the system of equations
    \[ 
        \partial_r\, h_{r,s}^{p,q} =  (r+1)\, h_{r+1,s-1}^{p,q} \ , \qquad   \overline{\partial}_s\, h_{r,s}^{p,q} =   (s+1)\, h_{r-1,s+1}^{p,q}
    \]
    \[
        \partial^p\, h_{r,s}^{p,q}  =  (p+1)\, h_{r,s}^{p+1,q-1}  \ ,\qquad \overline{\partial}^q\, h_{r,s}^{p,q}  =    (q+1)\, h_{r,s}^{p-1,q+1} \ .
    \] 
    The consistency of these equations is equivalent to the fact that each $h$ is an eigenfunction of the Laplacian in both $z$ and $w$:
    \begin{equation} \label{gfrspqLaplace}
        \left(\Delta_{r,s} + 2m \right) h_{r,s}^{p,q} =   \left(\Delta^{p,q} + 2m\right) h_{r,s}^{p,q} =0  \ . 
    \end{equation} 
    Note  that $ h^{p,q}_{r,s}(z,w) = h^{r,s}_{p,q}(w,z)$  and  $h^{p,q}_{r,s}(z,w)=  h^{q,p}_{s,r}(\overline{z},\overline{w})$. By \eqref{eq:raising-lowering-end},   equation \eqref{eq:dzGm} is equivalent to:
   \begin{equation} \label{partial2mgfpq} 
        \partial_{2m} \, h^{p,q}_{2m,0} =    (z- \overline{z}) \psi^{p,q}(z,w)     \qquad  \hbox{ and } \qquad 
        \overline{\partial}_{2m} \, h^{p,q}_{0,2m}   =      (z- \overline{z}) \psi^{p,q}(\overline{z},w) \  = (z- \overline{z})  \overline{\psi^{p,q}(z,\overline{w})} \ .  \nonumber 
    \end{equation}  
    We prove  the  equation on the left, since the other follows by symmetry. We prove by induction on $r\geq m$ that for  all $2m\geq r \geq m \geq s \geq 0$ such that $r+s=2m$, 
    \begin{equation}\label{inproof: gfmrs} 
        h^{m,m}_{r,s} =   (-1)^m  \frac{(2m)!}{m!\,  r!}\,        (z- \overline{z})^{-m} (w- \overline{w})^{-m}  \left( (-2)^{r-m+1}  Q_m^{(r-m)}(1- 2\cc)\,  \cc_z^{r-m}\right) \ ,
    \end{equation}
    where  $Q_m^{(d)}(t)$ denotes the $d$-fold derivative $\frac{\partial^d}{\partial t^d} Q_m(t)$ and where we write
    \[
        \cc_z  =  \  \partial_0 \cc \  =    \frac{(w-\overline{z})(\overline{z}-\overline{w}) }{ (z-\overline{z}) (w- \overline{w}) }\ .
    \]
    Equation \eqref{inproof: gfmrs}  is  true by definition \eqref{gfmmmm} in the case $r=s=m$.  Using $
    \partial_r (ab)= \partial_r(a) b + a \partial_0(b)$, one verifies that if $F$ is  any differentiable function on the reals with derivative $F'(t)$, one has 
    \begin{equation}  \label{technicaldifferentiatesimplecase}
        \partial_r       \left(   \cc^{r}_z   \, F(1- 2\cc) \right)    = (\partial_r \cc^r_z ) F(1-2\cc)  \   -2  \,\cc^{r+1}_z  \, F'(1- 2\cc) = -2  \,\cc^{r+1}_z  \, F'(1- 2\cc) 
    \end{equation}  
    since $\partial_r \cc^r_z=0$ by \eqref{eqn:shift}. 
    Using this equation we check that
    \begin{align*}
        \partial_r  \,  (z- \overline{z})^{-m}   \left(  Q_m^{(r-m)}(1- 2\cc)\,  \cc_z^{r-m}\right)  &=  (z- \overline{z})^{-m}  \, \partial_{r-m}   \left(  Q_m^{(r-m)}(1- 2\cc)\,  \cc_z^{r-m}\right)   \nonumber  \\ 
        &=   (z- \overline{z})^{-m} \left(-2 \, Q_m^{(r+1-m)}(1- 2\cc)\,  \cc_z^{r+1-m} \right) \nonumber
    \end{align*} 
    and use the equation $ h^{m,m}_{r+1,s-1} = (r+1)^{-1} \partial_r h^{m,m}_{r,s}$ to complete the induction step.  A further application of \eqref{technicaldifferentiatesimplecase}  implies that
    \[
        \partial_{2m} \, h^{m,m}_{2m,0} 
        =  \frac{2^{m+2} }{m!}\,        (z- \overline{z})^{-m} (w- \overline{w})^{-m}  \left(   Q_m^{(m+1)}(1- 2\cc)\,  \cc_z^{m+1}\right)\ . 
    \]
    The differential equation \eqref{Qdiffidentity} for the Legendre function implies that
    \[
        Q_m^{(m+1)}(1- 2\cc) =  \frac{2^m m!  }{ (4 \cc (1- \cc))^{m+1}}\ .
    \]
    Using the fact that
    \[
        \frac{\cc_z}{\cc(1-\cc)}  =  \frac{ (z-\overline{z}) (w- \overline{w})   }{ (  z-w ) ( z- \overline{w})  }
    \]
    we conclude that
    \[
        \partial_{2m} \, h^{m,m}_{2m,0}  =  \frac{ (z-\overline{z}) (w- \overline{w})   }{ (  z-w)^{m+1} ( z-\overline{w})^{m+1}  } = (z- \overline{z}) \,  \psi^{m,m}(z,w)
    \]
    which proves $\partial_{2m} \, h^{p,q}_{2m,0} =    (z- \overline{z}) \psi^{p,q}(z,w)  $   in the case $p=q=m$. All other instances for $p+q=2m$, $p,q\geq 0$ follow by applying raising and lowering operators with respect to $w$, for instance:
    \[
        \partial^{p} \partial_{2m} \,  h^{p,q}_{2m,0}   = \partial_{2m} \partial^{p}  \,  h^{p,q}_{2m,0}   = (p+1)\,  \partial_{2m}\, h^{p+1,q-1}_{2m,0}
    \]
    when  $q\geq1 $, and comparing with the equation  $\partial^p\psi^{p,q}=(p+1) \psi^{p+1,q-1}$.  The second equation of \eqref{eq:dzGm} follows similarly from $\partial^{2m} \, h^{2m,0}_{r,s} =  (w- \overline{w})  \psi^{r,s}(w,z)$ and its complex conjugate $\overline{\partial}^{2m} \, h^{2m,0}_{r,s} =  (w- \overline{w})  \overline{ \psi^{r,s}(w,\overline{z})}. $

    Finally,  we must check the anti-invariance in $z$ (or $w$).  For this, note that $\cc(\overline{z}, w) = 1- \cc(z,w)$ and hence $1- 2 \cc(\overline{z},w) = 2 \cc(z,w)-1$. The anti-invariance of $(z-\overline{z})^{-m} (w-\overline{w})^{-m} f_m(z,w) $ follows from the equation $Q_m(t) =  (-1)^{m+1} \,Q_m(-t)$ for solutions to the Legendre equation. By applying the raising and lowering operators, we deduce that  $h^{p,q}_{r,s}(\overline{z},w) + h^{q,p}_{r,s}(z,w)=0$  for all $p,q,r,s$, which completes the proof. 
\end{proof}

\begin{rem}
    The raising and lowering operators imply that:
    \[
        g^{m,m}_{m,m}(z,w) = \frac{1}{(m!)^2}   (\overline{\partial}_{m+1}  \overline{\partial}^{m+1})  \cdots  (\overline{\partial}_{2m}  \overline{\partial}^{2m} ) \,  g^{2m,0}_{2m,0} (z,w)  \ ,  
    \]
    which, in view of the explicit formulae \eqref{gkoformulaaslogarithm} and  \eqref{truncatedlogarithm}, provides another perspective on the Legendre functions. 
\end{rem}

\begin{cor}
    Let $k=2m$ be even. Then the central component of the higher vector-valued Green's function 
     \[
        \vec{G}_{\Gamma, 2m}(z,w):   \mathfrak{U}  \To (V_{2m}^{\B} \otimes V_{2m}^{\B})\otimes \mathbb{C}
    \]
    is proportional to the `classical' higher Green's function \eqref{eq:classical-GF}: 
    \begin{equation} \label{GFCentralequalsGZ} 
        {}^{\Gamma}\! G^{m,m}_{m,m}(z,w)  =(-1)^m \binom{2m}{m}\,  \frac{  G_{\Gamma,m+1}(z,w)   }{(z- \overline{z})^{m} (w- \overline{w})^{m}}\ .  
    \end{equation} 
\end{cor}

\section{Meromorphic modular forms and cohomology}\label{sec:cohomology-over-C}

In this section, we describe the algebraic de Rham cohomology over $\mathbb{C}$ of a punctured modular curve with coefficients, together with its Hodge filtration,  in terms of meromorphic modular forms with prescribed pole orders. We show in particular that the meromorphic modular forms $\Psi_{\Gamma}^{p,q}(z,w)$ define a canonical  splitting of the Hodge filtration. 

\subsection{Cohomology of the universal elliptic curve over $\mathfrak{H}$ and vector-valued functions} \label{par:cohomology-EC}

Let $p:\mathcal{E} \to \mathfrak{H}$ be the universal framed elliptic curve. Its fibre over $z \in \mathfrak{H}$ is the complex torus
\[
    \mathcal{E}_z = p^{-1}(z) =  \mathbb{C}/(\mathbb{Z} + \mathbb{Z}z)
\]
equipped with the $\mathbb{Z}$-basis $(\sigma_{1,z},\sigma_{2,z}) = (1,z)$ of the first homology group $H_1(\mathcal{E}_z,\mathbb{Z}) \cong \mathbb{Z} + \mathbb{Z}z$. The left action of $\SL_2(\mathbb{Z})$ on $\mathfrak{H}$ lifts to an action on the total space $\mathcal{E}$,
\[
    \mathcal{E}_z \stackrel{\sim}{\To} \mathcal{E}_{\gamma z}\text{, }\qquad u + \mathbb{Z} + \mathbb{Z}z\longmapsto j^{-1}_{\gamma}(z)u + \mathbb{Z} + \mathbb{Z}\gamma z,
\]
which transforms the above homology basis by $\gamma_*(\sigma_{1,z},\sigma_{2,z}) = (a\sigma_{1,\gamma z} - c\sigma_{2,\gamma z},-b \sigma_{1,\gamma z} + d\sigma_{2,\gamma z})$.

Let $\mathbb{V}_k^{\B}$ be the trivial $\mathbb{Q}$-local system over $\mathfrak{H}$ with fibre $V_k^{\B}$ as defined in \S \ref{par:harmonic}. There are unique isomorphisms
\begin{equation}\label{eq:betti-local-system}
    \mathbb{V}_k^{\B} \cong \Sym^k R^1p_*\mathbb{Q}_{\mathcal{E}}
\end{equation}
compatible with the symmetric product structure and which, for $k= 1$, identifies $(X,Y)$ with the frame on singular cohomology with rational coefficients which is dual to $(\sigma_{1,z},-\sigma_{2,z})$. Under  \eqref{eq:betti-local-system}, the right $\SL_2(\mathbb{Z})$-action \eqref{eq:right-action-SL2} on $V_k^{\B}$ is  dual to the natural left $\SL_2(\mathbb{Z})$-action on homology.

Let
\[
    \mathcal{V}_k^{\dR,\an} \defeq \Sym^k H^1_{\dR}(\mathcal{E}/\mathfrak{H}).
\]
Note that $H^1_{\dR}(\mathcal{E}/\mathfrak{H})$ is a rank 2 holomorphic vector bundle over $\mathfrak{H}$ with fibre $H^1_{\dR}(\mathcal{E}_z)$ at $z \in \mathfrak{H}$. In particular, $\mathcal{V}_k^{\dR,\an}$ is a holomorphic vector bundle of rank $k+1$ over $\mathfrak{H}$. It comes with a Hodge filtration 
\[
    \mathcal{V}_k^{\dR,\an} = F^0\mathcal{V}_k^{\dR,\an} \supset F^1\mathcal{V}_k^{\dR,\an} \supset \cdots \supset F^k\mathcal{V}_k^{\dR,\an} \supset F^{k+1}\mathcal{V}_k^{\dR,\an} = 0
\]
by holomorphic sub-bundles which are 
given by the $k$th symmetric power of the Hodge filtration
\begin{equation}\label{eq:hodge}
    \mathcal{V}_1^{\dR,\an} = F^0\mathcal{V}_1^{\dR,\an} \supset F^1\mathcal{V}_1^{\dR,\an} \cong p_*\Omega^1_{\mathcal{E}/\mathfrak{H}} \supset F^2\mathcal{V}_1^{\dR,\an} = 0 \ . 
\end{equation}
It also has an integrable holomorphic connection $\nabla^{\an}_k$ given by the $k$th symmetric power of the Gauss-Manin connection on $\mathcal{V}_1^{\dR,\an}$. 

It follows from \eqref{eq:betti-local-system} that $\mathbb{V}_k^{\B}\otimes \mathbb{C}$ is isomorphic to the $\mathbb{C}$-local system of horizontal sections of $\nabla^{\an}_k$, thus
\[
    (\mathcal{V}_k^{\dR,\an},\nabla^{\an}_k) \cong (V_k^{\B}\otimes \mathcal{O}_{\mathfrak{H}}, \id \otimes d),
\]
where $\mathcal{O}_{\mathfrak{H}}$ denotes the sheaf of holomorphic functions on $\mathfrak{H}$. In particular, holomorphic (resp. $C^{\infty}$, resp. real-analytic) vector-valued functions $F: \mathfrak{U} \to V_{k}^{\B}\otimes \mathbb{C}$ as discussed in \S \ref{par:harmonic} correspond to holomorphic (resp. $C^{\infty}$, resp. real-analytic) sections of $\mathcal{V}_{k}^{\dR,\an}$ over the open subset $\mathfrak{U}\subset \mathfrak{H}$, the action of $\nabla^{\an}_k$ on $V_k^{\B}\otimes \mathcal{O}_{\mathfrak{H}}$ being simply differentiation with respect to the variable $z \in \mathfrak{U}$ and acting by zero on the frame $X,Y$. The Betti structure induces a complex conjugation on $\mathcal{V}_{k}^{\dR,\an}$, as in \S\ref{par:complex-conjugation}, which we simply denote by $\alpha \mapsto \overline{\alpha}$. It corresponds to the naive complex conjugation of vector-valued functions. 

The Hodge line bundle $p_*\Omega^1_{\mathcal{E}/\mathfrak{H}}$ is trivialised by a global section which restricts in every fibre $\mathcal{E}_z = \mathbb{C}/(\mathbb{Z} + \mathbb{Z} z)$ to the 1-form $du$, where $u$ denotes the coordinate on $\mathbb{C}$. Seen as a vector-valued holomorphic function in the Betti frame $(X,Y)$ via $F^1\mathcal{V}_1^{\dR,\an} \subset \mathcal{V}_1^{\dR,\an} \cong V_k^{\B}\otimes \mathcal{O}_{\mathfrak{H}}$, this trivialisation is the class $[du]$, which via the integration pairing sends $(\sigma_{1,z}, \sigma_{2,z}) \mapsto (1,z)$, and therefore corresponds to $X-zY$:
\begin{equation} \label{ducorrespondsXzY}
    [du]  =   X - zY.
\end{equation}
Consequently,  $\mathcal{V}_k^{\dR,\an}$ admits a real-analytic frame 
\[
    (X-zY)^r(X-\overline{z}Y)^s\text{, }\qquad r+s =k\text{, }\  r,s\ge 0
\]
which splits the Hodge filtration \eqref{eq:hodge}. Expressing a vector-valued function in this frame corresponds to the decomposition \eqref{eq:Frs}. In particular, $F^p\mathcal{V}_k^{\dR,\an} $ is spanned by $ (X-zY)^r(X-\overline{z}Y)^s$ for $r\geq p$.

Note that $\mathcal{V}_1^{\dR,\an}$ admits the holomorphic frame $X-zY\text{, } -Y$. To see that $-Y$ is holomorphic, observe that $\nabla^{\an}_1 (X-zY) = (-Y)  dz$ is the image under $\nabla^{\an}_1$ of the holomorphic section   \eqref{ducorrespondsXzY} and hence holomorphic. In the above real-analytic frames, we can also write
\[
    -Y =  \frac{1}{z - \overline{z}}((X-zY) - (X-\overline{z}Y)).
\]
In particular,
\[
    (X-zY)^a(-Y)^b\text{, } \qquad a+b=k\text{, } \ a,b\ge 0
\]
is a holomorphic frame of $\mathcal{V}_k^{\dR,\an}$.

\begin{rem}
    Under the isomorphism $H^1_{\dR}(\mathcal{E}_{z}) \cong H^1_{\dR}(\mathcal{E}_z \setminus \{0\})$, which allows us to represent cohomology classes by forms of the second kind, the restriction of $-Y$ to the fibre at $z$ is given by the one form 
    \[
        \left(\wp_{z}(u) - \frac{E_2(z)}{12} \right)2\pi i \, du,
    \]
     which is of the second kind since it has a double pole at $u\in \ZZ + \ZZ z$, and where $E_2(z)$ is the weight 2 Eisenstein series of level 1 with first Fourier coefficient equal to 1; see \cite[\S A1]{katz}.
\end{rem}

\subsection{Analytic de Rham cohomology with coefficients for  open modular curves} \label{par:analytic-dR-open}

Let $\Gamma \le \SL_2(\mathbb{Z})$ be a finite-index subgroup and $\mathfrak{U}\subset \mathfrak{H}$ be an open subset given by the complement of a finite number of distinct $\Gamma$-orbits in $\mathfrak{H}$:
\[
    \mathfrak{U} = \mathfrak{H} \setminus S\ , \qquad \hbox{ where  }   S = \Gamma w_1 \cup \cdots \cup \Gamma w_r \ ,
\]
with $w_i \not\in \Gamma w_j$ for every $1\le i<j \le r$.  The restriction of the universal family
\[
    p|_{\mathfrak{U}} : p^{-1}(\mathfrak{U}) \To \mathfrak{U} ,
\]
and of the corresponding sheaves $\mathbb{V}_{k}^{\B}$, $(\mathcal{V}_k^{\dR,\an},\nabla^{\an}_k)$, and $F^{\bullet}\mathcal{V}_k^{\dR,\an}$ descend to the orbifold quotient
\[
    \mathcal{U}^{\an}_{\Gamma} \defeq \Gamma\backslash\!\! \backslash \mathfrak{U}\, .
\]
Note that $\mathcal{U}_{\Gamma}^{\an}$ is the open sub-orbifold of the open modular curve $\mathcal{Y}_{\Gamma}^{\an}\defeq \Gamma \backslash \!\!\backslash \mathfrak{H}$ obtained by `puncturing'  the points of
\[
    \Gamma \backslash \!\! \backslash S = \{\Gamma w_1,\ldots,\Gamma w_r\}\, .
\]
A $\Gamma$-invariant holomorphic (resp. real-analytic) vector-valued function $F: \mathfrak{U} \to V_{k}^{\B}\otimes \mathbb{C}$ then corresponds to a holomorphic (resp. real-analytic) section of $\mathcal{V}_k^{\dR,\an}$ over $\mathcal{U}^{\an}_{\Gamma}$.

We shall now describe $H^1_{\dR}(\mathcal{U}^{\an}_{\Gamma};\mathcal{V}_{k}^{\dR,\an})$ in terms of functions $f:\mathfrak{U} \to \mathbb{C}$ which are modular for $\Gamma$. If the action of $\Gamma$ on $\mathfrak{H}$ is not free, we may always choose a finite-index normal subgroup $\Gamma_0\le \Gamma$ which acts freely on $\mathfrak{H}$. Then, $\mathcal{U}^{\an}_{\Gamma_0} = \Gamma_0\backslash \mathfrak{U}$ is a \emph{bona fide} Riemann surface (so that the theory of Appendix \ref{sec:complex-conjugation} directly applies), and we may take as a definition
\[
    H^1_{\dR}(\mathcal{U}^{\an}_{\Gamma};\mathcal{V}_{k}^{\dR,\an}) \defeq H^1_{\dR}(\mathcal{U}^{\an}_{\Gamma_0};\mathcal{V}_{k}^{\dR,\an})^{G},
\]
where $G \defeq \Gamma/\Gamma_0$. It does not depend on the choice of $\Gamma_0$. In general, any geometric construction on $\mathcal{U}^{\an}_{\Gamma}$ will be given by the  $G$-invariants of the analogous construction over $\mathcal{U}^{\an}_{\Gamma_0}$.

By Lemma \ref{lem:analytic-smooth-cohomology}, $H^1_{\dR}(\mathcal{U}^{\an}_{\Gamma_0};\mathcal{V}_{k}^{\dR,\an})$ can be computed as the first cohomology group of the complex
\[
    \begin{tikzcd}[column sep = small]
        \Gamma(\mathcal{U}^{\an}_{\Gamma_0}, \mathbb{V}_k^{\B}\otimes \mathcal{A}^0_{\mathcal{U}^{\an}_{\Gamma_0}}) \arrow{r}{\id \otimes d} & \Gamma(\mathcal{U}^{\an}_{\Gamma_0}, \mathbb{V}_k^{\B}\otimes \mathcal{A}^1_{\mathcal{U}^{\an}_{\Gamma_0}}) \arrow{r}{\id \otimes d} & \Gamma(\mathcal{U}^{\an}_{\Gamma_0}, \mathbb{V}_k^{\B}\otimes \mathcal{A}^2_{\mathcal{U}^{\an}_{\Gamma_0}}).
    \end{tikzcd}
\]
Since
\[
    \Gamma(\mathcal{U}^{\an}_{\Gamma_0}, \mathbb{V}_k^{\B}\otimes \mathcal{A}^p_{\mathcal{U}^{\an}_{\Gamma_0}}) \cong \Gamma(\mathfrak{U}, V_k^{\B}\otimes \mathcal{A}^p_{\mathfrak{U}})^{\Gamma_0},
\]
 we see by taking $G$-invariants that a cohomology class in $H^1_{\dR}(\mathcal{U}^{\an}_{\Gamma};\mathcal{V}_{k}^{\dR,\an})$ can be represented by a $\Gamma$-equivariant vector-valued 1-form
\[
    \alpha = A_1(z)dz + A_2(z)d\overline{z},
\]
where $A_1,A_2 : \mathfrak{U} \to V_k^{\B}\otimes \mathbb{C}$ are $C^{\infty}$ vector-valued functions, satisfying
\[
    d\alpha = \left(\frac{\partial A_2}{\partial z} - \frac{\partial A_1}{\partial \overline{z}}\right)dz\wedge d\overline{z} = 0.
\]
Moreover, $[\alpha]=0$ if and only if $\alpha = dF$ for some $\Gamma$-equivariant,  $C^{\infty}$, vector-valued function $F: \mathfrak{U} \to V_k^{\B}\otimes \mathbb{C}$.

Since $\mathcal{U}^{\an}_{\Gamma_0}$ is Stein, $H^1_{\dR}(\mathcal{U}^{\an}_{\Gamma_0};\mathcal{V}_{k}^{\dR,\an})$ can also be computed as a  cokernel 
\[
 H^1_{\dR}(\mathcal{U}^{\an}_{\Gamma_0};\mathcal{V}_{k}^{\dR,\an})= \coker\Big(  \!\! \begin{tikzcd}[column sep = small]
        \Gamma(\mathcal{U}^{\an}_{\Gamma_0}, \mathbb{V}_k^{\B}\otimes \mathcal{O}_{\mathcal{U}^{\an}_{\Gamma_0}}) \arrow{r}{\id \otimes d} & \Gamma(\mathcal{U}^{\an}_{\Gamma_0}, \mathbb{V}_k^{\B}\otimes \Omega^1_{\mathcal{U}^{\an}_{\Gamma_0}})
    \end{tikzcd} \!\! \Big)
\]
which means, after taking $G$-invariants, that every cohomology class in $H^1_{\dR}(\mathcal{U}^{\an}_{\Gamma};\mathcal{V}_{k}^{\dR,\an})$ can equally  be represented by a $\Gamma$-equivariant vector-valued 1-form
\[
    \alpha = A(z)dz
\]
where $A: \mathfrak{U} \to V_k^{\B}\otimes \mathbb{C}$ is holomorphic. The next proposition implies that one can always take $A$ to be of the form
\[
    A(z) = f(z)(X-zY)^{k},
\]
for some holomorphic function $f:\mathfrak{U} \to \mathbb{C}$ which is  modular for $\Gamma$ of weight $k+2$.

\begin{prop}\label{prop:modular-representative}
    The natural map $\Gamma(\mathcal{U}^{\an}_{\Gamma_0}, F^k\mathcal{V}_k^{\dR,\an}\otimes \Omega^1_{\mathcal{U}^{\an}_{\Gamma_0}}) \to H^1_{\dR}(\mathcal{U}^{\an}_{\Gamma_0}; \mathcal{V}_k^{\dR,\an})$ is surjective.
\end{prop}

\begin{proof}
    Consider the complex of abelian sheaves $C = \Omega^{\bullet}_{\mathcal{U}_{\Gamma_0}^{\an}}(\mathcal{V}_k^{\dR,\an})$ and let $C = F^0C\supset F^1C\supset \cdots$ be the filtration by subcomplexes induced by the Hodge filtration on $\mathcal{V}_k^{\dR,\an}$. Explicitly, $F^pC$ is the complex concentrated in degrees $0,1$:
     \[
        F^pC : \qquad  F^p\mathcal{V}^{\dR,\an}_k  \overset{\nabla_k^{\an}}{\longrightarrow} F^{p-1}\mathcal{V}^{\dR,\an}_k \otimes \Omega^1_{\mathcal{U}_{\Gamma_0}^{\an}} \ . 
    \]
    Note that $F^{k+2}C = 0$. Moreover, for $1\le p \le k$, the graded piece $F^pC/F^{p+1}C$ is acyclic, since $\nabla_k^{\an}$ induces an isomorphism of $\mathcal{O}_{\mathcal{U}_{\Gamma_0}^{\an}}$-modules
    \begin{align} \label{Deltaancomputation}
        \begin{split}
        \nabla_k^{\an} : F^p\mathcal{V}_k^{\dR,\an}/ F^{p+1}\mathcal{V}_k^{\dR,\an} &\stackrel{\sim}{\longrightarrow} (F^{p-1}\mathcal{V}_k^{\dR,\an} / F^p\mathcal{V}_k^{\dR,\an})\otimes \Omega^1_{\mathcal{U}_{\Gamma_0}^{\an}}\\
        [f(z)(X-zY)^p(-Y)^{k-p}] &\longmapsto [pf(z)(X-zY)^{p-1}(-Y)^{k-p +1} dz] \  . 
        \end{split}
    \end{align}
    It follows that the first page $E^{p,q}_1 = H^{p+q}(F^pC/F^{p+1}C)$ of the spectral sequence associated to the filtered complex $(C,F^{\bullet}C)$ has only two non-vanishing terms: 
    \[
        E_1^{0,0} \cong \mathcal{V}_k^{\dR,\an}/F^1\mathcal{V}_k^{\dR,\an}\qquad \text{and}\qquad E_1^{k+1,-k} \cong F^k\mathcal{V}_k^{\dR,\an}\otimes \Omega^1_{\mathcal{U}_{\Gamma_0}^{\an}} \ ,
    \]
    and degenerates at page $k+2$. Thus, $E_{k+2}^{k+1,-k} \cong H^1(C)$ or, equivalently, we obtain a short exact sequence of abelian sheaves over $\mathcal{U}_{\Gamma_0}^{\an}$:
    \[
        \begin{tikzcd}[column sep = small]
            0 \arrow{r} & \mathcal{V}_k^{\dR,\an}/F^1\mathcal{V}_k^{\dR,\an} \arrow{r}{d_{k+1}} \arrow{r} & F^k\mathcal{V}_k^{\dR,\an}\otimes \Omega^1_{\mathcal{U}_{\Gamma_0}^{\an}} \arrow{r} &(\mathcal{V}_k^{\dR,\an}\otimes \Omega_{\mathcal{U}^{\an}_{\Gamma_0}}^1)/\nabla_k^{\an}(\mathcal{V}_k^{\dR,\an}) \arrow{r} & 0 \ . 
        \end{tikzcd}
    \]
    Since $\mathcal{U}_{\Gamma_0}^{\an}$ is Stein, we conclude by taking global sections in the above exact sequence (\emph{cf.} Example \ref{ex:cohomology-stein}).
\end{proof}

\begin{cor}\label{cor:existence-harmonic-lift}
    Every holomorphic function $f: \mathfrak{U} \to \mathbb{C}$ which is modular for $\Gamma$ of weight $k+2$ admits a harmonic lift $F: \mathfrak{U} \to V_{k}^{\B}\otimes \mathbb{C}$ as in Definition \ref{defn:harmoniclift}.
\end{cor}

\begin{proof}
    By an immediate application of Propositions \ref{prop:modular-representative} and \ref{prop:geometric-harmoniclift}, any $f$ as in the statement admits a $\Gamma_0$-equivariant harmonic lift $F_0$. Then, a $\Gamma$-equivariant harmonic lift can be defined by
    \[
        F(z) = \frac{1}{|G|}\sum_{g \in G}F_0(gz)|_g \ .\qedhere
    \]
\end{proof}

\subsection{Algebraicity and meromorphic modular forms} \label{sect: AlgMeroForms}

We keep the notation of the previous paragraph. Since $\Gamma_0$ has finite index in $\SL_2(\mathbb{Z})$, the Riemann surface $\mathcal{Y}_{\Gamma_0}^{\an}\defeq \Gamma_0\backslash \H$ is algebraic: there is a smooth algebraic curve $\mathcal{Y}_{\Gamma_0}$ over $\CC$ (in fact,  over a number field, see \S \ref{par: deRham-nb-field},  but in this section we work over $\CC$) whose analytification is $\mathcal{Y}_{\Gamma_0}^{\an}$. Since $\mathcal{U}^{\an}_{\Gamma_0}$ is obtained from $\mathcal{Y}^{\an}_{\Gamma_0}$ by deleting a finite number of points (namely, those of $\Gamma_0\backslash S$), it is also algebraic: $\mathcal{U}_{\Gamma_0}^{\an}$ is the analytifcation of an open subscheme $\mathcal{U}_{\Gamma_0} \subset \mathcal{Y}_{\Gamma_0}$.

The action of the finite group $G$ on $\mathcal{U}^{\an}_{\Gamma_0}$ and on $\mathcal{Y}^{\an}_{\Gamma_0}$ is algebraic as well, \emph{i.e.}, defined over $\mathcal{U}_{\Gamma_0}$ and $\mathcal{Y}_{\Gamma_0}$. Taking the stack-theoretic quotient by $G$, we obtain an open immersion of smooth 1-dimensional Deligne-Mumford stacks over $\mathbb{C}$
\[
    \mathcal{U}_{\Gamma} \defeq [G\backslash \mathcal{U}_{\Gamma_0}] \hookrightarrow \mathcal{Y}_{\Gamma} \defeq [G\backslash \mathcal{Y}_{\Gamma_0}]
\]
whose analytification is
\[
    \mathcal{U}^{\an}_{\Gamma} \cong \Gamma \backslash \!\!\backslash \mathfrak{U} \hookrightarrow \mathcal{Y}^{\an}_{\Gamma} \cong  \Gamma \backslash \!\!\backslash \mathfrak{H} \ .
\]

The holomorphic vector bundle with integrable connection $(\mathcal{V}_k^{\dR,\an},\nabla^{\an}_k)$ and the Hodge filtration $F^{\bullet}\mathcal{V}_k^{\dR,\an}$ are also algebraic, in the sense that they are the analytification of analogous algebraic objects $(\mathcal{V}_k^{\dR,\mathbb{C}},\nabla^{\mathbb{C}}_k)$ and $F^{\bullet}\mathcal{V}_k^{\dR,\mathbb{C}}$ over $\mathcal{Y}_{\Gamma_0}$. By taking $G$-invariants, they descend to the stack $\mathcal{Y}_{\Gamma}$. Since $(\mathcal{V}_k^{\dR,\mathbb{C}},\nabla^{\mathbb{C}}_k)$ has regular singularities at infinity, its algebraic de Rham cohomology is isomorphic to the analytic de Rham cohomology:
\begin{equation}
    H^1_{\dR}(\mathcal{U}_{\Gamma_0}; \mathcal{V}_k^{\dR,\mathbb{C}}) \cong H^1_{\dR}(\mathcal{U}^{\an}_{\Gamma_0};\mathcal{V}_{k}^{\dR, \an}) \ ,
\end{equation}
and, by taking $G$-invariants, we obtain a similar statement for $\mathcal{U}_{\Gamma}$:
\begin{equation}\label{eq:cohomolgy-open-complex}
    H^1_{\dR}(\mathcal{U}_{\Gamma};\mathcal{V}_{k}^{\dR,\mathbb{C}}) \defeq H^1_{\dR}(\mathcal{U}_{\Gamma_0}; \mathcal{V}_k^{\dR,\mathbb{C}})^G \cong H^1_{\dR}(\mathcal{U}^{\an}_{\Gamma};\mathcal{V}_{k}^{\dR, \an}) \ .
\end{equation}

\begin{prop} \label{prop:modular-representative-algebraic}
    The natural map $\Gamma(\mathcal{U}_{\Gamma_0}, F^k\mathcal{V}_k^{\dR,\mathbb{C}}\otimes \Omega^1_{\mathcal{U}_{\Gamma_0}/\mathbb{C}}) \to H^1_{\dR}(\mathcal{U}_{\Gamma_0}; \mathcal{V}_k^{\dR,\mathbb{C}})$ is surjective.
\end{prop}

\begin{proof}
    The proof is completely analogous to that of Proposition \ref{prop:modular-representative}, using the algebraic de Rham complex $\Omega^{\bullet}_{\mathcal{U}_{\Gamma_0}/\mathbb{C}}(\mathcal{V}_k^{\dR,\mathbb{C}})$ in  place of $\Omega^{\bullet}_{\mathcal{U}^{\an}_{\Gamma_0}}(\mathcal{V}_k^{\dR,\an})$, and the fact that $\mathcal{U}_{\Gamma_0}$ is affine in  place of the fact that $\mathcal{U}_{\Gamma_0}^{\an}$ is Stein, together with Serre's vanishing theorem for coherent cohomology of affine schemes. 
\end{proof}

By taking $G$-invariants, the above proposition implies that every cohomology class in $H^1_{\dR}(\mathcal{U}_{\Gamma}; \mathcal{V}_k^{\dR,\mathbb{C}})$ can be represented by an element of $\Gamma(\mathcal{U}_{\Gamma_0}, F^k\mathcal{V}_k^{\dR,\mathbb{C}}\otimes \Omega^1_{\mathcal{U}_{\Gamma_0}/\mathbb{C}})^G = \Gamma(\mathcal{U}_{\Gamma},F^k\mathcal{V}_k^{\dR,\mathbb{C}}\otimes \Omega^1_{\mathcal{U}_{\Gamma}/\mathbb{C}})$. Any such element can be written on the uniformisation as
\[
    f(z) (X-zY)^kdz
\]
for a unique meromorphic modular form $f: \mathfrak{U} \to \mathbb{C}$ for $\Gamma$ of weight $k+2$: the $\Gamma$-invariance amounts to the fact that $f$ is modular for $\Gamma$ of weight $k+2$, and the algebraicity to the fact that $f$ has at most poles  along the points of $S$ and the cusps.

For any $l \in \mathbb{Z}$, let $M^!_{l,\Gamma}(*S)$ be the $\mathbb{C}$-vector space of meromorphic modular forms for $\Gamma$ of weight $l$, which are holomorphic outside $S$ and the cusps.

\begin{prop} \label{prop:dR-cohomology-meromorphic-modular-forms}
    For any integer $k\ge 0$,
    \[
        \begin{tikzcd}[row sep = -0.1cm, column sep = small]
            M^!_{-k,\Gamma}(*S) \arrow{r} & M_{k+2,\Gamma}^!(*S) \arrow{r} & H^1_{\dR}(\mathcal{U}_{\Gamma}; \mathcal{V}_k^{\dR,\mathbb{C}}) \arrow{r} & 0\\
              h \rar[mapsto] & \frac{\partial^{k+1}h}{\partial z^{k+1}} \\
              & f \rar[mapsto] & {[f(z)(X-zY)^kdz]}
        \end{tikzcd}
    \]
    is an exact sequence of $\mathbb{C}$-vector spaces. 
\end{prop}

\begin{proof}The fact that the first map preserves modularity for $\Gamma$ follows from \eqref{eq:bol-formula} and properties of the raising and lowering operators \S\ref{sec: raisin-lowering}. 
    The exactness on the right follows from Proposition \ref{prop:modular-representative-algebraic} and the discussion which follows.  For the exactness in the middle, we observe that the condition $[f(z)(X-zY)^k dz] =0$ implies that $f$ admits a holomorphic harmonic lift. Thus, we may apply Proposition \ref{prop:holomorphic-lift} to obtain a holomorphic function $h:\mathfrak{U} \to \mathbb{C}$ modular for $\Gamma$ of weight $-k$ such that $f = \partial^{k+1}h/\partial z^{k+1}$. This equation also implies \emph{a posteriori} that $h$ can have at most poles along the cusps and the points of $S$.
\end{proof}

\begin{rem} \label{rem:weakly-holomorphic}
    The above proposition also applies to $S= \emptyset$, in which case we recover the description of $H^1_{\dR}(\mathcal{Y}_{\Gamma};\mathcal{V}_k^{\dR,\mathbb{C}})$ in terms of weakly holomorphic modular forms (\emph{cf.} \cite[Theorem 1.2]{brown-hain}, \cite[Theorem 5.1]{CoeffPoincare}, and references therein):
    \[
        H^1_{\dR}(\mathcal{Y}_{\Gamma};\mathcal{V}_k^{\dR,\mathbb{C}}) \cong M_{k+2,\Gamma}^!/ \textstyle{\frac{\partial^{k+1}}{\partial z^{k+1}}}M_{-k,\Gamma}^!.
    \]
\end{rem}

We now obtain a more precise version of Corollary \ref{cor:existence-harmonic-lift} for meromorphic modular forms.

\begin{cor} \label{cor:existence-harmonic-lift-meromorphic}
     If $f$ is a meromorphic modular form in $M^!_{k+2,\Gamma}(*S)$, then it admits a harmonic lift (Definition \ref{defn:harmoniclift}) for which the corresponding Betti-conjugate $g$ is also a meromorphic modular form in $M^!_{k+2,\Gamma}(*S)$.
\end{cor}

\begin{proof}
    This is an immediate consequence of Propositions \ref{prop:dR-cohomology-meromorphic-modular-forms} and \ref{prop:geometric-harmoniclift}, together with Remark \ref{rem:geometric-harmoniclift-algebraic}. 
\end{proof}

\begin{ex}
    When $S = \Gamma w$, we have seen in Proposition \ref{prop:Psi} that, as a function of $z$, every $\Psi_{\Gamma}^{p,q}(z,w)$ with $p+q = k$ is a meromorphic modular form in $M^!_{k+2,\Gamma}(*S)$. By Corollary \ref{cor: GisWHLofPsis}, it is a Betti-conjugate of $-\Psi_{\Gamma}^{q,p}(z,w)$, which is also an element of $M^!_{k+2,\Gamma}(*S)$.
\end{ex}

\subsection{Residue sequence}

Let $\mathcal{S}_{\Gamma_0} = \Gamma_0\backslash S$. It can be regarded as a finite set of closed points on $\mathcal{Y}_{\Gamma_0}$, and we denote by $i: \mathcal{S}_{\Gamma_0} \hookrightarrow \mathcal{Y}_{\Gamma_0}$ the inclusion. Since $\mathcal{Y}_{\Gamma_0}$ is affine and 1-dimensional, we have a residue exact sequence of complex vector spaces (see \S\ref{par:appendix-residue}):
\[
    \begin{tikzcd}[column sep = small]
        0 \arrow{r} & H^1_{\dR}(\mathcal{Y}_{\Gamma_0};\mathcal{V}_{k}^{\dR,\mathbb{C}}) \arrow{r} & H^1_{\dR}(\mathcal{U}_{\Gamma_0};\mathcal{V}_{k}^{\dR,\mathbb{C}}) \arrow{r} & H^0_{\dR}(\mathcal{S}_{\Gamma_0};i^*\mathcal{V}_{k}^{\dR,\mathbb{C}}) \arrow{r} & 0 \ .
    \end{tikzcd}
\]
By taking $G$-invariants, we obtain an exact sequence
\[
    \begin{tikzcd}[column sep = small]
        0 \arrow{r} & H^1_{\dR}(\mathcal{Y}_{\Gamma};\mathcal{V}_{k}^{\dR,\mathbb{C}}) \arrow{r} & H^1_{\dR}(\mathcal{U}_{\Gamma};\mathcal{V}_{k}^{\dR,\mathbb{C}}) \arrow{r} & H^0_{\dR}(\mathcal{S}_{\Gamma_0};i^*\mathcal{V}_{k}^{\dR,\mathbb{C}})^G \arrow{r} & 0 \ .
    \end{tikzcd}
\]
Recall that $S = \Gamma w_1 \cup \cdots \cup \Gamma w_r$ and note that 
\[
    H^0_{\dR}(\mathcal{S}_{\Gamma_0};i^*\mathcal{V}_{k}^{\dR,\mathbb{C}})^G \cong   \bigoplus_{i=1}^r\mathcal{V}_k^{\dR,\an}(w_i)^{\Gamma_{w_i}} \cong \bigoplus_{i=1}^r(\Sym^k H^1_{\dR}(\mathcal{E}_{w_i}))^{\Gamma_{w_i}} \ ,
\]
where $ \Gamma_{w_i}$  is the stabiliser of  $\Gamma$ at  $w_i$, so we can rewrite the residue sequence as
\begin{equation}\label{eq:residue-seq-over-C}
     \begin{tikzcd}[column sep = small]
        0 \arrow{r} & H^1_{\dR}(\mathcal{Y}_{\Gamma};\mathcal{V}_{k}^{\dR,\mathbb{C}}) \arrow{r} & H^1_{\dR}(\mathcal{U}_{\Gamma};\mathcal{V}_{k}^{\dR,\mathbb{C}}) \arrow{r}{\Res} & \bigoplus_{i=1}^r(\Sym^k H^1_{\dR}(\mathcal{E}_{w_i}))^{\Gamma_{w_i}} \arrow{r} & 0 \ .
    \end{tikzcd}
\end{equation}

\begin{rem}
    One can use this exact sequence to deduce dimension formulae for the space of meromorphic (quasi)modular forms  with prescribed pole orders, such as in \cite[Theorem 1.4]{Matthes}.
\end{rem}

Recall from Proposition \ref{prop:dR-cohomology-meromorphic-modular-forms} that a cohomology class $\alpha$ in $H^1(\mathcal{U}_{\Gamma};\mathcal{V}_k^{\dR,\mathbb{C}})$ can be  represented by a meromorphic modular form $f$. The residue of $\alpha$ is computed explicitly in terms of the principal part of $f$ at the points of $S$ as follows.

\begin{lem}\label{lem:formula-residue-meromorphic}
    Let $f \in M_{k+2, \Gamma}^!(*S)$ and, for every $i=1,\ldots,r$, denote by $f(z) = \sum_{n \gg 0} c_n(w_i) (z-w_i)^n$ its Laurent expansion at $w_i$. Then
    \[
        \Res([f(z)(X-zY)^kdz]) = \Bigg(\sum_{j=0}^k\binom{k}{j}c_{-j-1}(w_i) (X-w_iY)^{k-j}(-Y)^j \Bigg)_{i=1,\ldots,r} \ .
    \]
\end{lem}

\begin{proof}
    The residue in cohomology is induced by the residue of vector-valued functions (see \S \ref{par:appendix-residue}), which we compute in the Betti frame. By writing
    \[
        X - zY = (X - w_iY) + (z-w_i)(-Y)
    \]
    and applying the binomial formula, we obtain
    \[
        f(z) (X-zY)^k = \Bigg(\sum_{n \gg -\infty} c_n(w_i)(z-w_i)^n \Bigg)\Bigg(\sum_{j=0}^k\binom{k}{j}(X-w_iY)^{k-j}(-Y)^j(z-w_i)^j \Bigg),
    \]
    which gives
    \begin{align*}
        \Res_{z=w_i}f(z) (X-zY)^kdz &= \sum_{\substack{j+n = -1 \\ n\gg -\infty,\  0 \le j \le k}} \binom{k}{j}c_n(w_i)(X-w_iY)^{k-j}(-Y)^j \\
        &= \sum_{j=0}^k\binom{k}{j} c_{-j-1}(w_i) (X-w_iY)^{k-j}(-Y)^j. \qedhere
    \end{align*}
\end{proof}

Given $n \ge 0$, we denote by $M_{k+2,\Gamma}(n S)$  the $\mathbb{C}$-subspace of $M_{k+2,\Gamma}^!(*S)$ defined by meromorphic modular forms which are holomorphic at the cusps and have poles of order at most $n$ at every point of $S$. When $n=0$, we denote it simply by $M_{k+2,\Gamma}$; it is the space of holomorphic modular forms of weight $k+2$ for $\Gamma$.

\begin{lem}\label{lem:inclusion-modular-forms}
    If $n\le k+1$, then the natural map
    \begin{equation}\label{eq:inclusion-modular-forms}
        M_{k+2,\Gamma}(nS) \longrightarrow   M_{k+2,\Gamma}^!(*S)/\textstyle{\frac{\partial^{k+1}}{\partial z^{k+1}}M_{-k,\Gamma}^!(*S)} \cong H^1_{\dR}(\mathcal{U}_{\Gamma};\mathcal{V}_k^{\dR,\mathbb{C}})
    \end{equation}
    is injective.
\end{lem}

\begin{proof}
    A modular form $f \in M_{k+2,\Gamma}^!(nS)$ in the kernel of \eqref{eq:inclusion-modular-forms} is of the form $f = \partial^{k+1}h/\partial z^{k+1}$ for some $h \in M_{-k,\Gamma}^!(*S)$. Since $f$ is holomorphic at infinity, so is $h$. Moreover, if $h$ had a pole along $S$, then $f$ would have a pole of order at least $k+2$ along $S$, which would contradict the hypothesis $n \le k+1$. This shows that $h$ is holomorphic everywhere. Since there are no non-constant holomorphic modular forms for $\Gamma$ of weight $-k\le 0$, we conclude that $f=0$.
\end{proof}

It follows from the above lemma that, for $n\le k+1$, we may always identify $M_{k+2, \Gamma}(nS)$ as a subspace of $H^1_{\dR}(\mathcal{U}_{\Gamma};\mathcal{V}_k^{\dR,\mathbb{C}})$. By Lemma \ref{lem:formula-residue-meromorphic}, for $p=1,\ldots,k+1$ and $f \in M_{k+2,\Gamma}((k+2-p)S)$, we have
\[
    \Res([f(z)(X-zY)^kdz]) \ \in \ \bigoplus_{i=1}^rF^{p-1}(\Sym^k H^1_{\dR}(\mathcal{E}_{w_i}))^{\Gamma_{w_i}} \ ,
\]
since  $-j-1\geq -(k+2-p)$ implies that $k-j \geq p-1$. Here, $F^{p-1}(\Sym^k H^1_{\dR}(\mathcal{E}_{w_i}))^{\Gamma_{w_i}}$ is the $(p-1)$th step of the Hodge filtration, given explicitly by $\Gamma_{w_i}$-invariant degree $k$ homogeneous polynomials in $X-w_iY$ and $- Y$ which are divisible by $(X-w_iY)^{p-1}$.

\begin{prop}\label{prop:residue-seq-hodge-explicit}
    For every $p=1,\ldots,k+1$, the residue induces an exact sequence of $\mathbb{C}$-vector spaces
    \begin{equation}\label{eq:residue-seq-hodge-explicit}
        \begin{tikzcd}[column sep = small]
            0 \arrow{r} & M_{k+2,\Gamma}\arrow{r} & M_{k+2,\Gamma}((k+2-p)S) \arrow{r}{\Res} &  \bigoplus_{i=1}^rF^{p-1}(\Sym^k H^1_{\dR}(\mathcal{E}_{w_i}))^{\Gamma_{w_i}} \arrow{r} & 0 \ .
        \end{tikzcd}
    \end{equation}
\end{prop}

\begin{proof}
    The above sequence is well-defined and exact in the middle by Lemma \ref{lem:formula-residue-meromorphic}. It is trivially exact on the left, and we are left to show that the residue map is surjective. Using the $\mathbb{C}$-basis $X-w_iY$, $X-\overline{w}_iY$ of $H^1_{\dR}(\mathcal{E}_{w_i})$, and the notation \eqref{kappdef},  we see that $F^{p-1}(\Sym^k H^1_{\dR}(\mathcal{E}_{w_i}))^{\Gamma_{w_i}}$ is generated by
    \[
        \sum_{\gamma \in \Gamma_{w_i}}(X-w_iY)^m(X-\overline{w}_iY)^n\Big|_{\gamma} = \kappa_{\Gamma,w_i}^{m,n}(X-w_iY)^m(X-\overline{w}_iY)^n\text{, }\qquad m+n = k\text{, }m\ge p-1\text{, }n\ge 0.
    \]
    Now, for every $i=1,\ldots,r$ and $j=0,\ldots,k+1-p$, it follows from Proposition \ref{prop:Psi} that $\Psi_{\Gamma}^{j, k-j}(z,w_i) \in M_{k,\Gamma}((k+2-p)S)$. We conclude by an application of the residue computation of Corollary \ref{cor:residue-Psi}.
\end{proof}

We prove in the next paragraph that \eqref{eq:residue-seq-hodge-explicit} is actually given by the Hodge filtration on $H^1_{\dR}(\mathcal{U}_{\Gamma};\mathcal{V}_{k}^{\dR,\mathbb{C}})$.

\subsection{Hodge filtration and pole order}\label{par:hodge-and-pole}

The exact sequence of $\mathbb{C}$-vector spaces \eqref{eq:residue-seq-over-C} underlies an exact sequence of $\mathbb{Q}$-mixed Hodge structures (\emph{cf.} Remark \ref{rem:MHS-zucker}):
\[
    \begin{tikzcd}[column sep = small]
        0 \arrow{r} & H^1(\mathcal{Y}_{\Gamma};\mathcal{V}_k) \arrow{r} & H^1(\mathcal{U}_{\Gamma};\mathcal{V}_k) \arrow{r}{\Res} & \bigoplus_{i=1}^r(\Sym^kH^1(\mathcal{E}_{w_i}))^{\Gamma_{w_i}}(-1) \arrow{r} & 0.
    \end{tikzcd}
\]
Using the strictness of the Hodge filtration, we obtain for every $p\ge 1$ an exact sequence of $\mathbb{C}$-vector spaces
\begin{equation} \label{eq:residue-seq-hodge}
    \begin{tikzcd}[column sep = small]
        0 \arrow{r} & F^pH_{\dR}^1(\mathcal{Y}_{\Gamma};\mathcal{V}^{\dR,\mathbb{C}}_k) \arrow{r} & F^pH_{\dR}^1(\mathcal{U}_{\Gamma};\mathcal{V}^{\dR,\mathbb{C}}_k) \arrow{r}{\Res} & \bigoplus_{i=1}^rF^{p-1}(\Sym^kH_{\dR}^1(\mathcal{E}_{w_i}))^{\Gamma_{w_i}} \arrow{r} & 0 \ .
    \end{tikzcd}
\end{equation}
We shall prove that the above sequence agrees with \eqref{eq:residue-seq-hodge-explicit} for $p=1,\ldots,k+1$.

We recall the definition of the Hodge filtration on $H^1_{\dR}(\mathcal{U}_{\Gamma};\mathcal{V}_k^{\dR,\mathbb{C}})$, as in \S \ref{par:appendix-hodge-filtration}. Let $\mathcal{X}_{\Gamma}$ be the proper modular stack over $\mathbb{C}$ associated to $\Gamma$. With the above notation, it is the quotient stack  $\mathcal{X}_{\Gamma} = [G \backslash \mathcal{X}_{\Gamma_0}]$, where $\mathcal{X}_{\Gamma_0}$ is the smooth projective curve over $\mathbb{C}$ whose analytification is  the usual (Satake) compactification of a modular curve $\mathcal{X}_{\Gamma_0}^{\an} \cong \Gamma_0\backslash (\mathfrak{H} \cup \mathbb{P}^1(\mathbb{Q}))$.  Note that we have a chain of open immersions
\[
    \mathcal{U}_{\Gamma}\hookrightarrow\mathcal{Y}_{\Gamma}\hookrightarrow \mathcal{X}_{\Gamma},
\]
where $\mathcal{Y}_{\Gamma} \setminus \mathcal{U}_{\Gamma}$ is given by a closed substack $\mathcal{S}_{\Gamma}$ corresponding to the orbits $\Gamma w_1,\ldots,\Gamma w_r$, and $\mathcal{X}_{\Gamma}\setminus \mathcal{Y}_{\Gamma} $ is given by a closed substack $\mathcal{Z}_{\Gamma}$ corresponding to the cusps $\Gamma \backslash \mathbb{P}^1(\mathbb{Q})$. Formally, the closed substacks $\mathcal{S}_{\Gamma}$ and $\mathcal{Z}_{\Gamma}$ are given by finite unions of residue gerbes of the corresponding points of $\mathcal{X}_{\Gamma}$ (\emph{cf.} the discussion in \S\ref{par:moduli-pointed-EC}). In what follows, to simplify the exposition, we work directly with $\mathcal{X}_{\Gamma}$, but the reader can always take this to mean that we work $G$-equivariantly on $\mathcal{X}_{\Gamma_0}$.

The vector bundle with connection $(\mathcal{V}_k^{\dR,\mathbb{C}},\nabla_k^{\mathbb{C}})$ on $\mathcal{Y}_{\Gamma}$ extends canonically \cite{deligneEqDiff} to a vector bundle $\overline{\mathcal{V}}_{k}^{\dR,\mathbb{C}}$ over $\mathcal{X}_{\Gamma}$ with logarithmic connection
\[
    \overline{\nabla}_k^{\mathbb{C}} : \overline{\mathcal{V}}_k^{\dR,\mathbb{C}} \longrightarrow \overline{\mathcal{V}}_k^{\dR,\mathbb{C}} \otimes \Omega^1_{\mathcal{X}_{\Gamma}/\mathbb{C}}(\log \mathcal{Z}_{\Gamma}).
\]
The Hodge filtration on $\mathcal{V}_k^{\dR,\mathbb{C}}$ also extends canonically to a Hodge filtration $F^p\overline{\mathcal{V}}_k^{\dR,\mathbb{C}}$ satisfying Griffiths' transversality condition:
\[
    \overline{\nabla}_k^{\mathbb{C}} (F^p\overline{\mathcal{V}}_k^{\dR,\mathbb{C}}) \subset F^{p-1}\overline{\mathcal{V}}_k^{\dR,\mathbb{C}}\otimes \Omega^1_{\mathcal{X}_{\Gamma}/\mathbb{C}}(\log \mathcal{Z}_{\Gamma}) \ .
\]
Explicitly, using the isomorphism $\overline{\mathcal{V}}_k^{\dR,\mathbb{C}} \cong \Sym^k \overline{\mathcal{V}}_1^{\dR,\mathbb{C}}$, the Hodge filtration on $\mathcal{V}_k^{\dR,\mathbb{C}}$ is induced by the Hodge filtration on $\overline{\mathcal{V}}_1^{\dR,\mathbb{C}}$, which is given by
\[
    F^1\overline{\mathcal{V}}_1^{\dR,\mathbb{C}} = \overline{p}_*\Omega^1_{\overline{\mathcal{E}}/\mathcal{X}_{\Gamma}}(\log \overline{p}^{-1}\mathcal{Z}_{\Gamma}) \ ,
\]
where $\overline{p}: \overline{\mathcal{E}} \to \mathcal{X}_{\Gamma}$ is the `universal generalised elliptic curve' over $\mathcal{X}_{\Gamma}$.

\begin{rem}
    Since $dz = \frac{1}{2\pi i} \frac{dq}{q}$, if $q = e^{2\pi i z}$, and since $X-zY$ canonically extends to a section of $F^1\overline{\mathcal{V}}_1^{\dR,\mathbb{C}}$ on a neighbourhood of the cuspidal locus $\mathcal{Z}_{\Gamma}$, a meromorphic modular form $f$ which is holomorphic at the cusps defines a vector-valued 1-form $f(z)(X-zY)^kdz$ on $\mathcal{U}_{\Gamma}$ with logarithmic singularities along $\mathcal{Z}_{\Gamma}$. This yields, for every $m\ge 0$, an isomorphism of $\mathbb{C}$-vector spaces
    \begin{equation}\label{eq:holomorphic-at-cusps}
        M_{k+2,\Gamma}(m S) \stackrel{\sim}{\longrightarrow} H^0(\mathcal{X}_{\Gamma}, F^k \mathcal{V}_k^{\dR,\mathbb{C}}\otimes \mathcal{O}_{\mathcal{X}_{\Gamma}}(m\mathcal{S}_{\Gamma})\otimes \Omega^1_{\mathcal{X}_{\Gamma}/\mathbb{C}}(\log \mathcal{Z}_{\Gamma})))\ \text{, } \quad f \longmapsto f(z)(X-zY)^kdz \ .
    \end{equation}
\end{rem}

The Hodge filtration on $\overline{\mathcal{V}}_k^{\dR,\mathbb{C}}$ induces a filtration on the complex $\overline{C}\defeq \Omega^{\bullet}_{\mathcal{X}_{\Gamma}/\mathbb{C}}(\log (\mathcal{Z}_{\Gamma}\cup \mathcal{S}_{\Gamma});\overline{\mathcal{V}}_k^{\dR,\mathbb{C}})$:
\[ 
    F^p\overline{C}: \qquad  F^p\overline{\mathcal{V}}_k^{\dR,\mathbb{C}} \stackrel{\overline{\nabla}_k^{\mathbb{C}}}{\longrightarrow} F^{p-1}\overline{\mathcal{V}}_k^{\dR,\mathbb{C}} \otimes \Omega^{1}_{\mathcal{X}_{\Gamma}/\mathbb{C}}(\log (\mathcal{Z}_{\Gamma}\cup \mathcal{S}_{\Gamma}))
\] 
and the Hodge filtration on de Rham cohomology is defined by
\[
    F^pH^1_{\dR}(\mathcal{U}_{\Gamma};\mathcal{V}_k^{\dR,\mathbb{C}}) \coloneqq \im (\mathbb{H}^1(\mathcal{X}_{\Gamma};F^p\overline{C})\longrightarrow \mathbb{H}^1(\mathcal{X}_{\Gamma}; \overline{C}))
\]
under the identification $H^1_{\dR}(\mathcal{U}_{\Gamma};\mathcal{V}_k^{\dR, \mathbb{C}}) \cong \mathbb{H}^1(\mathcal{X}_{\Gamma};\overline{C})$ (see Lemma \ref{lem:deligne-theorem-compactifcation}).

\begin{lem}\label{lem:hodge-filtration-S-empty}
    Let $S=\emptyset$. Then for every $p=1,\ldots,k+1$, we have isomorphisms:
    \[
        F^pH^1_{\dR}(\mathcal{Y}_{\Gamma};\mathcal{V}_k^{\dR,\mathbb{C}}) \cong H^0(\mathcal{X}_{\Gamma},F^k\overline{\mathcal{V}}_k^{\dR,\mathbb{C}}\otimes \Omega^1_{\mathcal{X}_{\Gamma}/\mathbb{C}}(\log \mathcal{Z}_{\Gamma})) \stackrel{\eqref{eq:holomorphic-at-cusps}}{\cong} M_{k+2,\Gamma}.
    \]
\end{lem}

The above result is classical, see \cite[\S 12]{zucker} or \cite[\S 6]{CoeffPoincare} and references therein.

\begin{thm}\label{thm:hodge-filtration-open-complex}
    Under the isomorphism $H^1_{\dR}(\mathcal{U}_{\Gamma};\mathcal{V}_k^{\dR,\mathbb{C}}) \cong M_{k+2,\Gamma}^!(*S)/\frac{\partial^{k+1}}{\partial z^{k+1}}M_{-k,\Gamma}^!(*S)$ of Proposition \ref{prop:dR-cohomology-meromorphic-modular-forms}, the Hodge filtration is given by
    \begin{align*}
        &F^0H^1_{\dR}(\mathcal{U}_{\Gamma};\mathcal{V}_{k}^{\dR,\mathbb{C}}) \cong M_{k+2,\Gamma}^!(*S)/\textstyle{\frac{\partial^{k+1}}{\partial z^{k+1}}}M_{-k,\Gamma}^!(*S)\\
        &F^pH^1_{\dR}(\mathcal{U}_{\Gamma};\mathcal{V}_{k}^{\dR,\mathbb{C}}) \cong M_{k+2,\Gamma}((k+2-p)S) \text{, }\qquad p=1,\ldots,k+1\\
        &F^{k+2}H^1_{\dR}(\mathcal{U}_{\Gamma};\mathcal{V}_{k}^{\dR,\mathbb{C}}) = 0.
    \end{align*}
\end{thm}

\begin{proof}
    Let $p =1,\ldots,k+1$. By the exact sequences \eqref{eq:residue-seq-hodge} and \eqref{eq:residue-seq-hodge-explicit}, it suffices to prove one inclusion: if $f \in M_{k+2,\Gamma}((k+2-p)S)$, then its class in cohomology $\alpha \defeq [f(z)(X-zY)^kdz]$ lies in $F^pH^1_{\dR}(\mathcal{U}_{\Gamma};\mathcal{V}_k^{\dR,\mathbb{C}})$. Again by the same exact sequences, there is some $\beta \in F^pH^1_{\dR}(\mathcal{U}_{\Gamma};\mathcal{V}_k^{\dR,\mathbb{C}})$ such that $\Res(\alpha) = \Res(\beta)$, which means that $\alpha - \beta$ lies in $H^1_{\dR}(\mathcal{Y}_{\Gamma};\mathcal{V}_k^{\dR,\mathbb{C}})$. Since both $\alpha$ and $\beta$ are locally represented by vector-valued 1-forms with logarithmic singularities at the cuspidal locus $\mathcal{Z}_{\Gamma}$, the same holds for $\alpha - \beta$, which implies that $\alpha - \beta \in F^pH_{\dR}^1(\mathcal{Y}_{\Gamma};\mathcal{V}_k^{\dR,\mathbb{C}})$ by Lemma \ref{lem:hodge-filtration-S-empty}. Thus:
    \[
        \alpha \  \in  \ \beta + F^pH_{\dR}^1(\mathcal{Y}_{\Gamma};\mathcal{V}_k^{\dR,\mathbb{C}}) \subset F^pH_{\dR}^1(\mathcal{U}_{\Gamma};\mathcal{V}_k^{\dR,\mathbb{C}}) \ .\qedhere
    \]
\end{proof}

For simplicity, we state the next corollary only in the case $S = \Gamma w$.

\begin{cor}\label{cor:Psi-splits-residue-seq}
    Let $F^{\bullet}$ be the Hodge filtration on $H^1_{\dR}(\mathcal{U}_{\Gamma};\mathcal{V}_k^{\dR,\mathbb{C}})$. The meromorphic modular forms $\Psi^{m,n}_{\Gamma}(z,w)$,  with $m+n = k$ and $m,n\ge 0$,  satisfy
    \begin{equation} \label{eqn: PsiinFp}
        [  \Psi_{\Gamma}^{m,n}(z,w)(X-zY)^k dz ]   \  \in \     F^{n+1} \cap \overline{F}^{m+1}
    \end{equation}
   and induce a splitting of both the residue sequence \eqref{eq:residue-seq-over-C}:
    \begin{align*}
           (\Sym^kH^1_{\dR}(\mathcal{E}_{w}))^{\Gamma_{w}} &\longrightarrow H^1_{\dR}(\mathcal{U}_{\Gamma};\mathcal{V}_k^{\dR,\mathbb{C}})\\
        \sum_{\gamma \in \Gamma_{w}} (X-wY)^n  (X-\overline{w}Y)^m\Big|_{\gamma} &\longmapsto (-1)^m(w-\overline{w})^k\binom{k}{m}^{-1}[\Psi^{m,n}_{\Gamma}(z,w) (X-zY)^k  dz]
    \end{align*}
    and of the Hodge filtration on $F^1    \cong \bigoplus_{n=0}^{k}\mathrm{gr}^{n+1}_F H^1_{\dR}(\mathcal{U}_{\Gamma};\mathcal{V}_k^{\dR,\mathbb{C}}).$
\end{cor}

\begin{proof}
    That $[\Psi_{\Gamma}^{m,n}(z,w)(X-zY)^k dz] \in F^{n+1}$ follows  from  Theorem \ref{thm:hodge-filtration-open-complex}, together with the properties of $\Psi^{m,n}_{\Gamma}(z,w)$ (Proposition \ref{prop:Psi} and Corollary \ref{cor:residue-Psi}). To see that we also have $[\Psi_{\Gamma}^{m,n}(z,w)(X-zY)^k dz] \in \overline{F}^{m+1}$, we note that $\overline{F}^{m+1}$ is image of $F^{m+1}$ under Betti complex-conjugation. But, by Corollary \ref{cor: GisWHLofPsis}, $\Psi^{m,n}_{\Gamma}(z,w)$ is a Betti-conjugate of $-\Psi^{n,m}_{\Gamma}(z,w)$, which, by the first part of this proof, defines a class in $F^{m+1}$; this proves \eqref{eqn: PsiinFp}. The splitting statements follow from Proposition \ref{prop:residue-seq-hodge-explicit}.
\end{proof} 

\begin{rem}
    The splitting of the residue sequence \eqref{eq:residue-seq-over-C} is reminiscent of the theorem of Manin-Drinfeld with the points in the orbits of $w_i$ playing the role of the cusps, and the meromorphic modular forms $\Psi_{\Gamma}^{m,n}(z,w_i)$ playing the role of  Eisenstein series. Note however that, in contrast to  the Manin-Drinfeld case, which gives a splitting of the whole $\mathbb{Q}$-MHS, the above corollary concerns only the Hodge filtration.
\end{rem}

\section{Realisations, periods, and single-valued periods}

We discuss extensions  in a category of mixed Hodge structures, as a prototype for an abelian category of mixed motives and their periods. Since our main results  concern the computation of the associated periods, some details about the Hodge theory are omitted. 

\subsection{A category of realisations}\label{par:realisations}

In the absence of a suitable abelian category of mixed motives, it is convenient, following Deligne,  to work in a Tannakian category of realisations of motives. The variant most convenient for our purposes will be a category $\mathcal{H}_K$, where $K\subset \CC$ is a number field, whose objects are tuples:
\begin{equation} \label{ObjectinHk}
   \left( V_{\B} \ , V_{\overline{\B}} \ ,  V_{\dR}\ ; \comp_{\B,\dR} \ , \comp_{\overline{\B}, \dR}  \  , F_{\infty} \right)
\end{equation}
where $V_{\dR}$ is a finite dimensional $K$-vector space, equipped with an increasing (weight) filtration $W$ and a decreasing (Hodge) filtration $F$; $V_{\B}$ and $V_{\overline{\B}}$ are two finite-dimensional $\QQ$-vector spaces equipped with an increasing weight  filtration $W$; and 
\[
    \comp_{\B ,\dR} : V_{\dR} \otimes_{K,\iota} \CC \overset{\sim}{\To}   V_{\B} \otimes_{\QQ} \CC \qquad \hbox{ and } \qquad 
     \comp_{\overline{\B} ,\dR} : V_{\dR} \otimes_{K,\overline{\iota}} \CC \overset{\sim}{\To}   V_{\overline{\B}} \otimes_{\QQ} \CC
\]
are isomorphisms which respect the weight filtration $W$, where $\iota: K \hookrightarrow \mathbb{C}$ is the inclusion, and $\overline{\iota}: K \hookrightarrow \mathbb{C}$ is its complex-conjugate. Finally, the real Frobenius $F_{\infty}:  V_{\B} \overset{\sim}{\rightarrow} V_{\overline{\B}}$ is an isomorphism preserving $W$  which is compatible with $c_{\B,\dR}, c_{\overline{\B}, \dR}$ via the commutative diagram
\[
    \begin{tikzcd}
        V_{\dR}\otimes_{K,\iota} \CC \arrow{r}{\id \otimes c_{\dR}}\arrow{d}[swap]{\comp_{\B,\dR}} & V_{\dR}\otimes_{K,\overline{\iota}} \CC\arrow{d}{\comp_{\overline{\B},\dR}}\\
        V_{\B} \otimes_{\QQ} \CC \arrow{r}{F_{\infty} \otimes c_{\B}} & V_{\overline{\B}}\otimes_{\QQ} \CC 
    \end{tikzcd}
\]
and where  $c_{\dR}$, $c_{\B}$ both denote complex conjugation. 
Following convention, the inverse of $F_{\infty}$ will also be denoted by $F_{\infty}$ without ambiguity. The vector space $V_{\B}\otimes_{\QQ} \CC$ (respectively $V_{\overline{\B}}\otimes_{\QQ} \CC$) together with the filtrations $W$, $c_{\B,\dR} F$ (respectively $c_{\overline{\B}, \dR} F$) are required to be mixed Hodge structures. 

The category $\mathcal{H}_K$  is a $\QQ$-linear  Tannakian category which is neutralised by either of the two Betti realisation functors
which send an object \eqref{ObjectinHk} to $V_{\B}$ (or to $V_{\overline{\B}}$), respectively. 

The basic example of an object in $\mathcal{H}_K$ is given by the system of cohomology groups in degree $n$ of a smooth algebraic variety $X$ defined over $K$. We define
\begin{equation} \label{objectHn}
    H^n(X) =   \left( H^n_{\B}(X(\CC);\QQ) \ ,  \  H^n_{\B}( \overline{X}(\CC);\QQ) \ , \  H^n_{\dR}(X/K)  \  ; \   \comp_{\B, \dR} \ ,  \comp_{\overline{\B}, \dR}  \ , F_{\infty} \right) \ ,
\end{equation}
where $H^n_{\dR}(X/K) = \mathbb{H}^n(X; \Omega^{\bullet}_{X/k})$ denotes algebraic de Rham cohomology,  $\overline{X}$ is the base change $X\times_{K,\overline{\iota}}\overline{\iota}(K)$, and $\comp_{\overline{\B},\dR}$ is the comparison isomorphism for $\overline{X}$.  The fact that it defines an object of $\mathcal{H}_K$ follows from  theorems of Grothendieck (on the comparison between Betti and algebraic de Rham cohomology) and Deligne (on mixed Hodge theory). If $X$ is defined over a real subfield of $K$ then the two Betti realisations are canonically isomorphic
\begin{equation} \label{Bettisagree}
    H^n_{\B}(X(\CC);\QQ) = H^n_{\B}( \overline{X}(\CC);\QQ) \ ,
\end{equation}
and likewise the comparison isomorphism 
$\comp_{\B,\dR} = \comp_{\overline{\B},\dR}$. In this case, the real Frobenius $F_{\infty}$ is an involution of $H^n_{\B}(X(\CC);\QQ) $; in general, it is only an isomorphism between two different vector spaces.

\subsubsection{Periods} 

There are two notions of periods for the category $\mathcal{H}_K$. Given an object \eqref{ObjectinHk} and the data of $\omega \in V_{\dR}$, $\sigma \in V_{\B}^{\vee}$, the associated period is 
\[
    \sigma \, \comp_{\B,\dR}(\omega) \ \in\ \CC\ .
\]
Similarly, given $\sigma' \in  V_{\overline{\B}}^{\vee}$, the associated period is $\sigma' \, \comp_{\overline{\B},\dR}(\omega) \in \CC$. We shall call such periods `conjugate' periods, since the two notions are related as follows:
\[
    \overline{\sigma \, \comp_{\B,\dR}(\omega)}= (F^*_{\infty} \sigma) \,  \comp_{\overline{\B},\dR}(\omega) \ ,
\]
where $*$ denotes inverse transpose.
 
\subsubsection{Single-valued periods and `Betti-conjugation' periods}

The formalism of the category $\mathcal{H}_K$ is sufficient to express the notion of single-valued periods. The single-valued map  is 
\begin{align}  \label{svmapinH} 
    \mathsf{s}:  V_{\dR}\otimes_{K,\iota} \CC   \overset{\sim}{\To}   V_{\dR} \otimes_{K,\overline{\iota}} \CC \ , \quad \qquad   \mathsf{s} \coloneqq    \comp_{\overline{\B},\dR}^{-1} \circ  (F_{\infty}\otimes \id)  \circ \comp_{\B,\dR}\ .    \nonumber 
\end{align}
Given $\omega \in V_{\dR}$ and $\eta \in V_{\dR}^{\vee}$, the associated single-valued period is defined to be $ \eta \,\mathsf{s} (\omega) \in \CC$. Given a $K$-basis $(\omega_1,\ldots,\omega_r)$ of $V_{\dR}$, the corresponding matrix of $\mathsf{s}$ is given by
\[
    S = \overline{P}^{-1}P \in M_{r\times r}(\mathbb{C}) \ ,
\]
where $P$ is  the matrix of $\comp_{\B,\dR}$ with respect to $(\omega_1,\ldots,\omega_r)$ and any other $\mathbb{Q}$-basis of $V_{\B}$. It does not depend on the  latter choice of basis. It satisfies $\overline{S} = S^{-1}$.

\begin{defn} \label{defn: Betticonj}
    Define the \emph{Betti-conjugation map} relative to the inclusion $\iota: K \hookrightarrow \CC$ to be 
    \begin{equation}
        c_{\iota} :  V_{\dR}\otimes_{K,\iota}\CC  \To     V_{\dR}\otimes_{K,\iota}\CC  \ \text{, }\qquad 
    c_{\iota}  \defeq   \mathrm{comp}_{\B,\dR}^{-1} \circ  (\id \otimes c_{\B})   \circ \mathrm{comp}_{\B,\dR}  \ . 
    \end{equation}
    Note that it is complex anti-linear: $c_{\iota}(v \otimes \lambda) = c_{\iota}(v) \otimes \overline{\lambda}$ for $v\in V_{\dR}$ and $\lambda \in \CC.$
\end{defn}

\begin{lem} \label{lem: SVandCB}
    The Betti-conjugation map is related to the single-valued map via the formula
    \[
       c_{\iota} = (\id \otimes c_{\dR})\circ \mathsf{s}   \ . 
    \]
    In particular, the restriction of $\mathsf{s}$ and $c_{\iota}$ to $V_{\dR}$ coincide. 
\end{lem}
 
\begin{proof}
    It follows from the axiom $\id \otimes c_{\dR} = \comp_{\overline{\B},\dR}^{-1} \circ (F_{\infty} \otimes c_{\B})\circ \comp_{\B,\dR}$
    that $\mathsf{s} \circ c_{\iota}=\id \otimes c_{\dR}$. The result follows from  $c_{\iota}^2=\id,$ which follows from $c_{\B}^2=\id$, and  from $c_{\dR}^2=\id.$ 
\end{proof}

Betti-conjugation can be represented  by the matrix $\overline{S} =  P^{-1} \overline{P} = S^{-1}$, but note that it is  a $\CC$-antilinear map, so does  not  satisfy the usual formula for the linear action of a matrix on vectors. By abuse, we shall call such a matrix 
 a `Betti-conjugation period matrix'. 
 
\subsubsection{Variant with coefficients}

Let $F$ denote a finite extension of $\QQ$. We denote by $\mathcal{H}_{K} \otimes F$ the category defined by the Karoubian envelope of the $F$-linear category obtained from $\mathcal{H}_K$ by extending coefficients from $\QQ$ to $F$.  It is an $F$-linear Tannakian category such that 
\begin{equation} \label{eqn: ExtendHoms}
    \Hom_{\mathcal{H}_K \otimes F}(V,W)= \Hom_{\mathcal{H}_K}(V,W) \otimes_{\QQ} F
 \end{equation}
for all objects $V,W$ of $\mathcal{H}_K$. One may define the periods (and single-valued periods) of objects  in $\mathcal{H}_K \otimes F$: they  are $F$-linear combinations of periods of $\mathcal{H}_K$ taking values in $F\otimes_{\QQ} \CC$. 

\subsection{Examples of objects in $\mathcal{H}_K$}

Let $K$ be as above. 

\subsubsection{Tate objects}  \label{sect: TateObjects}

The Lefschetz object $H^2(\PP_K^1)$ in \eqref{objectHn} is denoted by $\QQ(-1)$. Since $\PP^1_K = \mathbb{P}^1_{\mathbb{Q}}\times_{\mathbb{Q}}K$ is defined over $\QQ$, its two Betti realisations are identical \eqref{Bettisagree}. For any object $V$ of $\mathcal{H}_K$, $V(n)$ denotes the object $V\otimes \QQ(-1)^{\otimes (-n)}$. The object $\QQ(-n)$ has periods (resp. single-valued periods) proportional to  $(2\pi i)^n$ (resp.  $(-1)^n$). 

\subsubsection{Kummer extensions} \label{sec:kummer}

Let $x\in K^{\times} \setminus \{1\}$ and  consider the object $\mathcal{K}_x=H^1(\PP_K^1\setminus \{0,\infty\}, \{1,x\})$ of $\mathcal{H}_K$. The relative long exact cohomology sequence gives 
\[
    \begin{tikzcd}[column sep = small]
        \cdots \arrow{r} & H^0(\PP_K^1\setminus \{0,\infty\}) \arrow{r} & H^0( \{1,x\}) \arrow{r} & \mathcal{K}_x \arrow{r} & H^1(\PP_K^1\setminus\{0,\infty\})\arrow{r} & 0
    \end{tikzcd}
\]
from which we deduce a simple extension in $\mathcal{H}_K$ of the form: 
\begin{equation}
    \begin{tikzcd}[column sep = small]
        0\arrow{r}& \QQ(0) \arrow{r}& \mathcal{K}_x \arrow{r} & \QQ(-1) \arrow{r} & 0 \ . 
    \end{tikzcd}
\end{equation}
A de Rham basis of $\mathcal{K}_x$ is given by $[\frac{dz}{1-x}], [\frac{dz}{z}]$; a Betti basis by the classes of  any path $\sigma_x$ in $\mathbb{P}^1(\mathbb{C})\setminus\{0,\infty\} = \CC^{\times}$ from $1$ to $x$, and  a small simple loop $\sigma_0$ around $0$. The period matrix and its single-valued version are 
\[
    P = \begin{pmatrix}
    1  & \log x \\
    0   &  2 \pi i  
    \end{pmatrix}   \qquad \hbox{ and } \qquad  S =   \overline{P}^{-1} P = \begin{pmatrix}
    1 & \log |x|^2  \\
    0   & -1   
    \end{pmatrix} \ . 
\]

\subsubsection{Elliptic curves}

Let $E$ be an elliptic curve defined over $K$, and denote its cohomology object \eqref{objectHn} by $H = H^1(E)$ in the category $\mathcal{H}_K$. It admits a surjective anti-symmetric polarisation morphism
\[
    \langle \ , \ \rangle : H \otimes H \longrightarrow \QQ(-1)
\]
which is equivalent to a canonical isomorphism 
\[
    H^{\vee} \cong  H(1)\ .
\]
Denote  $S^n H \coloneqq \mathrm{Sym}^n H$, where $S^0H=\QQ(0)$. The polarisation gives a canonical isomorphism  $\bigwedge^2 H \cong  \QQ(-1)$, and hence the semi-simple Tannakian subcategory of $\mathcal{H}_K$ generated by $H$ and its tensor powers has simple objects $(S^n H)(d)$ for integers $n\geq 0, d\in \ZZ$. We have canonical isomorphisms 
\[
    (S^n H)^{\vee} \cong S^n H^{\vee} \cong (S^n H )(n).
\]
For all $m,n\geq 0$ there is a natural surjective  map (multiplication of symmetric tensors)
\begin{equation} \label{eqn: multiplySnH}
    S^m H \otimes S^n H \To S^{m+n}H \ . 
\end{equation}

\begin{lem}  \label{lem: PeriodsE}
    Let $[\omega], [\eta]$ denote a de Rham basis for $H_{\dR}=H^1_{\dR}(E/K)$ such that $[\omega] \in F^1 H_{\dR}$ and which is symplectic, \emph{i.e.},  such that  $\langle [\omega],[\eta] \rangle_{\dR}=1$, where 
    \[
        \langle \  , \  \rangle_{\dR}  : H_{\dR}\otimes_K  H_{\dR}\To  \QQ(-1)_{\dR} =K
    \]
    is the pairing induced by the polarisation. Choose an isomorphism $E(\CC) \cong \CC/ (\ZZ+ z \ZZ) $, where $\Im(z)>0$, and consider the basis $\sigma_1,\sigma_2$ of $H_{\B}^{\vee} =H_1(E(\CC);\mathbb{Z}) \cong \mathbb{Z} +\mathbb{Z} z $ corresponding to the basis $1,z$ of the lattice. If we write the period matrix with respect to these bases in the form
    \[ 
        P =  \begin{pmatrix}   
                \int_{\sigma_1} \omega  &   \int_{\sigma_1} \eta \\
                \int_{\sigma_2} \omega  &   \int_{\sigma_2} \eta \\
            \end{pmatrix}  =
            \begin{pmatrix}
            \omega_1 & \eta_1 \\
            \omega_2 & \eta_2 
    \end{pmatrix}
    \]
    then $\det(P)=2\pi i $, $z = \omega_2/\omega_1$, and  the associated single-valued period matrix is 
    \[
        S = \overline{P}^{-1} P  =-\frac{1}{2\pi i} \begin{pmatrix} \omega_1 \overline{\eta}_2 -  \overline{\eta}_1\omega_2 &   \eta_1 \overline{\eta}_2 - \overline{\eta}_1 \eta_2 \\
        \overline{\omega}_1 \omega_2 - \omega_1 \overline{\omega}_2 & \overline{\omega}_1 \eta_2-  \eta_1 \overline{\omega}_2
        \end{pmatrix} \ . 
    \]
    Consider the following quantity, which is the quotient of two entries of $S$, and its complex conjugate: 
    \begin{equation} \label{rhodef} 
        \rho=  \frac{S_{11}}{S_{21}} =  \frac{\omega_1 \overline{\eta}_2 -  \overline{\eta}_1\omega_2}{\overline{\omega}_1 \omega_2 - \omega_1 \overline{\omega}_2} \   , \quad 
        \overline{\rho} = \frac{\overline{\omega}_1 \eta_2 -  \eta_1\overline{\omega}_2}{\omega_1 \overline{\omega}_2 - \overline{\omega}_1 \omega_2}   \ . 
    \end{equation}
    With these notations, we have 
    \begin{equation} \label{EllipticBettiBasis} 
        \comp_{\B,\dR}^{-1} \left( X- z Y \right) =   \frac{1}{\omega_1} \, [\omega] \qquad \hbox{ and } \qquad    \comp_{\B,\dR}^{-1} \left( X- \overline{z} Y \right)   =  \frac{\omega_1}{2 \pi i} \left(  \overline{z}-  z \right) \left([\eta] + \overline{\rho} [\omega] \right) \ .
    \end{equation}
\end{lem}

\begin{proof}
    The computation of the determinant (classically known as `Legendre's period relation') follows from the compatibility of the polarisation map with the comparison isomorphism and the fact that both the Betti basis and the de Rham basis are symplectic. Since, under the isomorphism $E(\CC) \cong \CC/ (\ZZ+ \ZZ z)$, the class $[\omega]$ is proportional to $[du]$, where $u$ is the coordinate on $\mathbb{C}$, it follows that $\omega_2/\omega_1= (\int_0^z du) / (\int_0^1 du)=z$. 

    Recall from \S \ref{par:cohomology-EC} that $(X,Y)$ denotes  the Betti basis dual to $(1,-z)$. Thus  we have
    \[
        \comp_{\B,\dR}^{-1} \left( X- zY \right) =  [du] =  \frac{1}{\omega_1} \, [\omega]
    \]
    and 
    \[
        \comp_{\B,\dR}^{-1} \left( X- \overline{z} Y \right) =  \comp_{\B,\dR}^{-1}  \circ (\id \otimes c_{\B}) \left( X- z Y \right)
        =     c_{\iota} \circ \comp^{-1}_{\B, \dR} \left( X- z Y \right) =  c_{\iota} \left(  \frac{1}{\omega_1} \, [\omega] \right)
    \]
    where the second equality follows from the definition of $c_{\iota}$. By applying Lemma \ref{lem: SVandCB} we obtain
    \begin{align*}
        c_{\iota} \left(  \frac{1}{\omega_1} \, [\omega] \right) = \frac{1}{\overline{\omega}_1} \, c_{\iota} \left(  [\omega] \right) = \frac{1}{\overline{\omega}_1}(\id \otimes c_{\dR})\circ \mathsf{s} ([\omega]) = \frac{1}{\overline{\omega}_1} \left(       \overline{S}_{11} [\omega] + \overline{S}_{21} [\eta] \right)
        =\frac{\overline{S}_{21}}{\overline{\omega}_1} \left( \overline{\rho} \, [\omega] +[\eta]\right)    \ . 
    \end{align*}
    since the action of $\mathsf{s}$ is computed by the matrix $S$. To complete the calculation, substitute the definition of $S_{21}$ and use the fact that $\omega_2= z \omega_1$. 
\end{proof}

\subsubsection{Elliptic curves with complex multiplication}

Let $E$ be an elliptic curve defined over $K$ with complex multiplication by an order $\mathcal{O}$ in a quadratic imaginary extension $F$ of $\QQ$, normalised such that the action of $\lambda \in \mathcal{O}$ is multiplication by $\lambda$  (as opposed to $\overline{\lambda}$) on $F^1 H_{\dR}$. We assume that $K$  contains $F$.

\begin{lem} \label{lem: periodsCMcase}
    Suppose that $E$ has complex multiplication as above, and let $[\omega],[\eta]$ be a basis of $H_{\dR}$ as in Lemma \ref{lem: PeriodsE}. Then, the corresponding quantity $\overline{\rho}$ (see \eqref{rhodef}) lies in $K\subset \mathbb{C}$. In particular, the de Rham class
    \[
        [\eta'] \defeq [\eta] + \overline{\rho}\,[\omega]
    \]
    is defined over $K$, so that $[\omega],[\eta']$ is an eigenbasis of $H_{\dR}$ for the action of $\mathcal{O}$, which  satisfies the conditions of Lemma \ref{lem: PeriodsE}. The period matrix $P'$ associated to the de Rham basis $[\omega],[\eta']$ and a symplectic basis of $H_{\B}$ satisfies
    \[ 
        P' = \begin{pmatrix} 
            \omega_1 & \eta'_1 \\ \omega_2 & \eta'_2 
            \end{pmatrix} \quad  \hbox{ with  }  \quad  \frac{\omega_2}{\omega_1} = z  \   \hbox{ and }  \   \frac{\eta'_2}{\eta'_1} = \overline{z} \ ,
    \]
    where $z$ lies in the quadratic imaginary field $F$.
\end{lem}

\begin{proof}
    Let $\varphi:H \to H$ be an endomorphism in $\mathcal{H}_K$ corresponding to some $\lambda \in \mathcal{O} \setminus \mathbb{Z}$. Since $\varphi_{\dR}: H_{\dR} \to H_{\dR}$ is compatible with the Hodge structure, its eigenvalues are $\lambda$ and $\overline{\lambda}$, and its action on the basis $[\omega],[\eta]$ satisfies
    \[
        \varphi_{\dR}\,[\omega] = \lambda\, [\omega]\qquad \varphi_{\dR}\,[\eta] = \mu\,[\omega] + \overline{\lambda} \, [\eta]
    \]
    for some $\mu \in K$. The compatibility between $\varphi_{\dR}$ and the single-valued map $\mathsf{s}$, namely $\mathsf{s} \circ (\varphi_{\dR}\otimes_\iota \id) = (\varphi_{\dR}\otimes_{\overline{\iota}}\id) \circ \mathsf{s}$, amounts to the matrix identity
    \[
        \begin{pmatrix}
            S_{11}\lambda & S_{11} \mu + S_{12} \overline{\lambda}\\
            S_{21} \lambda & S_{21} \mu + S_{22} \overline{\lambda}
        \end{pmatrix}
        =
        \begin{pmatrix}
            S_{11} & S_{12} \\
            S_{21} & S_{22}
        \end{pmatrix}
        \begin{pmatrix}
            \lambda & \mu \\
            0 & \overline{\lambda}
        \end{pmatrix}
        =
        \begin{pmatrix}
            \overline{\lambda} & \overline{\mu} \\
            0 & \lambda
        \end{pmatrix}
        \begin{pmatrix}
            S_{11} & S_{12} \\
            S_{21} & S_{22}
        \end{pmatrix}
        =
        \begin{pmatrix}
            \overline{\lambda} S_{11} + \overline{\mu} S_{21} & \overline{\lambda} S_{12} + \overline{\mu} S_{22} \\
            \lambda S_{21} & \lambda S_{22}
        \end{pmatrix} \ .
    \]
    By taking the quotient of the $(1,1)$ entry by the $(2,1)$ entry on both sides of the above equation, we obtain the relation
    \[
        \rho = \frac{S_{11}}{S_{21}} = \frac{\overline{\lambda} S_{11} + \overline{\mu} S_{21}}{\lambda S_{21}} = \frac{\overline{\lambda}}{\lambda} \rho + \frac{\overline{\mu}}{\lambda} \iff \overline{\rho} = \frac{\lambda}{\overline{\lambda}}\overline{\rho} + \frac{\mu}{\overline{\lambda}} \ .
    \]
    Since $\lambda \neq \overline{\lambda}$, we conclude that $\overline{\rho} \in K$. To see that $[\omega], [\eta']$ is an eigenbasis for $\varphi_{\dR}$, it suffices to use the previous formula for $\overline{\rho}$ to check that $\varphi_{\dR} [ \eta'] = \overline{\lambda} [\eta']$.

    The assertion that $z\defeq \omega_2/\omega_1$ lies in $F$ is classical; it follows from a period matrix computation, similar to the argument in the last paragraph, given by the compatibility $\comp_{\B,\dR}\circ (\varphi_{\dR}\otimes_{\iota}\id) =  (\varphi_{\B}\otimes \id) \circ \comp_{\B,\dR}$. From the formula \eqref{rhodef} for $\overline{\rho}$, we deduce that: 
    \[
        \frac{\eta_2'}{\eta_1'} = \frac{\eta_2 + \overline{\rho}\, \omega_2}{\eta_1 + \overline{\rho}\, \omega_1} = \frac{\overline{\omega}_2}{\overline{\omega}_1} = \overline{z}\ . \qedhere
    \]
\end{proof}

\begin{rem}
 The algebraicity of $\overline{\rho}$ fails if  $[\omega],[\eta]$ is not adapted to the Hodge filtration ($[\omega] \notin F^1H_{\dR}$). 
\end{rem}

\begin{rem} \label{rem: eigenbasisinBetti}
    It follows from Lemma \ref{lem: PeriodsE} that, for $[\omega],[\eta']$ as in Lemma \ref{lem: periodsCMcase}, we have
    \[ 
        \comp^{-1}_{\B,\dR} (X- zY)  = \frac{1}{\omega_1} [\omega]\qquad  \hbox{ and } \qquad   \comp^{-1}_{\B,\dR} (X- \overline{z} Y) = \frac{\omega_1 }{2\pi i}(\overline{z} -z)   [\eta']  \ .
    \]
\end{rem}

\begin{prop}  \label{prop: CMmotives}
    Let $E$ be as above and $H = H^1(E)$. There is a canonical decomposition  in  $\mathcal{H}_K$ 
    \begin{equation} \label{decompSnintoCn}
        S^n H = \bigoplus_{0\leq 2i \leq  n}  C^{n-2i} (-i) 
    \end{equation} 
    where $C^0 = \QQ(0)$, and, for $m\ge 1$, the objects $C^m$ are of rank two with Hodge numbers $(0,m)$ and $(m,0)$.
\end{prop}

\begin{proof}
    The graph of complex multiplication by  $\lambda \in \mathcal{O}$ defines an algebraic cycle $\Gamma_{\lambda} \subset E \times_K E$ defined  over $K$. The cycle class map defines a morphism $H^0(\Gamma_{\lambda})(-1) \rightarrow H^2(E\times_K E)$ in $\mathcal{H}_K$, inducing a map 
    \[
        \varphi_{\lambda}:  \QQ(-1)  \To S^2 H\ ,  
    \]
    after projecting onto the summand $H \otimes H \subset H^2(E\times_K E)$ by the  K\"unneth decomposition, and from there onto $S^2 H$. We  compute this map in de Rham cohomology. Let $e,f$ denote a symplectic eigenbasis for the action $\varphi_{\dR}$ of complex multiplication by $\lambda$ on $H_{\dR}$ (\emph{e.g.}, $[\omega],[\eta']$ as in Lemma \ref{lem: periodsCMcase}), which is represented by 
    \[
        \varphi_{\dR} =  e^{\vee} \otimes \lambda\,  e+   f^{\vee} \otimes \overline{\lambda}\,  f \ \in  \ \End(H_{\dR}) = H_{\dR}^{\vee} \otimes H_{\dR} \ ,
    \]
    where $\overline{\lambda}$ is the complex conjugate of $\lambda$. Since $e,f$ is symplectic (\emph{i.e.}, $\langle e,f\rangle_{\dR} = 1$), the isomorphism $H^{\vee} \cong  H(1)$ induced by Poincar\'e duality satisfies $e^{\vee} = f$ and $f^{\vee}=- e$. Therefore the  element $\varphi_{\dR}$ corresponds via this isomorphism to  
    \[
        f \otimes \lambda e - e \otimes \overline{\lambda}  f\  \in \ H_{\dR} \otimes H_{\dR} \ .
    \]
    Its image in $S^2 H_{\dR}$ is $(\lambda - \overline{\lambda}) ef$ and is non-zero if and only if $\Im(\lambda) \neq 0$, in which case $\varphi_{\lambda}: \QQ(-1)\rightarrow S^2 H$ is injective. Choose any $\lambda \in \mathcal{O}$ with non-vanishing imaginary part  (equivalently,  $\lambda \not \in \mathbb{Z}$) and write $\varphi = \varphi_{\lambda}$. By the above, it  only depends on $\mathrm{Im}\, \lambda$ and is therefore well-defined up to multiplication by a non-zero rational number. 
    Let $C^1 =H$ and for all $n\geq 2$ define 
    \[
        C^n = \coker \left( \varphi^n: \QQ(-1) \otimes S^{n-2} H \To S^n H \right) \ , 
    \]
    where the map $\varphi^n$ is the composition of the injection $ \varphi \otimes \id:  \QQ(-1)\otimes S^{n-2}H\rightarrow S^2H\otimes S^{n-2}H$ with the multiplication \eqref{eqn: multiplySnH} in the case $m=2.$ This map is injective. To see this, consider the eigenspace decomposition of the de Rham  realisation $H_{\dR}=H_{\dR}^{\lambda} \oplus H_{\dR}^{\overline{\lambda}}$ with respect to the complex multiplication.  Let us denote the $\lambda^p \overline{\lambda}^q$-eigenspace in $S^n H_{\dR}$ by $H_{\dR}^{p,q} \cong  (H_{\dR}^{\lambda})^{\otimes p} \otimes(H_{\dR}^{\overline{\lambda}})^{\otimes q}$.  By the above computation,  the image of $\varphi_{\dR}$ is the component $H^{1,1}_{\dR}$ of $S^2 H_{\dR}$  generated by $ef$. Therefore $\varphi^n_{\dR}$ is multiplication by $ef$ 
    \[
        \varphi^n_{\dR}  \ :  \  H^{1,1}_{\dR} \otimes \bigoplus_{p+q=n-2} H^{p,q}_{\dR} \overset{\sim}{\To} \bigoplus_{p+q=n-2} H^{p+1,q+1}_{\dR} \  \subset\  S^{n} H_{\dR} \ ,
    \]
    which is injective. Thus $\varphi^n$ is injective, and we see from the above that  $C^n_{\dR} \cong H^{n,0}_{\dR} \oplus H^{0,n}_{\dR}$ has rank $2$.

    Since $H^{\otimes n}$ and all the objects defined above lie in a semi-simple Tannakian subcategory of $\mathcal{H}_K$, we deduce from the definition of $C^n$ and injectivity of $\varphi^n$ a canonical decomposition 
    \[
        S^n H  =   C^n \oplus S^{n-2} H(-1)\ , 
    \]
    from which \eqref{decompSnintoCn} follows.  Since $H$ has Hodge types $(1,0)$ and $(0,1)$, $S^nH$ has Hodge numbers $(n-i,i)$ for $0\leq i \leq n$ with multiplicity $1$. By the previous equation,  $C^n$ has Hodge numbers $(0,n)$ and $(n,0)$.
\end{proof}

\begin{rem}
    The  previous proposition holds in the category of Chow motives over $K$ with trivial coefficients.
\end{rem}
 
\begin{rem}   \label{rem: defineMplusMminusCMcase}
    In the category $\mathcal{H}_{K} \otimes F$, complex multiplication by a non-real element of $\mathcal{O}$ defines an endomorphism of $H\otimes F$ with  distinct eigenvalues in $F$. Therefore it may be  decomposed  into eigenspaces
    \[
        H \otimes F \cong    M_+ \oplus M_{-} 
    \]
    in $\mathcal{H}_K\otimes F$. The objects $M_+, M_{-}$ are realisations of  rank one Chow motives over $K$ with coefficients in $F$.  The objects $C^m\otimes F$ in turn decompose into a direct sum $M^{\otimes m}_+ \oplus M^{\otimes m}_{-} $ of rank one objects. By permuting $M_+$ and $M_-$ if necessary, we may assume that  $ (M_{+})_{\dR}  \cong  H^{\lambda}_{\dR}$, $ (M_{-})_{\dR}  \cong  H^{\overline{\lambda}}_{\dR}$.    
    Since the complex multiplication respects the Hodge filtration, we have also  $(M_+)_{\dR}= F^1 H_{\dR}$.
   
\end{rem}

\section{Mixed modular objects in a category of realisations}\label{sec:MMR}
The first paragraphs of this section study the rational structure on algebraic de Rham cohomology of punctured modular curves and the field of definition of the meromorphic forms $\Psi^{r,s}_{\Gamma}(z,w)$.
The section concludes with the definition of mixed modular relative cohomology objects in a category of realisations.

\subsection{De Rham cohomology of punctured modular curves over a number field} \label{par: deRham-nb-field}

From now on, we assume for simplicity that $\Gamma \le \SL_2(\mathbb{Z})$ is a \emph{congruence} subgroup. We shall also assume that the finite union of $\Gamma$-orbits $S = \Gamma w_1 \cup \cdots \cup \Gamma w_r$ is such that the $j$-invariants $j(w_i)$ are algebraic numbers for $i=1,\ldots,r$. This implies that there is a number field $K\subset \mathbb{C}$ over which the universal elliptic curve $\mathcal{E}_{\Gamma} \to \mathcal{Y}_{\Gamma}$ and the closed substack $\mathcal{S}_{\Gamma}\subset \mathcal{Y}_{\Gamma}$ are defined; we denote by $\mathcal{E}_{\Gamma,K} \to \mathcal{Y}_{\Gamma ,K}$ and $\mathcal{S}_{\Gamma,K}\subset \mathcal{Y}_{\Gamma,K}$ their $K$-models to emphasize the base field. We denote  by  $\mathcal{U}_{\Gamma,K} $ the open complement $ \mathcal{Y}_{\Gamma,K}\setminus \mathcal{S}_{\Gamma,K}$.

\begin{example}
    If $\Gamma$ contains the principal congruence subgroup $\Gamma(N)$, then it is classical that we can take $\mathbb{Q}(e^{\frac{2\pi i}{N}})$ as a field of definition for the modular stack $\mathcal{Y}_{\Gamma}$, and $K= \mathbb{Q}(e^{\frac{2\pi i}{N}},j(w_1),\ldots,j(w_r))$. For the Hecke congruence subgroup $\Gamma_0(N)$, we can take $K = \mathbb{Q}(j(w_1),\ldots,j(w_r))$.
\end{example}

The system $(\mathcal{V}_{k}^{\dR,\mathbb{C}}, \nabla_{k}^{\mathbb{C}}, F^{\bullet}\mathcal{V}_{k}^{\dR,\mathbb{C}})$ discussed in \S\ref{sect: AlgMeroForms} comes from  a system $(\mathcal{V}_{k}^{\dR,K}, \nabla_{k}^{K}, F^{\bullet}\mathcal{V}_{k}^{\dR,K})$ defined over $K$, for instance:
\[
    \mathcal{V}_{k}^{\dR,K} \defeq \Sym^k H^1_{\dR}(\mathcal{E}_{\Gamma,K}/\mathcal{Y}_{\Gamma,K}) \ .
\]
By restricting it to $\mathcal{U}_{\Gamma,K}$ and computing de Rham cohomology, we obtain a $K$-structure on \eqref{eq:cohomolgy-open-complex}:
\[
    H^1_{\dR}(\mathcal{U}_{\Gamma,K};\mathcal{V}_{k}^{\dR,K}) \otimes_K \mathbb{C} \cong H^1_{\dR}(\mathcal{U}_{\Gamma};\mathcal{V}_{k}^{\dR,\mathbb{C}}) \ .
\]

Let $d$ be the width of the cusp at infinity, so that a meromorphic modular form $f \in M_{l,\Gamma}^!(*S)$ admits a Fourier expansion of the form
\[
    f(z) = \sum_{n \gg -\infty}a_nq^{\frac{n}{d}}\ \text{, }\qquad q = e^{2\pi i z}\ ,
\]
near $q=0$. We denote by $M_{l,\Gamma}^!(*S)_K \subset M_{l,\Gamma}^!(*S)$ (resp. $M_{l,\Gamma}(mS)_K \subset M_{l,\Gamma}(mS)$) the $K$-vector subspace given by the condition $a_n \in K$ for all $n \in \mathbb{Z}$.

Recall that the `classical Tate curve' admits the affine equation $y^2 = 4x^3 - \frac{E_4(q)}{12}x +\frac{E_6(q)}{216}$ over $\mathbb{Q}(\!(q)\!)$ and the class $\omega_q \defeq 2\pi i (X-zY)$ gets identified with $[dx/y]$. Thus, by the `$q$-expansion principle', there is an isomorphism
\[
    M_{k+2,\Gamma}^!(*S)_K \stackrel{\sim}{\longrightarrow} H^0(\mathcal{U}_{\Gamma,K},F^k\mathcal{V}_k^{\dR,K}\otimes \Omega^1_{\mathcal{U}_{\Gamma,K}/K})\ \text{, }\qquad f\longmapsto (2\pi i)^{k+1}f(z)(X-zY)^kdz = f \,\omega_q^k \,\frac{dq}{q} \ .
\]
The above map also induces, for every $m\ge 0$, a $K$-rational version of the isomorphism \eqref{eq:holomorphic-at-cusps}:
\[
    M_{k+2,\Gamma}(mS)_K \stackrel{\sim}{\longrightarrow} H^0(\mathcal{X}_{\Gamma,K}, F^{k}\mathcal{V}_k^{\dR,K}\otimes \mathcal{O}_{\mathcal{X}_{\Gamma,K}}(m\mathcal{S}_{\Gamma,K})\otimes \Omega^1_{\mathcal{X}_{\Gamma,K}/K}(\log \mathcal{Z}_{\Gamma,K})),
\]
where $\mathcal{X}_{\Gamma,K}$ (resp. $\mathcal{Z}_{\Gamma,K}$) denotes the compactified modular stack (resp. the cuspidal locus) defined over $K$. In particular, it follows that
\[
    M_{k+2,\Gamma}^!(*S)_K\otimes_K \mathbb{C} \cong M_{k+2,\Gamma}^!(*S)\ \text{, }\qquad M_{k+2,\Gamma}(mS)_K\otimes_K \mathbb{C} \cong M_{k+2,\Gamma}(mS).
\]

Note that we also have a residue sequence of $K$-vector spaces (\emph{cf.} \S \ref{par:appendix-residue}) whose base change to $\mathbb{C}$ is \eqref{eq:residue-seq-over-C}:
\begin{equation}\label{eq:residue-seq-over-K}
     \begin{tikzcd}[column sep = small]
        0 \arrow{r} & H^1_{\dR}(\mathcal{Y}_{\Gamma,K};\mathcal{V}_{k}^{\dR,K}) \arrow{r} & H^1_{\dR}(\mathcal{U}_{\Gamma,K};\mathcal{V}_{k}^{\dR,K}) \arrow{r}{\Res} & \bigoplus_{i=1}^r(\Sym^k H^1_{\dR}(\mathcal{E}_{w_i,K}/K))^{\Gamma_{w_i}} \arrow{r} & 0\ ,
    \end{tikzcd}
\end{equation}
where $\mathcal{E}_{w_i,K}$ is a $K$-model of the complex elliptic curve $\mathcal{E}_{w_i}$, which exists by the assumption that $j(w_i) \in K$ for all $i=1,\ldots,r$. The $K$-rational versions of the main results of  Section \ref{sec:cohomology-over-C} (namely Proposition \ref{prop:dR-cohomology-meromorphic-modular-forms} and Theorem \ref{thm:hodge-filtration-open-complex}, after rescaling by suitable powers of $2\pi i$) are summarised in the next proposition; the proofs are analogous.

\begin{thm}\label{thm:de-Rham-open-over-K}
    There is an exact sequence of $K$-vector spaces
    \[
        \begin{tikzcd}[row sep = -0.1cm, column sep = small]
              M^!_{-k,\Gamma}(*S)_K \arrow{r} & M_{k+2,\Gamma}^!(*S)_K \arrow{r} & H^1_{\dR}(\mathcal{U}_{\Gamma,K}; \mathcal{V}_k^{\dR,K}) \arrow{r} & 0\\
               h \rar[mapsto] & (q\frac{d}{dq})^{k+1}h \\
             & f \rar[mapsto] & {[f\, \omega_q^{k} \frac{dq}{q}]}
        \end{tikzcd}
    \]
    inducing isomorphisms 
    \[
        M_{k+2}((k+2-p)S)_K \stackrel{\sim}{\longrightarrow}F^pH^1_{\dR}(\mathcal{U}_{\Gamma,K};\mathcal{V}_k^{\dR,K})\text{, }\qquad p=1,\ldots,k+1.
    \]
    In particular, for every $p=1,\ldots,k+1$, we have a residue exact sequence 
    \[
        \begin{tikzcd}[column sep = small]
         0 \arrow{r} & (M_{k,\Gamma})_K  \arrow{r}\arrow{d}{\sim} & M_{k+2,\Gamma}((k+2-p)S)_K \arrow{r}{\Res}\arrow{d}{\sim} & \bigoplus_{i=1}^rF^{p-1}(\Sym^k H^1_{\dR}(\mathcal{E}_{w_i,K}/K))^{\Gamma_{w_i}} \arrow{r}\arrow{d}{\sim} & 0\\
            0 \arrow{r} & F^pH^1_{\dR}(\mathcal{Y}_{\Gamma,K};\mathcal{V}_{k}^{\dR,K}) \arrow{r} & F^pH^1_{\dR}(\mathcal{U}_{\Gamma,K};\mathcal{V}_{k}^{\dR,K}) \arrow{r}{\Res} & \bigoplus_{i=1}^rF^{p-1}(\Sym^k H^1_{\dR}(\mathcal{E}_{w_i,K}/K))^{\Gamma_{w_i}} \arrow{r} & 0.
        \end{tikzcd}
    \]
\end{thm}

\subsection{The field of definition of $\Psi^{r,s}_{\Gamma}(z,w)$}

In this paragraph, we only consider the case of a single orbit $S = \Gamma w$. Recall that $j(w)$ is algebraic and $K$ contains $\QQ(j(w))$. Furthermore, if $w$ is CM, then we shall also assume that $K$ contains the CM field $\mathbb{Q}(w)$. 

To fix ideas, we assume for the time being that $w$ is not an elliptic point, so that $\Gamma_{w} \le \{\pm I\}$. If $k$ is odd, we assume moreover that $\Gamma_w = \{I\}$ (otherwise the group  $\big(\Sym^k H^1_{\dR}(\mathcal{E}_{w,K}/K)\big)^{\Gamma_w}$ vanishes).  Then, if $[\omega],[\eta]$ is a $K$-basis of $H^1_{\dR}(\mathcal{E}_{w,K}/K)$ as in Lemma \ref{lem: PeriodsE}, we have $(\Sym^kH^1_{\dR}(\mathcal{E}_{w,K}/K))^{\Gamma_w} = \Sym^kH^1_{\dR}(\mathcal{E}_{w,K}/K)$ (\emph{cf.} Remark \ref{rem:kappa}), with a $K$-basis given by
\[
   [\eta]^m  [\omega]^n \text{, }\qquad m+n = k\text{, } m,n \ge 0.
\]
By the residue sequence in Theorem \ref{thm:de-Rham-open-over-K}, for every $m,n\ge 0$ satisfying $m+n = k $, there exists a meromorphic modular form with $K$-rational Fourier coefficients
\[
    f^{m,n} \in M_{k+2,\Gamma}((k+1-n)S)_K \cong F^{n+1}H^1_{\dR}(\mathcal{U}_{\Gamma,K};\mathcal{V}_{k}^{\dR,K})
\]
such that
\begin{equation}\label{eq:residue-fmn}
    \Res\, [(2\pi i)^{k+1}f^{m,n}(z)(X-zY)^kdz] = [\eta]^m [\omega]^n .
\end{equation}
Note that $f^{m,n}$ is not unique in general. By adding to $f^{m,n}$ a $K$-linear combination of Eisenstein series for $\Gamma$, we may assume that $f^{m,n}$ vanishes at the cusps.

\begin{prop}\label{prop:field-def-psi}
    Let $m,n\ge 0$ with $m+n=k$. With the above notation, and  $\omega_1$, $\rho$ as in Lemma \ref{lem: PeriodsE}, and $\kappa^{m,n}_{\Gamma,w} \in \{1,2\}$ as in Corollary \ref{cor:residue-Psi}, we have
    \begin{equation}\label{eq:field-def-psi}
        \Psi_{\Gamma}^{m,n}(z,w) = \binom{k}{m}\frac{\kappa^{m,n}_{\Gamma,w}(2\pi i)^{n+1}\omega_1^{m-n}}{(w-\overline{w})^n}\left(\sum_{j=0}^{m}\binom{m}{j}\overline{\rho}^jf^{m-j,n+j}(z)\right) + h^{m,n}(z)
    \end{equation}
    for some cusp form $h^{m,n}$  for $\Gamma$ of weight $k+2$ with  complex coefficients.
\end{prop}

\begin{proof}
    Let
    \[
        f(z) = \Psi_{\Gamma}^{m,n}(z,w)-\binom{k}{m}\frac{\kappa^{m,n}_{\Gamma,w}(2\pi i)^{n+1}\omega_1^{m-n}}{(w-\overline{w})^n}\sum_{j=0}^{m}\binom{m}{j}\overline{\rho}^jf^{m-j,n+j}(z) \in M_{k+2,\Gamma}((k+1-n)S).
    \]
    By Corollary \ref{cor:residue-Psi}, Lemma \ref{lem: PeriodsE}, and equation \eqref{eq:residue-fmn}, we have $\Res [f(z)(X-zY)^kdz] = 0$. Then, by Proposition \ref{prop:residue-seq-hodge-explicit}, there is a holomorphic modular form $h^{m,n}$ satisfying the equation in the statement. To see that $h^{m,n}$ is a cusp form, it suffices to remark that $\Psi^{m,n}_{\Gamma}$ (by Proposition \ref{prop:Psi}) and every $f^{m-j,n+j}$ (by construction) vanish at the cusps.
\end{proof}

\begin{ex} 
    We have in particular
    \[
        \Psi_{\Gamma}^{0,k}(z,w) = \frac{\kappa_{\Gamma,w}^{0,k}(2\pi i)^{k+1}}{\omega_1^k(w-\overline{w})^k}f^{0,k}(z) + h^{0,k}(z),
    \]
    where $f^{0,k}$ is a meromorphic modular form with $K$-rational Fourier coefficients, and $h^{0,k}$ is a cusp form. This is compatible with Example \ref{ex:psi0k-sl2}: when $\Gamma=\SL_2(\mathbb{Z})$ and $k\in\{2,4,6,8,12\}$, we have $\kappa_{\Gamma,w}^{0,k} = 2$ for non-elliptic $w$, $h^{0,k} = 0$, and $f^{0,k}(z)$ is $K$-proportional to $E_{k+2}(z) / (j(z)-j(w))$. The  value of the prefactor $j'(w)/E_{k+2}(w)$ is $K$-proportional to the quotient of periods $(2\pi i)^{k+1}/\omega_1^k$.
\end{ex}

If $w$ is CM (possibly elliptic), then $\overline{\rho} \in K$ by Lemma \ref{lem: periodsCMcase}, so that
\[
    [\eta']= [\eta] + \overline{\rho}[\omega] \in H^1_{\dR}(\mathcal{E}_{w,K}/K),
\]
and we can take $[\omega],[\eta']$ as a $K$-basis of $H^1_{\dR}(\mathcal{E}_{w,K}/K)$; it is a basis of eigenvectors for the action of complex multiplication (\emph{cf.} Lemmas \ref{lem: PeriodsE} and \ref{lem: periodsCMcase}). Moreover,
\[
    \kappa^{m,n}_{\Gamma,w}[\eta']^m [\omega]^n\text{, }\qquad m,n\ge 0\text{, } m+n=k
\]
generates the $K$-vector space $(\Sym^k H^1_{\dR}(\mathcal{E}_{w,K}/K))^{\Gamma_w}$. Again, by the residue sequence in Theorem \ref{thm:de-Rham-open-over-K}, and an argument involving Eisenstein series, for every $m,n\ge 0$ satisfying $m+n = k $, there is a meromorphic modular form with $K$-rational Fourier coefficients
\[
    g^{m,n} \in M_{k+2,\Gamma}((k+1-n)S)_K \cong F^{n+1}H^1_{\dR}(\mathcal{U}_{\Gamma,K};\mathcal{V}_{k}^{\dR,K})
\]
which vanishes at the cusps and satisfies (note the prefactors $\kappa^{m,n}_{\Gamma,w}$ compared to \eqref{eq:residue-fmn}) 
\begin{equation}\label{eq:residue-gmn}
    \Res\, [(2\pi i)^{k+1}g^{m,n}(z)(X-zY)^kdz] = \kappa^{m,n}_{\Gamma,w}[\eta']^m [\omega]^n.
\end{equation}

\begin{prop}\label{prop:rationality-psi}
    Let $m,n\ge 0$ with $m+n=k$ and  let $w \in \mathfrak{H}$ be a CM point. With the above notation, we have
    \[
        \Psi_{\Gamma}^{m,n}(z,w) = \binom{k}{m}\frac{(2\pi i)^{n+1}\omega_1^{m-n}}{(w-\overline{w})^n}g^{m,n}(z) + h^{m,n}(z)
    \]
    for some cusp form $h^{m,n}$ for $\Gamma$ of weight $k+2$. In particular, if $k=2m$ is even, then for the `middle term' $m=n$, there is no power of the period $\omega_1$:
    \begin{equation}\label{eq:field-def-psi-cm-middle}
        \Psi_{\Gamma}^{m,m}(z,w) = \binom{k}{m}\frac{(2\pi i)^{m+1}}{(w-\overline{w})^{m}}g^{m,m}(z) + h^{m,m}(z).
    \end{equation}
\end{prop}

\begin{proof}
    The proof is similar to that of Proposition \ref{prop:field-def-psi}: we set 
    \[
        g(z) = \Psi_{\Gamma}^{m,n}(z,w)-\binom{k}{m}\frac{(2\pi i)^{n+1}\omega_1^{m-n}}{(w-\overline{w})^n}g^{m,n}(z) \in M_{k+2,\Gamma}((k+1-n)S) , 
    \]
    and we use Corollary \ref{cor:residue-Psi}, Lemma \ref{lem: PeriodsE}, Remark \ref{rem: eigenbasisinBetti}, and equation \eqref{eq:residue-gmn}, to conclude that $\Res [g(z)(X-zY)^kdz] = 0$. The result then follows from Proposition \ref{prop:residue-seq-hodge-explicit}.
\end{proof}

Note that, in the above statement, the quantity $w - \overline{w}$ is algebraic. The proposition shows that, up to some powers of $2\pi i$ and $\omega_1$,  the meromorphic modular form $\Psi_{\Gamma}^{m,n}$ can be decomposed as the sum of an `algebraic part' (given by the meromorphic modular form $g^{m,n}$ with $K$-rational Fourier coefficients) and a `transcendental part' (given by the cusp form $h^{m,n}$, which has complex coefficients). It would be  interesting to study the Fourier coefficients of $\Psi_{\Gamma}^{m,n}$ from the point of view of single-valued periods of the corresponding  cohomology groups, in the same spirit as  \cite{CoeffPoincare}. We also remark that the cancellation of the powers of the period $\omega_1$ in equation \eqref{eq:field-def-psi-cm-middle} is a crucial point for the Gross-Zagier conjecture (see the proof of Theorem \ref{thm: GZ-conditional-proof}).

\begin{rem} \label{rem: Paloma}
    Let $\Gamma = \SL_2(\mathbb{Z})$ and $D<0$ be the discriminant of the CM point $w$. We may take $K = \mathbb{Q}(w,j(w))$ as the Hilbert class field of $\mathbb{Q}(w) = \mathbb{Q}(\sqrt{D})$. Let us recall the classical correspondence between binary quadratic  forms with integer coefficients  and CM points. The orbit $\Gamma w$ corresponds to a $\Gamma$-equivalence class $\mathcal{A}$ of primitive positive-definite binary quadratic forms $Q(U,V) = AU^2+BUV+CV^2$ of discriminant $D = B^2 - 4AC$ via
    \begin{equation}\label{eq:quadratic-form-to-w}
        Q\longmapsto  w(Q) = \frac{-B + \sqrt{D}}{2A}.
    \end{equation}
    Here, `primitive positive-definite' amounts to $A,B,C \in \mathbb{Z}$ coprime and $A>0$. Note that $w(Q)$ is the unique point of the upper half-plane such that $Q(w(Q),1)=0$; more precisely,
    \[
        Q(U,V) = \frac{\sqrt{D}}{w(Q)-\overline{w(Q)}}(U-w(Q)V)(U-\overline{w(Q)}V).
    \]
    Note also that  the map \eqref{eq:quadratic-form-to-w} is $\Gamma$-equivariant: $w(Q|_{\gamma^{-1}}) = \gamma w(Q)$ for every $\gamma \in \Gamma$. Let $Q$ be the quadratic form corresponding to $w$. Then
    \[
        (z-\gamma w)(z-\gamma \overline{w}) = \frac{\gamma w - \gamma \overline{w}}{\sqrt{D}}Q|_{\gamma}(z,1) = \frac{w-\overline{w}}{\sqrt{D}}\frac{Q|_{\gamma}(z,1)}{j_{\gamma}(w)j_{\gamma}(\overline{w})}
    \]
    and, by Proposition \ref{prop:Psi}, we have:
    \begin{align*}
        \Psi^{m,m}_{\Gamma}(z,w) &= \sum_{\gamma \in \Gamma}\frac{1}{j_{\gamma}^{m+1}(w)j_{\gamma}^{m+1}(\overline{w})}\frac{w-\overline{w}}{(z-\gamma w)^{m+1}(z-\gamma \overline{w})^{m +1}} = \frac{\sqrt{D}^{m +1}}{(w-\overline{w})^{m}}\sum_{\gamma \in \Gamma}\frac{1}{Q|_{\gamma}(z,1)^{m+1}}\\
        & = \frac{\sqrt{D}^{m +1}}{(w-\overline{w})^{m}}\sum_{R \in \mathcal{A}}\frac{1}{R(z,1)^{m+1}} = \frac{(\pi \sqrt{D})^{m+1}}{(w-\overline{w})^{m}}f_{m+1,D,\mathcal{A}}(z),
    \end{align*}
    where $f_{m+1,D,\mathcal{A}}$ is the meromorphic modular form considered in \cite[Equation (3)]{bengoechea}. In this case, our decomposition result \eqref{eq:field-def-psi-cm-middle} is essentially the same as \cite[Equation (10)]{bengoechea}.
\end{rem}

\subsection{Construction of general modular objects in $\mathcal{H}_K$} \label{par:general-modular-objects}

Let $\Gamma \le \SL_2(\mathbb{Z})$ be a congruence subgroup and $P, Q \subset \mathfrak{H}$ be unions of finitely many  $\Gamma$-orbits such that $P\cap Q = \emptyset$. We also assume that $j(w)$ is algebraic for every $w \in P\cup Q$ and we let $K\subset \mathbb{C}$ be a number field over which the open modular stack $\mathcal{Y}_{\Gamma}$ is defined and which also contains $\mathbb{Q}(j(w))$ for every $w \in P\cup Q$, and $\mathbb{Q}(w)$ for every CM point $w \in P\cup Q$. In particular, we have   a stack $\mathcal{Y}_{\Gamma,K}$ over $K$, with finite substacks $\mathcal{P}_{\Gamma,K},\mathcal{Q}_{\Gamma,K} \subset \mathcal{Y}_{\Gamma,K}$.

We shall describe an object of the category of realisations $\mathcal{H}_K$ (\emph{cf.} \S \ref{par:realisations}) whose de Rham realisation is the relative de Rham cohomology with coefficients
\[
    H^1_{\dR}(\mathcal{Y}_{\Gamma,K} \setminus \mathcal{P}_{\Gamma,K},\mathcal{Q}_{\Gamma,K}; \mathcal{V}_{k}^{\dR, K}).
\]
To study the action of complex conjugation, let 
\[
    \overline{\Gamma}  \defeq \epsilon \Gamma \epsilon^{-1} \leq \SL_2(\ZZ),\qquad \text{ where }\epsilon=\left( \begin{matrix} 1 & 0 \\ 0 & -1 \end{matrix}\right).
\]
Note that $\epsilon$ is idempotent and normalises $\SL_2(\mathbb{Z})$. 

\begin{lem}  
    The map $z \mapsto -\overline{z}$ induces a diffeomorphism of (real) orbifolds: $\mathcal{Y}_{\Gamma}^{\an} \cong  \mathcal{Y}_{\overline{\Gamma}}^{\an}$. Moreover, if a function $f(z)$ is modular for $\Gamma$ of weight $k$, then $\overline{f(-\overline{z})}$ is modular for $\overline{\Gamma}$ of weight $k$ and its Fourier coefficients are complex-conjugate to those of $f$.
\end{lem}

\begin{proof}
    One checks that $\gamma (-\overline{z}) = -\overline{(\epsilon \gamma \epsilon^{-1}) (z)}$ and $j_{\gamma}(-\overline{z}) = \overline{j_{\epsilon \gamma \epsilon^{-1}}(z)}$ for all $\gamma \in \SL_2(\ZZ)$.
\end{proof}

Let $\overline{K} = \overline{\iota}(K)\subset \mathbb{C}$ be the complex conjugate of the number field $K$ (see the notation in \S \ref{par:realisations}). Since modular curves defined over $K$ are parametrised by modular forms with $K$-rational Fourier coefficients, we deduce from the above lemma that the complex modular stack $\mathcal{Y}_{\overline{\Gamma}}$ admits an $\overline{K}$-model $\mathcal{Y}_{\overline{\Gamma},\overline{K}}$, which is isomorphic to the conjugate $\overline{\mathcal{Y}_{\Gamma,K}} = \mathcal{Y}_{\Gamma,K} \times_{K,\overline{\iota}}\overline{K}$, and that the complex conjugation
\begin{equation}\label{eq:complex-conj-modular-curve}
    \mathcal{Y}_{\Gamma}^{\an} \cong \mathcal{Y}_{\Gamma,K}(\mathbb{C}) \longrightarrow \overline{\mathcal{Y}_{\Gamma,K}}(\mathbb{C}) \cong \mathcal{Y}_{\overline{\Gamma}}^{\an}
\end{equation}
is induced by $z \mapsto -\overline{z}$ on the uniformisation.

The above complex conjugation extends to the level of universal elliptic curves $\mathcal{E}_{\Gamma}^{\an} \cong \mathcal{E}_{\overline{\Gamma}}^{\an}$: it is induced by $u\mapsto \overline{u}$ on the fibres $\mathcal{E}^{\an}_z \cong \mathbb{C}/(\mathbb{Z} + \mathbb{Z} z ) \to \mathcal{E}^{\an}_{-\overline{z}} \cong \mathbb{C}/(\mathbb{Z} - \mathbb{Z}\overline{z})$. In particular, it induces a real Frobenius between the Betti local systems $\mathbb{V}_{k}^{\B}$  on $\mathcal{Y}_{\Gamma}$ and $\mathcal{Y}_{\overline{\Gamma}}$,  given  explicitly  by
\[
    (X,Y) \longmapsto (X,Y)|_{\epsilon} = (X,-Y).
\]
Together with \eqref{eq:complex-conj-modular-curve}, this yields by functoriality a real Frobenius on cohomology
\[
    F_{\infty}: H^1_{\B}(\mathcal{Y}_{\Gamma}^{\an}\setminus \mathcal{P}_{\Gamma},\mathcal{Q}_{\Gamma};\mathbb{V}_k^{\B}) \longrightarrow H^1_{\B}(\mathcal{Y}_{\overline{\Gamma}}^{\an}\setminus \mathcal{P}_{\overline{\Gamma}},\mathcal{Q}_{\overline{\Gamma}};\mathbb{V}_k^{\B}),
\]
where $\mathcal{P}_{\overline{\Gamma}}$ denotes the image of $\mathcal{P}_{\Gamma}$ under $\eqref{eq:complex-conj-modular-curve}$ and similarly for $\mathcal{Q}_{\overline{\Gamma}}$.

\begin{defn} \label{defn: H1YPQ}
    With the above notation, we define for every integer $k\ge 0$ an object of the category of realisations $\mathcal{H}_K$ by
    \begin{multline} \label{eqn: defnH1YPQ}
        H^1 (\mathcal{Y}_{\Gamma,K} \setminus \mathcal{P}_{\Gamma,K}, \mathcal{Q}_{\Gamma,K} ; \mathcal{V}_k)  = \big( H^1_{\B}(\mathcal{Y}_{\Gamma}^{\an}\setminus \mathcal{P}_{\Gamma}, \mathcal{Q}_{\Gamma}  ; \mathbb{V}_k^{\B})  \ , \ H^1_{\B}(\mathcal{Y}_{\overline{\Gamma}}^{\an}\setminus \mathcal{P}_{\overline{\Gamma}} , \mathcal{Q}_{\overline{\Gamma}} ; \mathbb{V}_k^{\B})  \ , \\
        \ H_{\dR}^1(\mathcal{Y}_{\Gamma,K}\setminus \mathcal{P}_{\Gamma,K}, \mathcal{Q}_{\Gamma,K}; \mathcal{V}_k^{\dR,K})  \    ; \   \comp_{\B,\dR} , \comp_{\overline{\B},\dR}  , F_{\infty} \big)  
    \end{multline}
    where the real Frobenius was discussed above, and the comparison maps are induced by the usual Grothendieck comparison isomorphism on the universal elliptic curve.
\end{defn}

The Hodge filtration is defined as in \S\ref{par:appendix-hodge-filtration} (\emph{cf.} \S \ref{par:hodge-and-pole}). The weight filtration, which we do not use explicitly in this paper, as well as the fact that these form a mixed Hodge structure is worked out in \cite{zucker} (see Remark \ref{rem:MHS-zucker}). It is often more convenient to consider a \emph{cuspidal}, or \emph{parabolic}, variant
\[
    H^1_{\cusp}(\mathcal{Y}_{\Gamma,K} \setminus \mathcal{P}_{\Gamma,K}, \mathcal{Q}_{\Gamma,K} ; \mathcal{V}_k)
\]
of \eqref{eqn: defnH1YPQ}. This is the  sub-object  
of $H^1(\mathcal{Y}_{\Gamma,K} \setminus \mathcal{P}_{\Gamma,K}, \mathcal{Q}_{\Gamma,K} ; \mathcal{V}_k)$ whose (first) Betti realisation is 
\begin{align}\label{eq:betti-cuspidal-cohomology}
    \im (H^1(\mathcal{X}_{\Gamma}^{\an}\setminus \mathcal{P}_{\Gamma}, \mathcal{Q}_{\Gamma}; j_!\mathbb{V}_k^{\B}) \to H^1(\mathcal{Y}^{\an}_{\Gamma}\setminus \mathcal{P}_{\Gamma},\mathcal{Q}_{\Gamma}; \mathbb{V}_k^{\B})) \cong H^1(\mathcal{X}^{\an}_{\Gamma} \setminus \mathcal{P}_{\Gamma}, \mathcal{Q}_{\Gamma} ; j_* \mathbb{V}^{\B}_k),
\end{align}
where $j:\mathcal{Y}_{\Gamma,K}\rightarrow  \mathcal{X}_{\Gamma,K}$ is inclusion into the compactified modular stack (\emph{cf.} \cite[Lemma 5.3]{Looijenga} and \cite[Theorem 14.4]{zucker}). The de Rham realisation can also be explicitly described as the kernel of the `residue map at the cusps'
\[
    H^1_{\dR}(\mathcal{Y}_{\Gamma,K} \setminus \mathcal{P}_{\Gamma,K}, \mathcal{Q}_{\Gamma,K} ; \mathcal{V}^{\dR,K}_k)\longrightarrow H_{\dR}^0(\mathcal{Z}_{\Gamma,K};\overline{\mathcal{V}}_k^{\dR,K})/ \im (N)
\]
where $N$ is the nilpotent endomorphism given by the residue of the logarithimic connection $\overline{\nabla}_k^K$ on $\overline{\mathcal{V}}_k^{\dR,K}$ which extends $\nabla_k^K$ (\emph{cf.} \cite[Remark 14.5]{zucker}). Concretely, using Theorem \ref{thm:de-Rham-open-over-K}, one can check that cohomology classes in $H^1_{\dR,\cusp}(\mathcal{Y}_{\Gamma,K} \setminus \mathcal{P}_{\Gamma,K}, \mathcal{Q}_{\Gamma,K} ; \mathcal{V}^{\dR,K}_k)$ are represented by pairs
\[
    (f \omega_q^k \textstyle{\frac{dq}{q}}, v) = ((2\pi i)^{k+1}f(z) (X-zY)^kdz, (v_1,\ldots,v_s))
\]
where $f \in M_{\Gamma, k+2}^!(*P)_K$ is a meromorphic modular form which is holomorphic outside $P$ and the cusps, with $K$-rational Fourier coefficients at infinity, and vanishing constant Fourier coefficient at every cusp, \emph{i.e.}, $f$ is a \emph{cuspidal} meromorphic modular form. Furthermore, writing $Q$ as a disjoint union of $\Gamma$-orbits $Q= \Gamma z_1 \cup \cdots \cup \Gamma z_s$, the vector $v = (v_1,\ldots,v_s)$ is an element of $H_{\dR}^0(\mathcal{Q}_{\Gamma,K};\mathcal{V}_k^{\dR,K}) \cong \bigoplus_{i=1}^s \Sym^k (H^1_{\dR}(\mathcal{E}_{z_i,K}/K))^{\Gamma_{z_i}}$.

\section{Modular biextensions and their single-valued periods}
We study in more detail the structure of the relative cohomology modular objects defined in the previous section, and show that they are biextensions.  We first consider the case where they reduce to  simple  extensions before turning to the general case. We then explain how complex multiplication, and the action of Hecke correspondences, enables us to extract a subquotient which is a simple extension of Tate objects. 

\subsection{Simple modular extensions}

Consider the situation of Definition \ref{defn: H1YPQ} when $P=\emptyset$ and $Q = \Gamma z_1 \cup \cdots \cup \Gamma z_s$, with $z_i \not\in \Gamma z_j$ for $i\neq j$. The long exact relative cohomology sequence in the category $\mathcal{H}_K$ is:
\[
    \begin{tikzcd}[column sep = small]
         H^0(\mathcal{Y}_{\Gamma,K}; \mathcal{V}_k)\arrow{r} & H^0(\mathcal{Q}_{\Gamma,K};  \mathcal{V}_k) \arrow{r} & H^1 (\mathcal{Y}_{\Gamma,K} ,  \mathcal{Q}_{\Gamma,K}  ; \mathcal{V}_k) \arrow{r} & H^1 (\mathcal{Y}_{\Gamma,K} ; \mathcal{V}_k) \arrow{r} & 0.
    \end{tikzcd}
\]
If $k>0$ then $H^0(\mathcal{Y}_{\Gamma,K}; \mathcal{V}_k)$ vanishes and we obtain a short exact sequence: 
\[
    \begin{tikzcd}[column sep = small]
        0 \arrow{r} & \bigoplus_{i=1}^s \left(\Sym^k  H^1( \mathcal{E}_{z_i,K})\right)^{\Gamma_{z_i}} \arrow{r} & H^1 (\mathcal{Y}_{\Gamma,K} , \mathcal{Q}_{\Gamma,K}  ; \mathcal{V}_k) \arrow{r} & H^1 (\mathcal{Y}_{\Gamma,K} ; \mathcal{V}_k) \arrow{r} & 0 \ . 
    \end{tikzcd}
\]
By taking the image of  cohomology with compact supports and applying \eqref{eq:betti-cuspidal-cohomology}, we obtain an exact sequence
\begin{equation} \label{extension: RelcohomCusp}
    \begin{tikzcd}[column sep = small]
        0 \arrow{r} & \bigoplus_{i=1}^s \left( \Sym^k  H^1( \mathcal{E}_{z_i,K})\right)^{\Gamma_{z_i}} \To  H_{\cusp}^1 (\mathcal{Y}_{\Gamma,K} ,  \mathcal{Q}_{\Gamma,K}  ; \mathcal{V}_k) \To   H_{\cusp}^1 (\mathcal{Y}_{\Gamma,K}   ; \mathcal{V}_k) \To 0 \ . 
    \end{tikzcd}
\end{equation}
Furthermore, by the Manin--Drinfeld theorem, one can also deduce another exact sequence which will play a limited role in this paper but which is of significant arithmetic interest:  
\begin{equation}\label{extension: RelcohomEis}
    \begin{tikzcd}[column sep = small]
        0 \arrow{r}& \bigoplus_{i=1}^s \left(\Sym^k  H^1( \mathcal{E}_{z_i,K})\right)^{\Gamma_{z_i}} \arrow{r} &  H_{\eis}^1 (\mathcal{Y}_{\Gamma,K} ,  \mathcal{Q}_{\Gamma,K}  ; \mathcal{V}_k) \arrow{r} &  H^0(\mathcal{Z}_{\Gamma,K})(-k-1) \arrow{r} & 0 \ .
    \end{tikzcd}
\end{equation}
By Theorem \ref{thm:de-Rham-open-over-K}, in the simple case when there are no punctures, 
the periods of these extensions include generalised Eichler integrals
\begin{equation} \label{Eqn: generalisedEichlerint}
    (2 \pi i)^{k+1}  \, \int_{\delta}  f(z) (X-zY)^{k} dz
\end{equation}
where $\delta$ is a path in $\mathcal{Y}_{\Gamma,K}(\CC)$ whose endpoints are contained in $\mathcal{Q}_{\Gamma,K}$, and where $f$ is a cuspidal weakly holomorphic modular form with $K$-rational Fourier coefficients in the first case   \eqref{extension: RelcohomCusp}, and an Eisenstein series in the second
\eqref{extension: RelcohomEis}.

\begin{rem}
    The single-valued periods of the extensions \eqref{extension: RelcohomCusp} and \eqref{extension: RelcohomEis} are given by the single-valued versions of the integrals \eqref{Eqn: generalisedEichlerint}. In the former case, these include the values at $z_i$ of suitably normalised  weak harmonic lifts of cusp forms; in the latter, the values of real-analytic Eisenstein series, which are weak harmonic lifts of  holomorphic Eisenstein series.  This follows from Lemmas \ref{lem: SVandCB} and  \ref{lem:value-harmonic-lift}.

    In particular, let $Q=\Gamma z_1$ be a single $\Gamma$-orbit corresponding to an elliptic curve $\mathcal{E}_{z_1}$ with complex multiplication. We may assume for simplicity that $\Gamma_{z_1} \subset  \{ \pm I\}$ and $k$ is even. By Proposition \ref{prop: CMmotives}, the extension \eqref{extension: RelcohomEis} becomes
    \[
        \begin{tikzcd}[column sep = small]
            0 \arrow{r} & \bigoplus_{0\leq 2i \leq  k}  C^{k-2i}(-i) \arrow{r} & H_{\eis}^1 (\mathcal{Y}_{\Gamma,K} ,  \mathcal{Q}_{\Gamma,K}  ; \mathcal{V}_k) \arrow{r} &  H^0(\mathcal{Z}_{\Gamma,K})(-k-1) \arrow{r} & 0.
        \end{tikzcd}
    \]
    The periods of such an extension include the  (regularised) Eichler integrals \eqref{Eqn: generalisedEichlerint} where $f(z)$ is an Eisenstein series and $\delta$ a path  from the CM point $z_1$ to a cusp. Their single-valued versions are given by values of  real-analytic Eisenstein series at the CM point $z_1$.  Half of them correspond to regulators which are predicted by Beilinson's conjecture  to be values of $L$-functions of Hecke Grössencharakters. The other half are periods which do not have an interpretation as values of classical $L$-functions. 
   For example, in level $1$, when the point $z_1=i$, these additional periods may be interpreted as values  $\xi(2m,2n)$ of the double Riemann $\xi$-function studied in \cite{BrMultipleXi}, see \emph{loc. cit.}, Remark 9.8. They can be expressed as rational linear combinations of  `double quadratic sums' 
   $Q(k_1,k_2)= \sum_{n_1,n_2\geq 1}  \frac{1}{(n^2_1+n^2_2)^{k_1} (n_2^2)^{k_2} }$  for $k_1,k_2\geq 1$ (\emph{id.} Definition 9.3). 
\end{rem}

\subsection{Biextensions} \label{par:biextensions}

Now consider the fully general situation  of Definition  \ref{defn: H1YPQ},    where  $\mathcal{Q}_{\Gamma,K}$ is as above and 
 $\mathcal{P}_{\Gamma,K} = \Gamma w_1 \cup \cdots \cup \Gamma w_r$, with $w_i \not\in \Gamma w_j$ for $i\neq j$.
If $k>0$ we have an exact sequence:
\begin{equation}  \label{eqn: YminusPrelQmotive}
    \begin{tikzcd}[column sep = small]
        0 \arrow{r} & \bigoplus_{i=1}^s \left(\Sym^k \! H^1( \mathcal{E}_{z_i,K})\right)^{\Gamma_{z_i}} \arrow{r} &  H_{\cusp}^1 (\mathcal{Y}_{\Gamma,K} \setminus \mathcal{P}_{\Gamma,K} ,  \mathcal{Q}_{\Gamma,K} ; \mathcal{V}_k) \arrow{r} & H_{\cusp}^1 (\mathcal{Y}_{\Gamma,K}  \setminus \mathcal{P}_{\Gamma,K} ; \mathcal{V}_k) \arrow{r} & 0
    \end{tikzcd} 
\end{equation}
where the last term sits in a short exact sequence of its own: 
\begin{equation} \label{eqn: YminusPmotive}
    \begin{tikzcd}[column sep = small]
        0 \arrow{r} &   H_{\cusp}^1 (\mathcal{Y}_{\Gamma,K} ; \mathcal{V}_k) \arrow{r} &\arrow{r}  H_{\cusp}^1 (\mathcal{Y}_{\Gamma,K}  \setminus \mathcal{P}_{\Gamma,K}  ; \mathcal{V}_k) \To    \bigoplus_{j=1}^r  \left(\Sym^k  H^1( \mathcal{E}_{w_j,K})\right)^{\Gamma_{w_j}}(-1) \arrow{r}& 0  
    \end{tikzcd}
\end{equation}

\begin{rem}
    The periods of the extension \eqref{eqn: YminusPmotive} include Eichler-type integrals of meromorphic modular forms with poles along $w_i$, along paths between two cusps. There exists a `reciprocity formula' which relates such integrals to the periods of \eqref{extension: RelcohomCusp} after identifying $z_i$ with $w_i$. It interchanges the poles of the integrand with the end points of the path of integration  and  conversely, and follows from the fact that the latter extension is the Poincar\'e dual of the former. In particular, their period matrices with respect to suitable bases are mutually  inverse  (up to powers of $2\pi i$). Consequently,  Eichler integrals of meromorphic modular forms  with poles at $w$ are  expressible as integrals of Eisenstein series along paths ending at $w$, and conversely.  It would interesting to write down explicit formulae for such relations. 
\end{rem}

To describe the single-valued  periods of such extensions we introduce the following.

\begin{defn}\label{defn:greens-fct-matrix}
    With the notation \eqref{Gpqrsdefn}, define a $(k+1) \times (k+1)$ matrix:
    \begin{equation} \label{GmatrixDefn}
        \GM_{\Gamma,k}(z,w)= 
       \begin{pmatrix}
        {}^{\Gamma}\!G^{k,0}_{0,k}(z,w) & \cdots &  {}^{\Gamma}\! G^{0,k}_{0,k}(z,w)  \\ 
        \vdots & \reflectbox{$\ddots$} &  \vdots  \\
        {}^{\Gamma} \! G^{k,0}_{k,0}(z,w) & \cdots & {}^{\Gamma}\! G^{0,k}_{k,0}(z,w)    \\
        \end{pmatrix}
    \end{equation}
    whose entries are $(\GM_{\Gamma,k}(z,w))_{r,s}=   {}^{\Gamma}\!G^{k-s+1,s-1}_{r-1,k-r+1}(z,w)$ for $1\leq r,s \leq k+1$.  
\end{defn}

Since 
$\overline{{}^{\Gamma}\! G^{p,q}_{r,s}(z,w)} = {}^{\Gamma}\!G^{q,p}_{s,r}(z,w) $, the complex conjugate matrix  has entries 
$\overline{\GM}_{\Gamma,k}(z,w)_{r,s}=   {}^{\Gamma}\!G^{s-1, k-s+1}_{k-r+1,r-1}(z,w)$ for $1\leq r,s \leq k+1$.

\subsubsection{The case in which  $H_{\cusp}^1 (\mathcal{Y}_{\Gamma,K}  ; \mathcal{V}_k)$ vanishes}

Suppose that $H_{\cusp}^1 (\mathcal{Y}_{\Gamma,K}  ; \mathcal{V}_k)=0$, and let $P=\Gamma w_1$ and $Q=\Gamma z_1$ consist of a single orbit. Then the biextension \eqref{eqn: YminusPrelQmotive} reduces to a simple extension:
\begin{equation} \label{eqn: Extensionvanishingcuspforms}
    \begin{tikzcd}[column sep = small]
        0 \arrow{r} &\big( \Sym^k  H^1( \mathcal{E}_{z_1,K})\big)^{\Gamma_{z_1}} \arrow{r}&  H_{\cusp}^1 (\mathcal{Y}_{\Gamma,K} \setminus \mathcal{P}_{\Gamma,K} ,  \mathcal{Q}_{\Gamma,K} ; \mathcal{V}_k) \arrow{r} & \big(\Sym^k  H^1( \mathcal{E}_{w_1,K})\big)^{\Gamma_{w_1}}(-1)  \arrow{r}& 0
    \end{tikzcd}
\end{equation}

\begin{thm} \label{thm: PeriodsSkExtensionNoCusps}
    Let $M_{z_1,w_1}= H_{\cusp}^1 (\mathcal{Y}_{\Gamma,K} \setminus \mathcal{P}_{\Gamma,K} , \mathcal{Q}_{\Gamma,K}  ; \mathcal{V}_k)$ with $P = \Gamma w_1$ and $Q = \Gamma z_1$ as above, and assume that neither $z_1$ nor $w_1$ is  elliptic. Then we have an extension 
    \[
        \begin{tikzcd}[column sep = small]
            0 \arrow{r} & \Sym^k  H^1( \mathcal{E}_{z_1,K}) \arrow{r} & M_{z_1,w_1} \arrow{r} & \Sym^k  H^1( \mathcal{E}_{w_1,K})(-1) \arrow{r} & 0
        \end{tikzcd}
    \]
    in the category $\mathcal{H}_K$. A basis for $(M_{z_1,w_1})_{\dR}\otimes_K \CC$ 
    is given by relative cohomology classes
    \[ 
        [ (0,  \mathrm{comp}_{\B,\dR}^{-1} \, (X-z_1Y)^i (X-\overline{z}_1 Y)^{k-i})  ]_{i=0,\ldots, k}  \quad , \quad    [  (\Psi_{\Gamma}^{k-r,r}(z,w_1) (X-zY)^k dz,0 )]_{r=0,\ldots, k}     
    \]
    In order to compute the action of    $\mathrm{comp}_{\B,\dR}^{-1}$ on the left-hand elements, use \eqref{EllipticBettiBasis}. The single-valued period map $\sf{s}$ is given  with respect to this basis  by the matrix

    \[
   \begin{pmatrix}
        J  &- \GM_{\Gamma,k}(z_1,w_1) \\
        0 & -J 
    \end{pmatrix}  
    \]
    where $J_{r,s}=\delta_{r,k+1-s}$ is the  $(k+1)\times (k+1)$ square matrix with $1$'s along the antidiagonal and zeros elsewhere, and $\GM_{\Gamma,k}$ is defined in \eqref{GmatrixDefn}.   
\end{thm}

\begin{proof}
    Since $H^1_{\cusp}(\mathcal{Y}_{\Gamma,K};\mathcal{V}_k)$ vanishes, the residue map
    \[
        H^1_{\cusp}(\mathcal{Y}_{\Gamma,K}\setminus \mathcal{P}_{\Gamma,K}; \mathcal{V}_k) \To  \Big(\Sym^k  H^1( \mathcal{E}_{w_1,K})\Big)^{\Gamma_{w_1}}(-1) 
    \]
    is an isomorphism, and it follows from Corollaries  \ref{cor:residue-Psi} and \ref{cor:Psi-splits-residue-seq}
    that the classes of the vector-valued forms $\Psi_{\Gamma}^{k-r,r}(z, w_1) (X-zY)^k dz $, for $r=0,\ldots,k$, form a basis for $H^1_{\dR,\cusp}(\mathcal{Y}_{\Gamma,K}\setminus \mathcal{P}_{\Gamma,K}; \mathcal{V}^{\dR,K}_k) \otimes_K \CC$ which is adapted to the Hodge filtration. The statement about the basis of relative cohomology then follows from the description of the mapping cone complex \S\ref{sect: A1}.

  We compute the single-valued period matrix by first computing the Betti-conjugation period matrix and applying Lemma \ref{lem: SVandCB}.  By Definition \ref{defn: Betticonj}, the map $c_{\iota}$ is just the Betti complex-conjugation $c_{\B}$ transported to de Rham cohomology.  Note that $c_{\B}$ acts on the classes $(X-z_1Y)^i (X-\overline{z}_1Y)^j $ by interchanging $i$ and $j$, which accounts for the matrix $J$ in the top-left block and the zero block beneath it. Equation \eqref{eqn: GisWHLofPsis} implies that $\vec{G}_{\Gamma,k}^{p,q}(z,w_1) =\sum_{r+s=k} {}^{\Gamma}\!G^{p,q}_{r,s}(z,w_1)(X-zY)^r(X-\overline{z}Y)^s$ 
    is a harmonic lift of the pair $\Psi_{\Gamma}^{p,q}(z,w_1)(X-zY)^k dz$ and $-\Psi_{\Gamma}^{q,p}(z,w_1)(X-zY)^k dz$.  Applying  Lemma \ref{lem:value-harmonic-lift} implies that
    \[
        c_{\iota}  [(-\Psi_{\Gamma}^{q,p}(z,w_1)(X-zY)^k dz,0) ] = [(\Psi_{\Gamma}^{p,q}(z,w_1)(X-zY)^k dz , 0)] +   [(0,\vec{G}_{\Gamma,k}^{p,q}(z_1,w_1))],
    \] 
    which yields the remaining entries of the  Betti-conjugation period matrix. The single-valued period matrix is its complex-conjugate. Conclude using the fact that  $\overline{{}^{\Gamma}\! G^{p,q}_{r,s}} = {}^{\Gamma}\!G^{q,p}_{s,r}$.  
\end{proof}

Let us now assume that $\mathcal{E}_{w_1,K}$ has complex multiplication with respect to an order in a quadratic imaginary field $F \subset K$ and assume that $k=2m $ is even. Applying Proposition \ref{prop: CMmotives} we have 
\begin{equation}\label{eq:decompSnintoCn}
    \Sym^k H^1( \mathcal{E}_{w_1,K})   =  C^{2m } \oplus C^{2m-2}(-1) \oplus \cdots \oplus C^2(-m+1) \oplus \QQ(-m) .
\end{equation}

\begin{lem} \label{lem: invariantscontainsTate}
    Let $w_1$ be a CM  point. Then, in the decomposition \eqref{eq:decompSnintoCn}, the Tate piece is invariant:
    \[
        \QQ(-m)  \   \subset   \   \left(\Sym^{2m} H^1 (\mathcal{E}_{w_1,K})\right)^{\Gamma_{w_1}}  \ .
    \]
\end{lem}

\begin{proof} 
    We use the notation of Remark \ref{rem: defineMplusMminusCMcase}. Since $\Gamma_{w_1}$ acts via automorphisms of $\mathcal{E}_{w_1}$, it preserves the polarisation.  It follows that $M_+ \otimes M_-=\bigwedge^2 (H^1(\mathcal{E}_{w_1,K})\otimes F)=\QQ(-1)$ is $\Gamma_{w_1}$-invariant in the category $\mathcal{H}_K\otimes F$. It follows that all powers $(M_+ \otimes M_-)^{m} \subset \Sym^{2m} (H^1(\mathcal{E}_{w_1,K})\otimes F)$ are $\Gamma_{w_1}$-invariant. In particular, the copy of $\QQ(-m)$ inside $ \Sym^{2m} H^1(\mathcal{E}_{w_1,K})$ in the category $\mathcal{H}_K$ is $\Gamma_{w_1}$-invariant.
\end{proof}

\begin{thm}
    Suppose that $w_1$ and $z_1$ are CM. Then we can  extract from \eqref{eqn: Extensionvanishingcuspforms} an extension 
    \begin{equation} \label{NocuspsKummerExtension}
        \begin{tikzcd}[column sep = small]
            0  \arrow{r}& \QQ(-m) \arrow{r}& \mathcal{GZ}_{z_1,w_1} \arrow{r}& \QQ(-m-1) \arrow{r}& 0
        \end{tikzcd}
    \end{equation}
    in the category $\mathcal{H}_K$. A de Rham basis for $(\mathcal{GZ}_{z_1,w_1})_{\dR} \otimes_K \CC$ is given by the relative cohomology classes
    \[ 
        [ (0,  \mathrm{comp}_{\B,\dR}^{-1} (X-z_1Y)^{m} (X-\overline{z}_1 Y)^{m}  ]   \ , \quad    [  (\Psi_{\Gamma}^{m,m}(z,w_1) (X-zY)^k dz,0 )] 
    \]
    The single-valued  map $\mathsf{s}$ with respect to this basis is given by the  $2 \times 2$ matrix
    \[
        \begin{pmatrix}
            1 &    - {}^{\Gamma}\!G^{m,m}_{m,m}(z_1,w_1) \\
                & -1 
        \end{pmatrix}.
    \]
\end{thm}

\begin{proof}
    By Lemma \ref{lem: invariantscontainsTate}, we may pull back \eqref{eqn: Extensionvanishingcuspforms} along $\QQ(-m-1) \subset  \left(\Sym^{2m} H^1 (\mathcal{E}_{w_1,K})\right)^{\Gamma_{w_1}}(-1) $ to deduce an extension 
    \begin{equation}
        \begin{tikzcd}[column sep = small]
            0 \arrow{r} & \left( \Sym^k  H^1( \mathcal{E}_{z_1,K})\right)^{\Gamma_{z_1}} \arrow{r} & M \arrow{r} & \QQ(-m-1) \arrow{r} & 0.
        \end{tikzcd}
    \end{equation}
    Since $z_1$ is also CM, we apply a similar decomposition for $\Sym^{2m} H^1(\mathcal{E}_{z_1,K})$ and push out along the quotient $\QQ(-m)$ of $\left( \Sym^{2m} H^1( \mathcal{E}_{z_1,K})\right)^{\Gamma_{z_1}} $, yielding the extension \eqref{NocuspsKummerExtension}. 

    It follows from the proof of Proposition \ref{prop: CMmotives} that the action of complex multiplication on the de Rham realisation of $\Sym^{2m} H^1(\mathcal{E}_{z_1,K})$ splits the Hodge filtration. The Tate object inside it corresponds to the part of Hodge type $(m,m)$  which is spanned, over $\CC$, by the image of $(X-z_1Y)^{m} (X-\overline{z}_1Y)^{m}$ by Remark \ref{rem: eigenbasisinBetti}. The same argument applied to $\Sym^{2m} H^1(\mathcal{E}_{w_1,K})(-1)$ and the calculation of Theorem \ref{thm: PeriodsSkExtensionNoCusps}, proves that $[(\Psi_{\Gamma}^{m,m}(z,w_1)(X-zY)^{2m}dz,0)]$ spans the de Rham cohomology, over $\CC$, of the Tate piece  $\QQ(-m-1)$ of $\Sym^{2m} H^1(\mathcal{E}_{w_1,K})(-1)$. This proves that a de Rham basis of $(\mathcal{GZ}_{z_1,w_1})_{\dR} \otimes_{K} \CC$ is as stated. The computation of the 
    single-valued period matrix follows from Theorem \ref{thm: PeriodsSkExtensionNoCusps}.
\end{proof}

\subsection{Reminders on Hecke correspondences} \label{sect: HeckeCorres}

By a \emph{Hecke correspondence} on $\mathcal{Y}_{\Gamma,K}$, we mean a correspondence $T_{\alpha} $ given by a diagram  of the form
\[
    \begin{tikzcd}[column sep = small, row sep = small]
        &\mathcal{Y}_{\Gamma \cap \alpha^{-1}\Gamma \alpha ,K} \arrow{dl}[swap]{q_1} \arrow{dr}{q_2}&\\
        \mathcal{Y}_{\Gamma,K} & & \mathcal{Y}_{\Gamma,K}
    \end{tikzcd}
\]
where $\alpha \in M_{2\times 2}(\mathbb{Z})$ has positive determinant and $q_1$ (resp. $q_2$) is induced on the uniformisation  by $z\mapsto z$ (resp. $z \mapsto \alpha z$). Such a correspondence fits into a commutative diagram (\emph{cf.} \cite[3.16]{deligneFM})
\[
    \begin{tikzcd}[column sep = small, row sep = small]
    & q_1^*\mathcal{E}_{\Gamma,K} \arrow{rd} \arrow{ld} \arrow{rr}{\varphi_{\alpha}} & & q_2^*\mathcal{E}_{\Gamma,K} \arrow{ld} \arrow{rd} & \\
     \mathcal{E}_{\Gamma,K} \arrow{rd}  &  &\mathcal{Y}_{\Gamma \cap \alpha^{-1}\Gamma \alpha ,K} \arrow{dl}[swap]{q_1} \arrow{dr}{q_2}& & \mathcal{E}_{\Gamma,K}\arrow{ld} \\
       &  \mathcal{Y}_{\Gamma,K} & & \mathcal{Y}_{\Gamma,K} & 
    \end{tikzcd}
\]
where the action of $\varphi_{\alpha}$ on the fibres $\mathcal{E}_{z} \to \mathcal{E}_{\alpha z}$ is induced by $u\mapsto \frac{\det(\alpha)}{j_{\alpha}(z)}u$, where $u \in \mathbb{C}$ (\emph{cf.} computations in \S \ref{par:cohomology-EC}). In particular, $T_{\alpha}$ acts via $(q_1)_*\circ q_2^*$ on the modular object $H^1_{\cusp}(\mathcal{Y}_{\Gamma};\mathcal{V}_k)$ of $\mathcal{H}_K$.  More generally, $T_{\alpha}$ induces a morphism
\[
    T_{\alpha}: H^1_{\cusp}(\mathcal{Y}_{\Gamma,K} \setminus \mathcal{P}_{\Gamma,K},\mathcal{Q}_{\Gamma,K}; \mathcal{V}_k) \longrightarrow H^1_{\cusp}(\mathcal{Y}_{\Gamma,K} \setminus \mathcal{P}'_{\Gamma,K},\mathcal{Q}'_{\Gamma,K}; \mathcal{V}_k)
\]
in $\mathcal{H}_K$, which is compatible with the residue and relative cohomology long exact sequences; here, $\mathcal{P}'_{\Gamma,K} = T_{\alpha}^{-1}\mathcal{P}_{\Gamma,K}= q_1 q_2^{-1}\mathcal{P}_{\Gamma,K}$ and similarly for $\mathcal{Q}'_{\Gamma,K}$. Note that this means that $T_{\alpha}$ is also compatible with all the structures implicit in the definition of  $\mathcal{H}_K$, such as realisation functors, comparison and single-valued maps, and the Hodge and weight filtrations.

By representing de Rham cohomology classes by modular forms as in Theorem \ref{thm:de-Rham-open-over-K} (or Proposition \ref{prop:dR-cohomology-meromorphic-modular-forms}), we obtain the usual formula for the `Hecke operator': if $f(z)$ is modular of weight $k+2$, then
\[
    T_{\alpha}f(z) = \det(\alpha)^{k+1}\sum_{i=1}^n\frac{1}{j_{\beta_i}^{k+2}(z)}f(\beta_i z),
\]
where $\beta_i$ are chosen to satisfy $\Gamma \alpha \Gamma = \bigsqcup_{i=1}^n \Gamma \beta_i = \bigsqcup_{i=1}^n \beta_i \Gamma$. If $f(z,w)$ is modular in two variables, we shall always consider the action on $f$ as a function of the first variable $z$.

\begin{lem} \label{lem: kappainvarianceunderGL2}
    For every $\delta \in \GL_2(\mathbb{Q})$, \eqref{kappdef} satisfies $\kappa_{\delta^{-1}\Gamma\delta,\delta^{-1}w}^{p,q} = \kappa_{\Gamma,w}^{p,q}$.
\end{lem}

\begin{proof}
    Since $(\delta^{-1}\Gamma\delta)_{\delta^{-1}w} = \delta^{-1}\Gamma_w \delta$, we can write
    \[
        \kappa_{\delta^{-1}\Gamma\delta,\delta^{-1}w}^{p,q} = \sum_{\gamma \in \Gamma_w} \frac{1}{j_{\delta^{-1}\gamma\delta}^p(\delta^{-1}w)j_{\delta^{-1}\gamma\delta}^q(\delta^{-1}\overline{w})}.
    \]
    From the cocycle relations, we have on the one hand $j_{\delta^{-1}\gamma}(w) =  j_{\delta^{-1}\gamma\delta \delta^{-1}}(w) = j_{\delta^{-1}\gamma\delta}(\delta^{-1}w)j_{\delta^{-1}}(w)$, and on the other hand $j_{\delta^{-1}\gamma}(w)= j_{\delta^{-1}}(\gamma w) j_{\gamma}(w) = j_{\delta^{-1}}(w)j_{\gamma}(w)$, where the last equality comes from the fact that $\gamma$ stabilizes $w$. Comparing the two expressions for $j_{\delta^{-1}\gamma}(w)$  proves that $j_{\delta^{-1}\gamma\delta}(\delta^{-1}w) = j_{\gamma}(w)$, and analogously for $\overline{w}$ in the place of $w$, from which the statement follows.
\end{proof}

\begin{example}
    Given $p,q\ge 0$ such that $p+q = k$, it follows from part (3) of Proposition \ref{prop:Psi} that
    \[
        T_{\alpha}\Psi_{\Gamma}^{p,q}(z,w) = \sum_{i=1}^n\frac{1}{j_{\beta^{-1}_i}^p(w)j_{\beta^{-1}_i}^q(\overline{w})}\Psi_{\beta_i^{-1}\Gamma\beta_i}^{p,q}(z,\beta_i^{-1}w).
    \]
    By Lemma \ref{lem: kappainvarianceunderGL2},   \eqref{kappdef} satisfies $\kappa_{\beta_i^{-1}\Gamma\beta_i,\beta_i^{-1}w}^{p,q} = \kappa_{\Gamma,w}^{p,q}$. In particular, it follows from Corollary \ref{cor:residue-Psi} that, for every $i=1,\ldots,n$, we have
    \begin{equation}\label{eq:residue-hecke-Psi}
        \begin{aligned}
        &\Res  \left( T_{\alpha}\Psi_{\Gamma}^{p,q}(z,w)(X-zY)^{k}dz \right)= \\
        & \ \ \ \ \ \ \ \ \ \ \ \ \ \ \ \ \ \ \ \frac{(-1)^p\kappa_{\Gamma,w}^{p,q}\det(\alpha)^{k}}{(w-\overline{w})^{k}}  \binom{k}{p}\sum_{i=1}^n j_{\beta^{-1}_i}^q(w)j_{\beta^{-1}_i}^p(\overline{w})(X-  \beta_i^{-1}w Y)^q(X- \beta_i^{-1}\overline{w}Y)^p.
        \end{aligned}
    \end{equation}
    If $p=q$, the coefficients in this expression are real. If, furthermore,  $w$ is CM, then they are rational. 
\end{example}

The above example expresses, in the de Rham realisation, the commutativity of the diagram  with exact rows in $\mathcal{H}_K$
\[
    \begin{tikzcd}[column sep = small]
        0 \arrow{r} & H^1_{\cusp}(\mathcal{Y}_{\Gamma,K};\mathcal{V}_k) \arrow{r}\arrow{d} & H^1_{\cusp}(\mathcal{Y}_{\Gamma,K}\setminus \mathcal{P}_{\Gamma,K};\mathcal{V}_k) \arrow{r}{\Res}\arrow{d}{T_{\alpha}} & H^0(\mathcal{P}_{\Gamma,K};\mathcal{V}_k)(-1)\arrow{d} \arrow{r} & 0\\
        0 \arrow{r} & H^1_{\cusp}(\mathcal{Y}_{\Gamma,K};\mathcal{V}_k) \arrow{r} & H^1_{\cusp}(\mathcal{Y}_{\Gamma,K}\setminus \mathcal{P}'_{\Gamma,K};\mathcal{V}_k) \arrow{r}{\Res} & H^0(\mathcal{P}'_{\Gamma,K};\mathcal{V}_k)(-1) \arrow{r} & 0
    \end{tikzcd}
\]
in the case when $\mathcal{P}_{\Gamma,K} = \Gamma w$. Indeed, in this case, $\mathcal{P}_{\Gamma,K}'$ is given by the union of the orbits of $\beta_i^{-1}w$, the residue \eqref{eq:residue-hecke-Psi} lies in $H^0_{\dR}(\mathcal{P}_{\Gamma}';\mathcal{V}_k^{\dR,\mathbb{C}})$, the $i$th term in the sum being an element of $\Sym^kH^1_{\dR}(\mathcal{E}_{\beta_i^{-1}w}/\mathbb{C})$, and one can check that the vertical arrow on the right is given explicitly by
\begin{equation}\label{eq:hecke-vector}
    (X-wY)^r(X-\overline{w}Y)^s \longmapsto {\det}(\alpha)^k\sum_{i=1}^nj_{\beta_i^{-1}}^r(w)j_{\beta_i^{-1}}^s(\overline{w})(X-\beta_i^{-1}wY)^r(X-\beta_i^{-1}\overline{w}Y)^s.
\end{equation}

\subsection{The general case}

Using Hecke correspondences, we can extend the results of \S \ref{par:biextensions} to the fully general case when the cuspidal cohomology $H^1_{\cusp}(\mathcal{Y}_{\Gamma,K}; \mathcal{V}_k)$ does not vanish. Let $T$ be a $\mathbb{Z}$-linear combination of Hecke correspondences  on $\mathcal{Y}_{\Gamma,K}$ of the form $T_{\alpha}$, as above.

\begin{lem}\label{lem:vanishing-column}
    Let $w_1 \in \mathfrak{H}$ and $r,s\ge 0$ with $r+s = k> 0$. If
    \begin{equation}\label{eq:vanishing-residue}
        \Res \left(T\Psi^{r,s}_{\Gamma}(z,w_1) (X-zY)^kdz \right) = 0,
    \end{equation}
    then
    \[
        T\, {}^{\Gamma}\!G_{r,s}^{p,q}(z,w_1) \equiv  0
    \]
    for every $p,q \ge 0$ with $p+q=k$.
\end{lem}

\begin{proof}
    By applying the residue sequence of Proposition \ref{prop:residue-seq-hodge-explicit} and Proposition \ref{prop:Psi}, one sees that equation  \eqref{eq:vanishing-residue}  implies that $T\Psi^{r,s}_{\Gamma}(z,w_1)$ 
    is  a (holomorphic) cusp form for $\Gamma$. It therefore  defines a class in the pure Hodge structure  $H^1_{\dR,\cusp}(\mathcal{Y}_{\Gamma}; \mathcal{V}_k^{\dR,\mathbb{C}})$  of weight $k+1$.   By \eqref{eqn: PsiinFp} it lies in 
    $F^{s+1} \cap \overline{F}^{r+1} \cap W_{k+1} = 0 $ and so its cohomology class vanishes.   
    By Betti complex-conjugating the right-hand side of \eqref{eq:residue-hecke-Psi}, we see that \eqref{eq:vanishing-residue} implies the same equation with  $r$ and $s$ interchanged. Therefore the cohomology class of    $T\Psi^{s,r}_{\Gamma}(z,w_1) (X-zY)^kdz $ vanishes by the same argument.
        Then, since the space of cusp forms injects into cohomology by Lemma \ref{lem:inclusion-modular-forms}, we obtain
    \[
        T\Psi^{r,s}_{\Gamma}(z,w_1) \equiv 0 \equiv T\Psi^{s,r}_{\Gamma}(z,w_1).
    \]
    By Corollary \ref{cor: GisWHLofPsis}, the higher Green's functions in the statement are the components of a harmonic lift of the pair $(T\Psi^{r,s}_{\Gamma}(z,w_1), -T\Psi^{s,r}_{\Gamma}(z,w_1))$. We  conclude by applying the uniqueness statement of Lemma \ref{lem:uniqueness-harmonic}.
\end{proof}

From now on, we assume that the induced endomorphism on cuspidal cohomology
\begin{equation}  \label{Tactionzero}
    T \in \operatorname{End}( H^1_{\cusp}(\mathcal{Y}_{\Gamma,K}; \mathcal{V}_k))   \hbox{ is the zero map.}
\end{equation}
Note that, when there are no cusp forms (so that $H^1_{\cusp}(\mathcal{Y}_{\Gamma,K};\mathcal{V}_k) = 0$ in $\mathcal{H}_K$), one may take $T$ to be the identity. In particular, the remaining results of this section generalise those of \S \ref{par:biextensions}.

\begin{rem}
    In the literature, when $\Gamma = \Gamma_0(N)$ is a Hecke congruence subgroup, an operator satisfying the above condition \eqref{Tactionzero} is usually constructed by taking $T = \sum_{n \ge 0}\lambda_n T_n$, where $T_n$ are the classical $n$th Hecke operators, and $(\lambda_n)_{n \ge 0}$ is a relation in the space of cusp forms $S_{k+2,\Gamma}$ (see \cite[Section V.4, p. 316]{GrossZagier}). Such a relation may be obtained more concretely in terms of the principal part of a weakly holomorphic modular form of weight $-k$ (\emph{cf}. \cite[Proposition 1.1]{Viazovska} and \cite[Remark 2.5]{BruinierLiYang}).
\end{rem}

For simplicity, let us first consider the case where $P,Q$ are orbits of a single non-elliptic element.

\begin{thm} \label{thm: extensionmain}
    Let $P= \Gamma w_1$ and $Q = \Gamma z_1$, with $w_1$ and $z_1$ not elliptic, and $T$ be as above satisfying \eqref{Tactionzero}. 
     We construct an extension in the category $\mathcal{H}_K$ of the form: 
    \begin{equation} \label{extensionMTzw}
        \begin{tikzcd}[column sep = small]
            0 \arrow{r}&  \Sym^k H^1(\mathcal{E}_{z_1,K}) \arrow{r}& M^T_{z_1,w_1} \arrow{r}&   \Sym^k H^1(\mathcal{E}_{w_1,K})(-1)  \To 0 \ . 
        \end{tikzcd}
    \end{equation} 
If  none of the columns of the matrix $T\mathcal{G}_{\Gamma,k}(z,w_1)$ vanishes identically, then the extension 
 $M^T_{z_1,w_1}$ is a subobject of $H^1_{\cusp}(\mathcal{Y}_{\Gamma,K} \setminus T^{-1}\mathcal{P}_{\Gamma,K}, \mathcal{Q}_{\Gamma,K}; \mathcal{V}_k)$ and  a de Rham basis for $(M^T_{z_1,w_1})_{\dR}\otimes_K \CC$ is given by  
    \[
        [ (0,  \comp_{\B,\dR}^{-1}(X-z_1Y)^i (X-\overline{z}_1 Y)^{k-i})  ]_{i=0,\ldots, k}  \ , \quad    [  (T\Psi_{\Gamma}^{k-r,r}(z,w_1) (X-zY)^k dz,0 )]_{r=0,\ldots, k}.
    \]
   With notation as above,  the single-valued map $\mathsf{s}$  with respect to this basis is given by:
    \[
        \begin{pmatrix}
            J  &- T \GM_{\Gamma,k}(z_1,w_1) \\
            0 & -J 
        \end{pmatrix} \ . 
    \]
\end{thm}

\begin{proof} 
    Let $\mathcal{P}'_{\Gamma,K}  = T^{-1}\mathcal{P}_{\Gamma,K}$. The action of $T$ defines a morphism in $\mathcal{H}_K$ which we also denote by: 
    \begin{equation} \label{inproof: ExtendT}
        T :  H^1_{\cusp} ( \mathcal{Y}_{\Gamma,K} \setminus \mathcal{P}_{\Gamma,K}; \mathcal{V}_k) \To   H^1_{\cusp} ( \mathcal{Y}_{\Gamma,K} \setminus \mathcal{P}'_{\Gamma,K}; \mathcal{V}_k)  . 
    \end{equation}
    We have a  commutative diagram with exact rows
    \[
        \begin{tikzcd}[column sep = small]
            0 \arrow{r} & H^1_{\cusp} ( \mathcal{Y}_{\Gamma,K} ;\mathcal{V}_k) \arrow{d} \arrow{r} & H^1_{\cusp} ( \mathcal{Y}_{\Gamma,K} \setminus \mathcal{P}_{\Gamma,K}; \mathcal{V}_k) \arrow{r} \arrow{d}{\eqref{inproof: ExtendT}} & \Sym^k  H^1( \mathcal{E}_{w_1,K})(-1)\arrow{d} \arrow{r} & 0\\
            0 \arrow{r} & H^1_{\cusp} ( \mathcal{Y}_{\Gamma,K} ;\mathcal{V}_k) \arrow{r} & H^1_{\cusp} ( \mathcal{Y}_{\Gamma,K} \setminus \mathcal{P}'_{\Gamma,K}; \mathcal{V}_k) \arrow{r} & H^0(\mathcal{P}'_{\Gamma,K};\mathcal{V}_k)(-1) \arrow{r} & 0
        \end{tikzcd}
    \]
    where the vertical maps are given by the action of $T$ and the rows are residue short exact sequences \eqref{eqn: YminusPmotive}. Since the vertical map on the left is zero by hypothesis \eqref{Tactionzero}, the map \eqref{inproof: ExtendT} factors through 
    \begin{equation} \label{inproof: TonSkH1}
        T :    \Sym^k H^1(\mathcal{E}_{w_1,K})(-1) \To H^1_{\cusp} ( \mathcal{Y}_{\Gamma,K} \setminus \mathcal{P}'_{\Gamma,K}; \mathcal{V}_k)  . 
    \end{equation}
    
    Now consider the sequence \eqref{eqn: YminusPrelQmotive}:
    \[
        \begin{tikzcd}[column sep = small]
            0 \arrow{r} & \Sym^k  H^1( \mathcal{E}_{z_1,K}) \arrow{r} & H_{\cusp}^1 (\mathcal{Y}_{\Gamma,K} \setminus \mathcal{P}'_{\Gamma,K} ,  \mathcal{Q}_{\Gamma,K} ; \mathcal{V}_k) \arrow{r} & H_{\cusp}^1 (\mathcal{Y}_{\Gamma,K} \setminus \mathcal{P}'_{\Gamma,K} ; \mathcal{V}_k) \arrow{r} & 0.
        \end{tikzcd}
    \]
    By pulling it back along \eqref{inproof: TonSkH1} we deduce an extension 
    \[
        \begin{tikzcd}[column sep = small]
            0 \arrow{r} & \Sym^k H^1(\mathcal{E}_{z_1,K}) \arrow{r} & M^T_{z_1,w_1} \arrow{r} &  \Sym^k H^1(\mathcal{E}_{w_1,K})(-1)\arrow{r} & 0 
        \end{tikzcd}\ .
    \]

      We claim that, under the assumptions of the theorem, the map \eqref{inproof: TonSkH1} is injective. In fact, the induced  map $\overline{T}: \Sym^k H^1(\mathcal{E}_{w_1,K})(-1) \rightarrow  H^0(\mathcal{P}'_{\Gamma,K};\mathcal{V}_k)(-1)  $
    is \emph{a fortiori} injective.  To see this, observe that  $\overline{T}$ induces a morphism of pure Hodge structures of weight $k+2$ and hence preserves the Hodge decomposition. Since the  $(s+1,r+1)$ component $F^{s+1} \cap \overline{F}^{r+1}$ of $\Sym^k H_{\dR}^1(\mathcal{E}_{w_1,K})(-1)$ is one-dimensional, generated by the residue of $\Psi^{r,s}(z,w)(X-zY)^kdz$ (Corollary \ref{cor:Psi-splits-residue-seq}), it suffices to check that the residue of $T\Psi^{r,s}(z,w)(X-zY)^kdz$ is non-zero for every $r+s=k$. Therefore, injectivity of \eqref{inproof: TonSkH1}  follows from the assumption on $T\mathcal{G}_{\Gamma,k}(z,w_1) $ and Lemma \ref{lem:vanishing-column}.  It follows from this that $M^T_{z_1,w_1}$ is the subobject of $H_{\cusp}^1 (\mathcal{Y}_{\Gamma,K} \setminus \mathcal{P}'_{\Gamma,K} ,  \mathcal{Q}_{\Gamma,K} ; \mathcal{V}_k)$ which maps to the image of $\overline{T}$, which is isomorphic to $\Sym^k H^1(\mathcal{E}_{w_1,K})(-1)$ as established above. By Corollary \ref{cor:Psi-splits-residue-seq}, this shows that $(M^T_{z_1,w_1})_{\dR}\otimes_K \mathbb{C}$ admits a basis as in the statement. Finally, the calculation of the single-valued period matrix proceeds exactly as in Theorem \ref{thm: PeriodsSkExtensionNoCusps}. 
\end{proof}

\begin{rem}
    The hypothesis on $T\mathcal{G}_{\Gamma,k}(z,w_1)$ in the above theorem is always satisfied if $T\neq 0$ and $w_1$ is not CM. Indeed, in this case, by inspection of the formulae \eqref{eq:hecke-vector}, the injectivity of \eqref{inproof: TonSkH1} ultimately follows from the fact that a point of $\mathfrak{H}$ is CM if and only if it is fixed by a matrix in $M_{2\times 2}(\mathbb{Z})$ which is not proportional to the identity. 
\end{rem}

\begin{thm} \label{thm: partlyCMextension}
    Let $P= \Gamma w_1$ and $Q = \Gamma z_1$, with $w_1$ a CM point and $z_1$ not elliptic, and let $T$ be a combination of  Hecke operators as above satisfying \eqref{Tactionzero}. 
     Then we obtain an  extension in  $\mathcal{H}_K$ of the form 
    \[
        \begin{tikzcd}[column sep = small]
            0 \arrow{r} &  \Sym^{2m} H^1(\mathcal{E}_{z_1,K})  \arrow{r} &  N^T_{z_1,w_1} \arrow{r} &  \QQ (-m-1) \arrow{r} &  0\ . 
        \end{tikzcd}
    \]   
If the middle column
    \[
        v (z) \defeq (T \, {}^{\Gamma}\!G^{2m,0}_{m,m}(z,w_1), T \, {}^{\Gamma}\!G^{2m-1,1}_{m,m}(z,w_1), \ldots, T \, {}^{\Gamma}\!G^{0,2m}_{m,m}(z,w_1))^t 
    \]
    of the matrix $T\mathcal{G}_{\Gamma,k}(z,w_1)$ does not vanish identically as a function of $z$, then the extension
 $N^T_{z_1,w_1}$ is a subobject of $H^1_{\cusp}(\mathcal{Y}_{\Gamma,K} \setminus T^{-1}\mathcal{P}_{\Gamma,K}, \mathcal{Q}_{\Gamma,K}; \mathcal{V}_k)$, and  a de Rham basis for $(N^T_{z_1,w_1})_{\dR}\otimes_K \CC$ is given by 
    \[
        [ (0,  \comp_{\B,\dR}^{-1}(X-z_1Y)^i (X-\overline{z}_1 Y)^{2m-i})  ]_{i=0,\ldots, 2m}  \quad , \quad    [  (T\Psi_{\Gamma}^{m,m}(z,w_1) (X-zY)^{2m} dz,0 )] \ . 
    \]
    The single-valued map $\mathsf{s}$  is represented in this  basis by the matrix
    \[
        \begin{pmatrix}
            J  &- v(z_1) \\
            0 & -1 
        \end{pmatrix}.
    \]
\end{thm}

\begin{proof}
   We may  proceed as in the previous theorem by applying Proposition \ref{prop: CMmotives} to  $\Sym^k H^1(\mathcal{E}_{w_1,K})(-1)$,  and pulling back the resulting extension along its summand $\QQ(-m-1)$. It is non-zero by Lemma  \ref{lem: invariantscontainsTate}. This yields an extension:
    \[
        \begin{tikzcd}[column sep = small]
            0 \arrow{r} &  \Sym^{2m}   H^1(\mathcal{E}_{z_1,K}) \arrow{r} & N^T_{z_1,w_1} \arrow{r} &   \QQ(-m-1)  \arrow{r} & 0 \ .
        \end{tikzcd}
    \]
    The rest of the proof proceeds by an application of Lemma \ref{lem:vanishing-column} as before, noting that the copy of $\QQ(-m-1)$ in $\Sym^k H^1(\mathcal{E}_{w_1,K})(-1)$ corresponds to the component $[\Psi_{\Gamma}^{m,m}(z,w_1)(X-zY)^k dz]$.
\end{proof}

\begin{thm} \label{cor: GeneralKummerExtension}
    Assume now that both $z_1$ and $w_1$ are CM points, and possibly elliptic. We obtain an extension in $\mathcal{H}_K$ of the form: 
    \begin{equation}\label{eq:gz-extension}
        \begin{tikzcd}[column sep = small]
            0 \arrow{r} & \QQ(-m) \arrow{r} & \mathcal{GZ}^T_{z_1,w_1} \arrow{r} &   \QQ (-m-1) \arrow{r} & 0 \ .
        \end{tikzcd}
    \end{equation}  
   If the  function $T\, {}^{\Gamma}\!G^{m,m}_{m,m}(z,w_1)$ does not vanish identically as a function of $z$, then   $ \mathcal{GZ}^T_{z_1,w_1} $ is a subobject of $H^1_{\cusp}(\mathcal{Y}_{\Gamma,K} \setminus T^{-1}\mathcal{P}_{\Gamma,K}, \mathcal{Q}_{\Gamma,K}; \mathcal{V}_k)$, and   
    a basis for $(\mathcal{GZ}^T_{z_1,w_1})_{\dR} \otimes_K \CC$ is given by 
    relative cohomology classes
    \[
        [ (0,  \comp_{\B,\dR}^{-1}(X-z_1Y)^{m} (X-\overline{z}_1 Y)^{m} ) ]  \ , \quad    [  (T\Psi_{\Gamma}^{m,m}(z,w_1) (X-zY)^k dz,0 )]
    \]
    with corresponding single-valued period matrix
    \[
        \begin{pmatrix}
            1 &    - T\, {}^{\Gamma}\!G^{m,m}_{m,m}(z_1,w_1) \\
            & -1 
        \end{pmatrix}.
    \]
\end{thm}

\begin{proof}
    By proceeding  as in the previous theorem, we obtain an extension
    \[ 
        \begin{tikzcd}[column sep = small]
            0 \arrow{r} &  \left(\Sym^{2m}   H^1(\mathcal{E}_{z_1,K})\right)^{\Gamma_{z_1}} \arrow{r} & N^T_{z_1,w_1} \arrow{r} &   \QQ(-m-1)  \arrow{r} & 0 \ .
        \end{tikzcd}
    \]
    The extension \eqref{eq:gz-extension}  is obtained by pushing out the above extension along the projection onto the copy of $\QQ(-m)$ inside $\big(\Sym^{2m} H^1(\mathcal{E}_{z_1,K})\big)^{\Gamma_{z_1}}$. Since 
    the copy of $\QQ(-m)$ is of Hodge type $(m,m)$ and since $ [(0,(X-z_1Y)^{m} (X-\overline{z}_1 Y)^{m} )] \in F^m \cap \overline{F}^m $, it is spanned by precisely this class. The statement about the de Rham basis therefore follows from \eqref{eqn: PsiinFp}.  The period computation follows as for the previous theorems.    
\end{proof}

\section{Realisations of motives, and conjectures of Beilinson and Gross-Zagier}

A standard conjecture, in rough terms, states that every extension in the category $\mathcal{H}_K$ of $\QQ(-1)$ by $\QQ(0)$ which lies in the full subcategory generated by the cohomology of algebraic varieties is necessarily isomorphic to a Kummer extension. We first show, using the framework developed above,  that this  implies a generalised version of the Gross-Zagier conjecture for modular curves. 
We then briefly discuss how our matrix Green's functions, and their cohomological interpretation, actually leads to a much more general prediction involving special values of $L$-functions, via Beilinson's conjecture.  

\subsection{Gross-Zagier conjecture (the case of two CM points)}

For the next statement, we write $a\sim_{\mathbb{Q}^{\times}} b$ if there is a non-zero rational number $r$ such that $a = rb$.

\begin{thm}\label{thm: GZ-conditional-proof}
    With notation as in Theorem \ref{cor: GeneralKummerExtension}, assume that there exists a Kummer object $\mathcal{K}_x$ in $\mathcal{H}_K$  (\emph{cf.} \S \ref{sec:kummer}) such that the Gross-Zagier extension $\mathcal{GZ}_{z_1,w_1}^T$ is isomorphic to $\mathcal{K}_x(-m)$. Then, either 
    \begin{equation} \label{eqn: TGzwislogarithm}
        T  \,  {}^{\Gamma}\!G_{m,m}^{m,m}(z_1,w_1) \sim_{\mathbb{Q}^{\times}} \log |x| \, 
    \end{equation}
    or $ T  \,  {}^{\Gamma}\!G_{m,m}^{m,m}(z_1,w_1)=0$, in which case $ T  \,  {}^{\Gamma}\!G_{m,m}^{m,m}(z,w_1)$ vanishes identically as a function of $z$. In particular, if $d_{z_1}$ (resp. $d_{w_1}$) denotes the discriminant of the CM point $z_1$ (resp. $w_1$), then \eqref{eqn: TGzwislogarithm} is equivalent to
    \[
        TG_{\Gamma,m+1}(z_1,w_1) \sim_{\mathbb{Q}^{\times}} (d_{z_1}d_{w_1})^{m/2} \log |x|.
    \]
\end{thm}

\begin{proof} We may assume that $ T  \,  {}^{\Gamma}\!G_{m,m}^{m,m}(z,w_1)$ is not identically zero. 
    We must compare the single-valued period matrix of $\mathcal{K}_x(-m)$ in the basis $w_0 = [\textstyle{\frac{dz}{x-1}}](-m), w_1 = [\textstyle{\frac{dz}{z}}](-m)$ (note the Tate twist), namely
    \[
        \begin{pmatrix}
            (-1)^{m} & (-1)^{m}\log |x|^2\\
            0 & -(-1)^{m}
        \end{pmatrix},
    \]
    with the matrix in Corollary \ref{cor: GeneralKummerExtension}. Consider the $K$-basis of $(\mathcal{GZ}_{z_1,w_1}^T)_{\dR}$:
    \[
        v_0 = [(0,  (2\pi i)^{m} \comp_{\B,\dR}^{-1}(X-z_1Y)^{m} (X-\overline{z}_1 Y)^{m}  ]   \ , \quad v_1=    [  ((2\pi i)^{m}T\Psi_{\Gamma}^{m,m}(z,w_1) (X-zY)^{2m} dz,0 )].
    \]
    The fact that $v_0$ is de Rham $K$-rational follows from Lemmas \ref{lem: PeriodsE} and \ref{lem: periodsCMcase} (\emph{cf.} Remark \ref{rem: eigenbasisinBetti}). For $v_1$, this follows from the fact that $T$ annihilates cusp forms, together with Proposition \ref{prop:rationality-psi}, which implies that
    \[
        (2\pi i)^{m}T\Psi_{\Gamma}^{m,m}(z,w_1) (X-zY)^{2m} dz = \binom{2m}{m}\frac{(2\pi i)^{2m+1}}{(w_1-\overline{w}_1)^{m}}Tg^{m,m}(z) (X-zY)^{2m} dz,
    \]
    where $g^{m,m}$ is a meromorphic modular form with Fourier coefficients defined over $K$ (\emph{cf.} Theorem \ref{thm:de-Rham-open-over-K}).

    By construction, $v_0$ lies in the image of $\mathbb{Q}(-m)_{\dR} \to (\mathcal{GZ}_{z_1,w_1}^T)_{\dR}$ in the extension \eqref{eq:gz-extension}, which coincides with $W_{2m} (\mathcal{GZ}_{z_1,w_1}^T)_{\dR}$. Moreover, since $(X-z_1Y)^{m} (X-\overline{z}_1 Y)^{m}$ is Betti $\mathbb{Q}$-rational in $\Sym^k H^1(\mathcal{E}_{z_1,K})$ (since $X,Y$ are $\mathbb{Q}$-rational by definition, $z_1$ is quadratic imaginary by assumption, and the whole expression is $\operatorname{Gal}(\mathbb{Q}(z_1)/\mathbb{Q})$-invariant), we conclude that
    \[
        \comp_{\B,\dR}(v_0)  \in (2\pi i)^{m} W_{2m} (\mathcal{GZ}_{z_1,w_1}^T)_{\B}.
    \]
    The element $v_1$ lies in $F^{m+1}(\mathcal{GZ}_{z_1,w_1}^T)_{\dR}$, which splits the de Rham realisation of the extension \eqref{eq:gz-extension}. Its image $\overline{v}_1$ in the quotient $(\mathcal{GZ}_{z_1,w_1}^T)_{\dR}/ W_{2m} (\mathcal{GZ}_{z_1,w_1}^T)_{\dR} \cong \mathbb{Q}(-m-1)_{\dR}$ equals  
    \[
        (2\pi i)^{m} \,  \mathrm{Res} \left( T\Psi_{\Gamma}^{m,m}(z,w_1) (X-zY)^{2m} dz\right) (-1).
    \]
    Then, it follows from formula \eqref{eq:residue-hecke-Psi} and the comment immediately afterwards,  that
    \[
        \comp_{\B,\dR}(\overline{v}_1) \in (2\pi i)^{m +1} (\mathcal{GZ}_{z_1,w_1}^T)_{\B}/ W_{2m} (\mathcal{GZ}_{z_1,w_1}^T)_{\B}.
    \]
    All of the above properties imply that, under an isomorphism $\mathcal{GZ}_{z_1,w_1}^T \cong \mathcal{K}_x(-m)$ in $\mathcal{H}_K$, there are non-zero rational numbers $\lambda,\mu$ such that $v_0 = \lambda w_0$ and $v_1=\mu w_1$, where $w_0,w_1$ are the de Rham  basis of $\mathcal{K}_x(-m)$  defined above.  Since the basis $v_0,v_1$ differs from the $\mathbb{C}$-basis in Theorem \ref{cor: GeneralKummerExtension}  only by a factor of $(2\pi i)^{m}$, the single-valued period matrix with respect to $v_0,v_1$ is 
    \[
        \begin{pmatrix}
            (-1)^{m} &    - (-1)^{m}T\, {}^{\Gamma}\!G^{m,m}_{m,m}(z_1,w_1) \\
            & -(-1)^{m} 
        \end{pmatrix},
    \]
    (recall from \S\ref{sect: TateObjects} that the single-valued version of $2\pi i $ is $-1$) and we conclude that
    \[
        T \,  {}^{\Gamma}\! G^{m,m}_{m,m}(z_1,w_1) = -\frac{\mu}{\lambda} \log |x|^2 = -2\frac{\mu}{\lambda} \log |x|.
    \]
    The last assertion comes from the formula  \eqref{GFCentralequalsGZ}
    \[
        {}^{\Gamma}\!G^{m,m}_{m,m}(z,w) = (-1)^{m}\binom{2m}{m}\frac{1}{(z-\overline{z})^{m}(w-\overline{w})^{m}}G_{\Gamma,m+1}(z,w) 
    \]
    and from the fact that $\sqrt{d_{z_1}} \sim_{\mathbb{Q}^{\times}} (z_1-\overline{z}_1)$ and likewise for $w_1$. 
\end{proof}

\begin{rem}
    When the modular curve $\mathcal{Y}_{\Gamma}$ is defined over $\mathbb{Q}$, such as in the case $\Gamma = \Gamma_0(N)$, the number field $K$ in the above statement can be  taken to be $\mathbb{Q}(z_1,w_1,j(z_1),j(w_1))$, which, by the classical theory of complex multiplication, is the compositum of the Hilbert class fields of the imaginary quadratic fields $\mathbb{Q}(\sqrt{d_{z_1}})$ and $\mathbb{Q}(\sqrt{d_{w_1}})$. In particular, the above statement is equivalent to that of  \cite[Theorem 1.1]{Li}.
\end{rem}

\subsection{Beyond the Gross-Zagier conjecture}   \label{sect: BeyondGZ}
The extension \eqref{extensionMTzw}, and its variants with $z,w$ CM points, produces a variety of different types of simple extension in the category $\mathcal{H}_K$, which are of geometric origin. A version of  Beilinson's conjecture predicts that their  periods should be related to special values of $L$-functions. 
This suggests a range of new conjectures about the  values of vector-valued Green's functions which goes far beyond the original conjecture of Gross and Zagier, which concerns the very special case of Kummer extensions.  A more detailed study of such extensions will be postponed to a future paper, but we briefly give here an appetiser for the types of extensions which may occur.

Since $H^1 (\mathcal{E}_{w,K})$ is self-dual  (\emph{i.e.}, $H^{\vee} =  H(1)$)  the sequence \eqref{extensionMTzw} is equivalent  to an extension 
\[
    \begin{tikzcd}[column sep = small]
        0 \arrow{r} & M_{z,w} \arrow{r} & \mathcal{E} \arrow{r} & \mathbb{Q}(0) \arrow{r} & 0 \, ,
    \end{tikzcd}
\]
where we write $z=z_1, w=w_1$ for simplicity and 
\[
    M_{z,w}=  \Big(\Sym^k H^1(\mathcal{E}_{z,K})  \otimes \Sym^k H^1(\mathcal{E}_{w,K})\Big)(k+1)
\]
which has  weight $-2$. The corresponding $L$-function  $L(M_{z,w},s)$, which is a Rankin-Selberg convolution of two symmetric powers of motives of elliptic curves, can have a pole at $s=0$ of order $\dim \mathrm{Hom}_{\mathcal{H}_K}(\QQ(1), M_{z,w})$ which is non-zero if and only if $z,w$ are both CM by Proposition \ref{prop: CMmotives}. 
An interesting example is the case when $k=1$ and $M_{z,w}$ is defined over $\QQ$.  Then we expect a relation between the four entries of the $2\times 2$ matrix-valued Green's function $\mathcal{G}_{\Gamma,1}(z,w)$ defined in Appendix \ref{sect: AppendixC},  the  periods of $H^1(\mathcal{E}_{z,K})$ and $H^1(\mathcal{E}_{w,K})$ and their conjugates, and the special value of the Rankin-Selberg $L$-function  $L(H^1(\mathcal{E}_{z,K}) \otimes H^1(\mathcal{E}_{w,K}),2)$. 

Another case of interest is when $k=2m$ is even, $w$ has CM,  and $z$ does not. Then  the same proposition and the previous theorems enable us to cut out an extension of the form
\[
    \mathrm{Ext}^1_{\mathcal{H}_K} (\QQ(0), M'_{z,w} ) \quad \hbox{ where } \quad M'_{z,w}=  \left( \Sym^{2m} H^1(\mathcal{E}_{z,K}) \right)(m+1) 
\]
whose periods are related to the symmetric power $L$-value $L(\mathrm{Sym}^{2m} H^1(\mathcal{E}_z),m+1)$.  
An interesting case is when $M'_{z,w}$ is defined over a quadratic imaginary extension $K$ of $\QQ$, in which case the corresponding Deligne cohomology group has rank one, and Beilinson's conjecture predicts a relation between the values of our higher Green's functions and 
$L(\mathrm{Sym}^{2m} H^1(\mathcal{E}_{z,K}), m+1)$. For example, when $m=1$, we expect a relation between the three entries (${}^{\Gamma}\! G^{1,1}_{i,j}$ for $(i,j) = (0,2), (1,1), (2,0)$) in the central column of our higher   Green's function matrix  $\mathcal{G}_{\Gamma,2}(z,w)$ (see Appendix \ref{sect: AppendixC}),  the periods of $H^1(\mathcal{E}_{z,K})$, and the $L$-value $L(\mathrm{Sym}^2 H^1( \mathcal{E}_{z,K}), 2).$

More generally, Beilinson's conjecture predicts  that the space of motivic extensions has a precise finite rank, which implies the existence of relations between our higher Green's functions for fixed $z$ and varying $w$ a CM point, and also for varying $\Gamma$. An interesting question, which we plan to investigate in a future paper, is whether we can construct all such predicted extensions via the constructions \eqref{extensionMTzw}.  

\begin{rem}
    The discussion above suggests that the vector-valued Green's functions are universal for expressing certain  specific values of $L$-functions of symmetric powers of motives of elliptic curves. This is very much in the spirit of Zagier's conjecture, which states, for example, that the classical polylogarithm functions are the universal  functions for expressing the values of Dedekind zeta  functions of number fields.
\end{rem}

Note  that in \eqref{extension: RelcohomEis} we also considered extensions whose periods are Eichler integrals of meromorphic modular forms. These  produce interesting extensions of $\QQ$ by motives of arbitrary weights $\leq -2$, which should  also be related to special values of $L$-functions via Beilinson's conjecture, and  merit further study.

\section{Mixed modular motives via moduli stacks of pointed elliptic curves}

We prove that the mixed modular objects $H^1_{\cusp}(\mathcal{Y}_{\Gamma,K}\setminus \mathcal{P}_{\Gamma,K},\mathcal{Q}_{\Gamma,K};\mathcal{V}_k)$ constructed in \S\ref{par:general-modular-objects} are motivic, \emph{i.e.}, that they arise as realisations of objects in Voevodsky's triangulated category of motives over the number field $K$. We then explain how well-known and standard conjectures for the motives of cusp forms imply that the `Gross-Zagier extensions' constructed in Theorem \ref{cor: GeneralKummerExtension} are also motivic, which immediately implies the Gross-Zagier conjecture.
In Proposition \ref{prop: M13isMT}, we prove that the motive of $\overline{\mathcal{M}}_{1,3}$ is mixed Tate, which by our earlier results, gives a complete proof of the Gross-Zagier conjecture in  weight $4$ and level $1$. 

\subsection{Triangulated motives}

Given a base field $K$ of characteristic zero and a field of coefficients $F$, also of characteristic zero, we denote Voevodsky's triangulated category of mixed motives over $K$ with coefficients in $F$ by $\DM(K;F)$ \cite{Voevodsky,AyoubICM}.  When $F=\mathbb{Q}$, we write it simply $\DM(K)$. The category $\DM(K;F)$ is triangulated, $F$-linear, tensor, rigid, and pseudo-abelian. There is a canonical functor $\DM(K) \to \DM(K;F)$ given by extension of coefficients.

\subsubsection{Relative cohomology motives}

We borrow the following notation of \cite[Appendix A]{Dupont-Fresan}: for a pair $(X,Z)$, given by a $K$-variety $X$ and a closed subvariety $Z$ of $X$, we associate an object $M(X,Z)$ of $\DM(K)$ called \emph{relative cohomology motive} of the pair $(X,Z)$. If $Z=\emptyset$, we denote it by $M(X)$, and we call it the \emph{motive} of $X$. The relative cohomology motive sits in a natural distinguished triangle
\begin{equation}\label{eq:relative-motive-triangle}
    \begin{tikzcd}[column sep = small]
        M(Z)[-1] \arrow{r} & M(X,Z) \arrow{r} & M(X) \arrow{r}{+1} & {}.
    \end{tikzcd}
\end{equation}
More generally, if there is a decomposition $Z = Z' \cup Y$, where both $Z'$ and $Y$ are closed subvarieties of $X$, then there is a distinguished triangle given by `partial boundaries' \cite[Proposition A.5]{Dupont-Fresan}:
\begin{equation}\label{eq:partial-boundaries-triangle}
    \begin{tikzcd}[column sep = small]
        M(Y,Y\cap Z')[-1] \arrow{r} & M(X,Z) \arrow{r} & M(X,Z') \arrow{r}{+1} & {}. 
    \end{tikzcd}
\end{equation}

\begin{rem}
    Beware that we work with \emph{cohomological} motives, as in \cite{Dupont-Fresan}, contrary to Voevodsky's homological convention \cite{Voevodsky}. Thus, in this paper, the functor $(X,Z) \mapsto M(X,Z)$ is \emph{contravariant}.
\end{rem} 

\begin{example}[Tate motives]
    The trivial and Lefschetz motives in $\DM(K)$ are defined by
    \[
        \mathbb{Q}(0) \coloneqq M(\operatorname{Spec} K), \qquad \mathbb{Q}(-1) \coloneqq M(\mathbb{G}_{m,K},{1})[1].
    \]
    For $n\ge 0$, we set $\mathbb{Q}(-n) \coloneqq \mathbb{Q}(-1)^{\otimes n}$ and $\mathbb{Q}(n) \coloneqq \mathbb{Q}(-n)^{\vee}$. For any $n \in \mathbb{Z}$, the extension of coefficients of $\mathbb{Q}(n)$ to $\DM(K;F)$ is denoted by $F(n)$. Given any object $M$ of $\DM(K;F)$, we denote $M\otimes F(n) = M(n)$.
\end{example}

\begin{example}[Kummer motives]\label{ex:kummer-motive}
    For every $x \in K^{\times}$, the corresponding \emph{Kummer motive} is the object of $\DM(K)$ given by
    \[
        K_x \coloneqq M(\mathbb{G}_{m,K},\{1,x\})[1].
    \]
    By taking $Z'=\{1\}$ and $Y = \{x\}$ in \eqref{eq:partial-boundaries-triangle}, the Kummer motive sits in a distinguished triangle of the form:
    \[
        \begin{tikzcd}[column sep = small]
            \mathbb{Q}(0) \arrow{r} & K_x \arrow{r} & \mathbb{Q}(-1) \arrow{r}{+1} & {}.
        \end{tikzcd} 
    \]
    This is the \emph{Kummer extension}.
\end{example}

The triangulated subcategory generated by the Tate motives $\QQ(n)$ is denoted $\DMT(K;F)$, or $\DMT(K)$ when $F=\mathbb{Q}$. If $K$ is a number field, then it is known that $\DMT(K)$ is equipped with a canonical $t$-structure whose heart is the abelian category of \emph{mixed Tate motives} over $K$, denoted $\MT(K)$ (see \cite{Levine}). The Kummer motives of the above example are mixed Tate motives. Conversely, the isomorphisms
\[
    \operatorname{Ext}^1_{\MT(K)}(\mathbb{Q}(-1),\mathbb{Q}(0)) \cong \Hom_{\DMT(K)}(\mathbb{Q}(-1),\mathbb{Q}(0)[1]) \cong K^{\times}\otimes_{\mathbb{Z}}\mathbb{Q}
\]
assert that every extension of $\mathbb{Q}(-1)$ by $\mathbb{Q}(0)$ is a subobject of a finite direct sum of Kummer motives.

More generally, we shall consider motives given as `biextensions', as follows.

\begin{example}[Biextensions]\label{ex:motive-biextension}
    If $X$ is smooth and $A,B \subset X$ are smooth components of a normal crossings divisor in $X$, then the relative cohomology motive $M(X\setminus A,B\setminus (A\cap B))$ sits in a \emph{biextension} given by the relative cohomology distinguished triangle with respect to $B\setminus (A\cap B)$ followed by a Gysin distinguished triangle with respect to $A$:
    \[
        \begin{tikzcd}[column sep = small]
            M(B\setminus (A\cap B))[-1] \arrow{r} & M(X\setminus A,B\setminus (A\cap B)) \arrow{r} & M(X\setminus A) \arrow{r}{+1} & {}
        \end{tikzcd}
    \]
    \[
        \begin{tikzcd}[column sep = small]
            M(X) \arrow{r} & M(X\setminus A) \arrow{r}{\Res} & M(A)(-1)[-1] \arrow{r}{+1} & {}
        \end{tikzcd}
    \]
    Alternatively, one might consider first a Gysin triangle for $M(X\setminus A,B\setminus (A\cap B))$ (see \cite[Propositon A.18]{Dupont-Fresan}), followed by a relative cohomology triangle for $M(A,A\cap B)$.
\end{example} 

\subsubsection{Motives of smooth Deligne-Mumford stacks}

By \cite{Choudhury}, to every smooth separated Deligne-Mumford stack over $K$, we may associate a motive $M(\mathcal{X})$ in $\DM(K)$ satisfying the following properties (beware, again, that our convention is cohomological):
\begin{enumerate}
    \item The association $\mathcal{X}\mapsto M(\mathcal{X})$ is a contravariant functor.
    \item  If $\mathcal{X} \rightarrow X$ is the coarse moduli space of $\mathcal{X}$, then $M(X) \cong M(\mathcal{X})$ (\cite[Theorem 3.3]{Choudhury}).
    \item  If $X$ is a smooth quasi-projective scheme with an action of a \emph{finite} group $G$, then the motive of the quotient stack $[X/G]$ is isomorphic to the $G$-invariants of $M(X)$ (\cite[Lemma 3.1]{Choudhury}):
    \[
        M([X/G]) \cong M(X)^G.
    \]
\end{enumerate}

\begin{rem}
    By property (3) above, if the motive $M(X)$ is mixed Tate, then the motive of the quotient $M([X/G])$ is also mixed Tate. The converse is not true; see Remark \ref{rem:mixed-tate-elliptic-quotient}.
\end{rem}

If $\mathcal{X}$ is a smooth separated Deligne-Mumford stack over $K$, and $\mathcal{Z}$ is a smooth closed substack, we may also consider the relative cohomology motive $M(\mathcal{X},\mathcal{Z})$, which fits in a distinguished triangle
\[
    \begin{tikzcd}[column sep = small]
        M(\mathcal{Z})[-1] \arrow{r} & M(\mathcal{X},\mathcal{Z}) \arrow{r} & M(\mathcal{X}) \arrow{r}{+1} & {}.
    \end{tikzcd}
\]
If the coarse moduli spaces $X$ of $\mathcal{X}$ and $Z$ of $\mathcal{Z}$ are schemes, then $M(\mathcal{X},\mathcal{Z}) \cong M(X,Z)$, and the above distinguished triangle is isomorphic to \eqref{eq:relative-motive-triangle}. Similarly, if $\mathcal{A},\mathcal{B}\subset \mathcal{X}$ are smooth components of a normal crossings divisor in $\mathcal{X}$, the motive $M(\mathcal{X}\setminus \mathcal{A},\mathcal{B}\setminus (\mathcal{A}\cap \mathcal{B}))$ sits in a biextension as in Example \ref{ex:motive-biextension}.

\begin{example}
    Let $K[x_0,\ldots, x_n]$ be the graded polynomial ring where $x_i$ has degree $d_i\ge 1$. Equivalently, we equip $\mathbb{A}^{n+1} = \Spec K[x_0,\ldots,x_n]$ with the action of the multiplicative group $\mathbb{G}_{m}$ in which $x_i$ has weight $d_i$. Then the \emph{weighted projective stack}
    \[
        \mathcal{P}(d_0,\ldots, d_n) \coloneqq \left[ \frac{\mathbb{A}^{n+1} \setminus (0,\ldots,0)}{\mathbb{G}_m}  \right]
    \]
    is a separated smooth Deligne-Mumford stack. The open substack where $x_i\neq 0$ is isomorphic to a quotient $[\mathbb{A}^n/\mu_{d_i}]$ (see Lemma \ref{lem:slice}), and its complement is a weighted projective stack of smaller dimension.  The motive $M(\mathcal{P}(d_0,\ldots, d_n))$ is mixed Tate. 
\end{example}

\subsubsection{Realisation functors}

If $K\subset \mathbb{C}$ is a number field, there is an $\mathcal{H}$-realisation functor
\[
    \DM(K;F) \To D(\mathcal{H}_K\otimes F),
\]
which is $F$-linear and compatible with the triangulated and tensor structures. The realisation of an object of $\DM(K)$ of the form $M(X,Z)$ is an object in the derived category of $\mathcal{H}_K$ which computes the cohomology of the pair $(X,Z)$; by abuse, we shall denote it by $H^{\bullet}(X,Z)$.

\subsection{Moduli stacks of pointed elliptic curves and statement of the main result} \label{par:moduli-pointed-EC}

For simplicity, we work at level 1, and we refer to \cite{Petersen} for the necessary modifications in higher levels.

For every integer $n\ge 1$, let $\mathcal{M}_{1,n}$ be the smooth Deligne-Mumford stack over $\mathbb{Q}$ classifying smooth genus 1 curves with $n$ ordered marked points, and  let $\overline{\mathcal{M}}_{1,n}$ be its Deligne-Mumford compactification, which is a smooth and proper Deligne-Mumford stack over $\mathbb{Q}$ classifying stable genus 1 curves with $n$ ordered marked points. We denote by
\begin{equation}\label{eq:forgetful-projection}
    \pi_n: \overline{\mathcal{M}}_{1,n} \to \overline{\mathcal{M}}_{1,1}
\end{equation}
the canonical projection which  forgets the last $n-1$ marked points and contracts the possible destabilizing genus zero components. In the notation of Section \ref{sec:MMR}, we have $\mathcal{M}_{1,1} = \mathcal{Y}_{\SL_2(\mathbb{Z}),\mathbb{Q}}$ and $\overline{\mathcal{M}}_{1,1} = \mathcal{X}_{\SL_2(\mathbb{Z}),\mathbb{Q}}$.

To an elliptic curve $(E,0)$ with field of definition $K = \mathbb{Q}(j(E))$, there corresponds a $K$-point $\Spec K \to \mathcal{M}_{1,1}$, which factors through a reduced closed substack $\mathcal{G}_E\subset \mathcal{M}_{1,1}$, called its \emph{residue gerbe}, whose underlying space is a single point:
\[
    \begin{tikzcd}[column sep = small]
        \Spec K \arrow{r}\arrow[bend right]{rr}[swap]{E} & \mathcal{G}_E \arrow[hook]{r} &\mathcal{M}_{1,1}.
    \end{tikzcd}
\]
If $\Gamma_E$ denotes the group of automorphisms of $(E,0)$, then the residue gerbe $\mathcal{G}_E$ is isomorphic to the classifying stack $B_K \Gamma_E = [(\Spec K)/ \Gamma_E]$. The closed substack given by the `cusp at infinity', corresponding to a nodal cubic curve, is denoted $\infty\subset \overline{\mathcal{M}}_{1,1}$ (see also \eqref{eq:cusp_mu2}). 

Let us denote by $E' = E\setminus \{0\}$ the punctured elliptic curve. The fibre of $\pi_n$ at a  $K$-point corresponding to $(E,0)$, as above, is isomorphic to a compactification $\overline{C}_{n-1}(E')$ of the configuration space
\begin{align*}
    C_{n-1}(E') &= \{(x_2,\ldots,x_{n}) \in (E')^{n-1} : x_i\neq x_j\text{ for all }i\neq j\}.
\end{align*}
Note that $\Gamma_E$ acts on $\overline{C}_{n-1}(E')$ and the fibre of $\pi_n$ at the residue gerbe $\mathcal{G}_E$ is the quotient stack
\[
    \overline{\mathcal{M}}_{1,n}\times_{\overline{\mathcal{M}}_{1,1}}\mathcal{G}_E\cong [\overline{C}_{n-1}(E')/\Gamma_E].
\]

The symmetric group $\mathfrak{S}_n$ acts on $\overline{\mathcal{M}}_{1,n}$ by permuting the marked points. The open substack $\mathcal{M}_{1,n}\subset \overline{\mathcal{M}}_{1,n}$ is preserved by this action. The fibres of $\pi_n$ are also preserved by the action of $\mathfrak{S}_n$. This follows from the fact that, if $(C,p)$ is a stable genus 1 curve with one marked point, and $q$ is any other smooth point of $C$, then there is a natural isomorphism $(C,p) \cong (C,q)$: if $C$ is smooth, then $(C,p)$ is an elliptic curve with identity $p$, and the isomorphism is translation by $q-p$; if $C$ is singular, then it is isomorphic to the nodal curve $\mathbb{P}^1/\{0,\infty\}$ and the statement follows from the fact that $\operatorname{Aut}(\mathbb{P}^1)$ acts $3$-transitively on $\mathbb{P}^1$. Note that the previous argument also works in families.

\begin{rem}
    Explicitly, on the fibre at $(E,0)$, the induced action on $C_{n-1}(E')$ is given as follows: by setting $x_1=0$, we have $\sigma \cdot (x_2,\ldots,x_n) = (x_{\sigma(2)} - x_{\sigma(1)},\ldots,x_{\sigma(n)}- x_{\sigma(1)})$. 
\end{rem}

Let $\varepsilon : \mathfrak{S}_n \to \{\pm 1\}$ be the alternating (sign) character. If $V$ is an object of a pseudo-abelian category with an action of $\mathfrak{S}_n$, we let $V_{\varepsilon}$ denote the `alternating part' of $V$, namely the object obtained as the image of the projector $\frac{1}{n!}\sum_{\sigma \in \mathfrak{S}_n}\varepsilon(\sigma)\sigma$.

The main result of \cite{ConsaniFaber} implies that, for any integer $n\ge 1$, the object $H^1_{\cusp}(\mathcal{M}_{1,1};\mathcal{V}_{n-1})$ of $\mathcal{H}_{\mathbb{Q}}$ can be obtained as the realisation of the motive
\[
   M(\overline{\mathcal{M}}_{1,n})_{\varepsilon}[n]  \in \DM(\mathbb{Q}).
\]
We now explain how to extend this to the mixed modular objects of Section \ref{sec:MMR}. Let $K$ be a number field, and let $\mathcal{P} \defeq \mathcal{P}_{\SL_2(\mathbb{Z}),K}$ and $\mathcal{Q} \defeq \mathcal{Q}_{\SL_2(\mathbb{Z}),K}$ be the closed substacks  of $\mathcal{M}_{1,1}\times_{\mathbb{Q}}K = \mathcal{Y}_{\mathrm{SL}_2(\ZZ),K}$ (given by finite unions of residue gerbes of points)  as in \S \ref{par:general-modular-objects}. Note that we may also allow $\mathcal{P},\mathcal{Q}$ to be the empty stack $\emptyset$.

\begin{thm}\label{thm:mixed-modular-motive}
    The object of $\DM(K)$ defined by
    \[
        M^{n}(\mathcal{P},\mathcal{Q}) \defeq  M ( \overline{\mathcal{M}}_{1,n}  \times_{\QQ}K  \setminus \pi_n^{-1}(\mathcal{P} ),  \pi_n^{-1}(\mathcal{Q}))_{\varepsilon} [n] 
    \]
    has cohomology concentrated in degree zero and satisfies:
    \[
        H^0(M^{n}(\mathcal{P},\mathcal{Q}))\cong H_{\cusp}^1 (\mathcal{M}_{1,1} \setminus \mathcal{P}, \mathcal{Q} ; \mathcal{V}_{n-1})
    \]
    in the category of realisations $\mathcal{H}_K$.
\end{thm}

The proof will rely on a computation of the cohomology of the fibres of $\pi_n$, which is set out in the next paragraph.There is a variant of the above statement for higher levels.

\subsection{Alternating cohomology of the fibres of $\pi_n$}

We recall some standard preliminaries concerning the stratification of the  boundary of the Deligne-Mumford compactification $\overline{\mathcal{M}}_{1,n}$. 

The boundary strata of the Deligne-Mumford compactification $\overline{\mathcal{M}}_{1,n}$ correspond to connected, stable, weighted graphs $G$ of genus 1 with $n$ external half-edges, which are labeled $1,\ldots,n$. Each vertex $v \in V_G$ is assigned an integer weight $w(v)\ge 0$ and the total genus is $g(G) = h_G + \sum_{v \in V_G}w(v)$, where $h_G$ denotes the Betti number of $G$. Since $g(G)= 1$, there is at most one vertex with a non-zero weight. The codimension of the corresponding boundary stratum is the number of edges of $G$.  Inclusions of (closed) boundary strata  are dual to contractions of internal edges: the contraction of a self-edge produces a vertex of weight one, and the contraction of an edge whose endpoints have weights $w_1,w_2$ is a vertex of weight $w_1+w_2$.

A boundary stratum of $\overline{\mathcal{M}}_{1,n}$ lies over the cusp $\infty \subset \overline{\mathcal{M}}_{1,1}$ if and only if the corresponding graph has a cycle, i.e., $h_G=1$. Thus, the remaining strata in $\overline{\mathcal{M}}_{1,n} \setminus \pi_n^{-1}(\infty)$ correspond to stable weighted trees $T$ of genus $g(T)=1$, that is, with a single vertex of weight $1$. Such a graph has no automorphisms. In codimension $k$, the corresponding (closed) strata of $\overline{\mathcal{M}}_{1,n}\setminus \pi^{-1}(\infty)$ are therefore isomorphic to
\[
   (\overline{\mathcal{M}}_{1,m} \setminus \pi_m^{-1}(\infty) )\times \prod_{j=1}^k \overline{\mathcal{M}}_{0,m_j},
\]
where $m + m_1 + \cdots + m_k = n + 2k$.

\begin{prop} \label{prop: alternatingpartofcohom}
    Let $\pi_n^{-1}(\mathcal{G}_E)\subset \overline{\mathcal{M}}_{1,n}$ be the fibre of $\pi_n:\overline{\mathcal{M}}_{1,n} \to \overline{\mathcal{M}}_{1,1}$ above the residue gerbe $\mathcal{G}_E\subset \overline{\mathcal{M}}_{1,1}$ of an elliptic curve $(E,0)$ defined over a number field. Then, the alternating part of the cohomology satisfies:
    \[
        H^{i}_{\B}\left( \pi_n^{-1}(\mathcal{G}_E) \right)_{\varepsilon}   \cong  \begin{cases}
            \left( \Sym^{n-1}\,  H^1_{\B}(E)\right)^{\Gamma_E} & \hbox{ if } i = n-1 \\ 0 & \hbox{ otherwise } \ .  \end{cases}  \ .
    \] 
\end{prop} 

 \begin{proof}
     The following argument paraphrases results in \cite{ConsaniFaber}. It follows from the preliminaries preceding the statement of this proposition that the fibre of $\pi_n$ at $(E,0)$, which is isomorphic to $\overline{C}_{n-1}(E')$, admits an open stratification by:
     \[
        C_m(E') \times \prod_{j=1}^k \mathcal{M}_{0,m_j}.
     \]
     The symmetric group $\mathfrak{S}_n$ acts by permuting the marked points, and in particular permutes the boundary strata. As in \cite{ConsaniFaber}, consider an extremal vertex of weight 0 in the dual graph corresponding to any open boundary stratum. It corresponds to a component $\mathcal{M}_{0,m_j}$, where one of the $m_j$ points corresponds to an internal edge, and the remaining $m_j-1$ points correspond to external half-edges. The subgroup $\mathfrak{S}_{m_j-1}\subset \mathfrak{S}_n$ permutes the latter. As shown in \cite[Lemma 5]{ConsaniFaber},  the alternating representation of $\mathfrak{S}_{m_j-1}$ does not occur in $H^i(\mathcal{M}_{0,m_j})$. It follows via the Künneth formula that $H^i(C_m(E')\times \prod_{j=1}^k\mathcal{M}_{0,m_j})_{\varepsilon} = 0$ whenever $k>0$. By a relative cohomology spectral sequence relative to the strict normal crossing divisor $\overline{C}_m(E') \setminus C_m(E')$, or by using excision exact sequences for cohomology with compact supports, we deduce that
     \[
        H^i(\overline{C}_{n-1}(E'))_{\varepsilon} = H_c^i(C_{n-1}(E'))_{\varepsilon}.
     \]
     Now, \cite[Lemma 2]{ConsaniFaber} states that $H_c^i(C_{n-1}(E'))_{\varepsilon} \cong \Sym^{n-1}H^1(E)$ for $i=n-1$, and vanishes otherwise. We conclude that
     \[
        H^i(\pi_n^{-1}(\mathcal{G}_E))_{\varepsilon} \cong \left(H_c^i(\overline{C}_{n-1}(E'))^{\Gamma_E}\right)_{\varepsilon},
     \]
     which implies the statement of the proposition, since the action of $\Gamma_E$ and $\mathfrak{S}_n$ commute. 
 \end{proof}

From now on, a sheaf on a stack will  mean a sheaf  on the underlying complex orbifold.

\begin{cor} \label{cor: DerivedPushforwardandSym}
    Let $\QQ$ denote the constant sheaf on $\overline{\mathcal{M}}_{1,n} \setminus \pi_n^{-1} (\infty)$. Then in the derived category of sheaves on   $\mathcal{M}_{1,1}$ we have an isomorphism
    \[
        \mathbb{V}^{\B}_{n-1}[1-n] \cong  (R \pi_{n*} \QQ)_{\varepsilon},
    \]
    where $\mathbb{V}_{n-1}^{\B}$ is the Betti local system on  $\mathcal{M}^{\an}_{1,1} = \SL_2(\mathbb{Z})\backslash\!\!\backslash \mathfrak{H}$ defined in \S \ref{par:analytic-dR-open}.
\end{cor}

\begin{proof}
    By Proposition \ref{prop: alternatingpartofcohom},  the $k$th cohomology sheaf of the object $(R {\pi_n}_* \QQ)_{\varepsilon}$ in the derived category of sheaves on the analytification of $\mathcal{M}_{1,1}$ vanishes if $k\neq n-1$ and  satisfies $\Sym^{n-1} R^1 {\pi_n}_* \QQ \cong (R^{n-1} {\pi_n}_* \QQ)_{\varepsilon}$, because the natural map given by the cup product is an isomorphism  on stalks, again by  Proposition \ref{prop: alternatingpartofcohom}. Conclude using the standard fact that an object in a derived category is  isomorphic  to its cohomology  when the latter is concentrated in a single degree. 
  \end{proof}

\subsection{Proof of Theorem \ref{thm:mixed-modular-motive}}

First, some notation. In this proof we shall assume for simplicity of notation that $\mathcal{P}$ and $\mathcal{Q}$ are defined over $\QQ$. In the general case, simply replace every occurrence of $\mathcal{M}_{1,n}$,  $\overline{\mathcal{M}}_{1,n}$ and so on, with their extensions of scalars  $\mathcal{M}_{1,n} \times_{\QQ} K$, \emph{etc}. Moreover, to fix ideas and avoid introducing further notation, we shall only work with Betti realisations.

For any closed substack $\mathcal{S}\subset \mathcal{M}_{1,1} $ given by a finite union of residue gerbes of points in $\mathcal{M}_{1,1}$ (typically $\mathcal{S}\subseteq \mathcal{P} \cup \mathcal{Q}$), consider the following commutative diagram
  \[
        \begin{tikzcd}[column sep = small]
           \overline{\mathcal{M}}_{1,n}  \setminus \pi_n^{-1} (\mathcal{S} \cup \infty)  \arrow{d}    \arrow{r}{j}  &   \overline{\mathcal{M}}_{1,n} \setminus \pi_n^{-1}(\mathcal{S}) \arrow{d} &   \\
            \mathcal{M}_{1,1}     \setminus \mathcal{S}  \arrow{r}{j} &   \overline{\mathcal{M}}_{1,1} \setminus \mathcal{S}   & 
        \end{tikzcd}
    \]
where the vertical maps are given by $\pi_n$. The map $j$ therefore denotes, without ambiguity, the inclusion of the complement of the cusp (resp. the fiber of $\pi_n$  over the cusp).  

Now consider the commutative diagram:
  \begin{equation} \label{M1nrelativespacesdiagram}
        \begin{tikzcd}[column sep = small]
           \overline{\mathcal{M}}_{1,n}  \setminus \pi_n^{-1} (\mathcal{P} \cup \mathcal{Q})  \arrow{d}   \arrow{r}{k}  &   \overline{\mathcal{M}}_{1,n} \setminus \pi_n^{-1}(\mathcal{P}) \arrow{d}  &   \arrow{l}[swap]{i} \pi_n^{-1}(\mathcal{Q})\arrow{d} \\
            \overline{\mathcal{M}}_{1,1}     \setminus (\mathcal{P} \cup \mathcal{Q})   \arrow{r}{k} &   \overline{\mathcal{M}}_{1,1} \setminus \mathcal{P}    & \arrow{l}[swap]{i}   \mathcal{Q}
        \end{tikzcd}
    \end{equation}
where $i$ and $k$ also denote inclusions. Restricting to the complement of the cusp $\infty$ in the bottom row (resp. the complement of $\pi_n^{-1}(\infty)$ in the top row), yields an analagous commutative diagram, whose maps will be denoted without ambiguity using  the same symbols.   This new commutative diagram forms, together with \eqref{M1nrelativespacesdiagram}, a commutative diagram which consists of a pair of commutative cubes glued together along a face.

We have at our disposal the following extensions of the constant sheaf $\mathbb{Q}$ over  $\overline{\mathcal{M}}_{1,n} \setminus \pi_n^{-1}(\mathcal{P}\cup \infty)$ to a sheaf, or an object of the derived category of sheaves, on $\overline{\mathcal{M}}_{1,n} \setminus \pi_n^{-1}(\mathcal{P})$:
\[
    \begin{tikzcd}[column sep = small]
        j_!\mathbb{Q} \arrow{r} & j_*\mathbb{Q} \arrow{r}& Rj_*\mathbb{Q} \ .
    \end{tikzcd}
\]
For now we shall focus on the extension by zero, $j_! \QQ$. With reference to the top line of the diagram \eqref{M1nrelativespacesdiagram}, we deduce the following exact sequence of sheaves on $\overline{\mathcal{M}}_{1,n} \setminus \pi_n^{-1}(\mathcal{P})$:
\[
    \begin{tikzcd}[column sep = small]
        0 \arrow{r}& k_! k^* j_! \QQ \arrow{r}& j_! \QQ \arrow{r}& i_* i^* j_! \QQ \arrow{r}& 0 \ .
    \end{tikzcd}
\]
Consider the motive $M^{n}_{!} (\mathcal{P},\mathcal{Q}) \coloneqq  M ( \overline{\mathcal{M}}_{1,n} \setminus \pi_n^{-1}(\mathcal{P}),  \pi_n^{-1}(\mathcal{Q} \cup \infty))_{\varepsilon} [n]$. By Corollary \ref{cor: DerivedPushforwardandSym}, $(R{\pi_n}_*j_! \QQ)_{\varepsilon}$ is isomorphic, in the derived category of sheaves, to $j_!\mathbb{V}^{\B}_{n-1} [n-1]$. This follows from the properness of $\pi_n$, \emph{i.e.}, ${\pi_n}_* = {\pi_n}_!$, and the exactness of $j_! = Rj_!$, which together imply that  $ R {\pi_n}_! j_! =  R( \pi_n \circ j)_! = R (j\circ \pi_n)_! =  j_! R{\pi_n}_!$.
The action of the symmetric group preserves  the fibers of $\pi_n$.  

By definition of $j$ and slight abuse of notation, $k_! k^* j_! \QQ =  k_! j_! j^* k^* \QQ $ and hence
\begin{align*}
    H^{m-n}_{\B}(M^{n}_!(\mathcal{P},\mathcal{Q})) 
    & \cong H^m ( \overline{\mathcal{M}}_{1,n}  \setminus \pi_n^{-1}(\mathcal{P});   k_! k^* j_! \QQ )_{\varepsilon}  \\ 
    & \cong  H^m ( \overline{\mathcal{M}}_{1,1}  \setminus  \mathcal{P} ;    (R{\pi_n}_* ( k_! k^* j_! \QQ))_{\varepsilon}) \\
    & \cong  H^m ( \overline{\mathcal{M}}_{1,1}  \setminus  \mathcal{P} ;   k_! k^*   ( R{\pi_n}_*   j_! \QQ)_{\varepsilon})\\
    & \cong \begin{cases} 
        H^1(\overline{\mathcal{M}}_{1,1} \setminus \mathcal{P},  \mathcal{Q};  j_!  \mathbb{V}^{\B}_{n-1}  ),  & \hbox{if } m=n \\
        0,  & \hbox{otherwise,} 
    \end{cases} 
\end{align*}
where the commutation $(R{\pi_n}_* ( k_! k^* j_! \QQ))_{\varepsilon} \cong k_! k^*   ( R{\pi_n}_*   j_! \QQ)_{\varepsilon}$  follows from a similar argument to the one given above for $j_!$.  
We now consider the motive $M^{n}_{\circ} (\mathcal{P},\mathcal{Q}) \coloneqq   M ( \overline{\mathcal{M}}_{1,n}  \setminus \pi_n^{-1}(\mathcal{P} \cup \infty),  \pi_n^{-1}(\mathcal{Q}))_{\varepsilon} [n]$ and compute its cohomology along similar lines applying Corollary \ref{cor: DerivedPushforwardandSym} (here, again by abuse of notation, $k$ is restricted along $j$):
\begin{align*}
    H_{\B}^{m-n} (M^{n}_{\circ}(\mathcal{P},\mathcal{Q})) &\cong  H^m (\overline{\mathcal{M}}_{1,n}  \setminus \pi_n^{-1}(\mathcal{P} \cup \infty);   k_! k^* \QQ )_{\varepsilon}\\
    & \cong H^m(\mathcal{M}_{1,1}\setminus \mathcal{P} ; (R{\pi_n}_*(k_! k^* \QQ))_{\varepsilon} )\\
    &\cong H^m ( \mathcal{M}_{1,1}  \setminus  \mathcal{P} ;   k_! k^*   ( R{\pi_{n}}_*   \QQ)_{\varepsilon})\\
    & \cong \begin{cases}
            H^1(\mathcal{M}_{1,1}\setminus \mathcal{P},\mathcal{Q};\mathbb{V}_{n-1}^{\B}), & \text{if } m=n \\
            0, &\text{otherwise}
    \end{cases}
\end{align*} 
  
Now consider the natural maps 
\[
    \begin{tikzcd}[column sep = small]
        H_{\B}^{m-n} (M_!^n(\mathcal{P},\mathcal{Q})) \arrow{r}& H_{\B}^{m-n} (M^n(\mathcal{P},\mathcal{Q})) \arrow{r}&  H_{\B}^{m-n} (M_{\circ}^n(\mathcal{P},\mathcal{Q})).
    \end{tikzcd}
\]
To conclude, it suffices to show that the first is surjective, and the second is injective, for all $m$ (see \eqref{eq:betti-cuspidal-cohomology}). To establish this, we note that the cohomology of each of the three objects $M^n(\mathcal{P},\mathcal{Q}), M^n_!(\mathcal{P},\mathcal{Q}), M^n_{\circ}(\mathcal{P},\mathcal{Q})$ lies in two  exact sequences (relative cohomology and Gysin):
\[
    \begin{tikzcd}[column sep = small]
        0 \arrow{r}& H_{\B}^{m-1}(\pi_n^{-1}(\mathcal{Q}))_{\varepsilon} \arrow{r}&  H_{\B}^{m-n}(M^n_{\bullet}(\mathcal{P},\mathcal{Q}))\arrow{r}&  H_{\B}^{m-n}(M^n_{\bullet}(\mathcal{P},\emptyset))  \arrow{r}& 0 
    \end{tikzcd}
\]
\[
    \begin{tikzcd}[column sep = small]
        0 \arrow{r}& H_{\B}^{m-n}(M^n_{\bullet}(\emptyset,\emptyset)) \arrow{r}&  H_{\B}^{m-n}(M^n_{\bullet}(\mathcal{P},\emptyset)) \arrow{r}&  H_{\B}^{m-1}(\pi_n^{-1}(\mathcal{P}))_{\varepsilon}(-1)  \arrow{r}& 0
    \end{tikzcd} 
\]
where $\bullet\in \{\emptyset,!, \circ\}$.  It suffices, therefore, to show that in the natural maps
\[
    \begin{tikzcd}[column sep = small]
        H_{\B}^{m-n} (M_!^n(\emptyset,\emptyset))\arrow{r}&   H_{\B}^{m-n} (M^n(\emptyset,\emptyset))\arrow{r}&   H_{\B}^{m-n} (M_{\circ}^n(\emptyset,\emptyset)) ,
    \end{tikzcd}
\]
the first is surjective, and the second injective, for all $m$. This is equivalent to the statement:
\[
    H_{\B}^{m-n} (M^n(\emptyset,\emptyset)) =   \im\left(  H_{\B}^{m-n} (M_!^n(\emptyset,\emptyset)) \rightarrow  H_{\B}^{m-n} (M_{\circ}^n(\emptyset,\emptyset))  \right)
\]
By the above computations, the right-hand side vanishes if $m\neq n$ and for $m=n$ equals
\[
    \mathrm{im}\left(H^1 ( \overline{\mathcal{M}}_{1,1};j_! \mathbb{V}_{n-1}^{\B} ) \rightarrow  H^1( \mathcal{M}_{1,1} ; \mathbb{V}_{n-1}^{\B} ) \right)  \cong H^1_{\cusp}(\mathcal{M}_{1,1};\mathbb{V}_{n-1}^{\B}) \ . 
\]
The assertion that, indeed, $H_{\B}^0 (M^n(\emptyset,\emptyset))= H^1_{\cusp}(\mathcal{M}_{1,1};\mathbb{V}_{n-1}^{\B}) $ and vanishes in all other degrees is the main result of \cite{ConsaniFaber}, and thus completes the proof. 

\subsection{Motivic biextensions for cusp form motives}

Let $M^n(\mathcal{P},\mathcal{Q})$ be defined as above and write $M^n(\mathcal{P})$ for the case when $\mathcal{Q}=\emptyset$. Let $M^n_{\cusp}$ denote the `cusp form motive' which corresponds to the case when $\mathcal{P}=\mathcal{Q}=\emptyset.$ With notations as in \S\ref{sect: HeckeCorres},  any Hecke correspondence $T_{\alpha}$ defines an endomorphism of cusp form motives
\[
    T_{\alpha}:   M_{\cusp}^{n} \To M_{\cusp}^{n}\ . 
\]

\begin{thm} \label{thm: SimpleMotivicExtension}
    Suppose that $T$ is a $\mathbb{Z}$-linear combination of Hecke correspondences $T_{\alpha}$ such that 
    \begin{equation} \label{motivicTiszero}
        T \in \mathrm{End}(M_{\cusp}^{n}) \hbox{ is the zero map.}
    \end{equation}
    Suppose that $\mathcal{P} = \SL_2(\ZZ) w_1$ and $\mathcal{Q}= \SL_2(\ZZ) z_1$, where $z_1,w_1$ are defined over $K$.  Then there exists a distinguished triangle in $\DM(K)$ 
    \begin{equation} \label{eqn: motivicextensionE}
        \begin{tikzcd}[column sep = small]
            \left(\mathrm{Sym}^{n-1} \,H^1(E_{z_1}) \right)^{\Gamma_{z_1}} \arrow{r}& \mathcal{E} \arrow{r}& \left( \mathrm{Sym}^{n-1} \,\,  H^1(E_{w_1})\right)^{\Gamma_{w_1}} (-1)\arrow{r}{+1} & {}
        \end{tikzcd}
    \end{equation}
     whose realisation, in the case where $w_1,z_1$ are not elliptic, is the extension \eqref{extensionMTzw}. 
\end{thm}

\begin{proof}
    There exists a closed substack $\mathcal{P}'$  of $\mathcal{M}_{1,1}$, given by a finite union of residue gerbes of points, such that the following map is a morphism of motives:
    \[
        T:  M^{n}(\mathcal{P}) \To M^{n}(\mathcal{P}') 
    \]
  We shall write $R_{\mathcal{P}}$  for $\left(\Sym^{n-1} \,H^1(E_{w_1}) \right)^{\Gamma_{w_1}}(-1)$.  Consider the diagram of distinguished triangles 
    \[
        \begin{tikzcd}[column sep = small]
            M^{n}_{\cusp} \arrow{r}\arrow{d} & M^{n}(\mathcal{P}) \arrow{r}\arrow{d}{T} & R_\mathcal{P}  \arrow{r}{+1} & {} \\
            0  \arrow{r} & M^{n}(\mathcal{P}')  \arrow{r} & M^{n}(\mathcal{P}') \arrow{r}{+1}& {}
        \end{tikzcd}
    \]
    where the vertical morphism is induced by the Hecke correspondence $T$, which gives rise to a morphism $R_\mathcal{P}  \rightarrow M^{n}(\mathcal{P}')$.  Now consider the relative cohomology distinguished triangle 
    \[
        \begin{tikzcd}[column sep = small]
            S_\mathcal{Q} \arrow{r} & M^n(\mathcal{P'}, \mathcal{Q}) \arrow{r} & M^n(\mathcal{P}') \arrow{r}{+1} & {} 
        \end{tikzcd}
    \]
    where    $S_{\mathcal{Q}} = \left(\Sym^{n-1} \,H^1(E_{z_1}) \right)^{\Gamma_{z_1}} $ . By composing the morphism $M^{n}(\mathcal{P}') \rightarrow S_Q[1]$ with $R_P \rightarrow M^{n}(\mathcal{P}')$, we deduce the existence of $\mathcal{E}$ and a commutative diagram with two distinguished triangles:  
    \[
        \begin{tikzcd}[column sep = small]
            S_\mathcal{Q} \arrow{r}\arrow[equal]{d} & \mathcal{E} \arrow{r}\arrow{d} & R_\mathcal{P}\arrow{d} \arrow{r}{+1} & {}  \\
            S_\mathcal{Q} \arrow{r} & M^{n}(\mathcal{P}',\mathcal{Q})  \arrow{r} & M^{n}(\mathcal{P}')  \arrow{r}{+1}&{}
        \end{tikzcd}
        \qedhere
    \]
\end{proof}

\begin{cor} \label{cor: MotivicKummerExtension}
    In the setting of Theorem \ref{thm: SimpleMotivicExtension}, suppose that $z_1,w_1$ have complex multiplication defined over $K$. Then the extension \eqref{eq:gz-extension} is the realisation of a Kummer extension in $\DM(K)$.   Consequently,  Theorem  \ref{thm: GZ-conditional-proof} applies and the Gross-Zagier conjecture is true. 
     
\end{cor}

\begin{proof}
    By the remark  which follows Proposition \ref{decompSnintoCn}, the decomposition \eqref{eq:decompSnintoCn} is motivic  and therefore the Kummer extension
    \[
        \begin{tikzcd}[column sep = small]
            \QQ(0) \arrow{r}& \mathcal{E}' \arrow{r}& \QQ(-1) \arrow{r}{+1}&{}
        \end{tikzcd} 
    \]
    may be extracted from \eqref{eqn: motivicextensionE} by 
   pushout and pullback. The conclusion then follows from  Theorem \ref{thm: GZ-conditional-proof}. 
\end{proof}

The previous results essentially prove a weaker version of the Gross-Zagier conjecture under the assumption \eqref{motivicTiszero}. To deduce the full conjecture requires the statement that \eqref{motivicTiszero} follows from \eqref{Tactionzero}. 

\subsubsection{Vanishing of the endomorphism $T$}

If the Betti realisation functor is faithful on the motives of cusp forms, then  \eqref{Tactionzero} indeed implies  \eqref{motivicTiszero}. In a weight $n+1$ in which there are no cusp forms,  this follows from the conservativity conjecture  (cf. \cite[Conjecture 2.12]{AyoubConjectures}) for the restriction of the Betti realisation functor to the subcategory of motives generated by $M^n_{\cusp}.$ 
In this case, the conservativity conjecture is equivalent to showing that the corresponding cusp form motive $M^{n}_{\cusp}$ is zero, which appears to be an open problem.  In Appendix \ref{sect: AppendixD},  we show that this is indeed true in the first  interesting case, which gives a complete proof of the Gross-Zagier conjecture using the geometric techniques of this paper.

\begin{thm}\label{thm: GZtrueinlevel1weight4} 
    The motive $M^{3}_{\cusp}$ is zero. Therefore the Gross-Zagier conjecture is true in level 1 and weight $4$ (\emph{i.e.}, for the function $G_{\SL_2(\mathbb{Z}),2}(z,w)$), for any pair of CM points $z\notin \SL_2(\mathbb{Z}) w$.  
\end{thm}
 
\begin{proof}
    By Proposition \ref{prop: M13isMT}, the motive of $\mathcal{M}_{1,3}$ is mixed Tate and hence so is $M^{3}_{\cusp}$. Since the Betti realisation functor restricted to the triangulated category of mixed Tate motives is conservative, it follows that the cusp form motive $M^{3}_{\cusp}$ is zero since its Betti realisation vanishes. It follows that the hypotheses of Theorem \ref{thm: SimpleMotivicExtension} and Corollary \ref{cor: MotivicKummerExtension} hold, and the statement then follows from Theorem \ref{thm: GZ-conditional-proof}.
\end{proof}

We expect that our methods can be extended to prove a similar theorem for $n\leq 10$ by showing that  $\mathcal{M}_{1,n}$ is mixed Tate in this range.\footnote{A well-known result of Belorousski \cite{belorousski} shows that the coarse moduli spaces $M_{1,n}$ are rational algebraic varieties for $n\le 10$, but this does not imply the mixed Tate property.} The conservativity conjecture is open in general, but is known to hold on the thick rigid tensor triangulated subcategory generated by motives of smooth projective curves; see (\cite{Wildeshaus,Wildeshaus2,Ancona} and \cite[Corollary 2.33]{AyoubConjectures}).

\subsubsection{Higher levels and alternative approaches} \label{sect: HigherLevels}

In the above discussion we considered modular forms of level 1. However, by using the results of Petersen \cite[Theorem 5.1]{Petersen} one may use the moduli spaces of curves with level structure $\mathcal{M}_{1,n}(m)$ to obtain similar results to Theorem  \ref{thm: SimpleMotivicExtension} and Corollary \ref{cor: MotivicKummerExtension} but for the congruence group $\Gamma(m)$. By Remark 6.4 of \emph{loc. cit.}, this can be extended to other congruence subgroups of $\mathrm{SL}_2(\ZZ)$, specifically $\Gamma_1(m)$ and $\Gamma_0(m)$ giving a motivic interpretation of the Gross-Zagier conjecture in an identical manner.  In particular, in every weight, level and nebentypus for which the corresponding space of cusp form vanishes (see, \emph{e.g.}, \cite[Appendix 2]{ZhouI}), one hopes to prove, in the spirit of Theorem \ref{thm: GZtrueinlevel1weight4}, that the corresponding motive vanishes.

\begin{rem}
    We have  constructed our modular biextensions out of the motives of moduli spaces of curves. An alternative approach is to use the motives of Kuga-Sato varieties, following Scholl \cite{Scholl}. 
\end{rem}

\begin{rem}
    Instead of appealing to the faithfulness of the Betti realisation functor, it is enough to show that the cusp form  motives (in suitable weight and level) decompose into finitely many Hecke eigenspaces, which implies that 
    \eqref{motivicTiszero} follows from   \eqref{Tactionzero}. Whilst this can be shown for motives modulo homological equivalence, the general statement follows from the standard conjectures, as pointed out in \cite[Remark 1.2.6]{Scholl} and  \cite[Remark 6.3]{Petersen}.    
\end{rem}

\appendix 

\section{Complements in algebraic de Rham cohomology}\label{sec:algebraic-de-Rham}

For lack of a suitable reference, we collect in this appendix a few basic definitions and results concerning relative algebraic de Rham cohomology with non-trivial coefficients.

\subsection{Relative algebraic de Rham cohomology with coefficients on smooth varieties} \label{sect: A1}

Let $X$ be a smooth algebraic variety over a field $k$ of characteristic zero, $i:Y \hookrightarrow X$ be a smooth closed $k$-subscheme of positive codimension, and $(\mathcal{V},\nabla)$ be a vector bundle with integrable connection on $X$, defined over $k$. 

We denote by $\Omega^{\bullet}_{X/k}(\mathcal{V})$ the $k$-linear complex of  $\mathcal{O}_X$-modules induced by the integrable connection $\nabla$:
\begin{equation}\label{eq: de-rham-complex}
    \begin{tikzcd}[column sep = small]
        0 \rar & \mathcal{V} \rar{\nabla} & \mathcal{V}\otimes_{\mathcal{O}_X}\Omega^1_{X/k} \rar{\nabla^1} & \mathcal{V}\otimes_{\mathcal{O}_X}\Omega^{2}_{X/k} \rar{\nabla^2} & \cdots
    \end{tikzcd}
\end{equation}
and by $\Omega^{\bullet}_{Y/k}(i^*\mathcal{V})$ the analogous complex on $Y$ given by restricting $(\mathcal{V},\nabla)$ via $i$. The pullback of differential forms then induces a $k$-linear morphism of complexes on $X$
\[
    i^*: \Omega^{\bullet}_{X/k}(\mathcal{V}) \To i_*\Omega^{\bullet}_{Y/k}(i^*\mathcal{V})\text{, }\qquad \alpha \longmapsto \alpha|_{i(Y)} \ ,
\]
The mapping cone of $i^*$ is the $k$-linear complex of $\mathcal{O}_X$-modules denoted by
\[
    \Omega^{\bullet}_{(X,Y)/k}(\mathcal{V}) \defeq C(i^*).
\]
Concretely, $\Omega^{n}_{(X,Y)/k}(\mathcal{V}) = \Omega^n_{X/k}(\mathcal{V})\oplus i_*\Omega^{n-1}_{Y/k}(i^*\mathcal{V})$ and the differential is given by
\[
    (\alpha,\beta)\longmapsto (\nabla^n \alpha, \alpha|_{i(Y)} - \nabla^{n-1} \beta)
\]

\begin{defn}
    The \emph{$n$th algebraic de Rham cohomology of $X$ relative to $i: Y\hookrightarrow X$ with coefficients in $\mathcal{V}$} is the $n$th hypercohomology
    \[
        H^n_{\dR}(X,Y;\mathcal{V}) \defeq \mathbb{H}^{n}(X,\Omega^{\bullet}_{(X,Y)/k}(\mathcal{V}))\, .
    \]
\end{defn}

By abuse, we omit in the above notation the base field $k$, the connection $\nabla$, and the closed immersion $i$, as they are often clear from context.

\begin{rem}
    Note that $H^n_{\dR}(X,Y;\mathcal{V})$ is always a finite-dimensional $k$-vector space. When $(\mathcal{V},\nabla)=(\mathcal{O}_X,d)$ is trivial, then $H^n_{\dR}(X,Y;\mathcal{V})$ coincides with the usual relative de Rham cohomology $H^n_{\dR}(X,Y)$. When $Y$ is empty, we recover the `absolute' de Rham cohomology with coefficients $H^n_{\dR}(X;\mathcal{V})$.
\end{rem} 

The relative de Rham cohomology with coefficients is related to the de Rham cohomology with coefficients of $X$ and $Y$ via the long exact sequence of $k$-vector spaces
\begin{equation}\label{eqn:les-relative}
    \begin{tikzcd}[column sep = small]
        \cdots  \arrow{r} & H^{n-1}_{\dR}(Y;i^*\mathcal{V})\arrow{r} & H_{\dR}^n (X,Y;\mathcal{V}) \arrow{r} & H^n_{\dR}(X;\mathcal{V}) \arrow{r} & H^{n}_{\dR}(Y;i^*\mathcal{V})\arrow{r} & \cdots   
    \end{tikzcd}
\end{equation}
induced by the short exact sequence of complexes
\[
    \begin{tikzcd}[column sep = small]
        0 \rar & i_*\Omega^{\bullet}_{Y/k}(i^*\mathcal{V})[-1] \rar & \Omega^{\bullet}_{(X,Y)/k}(\mathcal{V}) \rar & \Omega^{\bullet}_{X/k}(\mathcal{V}) \rar & 0
    \end{tikzcd}.
\]

\subsection{Residue exact sequence} \label{par:appendix-residue}

Let $\iota: D \hookrightarrow X$ be the inclusion of a smooth divisor in $X$ such that $D\cap Y$ is a smooth divisor in $Y$. Recall that there is a complex
\begin{equation}\label{eqn: residue}
    \begin{tikzcd}[column sep = small]
        0 \rar & \Omega^{\bullet}_{X/k} \rar & \Omega^{\bullet}_{X/k}(\log D) \rar{\Res_D} & \iota_*\Omega^{\bullet}_{D/k}[-1] \rar & 0 \ ,
    \end{tikzcd}
\end{equation}
where $\Res_D : \Omega^n_{X/k}(\log D) \to \iota_*\Omega^{n-1}_{D/k}$ is the Poincaré residue map. We denote by $\Omega^{\bullet}_{X/k}(\log D, \mathcal{V})$ the complex analogous to \eqref{eq: de-rham-complex} defined using the complex of logarithmic forms $\Omega^{\bullet}_{X/k}(\log D)$ in place of $\Omega^{\bullet}_{X/k}$. Pullback by $i$ yields a map
\[
    i^*:  \Omega^{\bullet}_{X/k}(\log D, \mathcal{V}) \To i_*\Omega^{\bullet}_{Y/k}(\log D\cap Y,i^*\mathcal{V}) \ ,
\]
the mapping cone of which is denoted by $\Omega^{\bullet}_{(X,Y)/k}(\log D,\mathcal{V})$. We get from \eqref{eqn: residue} an exact sequence
\begin{equation}\label{eqn:residue-relative}
    \begin{tikzcd}[column sep = small]
        0 \rar & \Omega^{\bullet}_{(X,Y)/k}(\mathcal{V}) \rar & \Omega^{\bullet}_{(X,Y)/k}(\log D,\mathcal{V}) \rar{\Res_D} & \iota_*\Omega^{\bullet}_{(D,D\cap Y)/k}(\iota^*\mathcal{V})[-1] \rar & 0 \ .
    \end{tikzcd}
\end{equation}

\begin{prop}\label{prop:deligne}
    The restriction from $X$ to $X\setminus D$ induces an isomorphism of $k$-vector spaces
    \[
        \mathbb{H}^{n}(X, \Omega^{\bullet}_{(X,Y)/k}(\log D,\mathcal{V})) \stackrel{\sim}{\To} H^n_{\dR}(X\setminus D, Y\setminus (D\cap Y); \mathcal{V})\, .   
    \]
\end{prop}

\begin{proof}
    When $Y$ is empty, this is  a special case of Deligne's theorem \cite[II, Corollaire 3.15]{deligneEqDiff}. The general case reduces to this one by considering the hypercohomology long exact  sequence associated to
    \[
        \begin{tikzcd}[column sep = small]
            0 \rar & i_{*}\Omega^{\bullet}_{Y/k}(\log D\cap Y,i^*\mathcal{V})[-1] \rar & \Omega^{\bullet}_{(X,Y)/k}(\log D, \mathcal{V}) \rar & \Omega^{\bullet}_{X/k}(\log D, \mathcal{V}) \rar & 0
        \end{tikzcd}
    \]
    and applying the five lemma.
\end{proof}

It follows from the above that the hypercohomology long exact sequence associated to \eqref{eqn:residue-relative} is 
\begin{equation}\label{eq:residue-exact-seq}
    \begin{tikzcd}[column sep = small]
       \cdots \arrow{r} & H^{n}_{\dR}(X,Y;\mathcal{V}) \arrow{r} & H^{n}_{\dR}(X\setminus D,Y\setminus (D\cap Y);\mathcal{V}) \arrow{r} & H^{n-1}_{\dR}(D,D\cap Y;\iota^*\mathcal{V}) \arrow{r} & \cdots .
    \end{tikzcd}
\end{equation}

\subsection{The case of a smooth affine curve}

We keep the above notation and assume moreover that $X$ is a smooth affine algebraic variety over $k$ of dimension 1. Since $X$ is affine, it follows from Serre's vanishing theorem that 
\[
    H^n_{\dR}(X,Y;\mathcal{V}) = H^n(\Gamma(X, \Omega^{\bullet}_{(X,Y)/k}(\mathcal{V})) .
\]
Since $X$ (resp. $Y$) has dimension $1$ (resp. $0$), the complex $\Gamma(X, \Omega^{\bullet}_{(X,Y)/k}(\mathcal{V}))$ is simply
\[
    \begin{tikzcd}[column sep = small]
        0 \rar & \Gamma(X,\mathcal{V}) \rar{\partial} & \Gamma(X, \Omega^1_{X/k}(\mathcal{V}))\oplus \Gamma(Y,i^*\mathcal{V}) \rar & 0\, ,
    \end{tikzcd}
\]
where $\partial s = (\nabla s, s|_{i(Y)})$, so that
\begin{equation}\label{eqn:cohomology-curve}
    H^n_{\dR}(X,Y;\mathcal{V}) =\begin{cases}
                                    \ker \partial & n=0\\
                                    \coker \partial & n=1 \\
                                    0 & n\ge 2\, .
                                \end{cases}
\end{equation}
Elements of $H^1_{\dR}(X,Y;\mathcal{V})$ are thus given by equivalence classes of pairs $(\alpha,v) \in \Gamma(X,\Omega^1_{X/k}(\mathcal{V}))\oplus \Gamma(Y,i^*\mathcal{V})$, where $(\alpha,v)$ is equivalent to $(\beta, w)$ if and only if there is $s \in \Gamma(X,\mathcal{V})$ such that $(\alpha - \beta, v-w) = (\nabla s, s|_{i(Y)})$.

\begin{rem}
    If $Y$ is non-empty, then it is given by a $\Gal(\overline{k}/k)$-invariant finite subset $\{y_1,\ldots,y_m\}$ of $X(\overline{k})$, where $\overline{k}$ denotes an algebraic closure of $k$.  Note that we can identify $\Gamma(Y,i^*\mathcal{V})$ with a subset of $\bigoplus_{i=1}^m\mathcal{V}(y_i)$, where $\mathcal{V}(y_i) \defeq \mathcal{V}\otimes_{\mathcal{O}_{X,y_i}} k(y_i)$ denotes the fibre of $\mathcal{V}$ at the point $y_i$. Then, the restriction $s|_{i(Y)}$ is identified  with the tuple $( s(y_1),\ldots,s(y_m))$.
\end{rem} 

The long exact sequence in relative cohomology \eqref{eqn:les-relative} boils down to
\begin{equation}\label{eqn:les-relative-curve}
    \begin{tikzcd}[column sep = small]
        0 \arrow{r} & H^0_{\dR}(X,Y;\mathcal{V}) \arrow{r} & H^0_{\dR}(Y;i^*\mathcal{V}) \arrow{r}{\delta} &  H^1_{\dR}(X ,Y;\mathcal{V}) \arrow{r} & H^1_{\dR}(X;\mathcal{V}) \arrow{r} & 0
    \end{tikzcd}
\end{equation}
where $\delta$ sends $v \in H^0_{\dR}(Y;i^*\mathcal{V}) = \Gamma(Y,i^*\mathcal{V})$ to the class of $(0,v)$ in $H^1_{\dR}(X,Y;\mathcal{V})$.

\begin{lem}\label{lem:relative-sequence-curve}
    If $\ker (\nabla : \Gamma(X,\mathcal{V}) \to \Gamma(X, \Omega^1_{X/k}(\mathcal{V})))  = 0$, then
    \[
        \begin{tikzcd}[row sep = -0.1cm, column sep = small]
            0 \arrow{r} & H^0_{\dR}(Y;i^*\mathcal{V}) \arrow{r} & H^1_{\dR}(X,Y;\mathcal{V}) \arrow{r} & H^1_{\dR}(X;\mathcal{V}) \arrow{r} & 0\\
             & v \rar[mapsto] & \left [(0,v)\right ] \\
             & & \left [(\alpha,v)\right ] \rar[mapsto] & \left [\alpha\right ]
        \end{tikzcd}
    \]
    is an exact sequence of $k$-vector spaces.
\end{lem}

\begin{proof}
    By \eqref{eqn:cohomology-curve}, the hypothesis implies that $H^0_{\dR}(X,Y;\mathcal{V}) = \ker \partial \subset \ker \nabla =0$. The result immediately follows from \eqref{eqn:les-relative-curve}.
\end{proof}

\subsection{Hodge filtration}\label{par:appendix-hodge-filtration}

We keep the hypothesis that $X$ is a smooth affine curve over $k$, and we denote by $\overline{X}$ the smooth compactification of $X$ over $k$, with divisor at infinity $Z$, so that $X = \overline{X} \setminus Z$. As before, let $(\mathcal{V},\nabla)$ be a vector bundle with integrable connection on $X$. We further impose the following:
\begin{align}\label{hyp:canonical-extension}
    \begin{split}
        &(\mathcal{V},\nabla) \text{ extends to a vector bundle }\overline{\mathcal{V}} \text{ with a logarithmic connection} \\ & \,\overline{\nabla} : \overline{\mathcal{V}} \To \overline{\mathcal{V}} \otimes\Omega^1_{\overline{X}/k}(\log Z) \text{ whose residue along }Z\text{ is nilpotent}.
    \end{split}
\end{align}
The extension $(\overline{\mathcal{V}},\overline{\nabla})$ is uniquely determined by the above properties. Let
\begin{equation}\label{eq:de-Rham-log-Z}
    \Omega^{\bullet}_{(\overline{X},Y)/k}(\log Z,\overline{\mathcal{V}}) : \qquad \overline{\mathcal{V}} \longrightarrow  (\overline{\mathcal{V}}\otimes_{\mathcal{O}_{\overline{X}}}\Omega^1_{\overline{X}/k}(\log Z))\oplus \overline{i}_*\overline{i}^*\overline{\mathcal{V}}
\end{equation}
be the two-term complex defined as the cone of 
\[
    \overline{i}^*: \Omega^{\bullet}_{\overline{X}/k}(\log Z,\overline{\mathcal{V}}) \longrightarrow \overline{i}_*\Omega^{\bullet}_{Y/k}(\overline{i}^*\overline{\mathcal{V}}),
\]
with $\overline{i}: Y \stackrel{i}{\hookrightarrow} X \hookrightarrow \overline{X}$ the inclusion.

\begin{lem}\label{lem:deligne-theorem-compactifcation}
    The restriction induces a canonical isomorphism
    \[
        \mathbb{H}^1(\overline{X},\Omega^{\bullet}_{(\overline{X},Y)/k}(\log Z,\overline{\mathcal{V}})) \stackrel{\sim}{\longrightarrow} H^1_{\dR}(X,Y;\mathcal{V}).
    \]
\end{lem}

\begin{proof}
    This follows from Deligne's theorem \cite[II, Corollaire 3.15]{deligneEqDiff} (\emph{cf.} Proposition \ref{prop:deligne}).
\end{proof}

Assume that the pair $(\overline{\mathcal{V}},\overline{\nabla})$ is equipped with a filtration by sub-bundles
\begin{equation}\label{eq:hodge-filtration-bundle}
    \overline{\mathcal{V}} = F^0 \overline{\mathcal{V}} \supset F^1\overline{\mathcal{V}}\supset \cdots \supset F^m\overline{\mathcal{V}} \supset F^{m+1}\overline{\mathcal{V}} = 0
\end{equation}
satisfying Griffiths' transversality condition
\[
    \overline{\nabla} (F^p\overline{\mathcal{V}}) \subset F^{p-1}\overline{\mathcal{V}}\otimes_{\mathcal{O}_{\overline{X}}}\Omega^1_{\overline{X}/k}(\log Z), \qquad p=1,\ldots,m.
\]
Then, \eqref{eq:hodge-filtration-bundle} induces a filtration by $k$-linear subspaces 
\[
    H^1_{\dR}(X,Y;\mathcal{V}) = F^0H^1_{\dR}(X,Y;\mathcal{V}) \supset \cdots \supset F^{m+1}H^1_{\dR}(X,Y;\mathcal{V}) \supset F^{m+2}H^1_{\dR}(X,Y;\mathcal{V}) = 0
\]
as follows. Consider the filtration of \eqref{eq:de-Rham-log-Z} by the subcomplexes
\[
    F^p\Omega^{\bullet}_{\overline{X}/k}(\log Z,\overline{\mathcal{V}}) :\qquad  F^p\overline{\mathcal{V}} \longrightarrow ( F^{p-1}\overline{\mathcal{V}} \otimes_{\mathcal{O}_{\overline{X}}}\Omega^1_{\overline{X}/k}(\log Z) ) \oplus \overline{i}_*\overline{i}^*F^p\overline{\mathcal{V}},
\]
and set, under the isomorphism of Lemma \ref{lem:deligne-theorem-compactifcation},
\begin{equation}\label{eq:def-hodge-filtration}
    F^pH^1_{\dR}(X,Y;\mathcal{V}) \coloneqq \im (\mathbb{H}^1(\overline{X},F^p\Omega^{\bullet}_{\overline{X}/k}(\log Z,\overline{\mathcal{V}}) ) \longrightarrow H^1_{\dR}(X,Y;\mathcal{V})).
\end{equation}

\begin{rem}\label{rem:MHS-zucker}
    When $k=\mathbb{C}$ and $(\mathcal{V},\nabla, F^{\bullet}\mathcal{V})$ underlies a variation of Hodge structures, then $H^1_{\dR}(X,Y;\mathcal{V})$ has a natural mixed Hodge structure with Hodge filtration given by \eqref{eq:def-hodge-filtration}, and weight filtration related to cuspidal cohomology  (\emph{cf.} \cite[\S 13]{zucker}, \cite[\S6]{hain}). This mixed Hodge structure is functorial and compatible with the relative cohomology and residue exact sequences \eqref{eqn:les-relative-curve} and \eqref{eq:residue-exact-seq}. 
\end{rem}

\section{Complex conjugation in cohomology with coefficients}\label{sec:complex-conjugation}

The purpose of this appendix is to cast the notion of harmonic lifts of modular forms, studied in Section \ref{sec: HarmonicLiftsGeneralities}, in a more general setting concerning complex conjugation on the  de Rham cohomology with non-trivial coefficients of complex manifolds.

\subsection{Reminders on analytic de Rham cohomology with coefficients} \label{sect: B1}

Let $M$ be a complex manifold. We denote its sheaf of holomorphic functions by $\mathcal{O}_M$ and its complex of sheaves of holomorphic differential forms by $\Omega_M^{\bullet}$.

Let $\mathcal{V}$ be a coherent locally free sheaf of $\mathcal{O}_M$-modules, which corresponds to the sheaf of sections of a holomorphic vector bundle on $M$. From now on, we use these two notions interchangeably. A holomorphic integrable connection $\nabla : \mathcal{V} \to \mathcal{V} \otimes \Omega^1_M$ induces a complex
\[
    \begin{tikzcd}[column sep = small]
        0 \rar & \mathcal{V} \rar{\nabla} & \mathcal{V}\otimes_{\mathcal{O}_M}\Omega^1_{M} \rar{\nabla^1} & \mathcal{V}\otimes_{\mathcal{O}_M}\Omega^{2}_{M} \rar{\nabla^2} & \cdots
    \end{tikzcd}
\]
which we denote by $\Omega^{\bullet}_M(\mathcal{V})$. Recall that the analytic de Rham cohomology with coefficients in $(\mathcal{V},\nabla)$ is the finite-dimensional complex vector space defined by the hypercohomology
\[
    H^n_{\dR}(M;\mathcal{V}) \defeq \mathbb{H}^n(M, \Omega^{\bullet}_M(\mathcal{V})).
\]

\begin{example}\label{ex:cohomology-stein}
    If $M$ is a Stein manifold, then Cartan's theorem B implies that the sheaves $\Omega^p_M(\mathcal{V})$, $p\ge 0$, are $\Gamma(M,-)$-acylic, so that the de Rham cohomology can be computed by global holomorphic forms:
    \[
        H^n_{\dR}(M;\mathcal{V}) \cong H^n(\Gamma(M,\Omega_M^{\bullet}(\mathcal{V}))).
    \]
\end{example}

\begin{rem}\label{rem:algebraic-analytic}
    If $M=X^{\an}$ is the analytification of a smooth algebraic variety $X$ over $\mathbb{C}$, and $(\mathcal{V},\nabla)$ is the analytification of an algebraic vector bundle with integrable connection $(\mathcal{V}^{\alg},\nabla^{\alg})$ on $X$ with regular singularities, then a theorem of Deligne \cite[\S II.6]{deligneEqDiff} shows that there is a canonical isomorphism of complex vector spaces
    \[
        H^n_{\dR}(X;\mathcal{V}^{\alg}) \cong H^n_{\dR}(M;\mathcal{V}),
    \]
    where $H^n_{\dR}(X;\mathcal{V}^{\alg})$ is the algebraic de Rham cohomology of Appendix \ref{sec:algebraic-de-Rham}.
\end{rem}

Consider the sheaf of horizontal sections
\begin{equation} \label{eq: horizsections}
    \mathbb{V} \defeq \ker(\nabla : \mathcal{V} \To \mathcal{V} \otimes \Omega^1_M),
\end{equation}
which is a $\mathbb{C}$-local system such that $(\mathbb{V} \otimes_{\mathbb{C}}\mathcal{O}_M, \id\otimes d) \cong (\mathcal{V},\nabla)$. The holomorphic Poincaré lemma implies that $\Omega_M^{\bullet}(\mathcal{V})$ is a resolution of $\mathbb{V}$, yielding the canonical isomorphism with the sheaf cohomology of $M$ 
\begin{equation}\label{eq:holomorphic-poincare-comparison}
   H^n_{\dR}(M;\mathcal{V}) \cong H^n(M,\mathbb{V}) 
\end{equation}

Finally, we recall that the analytic de Rham cohomology can also be computed via smooth differential forms,.  Let $\mathcal{A}_M^{\bullet}$ be the complex of $C^{\infty}$ complex-valued differential forms on $M$, which contains $\Omega^{\bullet}_M$ as a subcomplex, and let
\[
    \mathcal{A}_M^{\bullet}(\mathcal{V}) = \mathcal{V}\otimes_{\mathcal{O}_M}\mathcal{A}^{\bullet}_M \cong \mathbb{V}\otimes_{\mathbb{C}}\mathcal{A}^{\bullet}_M
\]
be the complex induced by $\nabla$.

\begin{lem}\label{lem:analytic-smooth-cohomology}
    There is a canonical isomorphism
    \[
        H^n_{\dR}(M;\mathcal{V}) \cong H^n(\Gamma(M, \mathcal{A}_M^{\bullet}(\mathcal{V}))).
    \]
\end{lem}

\begin{proof}
    From \eqref{eq:holomorphic-poincare-comparison} and the $C^{\infty}$ Poincaré lemma, we obtain an isomorphism $H^n_{\dR}(M;\mathcal{V}) \cong \mathbb{H}^n(M,\mathcal{A}^{\bullet}_M(\mathcal{V}))$. The statement then comes from the existence of partitions of unity, which shows that, for every $p\ge 0$, the sheaf $\mathcal{A}^p_{M}(\mathcal{V})$ is fine, thus $\Gamma(M,-)$-acyclic.
\end{proof}

\subsection{Complex conjugation} \label{par:complex-conjugation}

Let us assume that the vector bundle with integrable connection $(\mathcal{V},\nabla)$ over $M$ as above is equipped with the following additional structure:
\begin{enumerate}[(i)]
    \item a holomorphic sub-bundle $\mathcal{W} \subset \mathcal{V}$, which is not necessarily horizontal,
    \item a real structure $\mathbb{V}_{\mathbb{R}}$ on the sheaf of horizontal sections $\mathbb{V}$, \emph{i.e.}, $\mathbb{V}_{\mathbb{R}}\subset \mathbb{V}$ is a sub-$\mathbb{R}$-local system satisfying 
    \[
        \mathbb{V}_{\mathbb{R}}\otimes_{\mathbb{R}}\mathbb{C} \cong \mathbb{V}.
    \]
\end{enumerate}

\begin{lem}
    The datum (ii) is also equivalent to:
    \begin{enumerate}
        \item[(ii)'] a complex anti-linear involution
            \[
                c : \mathcal{A}^0_{M}(\mathcal{V}) \To \mathcal{A}^0_{M}(\mathcal{V}),
            \]
        which is horizontal for $\nabla$.
    \end{enumerate}
\end{lem}

\begin{proof}
    We use the identification $\mathcal{A}^0_M(\mathcal{V}) \cong \mathbb{V} \otimes_\mathbb{C} \mathcal{A}_M^0$. If $\mathbb{V}_{\mathbb{R}}$ is a real structure on $\mathbb{V}$, then $c$ is defined by $c(s\otimes \varphi) = s \otimes \overline{\varphi}$ whenever $s$ is a section of $\mathbb{V}_{\mathbb{R}}$.  Conversely, given an involution $c$ as in the statement, the real structure $\mathbb{V}_{\mathbb{R}}$ is given by the sections $s$ of $\mathbb{V}$ such that $c(s\otimes 1) = s\otimes 1$.
\end{proof}

The involution $c$ on $\mathcal{A}^0_M(\mathcal{V})$ extends to $\mathcal{A}^p_M(\mathcal{V}) \cong \mathbb{V}_{\mathbb{R}}\otimes_{\mathbb{R}}\mathcal{A}_M^p$ by
\[
    c(s \otimes \alpha) = s \otimes \overline{\alpha},
\]
where $\alpha \mapsto \overline{\alpha}$ is the usual complex conjugation of $C^{\infty}$ complex-valued differential forms.

\begin{defn} \label{defn: B1HarmonicLift}
    A section $\xi$ of $\mathcal{A}^p_M(\mathcal{V})$ is said to be \emph{harmonic}, relative to the sub-bundle $\mathcal{W}$ and the real structure 
    $\mathbb{V}_{\mathbb{R}}$, if
    \begin{equation}\label{eq:harmonic-lift-geometric}
        \nabla^p\xi = \omega - c\eta
    \end{equation}
    where both $\omega,\eta$ are sections of $\mathcal{W}\otimes\Omega^{p+1}_{M} \subset \mathcal{A}^{p+1}_M(\mathcal{V})$.
\end{defn}

\begin{defn} \label{defn: B.6}
    Given a section $\omega$ of $\mathcal{W}\otimes\Omega^{p+1}_{M}$, a \emph{harmonic lift} of $\omega$ is a section $\xi$ of $\mathcal{A}^p_M(\mathcal{V})$ satisfying \eqref{eq:harmonic-lift-geometric} for some $\eta$ in $\mathcal{W}\otimes\Omega^{p+1}_{M}$; we may also say that $\xi$ is a harmonic lift of the pair $(\omega,\eta)$. When such a harmonic lift exists, we say that $\omega$ is a \emph{Betti-conjugate} of $\eta$.
\end{defn}

Since $c$ is horizontal for $\nabla$, it induces a complex anti-linear involution on cohomology, which we denote by the same symbol:
\begin{equation}\label{eq:involution-abs-cohomology}
    c: H^n_{\dR}(M;\mathcal{V}) \To H^n_{\dR}(M;\mathcal{V}).
\end{equation}
The involution $c$ coincides, under the isomorphisms
\[
    H^n_{\dR}(M;\mathcal{V}) \cong H^n(M,\mathbb{V}) \cong  H^n(M,\mathbb{V}_{\mathbb{R}})\otimes_{\mathbb{R}}\mathbb{C} 
\]
with the complex conjugation induced by the real structure given by the `Betti cohomology' $H^n(M,\mathbb{V}_{\mathbb{R}})$; this justifies the terminology `Betti-conjugate'. 

\begin{prop}\label{prop:geometric-harmoniclift}
    Let $n\ge 1$ be the (complex) dimension of $M$. If the natural map $\Gamma(M,\mathcal{W}\otimes \Omega^{n}_M) \to H^{n}_{\dR}(M;\mathcal{V})$ is surjective, then every section $\omega$ of $\mathcal{W}\otimes \Omega^n_M$ admits a harmonic lift. 
\end{prop}

\begin{proof}
    Note that $\Gamma(M,\mathcal{W}\otimes \Omega^{n}_M) \to H^{n}_{\dR}(M;\mathcal{V})$, $\omega \mapsto [\omega]$, is well-defined since $\Omega^{p}_M = 0$ for $p\ge n+1$. By the surjectivity hypothesis, given $\omega \in\Gamma(M,\mathcal{W}\otimes \Omega^{n}_M)$ there is $\eta\in\Gamma(M,\mathcal{W}\otimes \Omega^{n}_M)$ such that
    \[
        c[\omega] = [\eta]
    \]
    in $H^{n}_{\dR}(M;\mathcal{V})$. Since $c$ is an involution, this is equivalent to $[\omega] = c[\eta]$, which also amounts to $[\omega - c\eta]=0$. By Lemma \ref{lem:analytic-smooth-cohomology}, there exists $\xi$ in $\mathcal{A}^{n-1}_M(\mathcal{V})$ satisfying equation \eqref{eq:harmonic-lift-geometric}.
\end{proof}

\begin{rem}\label{rem:geometric-harmoniclift-algebraic}
    In the algebraic setting of Remark \ref{rem:algebraic-analytic}, we can also ask the sub-bundle $\mathcal{W}$ of $\mathcal{V}$ to be the analytification of an algebraic sub-bundle $\mathcal{W}^{\alg}$ of $\mathcal{V}^{\alg}$. It follows from the same proof above that, if the natural map
    \[
        \Gamma(X,\mathcal{W}^{\alg}\otimes \Omega^n_{X/\mathbb{C}}) \To H^n_{\dR}(X;\mathcal{V}^{\alg}) \cong H^n_{\dR}(M;\mathcal{V})
    \]
    is surjective, then any $\omega \in \Gamma(X,\mathcal{W}^{\alg}\otimes \Omega^n_{X/\mathbb{C}})$ admits a harmonic lift $\xi$ relative to $\mathcal{W}^{\alg}$ and $\mathbb{V}_{\mathbb{R}}$, meaning that $\eta$ in equation \eqref{eq:harmonic-lift-geometric} can also be assumed to lie in $\Gamma(X,\mathcal{W}^{\alg}\otimes \Omega^n_{X/\mathbb{C}})$.
\end{rem}

\subsection{Harmonic lifts and complex conjugation in relative cohomology}

We keep the above notation, and consider a closed embedded complex submanifold $i:N\hookrightarrow M$. The \emph{$n$th relative analytic de Rham cohomology with coefficients in $\mathcal{V}$} is defined as in \S \ref{sect: A1}:
\[
    H^n_{\dR}(M,N;\mathcal{V}) \defeq \mathbb{H}^n(M,\Omega^{\bullet}_{(M,N)}(\mathcal{V})),
\]
where $\Omega^{\bullet}_{(M,N)}(\mathcal{V})$ is the mapping cone of $i^*: \Omega^{\bullet}_M(\mathcal{V}) \to i_*\Omega^{\bullet}_N(i^*\mathcal{V})$. As in the algebraic case, there is a long exact sequence of complex vector spaces
\begin{equation}\label{eqn:les-relative-analytic}
    \begin{tikzcd}[column sep = small]
        \cdots  \arrow{r} & H^{n-1}_{\dR}(N;i^*\mathcal{V})\arrow{r} & H_{\dR}^n (M,N;\mathcal{V}) \arrow{r} & H^n_{\dR}(M;\mathcal{V}) \arrow{r} & H^{n}_{\dR}(N;i^*\mathcal{V})\arrow{r} & \cdots   
    \end{tikzcd}
\end{equation}
induced by the short exact sequence of complexes
\[
    \begin{tikzcd}[column sep = small]
        0 \rar & i_*\Omega^{\bullet}_{N}(i^*\mathcal{V})[-1] \rar & \Omega^{\bullet}_{(M,N)}(\mathcal{V}) \rar & \Omega^{\bullet}_{M}(\mathcal{V}) \rar & 0
    \end{tikzcd}.
\]

\begin{lem}\label{lem:algebraic-analytic-relative}
    If $M=X^{\an}$ and $(\mathcal{V},\nabla)$ is the analytification of $(\mathcal{V}^{\alg},\nabla^{\alg})$ over $X$, with hypotheses as in Remark \ref{rem:algebraic-analytic}, and moreover the submanifold $N\hookrightarrow M$ is the analytification of a smooth closed subvariety $Y \hookrightarrow X$, then there is a canonical isomorphism
    \[
        H^n_{\dR}(X,Y;\mathcal{V}^{\alg}) \cong H^n_{\dR}(M,N;\mathcal{V}).
    \]
\end{lem}

\begin{proof}
    This follows from Deligne's theorem in the absolute case, as explained in Remark \ref{rem:algebraic-analytic}, together with an application of the five lemma to the long exact sequences \eqref{eqn:les-relative} and \eqref{eqn:les-relative-analytic}.
\end{proof}

 If   $M$ is Stein as in Example \ref{ex:cohomology-stein}, then the sheaves $\Omega^p_{(M,N)}(\mathcal{V})$ are acyclic, and the relative analytic de Rham cohomology is computed by global sections: $H^n_{\dR}(M,N;\mathcal{V}) = H^n(\Gamma(M,\Omega^{\bullet}_{(M,N)}(\mathcal{V})))$.

More generally, without any hypothesis on $M$, the relative analytic de Rham cohomology can be computed by global $C^{\infty}$ sections. Let $\mathcal{A}^{\bullet}_{(M,N)}(\mathcal{V})$ denote the cone of $i^*: \mathcal{A}^{\bullet}_M(\mathcal{V}) \to i_*\mathcal{A}^{\bullet}_N(i^*\mathcal{V})$; we may also write
\[
    \mathcal{A}^{\bullet}_{(M,N)}(\mathcal{V}) \cong \mathbb{V} \otimes_{\mathbb{C}}\mathcal{A}_{(M,N)}^{\bullet},
\]
where $\mathbb{V}$ is the local system of horizontal sections of $(\mathcal{V},\nabla)$, as before. Then, there is a canonical isomorphism
\[
    H^n_{\dR}(M,N;\mathcal{V}) \cong H^n(\Gamma(M,\mathcal{A}^{\bullet}_{(M,N)}(\mathcal{V})).
\]
The proof is similar to that of Lemma \ref{lem:analytic-smooth-cohomology}.

Now, assume that $(\mathcal{V},\nabla)$ is equipped with the additional structure (i) and (ii) of \S\ref{par:complex-conjugation}. The complex conjugation $c: \mathcal{A}_M^0(\mathcal{V}) \to \mathcal{A}_M^0(\mathcal{V})$ induces, for all $n$, a $\mathbb{C}$-antilinear involution on relative cohomology
\[
    c: H^n_{\dR}(M,N;\mathcal{V}) \longrightarrow H^n_{\dR}(M,N;\mathcal{V})
\]
which is compatible with the involution on absolute cohomology \eqref{eq:involution-abs-cohomology} by the long exact sequence \eqref{eqn:les-relative-analytic}. The following lemma implies that the map $c$ may be computed in top degree by a harmonic lift.

\begin{lem}\label{lem:value-harmonic-lift}
    Suppose that $M$ has dimension $n$ and let $\omega, \eta \in \Gamma(M, \mathcal{W}\otimes_{\mathcal{O}_M} \Omega^n_{M} )$. They define classes $[(\omega, 0)],[(\eta,0)] \in H^n_{\dR} ( M, N ; \mathcal{V})$. Let $\xi \in \Gamma(M, \mathcal{A}^{n-1}_{M}(\mathcal{V}))$ be a harmonic lift of $(\omega,\eta)$. Then we have 
    \begin{equation}\label{eq:conjugation-harmonic-lift}
        c [(\eta,0) ]  =  [(\omega,0)]  +   [(0,i^*\xi)],
    \end{equation}
    where  $[(0,i^*\xi)]$ is the image of the class $[i^*\xi] \in H^{n-1}_{\dR} (N; i^*\mathcal{V})$  under the boundary map
    \[
        H^{n-1}_{\dR} (N; i^*\mathcal{V})   \To  H^{n}_{\dR} (M, N; \mathcal{V}).
    \]
    Equivalently, one may write $c[(\omega,0)]= [(\eta,0)] - [(0,i^*c\, \xi)]$. 
\end{lem}

\begin{proof}
    Since $\omega,\eta $ are holomorphic and of top degree, they are closed and hence the pairs $(\omega,0)$ and $(\eta,0)$ are closed in the mapping cone $\Omega^{\bullet}_{(M,N)}(\mathcal{V}) = C(i^*)$.  By Definition \ref{defn: B1HarmonicLift},  $\nabla^{n-1} \xi = \omega - c \eta$. We deduce that the following relation 
    \[
        d (\xi,0 ) =  (\nabla^{n-1} \xi,  \xi|_{i(N)} ) = (\omega- c \eta,0) + (0,\xi|_{i(N)} )
    \]
    holds in $\Omega^{\bullet}_{(M,N)}(\mathcal{V})$. Therefore,  in relative analytic de Rham cohomology, we obtain
    \[ 
        c[(\eta,0)] =  [(\omega,0)] + [( 0, i^*\xi ) ].
    \]
    Note that $i^*\xi$ is closed, since $\xi$ is harmonic and $\dim N<n$.  The final equation follows from $c^2=\id.$
\end{proof}

\begin{rem}
    There is also an algebraic version of the above result. In the situation of Lemma \ref{lem:algebraic-analytic-relative} and Remark \ref{rem:geometric-harmoniclift-algebraic}, we may take $\omega,\eta \in \Gamma(X,\mathcal{W}^{\alg}\otimes_{\mathcal{O}_X}\Omega^n_{X/\mathbb{C}})$ to be algebraic Betti-conjugate sections. Then, the equation \eqref{eq:conjugation-harmonic-lift} holds in algebraic de Rham cohomology $H^n_{\dR}(X,Y;\mathcal{V}^{\alg})$. Note that the harmonic lift $\xi$ is only smooth, and  \emph{not} algebraic, but $[i^*\xi] \in H^{n-1}_{\dR}(Y;i^*\mathcal{V}^{\alg})$ does indeed define a class in algebraic de Rham cohomology via Lemma \ref{lem:analytic-smooth-cohomology}.
\end{rem}

\section{Examples of matrix-valued higher Green's functions} \label{sect: AppendixC}
We provide  examples of our $(k+1) \times (k+1)$ matrix-valued Green's functions for $k=1,2$. 
Their existence and properties follow from the results of \S\ref{sect: existenceHigherGreensMatrix}.  When $k=2m$ is even, the  central element of this  matrix is proportional to the `classical' higher Green's function \eqref{eq:classical-GF}, and all other entries are  deduced from it by applying raising and lowering operators. However, when $k$ is odd,  corresponding to  modular forms of odd weight, these matrices seem to be new.    We give here the simplest non-trivial example in each case. 

\subsection{The $2\times 2$ case}

First, recall that for $z,w$ in the upper or lower half planes, 
\[
    \cc =\frac{(z-w)(\overline{z}-\overline{w})}{(z-\overline{z})(w-\overline{w})}\ \qquad \hbox{ and } \qquad 
    \psi^{p,q}(z,w) = \frac{(w- \overline{w})}{ (z-w)^{p+1} (z- \overline{w})^{q+1}   }\ . 
\]
Define the following function  
\[
    f_2(z,w)=  \frac{1}{ (w-\overline{w})(z-\overline{z})  } \log \left( \frac{\cc}{\cc-1} \right)   =   \frac{1}{ (w-\overline{w})(z-\overline{z})  } \log  \frac{(z-w)(\overline{z}-\overline{w})}{(z-\overline{w})(\overline{z}-w)}\ .  
\]
It takes values in $\RR$, which is clear from writing it as a single-valued logarithm: 
\[
    f_2(z,w)= \frac{-1}{4\,\mathrm{Im}(z) \mathrm{Im}(w) }\log \left|  \frac{z-w}{z- \overline{w}} \right|^2 . 
\]
Consider the $2\times 2$ matrix whose entries are  
\[
    \mathsf{g}_1(z,w)= \begin{pmatrix} g^{1,0}_{0,1} & g^{0,1}_{0,1} \\  g^{1,0}_{1,0}&  g^{0,1}_{1,0}  \end{pmatrix}
\] 
where 
\[
    g^{0,1}_{0,1} =\frac{1}{\overline{w}-\overline{z}}       +  (z-w) \, f_2(z,w)    \ ,  \quad    g^{0,1}_{1,0} =\frac{1}{z-\overline{w}}      +    (w- \overline{z}) \, f_2(z,w)
\] 
\[
    g^{1,0}_{0,1}=    \frac{1}{\overline{z}-w}     +  (\overline{w}-z) \, f_2(z,w)   \ ,  \quad  g^{1,0}_{1,0}=   \frac{1}{w-z}     +  (\overline{z}-\overline{w}) \, f_2(z,w)   \ .
\] 
The formula for $g^{1,0}_{1,0}(z,w)$ is given by 
\eqref{gkoformulaaslogarithm} and \eqref{truncatedlogarithm}, since $\mathcal{L}_1(x) = -1 -\log(1-x)/x, $ where $x=\cc^{-1}.$
Then we may check that all the entries of the matrix are compatible with raising and lowering operators:
\[
    \partial_r g^{p,q}_{r,s} = g^{p,q}_{r+1,s-1}\ ,  \quad  \overline{\partial}_s g^{p,q}_{r,s} = g^{p,q}_{r-1,s+1}
\]
\[ 
    \partial^p g^{p,q}_{r,s} = g^{p+1,q-1}_{r,s} \ ,   \quad  \overline{\partial}^q g^{p,q}_{r,s} = g^{p-1,q+1}_{r,s}
\]
where all indices lie in the range of the matrix. Furthermore we verify that
\[ 
    \partial_1 g^{0,1}_{1,0} =  \frac{(z-\overline{z})(w-\overline{w})}{(z-w)(z-\overline{w})^2} \ , \quad 
    \partial_1 g^{1,0}_{1,0} =  \frac{(z-\overline{z})(w-\overline{w})}{(z-w)^2(z-\overline{w})}
\]
hold, as do  similar equations  for $\partial^1$ and their complex conjugates. The situation is  completely summarised by the following diagram:

\[ \renewcommand{\arraystretch}{2.0}
    \begin{array}{c|cc|c}
    \frac{(z- \overline{z})(w-\overline{w})}{(\overline{z}-w)^3}  &  \frac{(z- \overline{z})(w-\overline{w})}{(\overline{z}-w)^2 (\overline{z}-\overline{w})}   &  \frac{(z- \overline{z})(w-\overline{w})}{(\overline{z}-w) (\overline{z}- \overline{w})^2}      &      \frac{(z- \overline{z})(w-\overline{w})}{(\overline{z}-\overline{w})^3 }\\ \hline
    \frac{(z- \overline{z})(w-\overline{w})}{(z-w) (\overline{z}-w)^2}  &  g^{1,0}_{0,1}     &   g^{0,1}_{0,1}    &   
    \frac{(z- \overline{z})(w-\overline{w})}{(z-\overline{w}) (\overline{z}-\overline{w})^2}\\ 
    \frac{(z- \overline{z})(w-\overline{w})}{(z-w) ^2(\overline{z}-w)}   &  g^{1,0}_{1,0}     &    g^{0,1}_{1,0} &   \frac{(z- \overline{z})(w-\overline{w})}{(z -\overline{w})^2(\overline{z}-\overline{w})} \\  \hline
    \frac{(z- \overline{z})(w-\overline{w})}{(z-w)^3}  & \frac{(z- \overline{z})(w-\overline{w})}{(z-w)^2 (z-\overline{w})} &   \frac{(z- \overline{z})(w-\overline{w})}{(z -w)(z-\overline{w})^2} &    \frac{(z- \overline{z})(w-\overline{w})}{(z-\overline{w})^3}   
    \end{array}
\]

For example,  the entries in all four rows satisfy  the following relations with respect to raising and lowering operators in the variable  $w$
\[ 
    (\overline{w}-w) \psi^{1,0}(w,z)  \  \overset{\partial^1}{\longleftarrow}   \  \  g^{1,0}_{1,0}\overset{\partial^0,\overline{\partial}^0} {\longleftrightarrow}   \  g^{0,1}_{1,0}     \   \overset{\overline{\partial}^1}{\To}  \   (\overline{w}-w) \psi^{0,1}(\overline{w},z) 
\]
Every column satisfies a similar relation with respect to $z$.

Let $\Gamma \leq \SL_2(\ZZ)$ be a finite index subgroup. Then the higher Green's matrix is defined by  averaging over the action of $\Gamma$ as follows: 
\[
    \mathcal{G}_{\Gamma,1}(z,w) = \sum_{\gamma \in \Gamma} \begin{pmatrix} j_{\gamma}^{-1}(\overline{z}) & 0 \\ 0 & j_{\gamma}(z)^{-1} \end{pmatrix} \mathsf{g}_1(\gamma z , w) \ .  
\]

\subsection{The $3\times 3$ case}

Consider the  $3\times 3$ array of functions in two variables starting  with the  following function,  which is to be  placed in the middle:
\[
    g^{1,1}_{1,1}(z,w)  = \frac{4\,  Q_1(1- 2\, \cc(z,w)) }{(z-\overline{z})(w-\overline{w})}  \ ,  \ 
\]
where $Q_1(t) =\frac{t}{2} \log (\frac{1+t}{t-1} ) -1 $ is the Legendre function. We define: 
\[ \renewcommand{\arraystretch}{3.0}
    \mathsf{g}_2(z,w)= \left(\begin{array}{c|c|c} 
   g^{2,0}_{0,2}    &  g^{1,1}_{0,2} &     g^{0,2}_{0,2}   \\  \hline
   g^{2,0}_{1,1}    & g^{1,1}_{1,1}  &  g^{0,2}_{1,1}  \\  \hline
 g^{2,0}_{2,0}      &    g^{1,1}_{2,0}    & g^{0,2}_{2,0}     \\ 
 \end{array}\right)  =\left(\begin{array}{c|c|c} 
   \frac{1}{4} \partial^1 \overline{\partial}_1   g^{1,1}_{1,1}       &   \frac{1}{2} \overline{\partial}_1 g^{1,1}_{1,1}    &   \frac{1}{4} \overline{\partial}^1\overline{\partial}_1   g^{1,1}_{1,1}  \\  \hline
 \frac{1}{2} \partial^1 g^{1,1}_{1,1}      & g^{1,1}_{1,1}  &   \frac{1}{2} \overline{\partial}^1 g^{1,1}_{1,1}  \\  \hline
 \frac{1}{4} \partial^1\partial_1   g^{1,1}_{1,1}        &  \frac{1}{2} \partial_1 g^{1,1}_{1,1}     &   \frac{1}{4} \overline{\partial}^1\partial_1   g^{1,1}_{1,1}     \\ 
 \end{array}\right)
\]
which, if one wishes,  can be suitably interpreted in matrix form: 
\[  \begin{pmatrix} \frac{1}{2} \overline{\partial}_1 & & \\   & 1 & \\  & &\frac{1}{2} \partial_1  \end{pmatrix}
\begin{pmatrix}   g^{1,1}_{1,1}  &  g^{1,1}_{1,1}  &  g^{1,1}_{1,1} \\
 g^{1,1}_{1,1}  &  g^{1,1}_{1,1}  &  g^{1,1}_{1,1}  \\
  g^{1,1}_{1,1} &  g^{1,1}_{1,1}  &  g^{1,1}_{1,1} 
\end{pmatrix}  
 \begin{pmatrix} \frac{1}{2} \partial^1 & & \\   & 1 & \\  & &\frac{1}{2} \overline{\partial}^1  \end{pmatrix}
\]
The vertical and horizontal lines are merely there for the sake of legibility. Nevertheless, crossing such a line, in either direction, corresponds to the action of an operator  $\partial_{\bullet}, \overline{\partial}_{\bullet}$ to move  up  and down  (with an appropriate scaling factor), and to the action of  the operators $\partial^{\bullet}, \overline{\partial}^{\bullet}$ to move left and right.
For example:
\[ 
    \partial_r g^{p,q}_{r,s} = (r+1) g^{p,q}_{r+1,s} \ , \quad   \overline{\partial}_s g^{p,q}_{r,s} = (s+1) g^{p,q}_{r-1,s+1}
\]
for all indices in the range above, and similarly for $\partial^p$, $\overline{\partial}^q$.

The bottom-left entry, for example, can be written more simply
\[
    g^{2,0}_{2,0}(z,w) =    \left(  \frac{\overline{z}-\overline{w}}{ (z-\overline{z})(w- \overline{w})}\right)^2 \log    \left|  \frac{z-w}{z-\overline{w}}\right|^2    -  \frac{1}{2}   \frac{2 \cc(z,w)+ 1}{ (z-w)^{2} }\ ,
\]
which is equivalent to \eqref{gkoformulaaslogarithm}, since $\mathcal{L}_2(x) = -\frac{1}{2} - \frac{1}{x} - \log(1-x)/x^2$ in \eqref{truncatedlogarithm}.

If we  apply the raising and lowering operators in $z,w$ again  to take us one further step outside the range of this array, we obtain only rational functions of $z,w$ and their complex conjugates,  as shown:
\[ 
 \renewcommand{\arraystretch}{2.0}
    \begin{array}{c|ccc|c}
   \frac{(z- \overline{z})(w-\overline{w})}{(\overline{z}-w )^4}   &  \frac{(z- \overline{z})(w-\overline{w})}{(\overline{z}-w )^3(\overline{z}-\overline{w} )}  &   \frac{(z- \overline{z})(w-\overline{w})}{(\overline{z}-w )^2(\overline{z}-\overline{w} )^2}  & \frac{(z- \overline{z})(w-\overline{w})}{(\overline{z}-w )(\overline{z}-\overline{w} )^3}   &   \frac{(z- \overline{z})(w-\overline{w})}{(\overline{z}-\overline{w} )^4}  \\ \hline
   \frac{(z- \overline{z})(w-\overline{w})}{(z-w)(\overline{z}-w )^3}   &  g^{2,0}_{0,2}   &  g^{1,1}_{0,2}    &  g^{0,2}_{0,2}  &   \frac{(z- \overline{z})(w-\overline{w})}{(z-\overline{w})(\overline{z}-\overline{w} )^3}  \\ 
  \frac{(z- \overline{z})(w-\overline{w})}{(z-w)^2(\overline{z}-w )^2}  &   g^{2,0}_{1,1}     &  g^{1,1}_{1,1}  &  g^{0,2}_{1,1} &  \frac{(z- \overline{z})(w-\overline{w})}{(z-\overline{w})^2(\overline{z}-\overline{w} )^2}   \\ 
  \frac{(z- \overline{z})(w-\overline{w})}{(z-w)^3(\overline{z}-w )}     &g^{2,0}_{2,0}    &  g^{1,1}_{2,0}   &    g^{0,2}_{2,0}  &  \frac{(z- \overline{z})(w-\overline{w})}{(z-\overline{w})^3(\overline{z}-\overline{w} )}  \\ \hline
   \frac{(z- \overline{z})(w-\overline{w})}{(z-w)^4 } &      \frac{(z- \overline{z})(w-\overline{w})}{(z-w)^3 (z-\overline{w})}    &     \frac{(z- \overline{z})(w-\overline{w})}{(z-w)^2(z-\overline{w})^2}   &    \frac{(z- \overline{z})(w-\overline{w})}{(z-w)(z-\overline{w})^3 }  &  \frac{(z- \overline{z})(w-\overline{w})}{(z-\overline{w})^4 }  
 \end{array}
\] 
 For example, we check that 
\[ \partial_2 g^{p,q}_{2,0}(z,w) =  \frac{(z- \overline{z})(w-\overline{w})}{(z-w)^{p+1} (z-\overline{w})^{q+1}}  = (z- \overline{z}) \psi^{p,q}(z,w) \]
for all $p+q=2$, $p,q\geq  0 .$

Let $\Gamma \leq \SL_2(\ZZ)$ be a finite index subgroup. Then the higher Green's function matrix is defined by  averaging over the action of $\Gamma$ as follows: 
\[ \mathcal{G}_{\Gamma,2}(z,w) = \sum_{\gamma \in \Gamma} \begin{pmatrix} j_{\gamma}^{-2}(\overline{z}) & 0  & 0 \\  0 &  j_{\gamma}(z)^{-1}j_{\gamma}(\overline{z})^{-1}  \\ 0 & 0 & j_{\gamma}(z)^{-2}  \end{pmatrix} \mathsf{g}_2(\gamma z , w) \ .  
\]

\section{Geometry of $\mathcal{M}_{1,3}$} \label{sect: AppendixD}

Let  $\mathcal{M}^{\node}_{1,n}$ denote the moduli stack of geometrically integral genus one curves with at most nodal singularities, and $n$ distinct marked points. It is the union of $\mathcal{M}_{1,n} $ with  the configuration space of $n$ distinct  ordered points on the nodal cubic. We shall give explicit stratifications of $\mathcal{M}^{\node}_{1,n}$ for $n\leq  3$.  

Throughout  this section, we work over the base ring $R = \mathbb{Z}[1/6]$.

\subsection{The moduli stack $\mathcal{M}_{1,1}$ of elliptic curves}\label{par:M_1,1}

Recall that the moduli stack of geometrically integral projective curves of arithmetic genus one and one marked smooth point $(C,p)$ is isomorphic to the quotient stack $[\mathbb{A}^2 / \mathbb{G}_m]$, where $\mathbb{A}^2 = \Spec R[g_2,g_3]$ is equipped with the $\mathbb{G}_m$-action for which $g_2$ and $g_3$ have weights $4$ and $6$, respectively. Explicitly, to a triple $(C,p, v)$, where $v$ is a non-zero tangent vector at $p$, there corresponds a unique Weierstrass equation
\begin{equation}\label{eq:Weierstrass-eq}
   y^2 = 4x^3 -g_2x - g_3
\end{equation}
where $p$ is identified with the point at infinity $(0:1:0)$, and $v$ with the tangent vector dual to $dx/y$. The action of $\mathbb{G}_m = \Spec R[t,t^{-1}]$ comes from the different choices of $v$: the triple $(C,p,tv)$ corresponds to the Weierstrass coefficients $(t^4g_2, t^6g_3)$ via the change of variables $(x,y) \mapsto (t^2x,t^3y)$. 

Let $\Delta = g_2^3 - 27g_3^2$. The projective curve defined by \eqref{eq:Weierstrass-eq} is smooth if and only if $\Delta \neq 0$. It has nodal singularities if $\Delta=0$ and $(g_2,g_3) \neq (0,0)$, and it has cuspidal singularities if $(g_2,g_3) = (0,0)$. Then,
\begin{equation} \label{eq:M_1,1node}
    \mathcal{M}^{\node}_{1,1} =\overline{\mathcal{M}}_{1,1}  \cong    \left[  \frac{\mathbb{A}^2 \setminus  (0,0) }{\mathbb{G}_m}  \right] = \mathcal{P}(4,6)
\end{equation}
and 
\[
    \mathcal{M}_{1,1} \cong   \left[  \frac{\mathbb{A}^2 \setminus V(\Delta)}{\mathbb{G}_m}  \right]  \cong  \mathcal{M}^{\node}_{1,1}  \setminus \infty ,
\]
where 
\begin{equation}\label{eq:cusp_mu2}
    \infty \cong \left[\frac{V(\Delta) \setminus (0,0)}{\mathbb{G}_m}\right] \cong \left[\frac{\Spec R}{\mu_2}\right]
\end{equation}
is the cusp. Here, and in what follows, we denote by $\mu_n=\Spec  R[t]/(t^n-1)$  the affine subgroup scheme of $\mathbb{G}_m$ given by $n$th roots of unity. The isomorphism \eqref{eq:cusp_mu2} may be obtained by the following auxiliary result, which will be used multiple times in the following paragraphs.

\begin{lem}[Slicing lemma]\label{lem:slice}
    Consider $\mathbb{A}^n = \Spec R[x_1,\ldots,x_n]$ with the $\mathbb{G}_m$-action for which $x_1,\ldots,x_n$ have weights $r_1,\ldots,r_n\ge 1$, and let $X\subset \mathbb{A}^n$ be a $\mathbb{G}_m$-invariant $R$-subscheme. Let $f \in R[x_1,\ldots,x_n]\setminus\{0\}$ be a homogeneous polynomial of degree $r$ with respect to the $(r_1,\ldots,r_n)$-grading, and assume that  $X\subset \mathbb{A}^n \setminus V(f)$. Then the natural map
    \[
        \left[\frac{X\cap V(f-1)}{\mu_r}\right] \to \left[\frac{X}{\mathbb{G}_m}\right]
    \]
    is an isomorphism of stacks.
\end{lem}

\begin{proof}
    An explicit inverse for the map in the statement is induced by the morphism $X \to [X\cap V(f-1) / \mu_r]$ corresponding to the $\mu_r$-torsor $P \to X$ defined by the equation $f(x_1,\ldots,x_n) = t^r$ in $X\times \mathbb{A}^1$. It is  equipped with the $\mu_r$-equivariant map $P \to X\cap V(f-1)$  defined on  points by $(x_1,\ldots,x_r,t)\mapsto t^{-1}\cdot (x_1,\ldots,x_r)$. 
\end{proof}

\subsection{The universal punctured elliptic curve $\mathcal{M}_{1,2}$}

It follows from the discussion at the beginning of \S \ref{par:M_1,1} that a quadruple $(C,p_1,p_2,v)$, where $C$ is a (geometrically integral, projective) non-cuspidal genus 1 curve, $p_1,p_2$ are distinct marked points, and $v$ is a tangent vector at $p_1$, corresponds to a unique $(x,y,g_2,g_3)$ such that \eqref{eq:Weierstrass-eq} holds and $(g_2,g_3) \neq (0,0)$. Since \eqref{eq:Weierstrass-eq} may be solved for $g_3$, we can identify:
\begin{equation}\label{eq:M_1,2node}
    \mathcal{M}^{\node}_{1,2} \cong \left[  \frac{ \mathbb{A}^3 \setminus  V(g_2, y^2-4x^3)}{\mathbb{G}_m}  \right]
\end{equation}
where $\mathbb{A}^3 =\Spec  R[x,y,g_2]$ is equipped with the $\mathbb{G}_m$-action for which $x$, $y$, and $g_2$ have weights $2$, $3$, and $4$, respectively.

\begin{lem} \label{lem:M12}
    The stack $\mathcal{M}^{\node}_{1,2}$ admits a stratification
    \[
        \mathcal{M}^{\node}_{1,2} \cong  [\mathbb{A}^2/\mu_4]  \cup [Z/\mu_6]
    \]
    where $\mathbb{A}^2 = \Spec R[x,y]\cong V(g_2-1)$, and $Z= V(g_2,y^2-4x^3-1)$.
    It  may also be identified with 
    \begin{equation} \label{eqn: M12node} 
        \mathcal{M}^{\node}_{1,2} \cong \mathcal{P}(2,3,4) \setminus \Spec R ,
    \end{equation}
    the copy of $\Spec R$ in $\mathcal{P}(2,3,4)$ being given by the quotient $[V(g_2,y^2-4x^3)/\mathbb{G}_m]$.  
\end{lem}

\begin{proof} 
    The scheme $\mathbb{A}^3 \setminus  V(g_2, y^2-4x^3)$ is stratified by the open subscheme $\mathbb{A}^3 \setminus  V(g_2)$ and its complement $V(g_2) \setminus V(g_2, y^2-4x^3)$. Since these subschemes are $\mathbb{G}_m$-invariant, we obtain from \eqref{eq:M_1,2node}:
    \[
        \mathcal{M}^{\node}_{1,2} \cong  \left[ \frac{\Spec R[x,y,g_2, g_2^{-1}]}{\mathbb{G}_m} \right]  \cup    \left[ \frac{\Spec R[x,y, (y^2-4x^3)^{-1}]}{\mathbb{G}_m} \right] . 
    \] 
    Since $g_2$ is homogeneous of degree 4, it follows from Lemma \ref{lem:slice} that
    \[
        \left[ \frac{\Spec R[x,y,g_2, g_2^{-1}]}{\mathbb{G}_m} \right]  \cong  \left[ \frac{\Spec R[x,y]}{\mu_4} \right] .
    \]
    Likewise, as $y^2-4x^3$ is homogeneous of degree 6, Lemma \ref{lem:slice} yields
    \[
        \left[ \frac{\Spec R[x,y, (y^2-4x^3)^{-1}]}{\mathbb{G}_m} \right] \cong   \left[ \frac{V(y^2-4x^3-1)}{\mu_6} \right].
    \] 

    For the second part, we stratify differently and write 
    \[
        \mathcal{M}^{\node}_{1,2} = \left[  \frac{ \mathbb{A}^3 \setminus (0,0,0) }{\mathbb{G}_m}  \right] \setminus   
        \left[  \frac{V(g_2, y^2-4x^3) \setminus (0,0,0) }{\mathbb{G}_m}  \right] . 
    \]
    The stratum on the left is  $\mathcal{P}(2,3,4)$. To compute the one on the right, normalise the plane affine cuspidal cubic $y^2 =4x^3$ by $\mathbb{A}^1 = \Spec R[z]$ via $z \mapsto (z^2,2z^3)$. The normalisation is $\mathbb{G}_m$-equivariant (where $z$ is taken with weight one) and induces an isomorphism outside the origin, thus:
    \[
        \left[ \frac{ V(g_2,y^2-4x^3) \setminus (0,0,0) }{ \mathbb{G}_m} \right]    \cong     \left[ \frac{\Spec R [z,z^{-1}]  }{ \mathbb{G}_m} \right] = \Spec R.\qedhere
    \]
\end{proof}

\begin{rem}\label{rem:mixed-tate-elliptic-quotient}
    If $Z$ is as in the above statement, one  checks  that 
    \[
        [Z/\mu_6] =   \mathcal{P}(2,3) \setminus \Spec R 
    \]
    where $\Spec R$ corresponds to the $\mathbb{G}_m$-orbit given by the punctured cuspidal cubic, as in the above proof. Note that, although the motive of $Z$ is not mixed Tate (since $Z$ is a  punctured  elliptic curve), its invariant submotive  for the action of $\mu_6$, is. 
\end{rem}

\begin{rem}
    We refer to \cite{Inchiostro} for a description of $\overline{\mathcal{M}}_{1,2}$ as a weighted blow-up of $\mathcal{P}(2,3,4).$ 
\end{rem}

\begin{cor} \label{cor: M12nodeminusVy}
    Under the isomorphism \eqref{eq:M_1,2node}, we have
    \[
        \mathcal{M}_{1,2}^{\node} \setminus [V(y)/\mathbb{G}_m] \cong   (\mathcal{P}(2,3,4) \setminus  \mathrm{Spec}(R)) \setminus \mathcal{P}(2,4).
    \]
\end{cor}

\begin{proof}
    By the discussion preceding  Lemma  \ref{lem:M12}, we can identify
    \[
        \mathcal{M}^{\node}_{1,2} \cap \left[\frac{V(y)}{\mathbb{G}_m}\right] \cong \left[\frac{\Spec R[x,g_2] \setminus V(x,g_2) }{\mathbb{G}_m}\right] \cong \mathcal{P}(2,4).\qedhere
    \]
\end{proof}

\subsection{Three marked points $\mathcal{M}_{1,3}$}

The data of three distinct marked points on a non-cuspidal curve of genus one is encoded by  two distinct pairs of points $(x_1,y_1)$, $(x_2,y_2)$, and $(g_2, g_3)\neq (0,0)$,  subject to the equations \eqref{eq:Weierstrass-eq}. 

\begin{lem} \label{lem: M13node}
    Let $\mathcal{U}_3 \subset \mathcal{M}^{\node}_{1,3}$ denote the open substack  where $x_1 \neq x_2$, and let $\mathcal{Z}_3$ denote its complement, so that $\mathcal{M}_{1,3}^{\node} = \mathcal{U}_3\cup \mathcal{Z}_3$. Then
    \[
        \mathcal{Z}_3 \cong \mathcal{P}(2,3,4) \setminus  (\mathrm{Spec} (R)  \cup  \mathcal{P}(2,4)), \quad  \mathcal{U}_3 \cong  \left[ \mathbb{A}^3/\mu_2 \right]  \setminus W,
    \]
    where $W$ is an $R$-scheme with a stratification of the form
    \[
        W \cong  (\mathbb{A}^1\setminus \{-1,0,1\}) \cup (\Spec R \cup \Spec R). 
    \]
The space  $W$ is the moduli space of  two points with distinct $x$ coordinate on a  cuspidal cubic; the points $\Spec R$ correspond to the locus where one or other of these point is the cusp. 
\end{lem}

\begin{proof}
    The equations \eqref{eq:Weierstrass-eq} for $(x_i,y_i)$ can be written in matrix form: 
    \[
        \begin{pmatrix}  x_1 & 1 \\
                    x_2 & 1 
        \end{pmatrix} \begin{pmatrix} g_2 \\ g_3 \end{pmatrix} = \begin{pmatrix}   -y_1^2+4x_1^3 \\
                     -y_2^2+4x_2^3
        \end{pmatrix}  . 
    \]
    When $x_1\neq x_2$, they may be uniquely solved for $(g_2,g_3)$, and hence
    \[ 
        \mathcal{U}_3 = \left[ \frac{ \Spec R [x_1,y_1,x_2,y_2, (x_1-x_2)^{-1}] \setminus V(y^2_1-4x^3_1, y^2_2-4x^3_2) }{\mathbb{G}_m} \right] .    
    \]
    Since $x_1,x_2$ have weight $2$ for the action of $\mathbb{G}_m$,  slicing by $x_1-x_2 = 1$ (Lemma \ref{lem:slice}) yields
    \[   
        \left[ \frac{ \Spec R [x_1,y_1,x_2,y_2, (x_1-x_2)^{-1}] }{\mathbb{G}_m} \right]  \cong  \left[ \frac{\Spec R [x_1,y_1,y_2]}{\mu_2} \right] \cong  \left[ \frac{\mathbb{A}^3 }{\mu_2} \right] . 
    \]
    By normalising the cuspidal cubics $y_i^2=4x_i^3$ and deleting $V(x_1x_2)$,  we have furthermore 
    \[
        \left[ \frac{V(y^2_1-4x^3_1, y^2_2-4x^3_2) \setminus V(x_1x_2(x_1-x_2))}{\mathbb{G}_m} \right]   \cong  \left[ \frac{\Spec R[t_1,t_2,t_1^{-1},t_2^{-1}] \setminus V(t_1^2-t_2^2)}{\mathbb{G}_m} \right]   \cong   \mathbb{A}^1\setminus \{1,0,-1\},
    \] 
    where the last isomorphism is obtained by slicing the quotient by $t_2=1$ (Lemma \ref{lem:slice}). The stratum where $x_1=0$ (and similarly $x_2=0$) is isomorphic to a point (again by slicing):
    \[
        \left[ \frac{  V(y^2_2-4x^3_2) \setminus V(x_2)   }{\mathbb{G}_m}\right] \cong \Spec R.
    \]

    The locus $\mathcal{Z}_3$ is defined by $x_1=x_2$. Then, equation \eqref{eq:Weierstrass-eq} implies that $y_1^2 = y_2^2$, and hence the condition $(x_1,y_1) \neq (x_2,y_2)$ amounts to $y_2=-y_1$. In particular, $y_1,y_2\neq 0$. The data of the pair of points $(x_1,y_1)$, $(x_2,y_2)$ is uniquely determined by $(x_1,y_1)$, and hence
    \[
        \mathcal{Z}_3 \cong \mathcal{M}^{\node}_{1,2} \setminus \left[ \frac{V(y)}{\mathbb{G}_m} \right] \cong \mathcal{P}(2,3,4) \setminus (\mathrm{Spec}(R) \cup \mathcal{P}(2,4)),
    \]
    where the last isomorphism is given by Corollary \ref{cor: M12nodeminusVy}.
\end{proof}

\subsection{Mixed Tate motives in genus $1$}

To show that the motive of $\mathcal{M}_{1,n}$  over $\mathbb{Q}$ is mixed Tate, it suffices to find a stratification whose strata are mixed Tate. 

\begin{prop} \label{prop: M13isMT}
    The motives $M(\overline{\mathcal{M}}_{1,n})$ and  $M(\mathcal{M}_{1,n})$ are mixed Tate for $n\leq 3$.
\end{prop}

\begin{proof}
    The compactification $\overline{\mathcal{M}}_{1,n}$ is stratified by finite quotients of products of the moduli stacks $\mathcal{M}_{0,r}$ and $\mathcal{M}_{1,s}$, for $3\le r \le n+2$ and $1\leq s\leq n$. Since the motives $M(\mathcal{M}_{0,r})$ are mixed Tate ($\mathcal{M}_{0,r}$ is a complement of a hyperplane arrangement in a projective space), it follows that $M(\overline{\mathcal{M}}_{1,s})$  is mixed Tate for all $1\leq s \leq n$ if and only if the same is true for $ M(\mathcal{M}_{1,s})$ for all $1\leq s \leq n$. Since $\mathcal{M}_{1,s}^{\node}$ is the complement in  $\overline{\mathcal{M}}_{1,s}$ of a finite number of strata of the above form, this is in turn equivalent to  $M(\mathcal{M}_{1,s}^{\node})$ being mixed Tate for all $1\leq s \leq n$. But this holds for $n=3$ by \eqref{eq:M_1,1node} and the stratifications given in Lemmas \ref{lem:M12} and \ref{lem: M13node}.
\end{proof}

One expects that $M(\overline{\mathcal{M}}_{1,n})$ and $M(\mathcal{M}_{1,n})$ are mixed Tate for all $n\leq 10$. They are not for $n\geq 11$.

\bibliographystyle{alpha}

\bibliography{biblio}

\end{document}